\documentclass[11pt]{article}

\newcommand{\Old}[1]{}

\newcommand{\New}[1]{{\color{purple}#1}}
\newcommand{\nickc}[1]{{\color{red} {\bf [} {\color{red} COMMENT Nick  Jan 11: } #1{\bf ]}}}
\newcommand{\nickk}[1]{{\color{blue} #1}}
\newcommand{\nickca}[1]{{\color{purple} {\bf [} {\color{blue} COMMENT Nick  Jan 11: } #1{\bf ]}}}
\newcommand{\nickka}[1]{{\color{purple} #1}}
\newcommand{\nickcb}[1]{{\color{red} {\bf [} {\color{blue} COMMENT Nick  Jan 11: } #1{\bf ]}}}
\newcommand{\nickkb}[1]{{\color{blue} #1}}
\newcommand{\crisc}[1]{{\color{orange} {\bf [}{\color{orange} COMMENT Cris  Jan 11: } #1{\bf ]}}}
\newcommand{\textred}[1]{{\color{red}#1}}

\if00 %%% change to 00 to remove all comments and make all new text permanent.
\renewcommand{\nickc}[1]{{}}
\renewcommand{\nickk}[1]{{#1}}
\renewcommand{\nickca}[1]{{}}
\renewcommand{\nickka}[1]{{#1}}
\renewcommand{\nickcb}[1]{{}}
\renewcommand{\nickkb}[1]{{#1}}
\renewcommand{\crisc}[1]{}
\renewcommand{\textred}[1]{{#1}}
\renewcommand{\New}[1]{{#1}}
\fi

\usepackage[english]{babel}
\usepackage{nicefrac}
\usepackage{verbatim}
\usepackage{amsfonts,amsmath,amssymb,amsthm}
\usepackage{url}
\usepackage{mathrsfs}
\usepackage{bbm}
\usepackage{enumitem}

\usepackage{textcomp}
\usepackage{ifthen}
\usepackage{ifpdf}
\usepackage{aeguill}

\usepackage{tikz,subfigure}
\usetikzlibrary{snakes,calc}

\usepackage[top = 1in, bottom = 1in, left = .75in, right =
.75in]{geometry}

\theoremstyle{definition}
\newtheorem{thm}{Theorem}[section]
\newtheorem{lem}[thm]{Lemma}
\newtheorem{prop}[thm]{Proposition}
\newtheorem{cor}[thm]{Corollary}

\newcommand{\textdef}[1]{\textit{#1}}

\newcommand{\setR}{\ensuremath{\mathbb{R}}}

\newcommand{\setZ}{\ensuremath{\mathbb{Z}}}
\newcommand{\setN}{\ensuremath{\mathbb{N}}}

\newcommand{\set}[2][]{#1\{ {#2} #1\}}

\newcommand{\abs}[2][]{#1| #2 #1|}
\newcommand{\ceil}[2][]{#1\lceil #2 #1\rceil}
\newcommand{\floor}[2][]{#1\lfloor #2 #1\rfloor}
\newcommand{\paren}[2][]{#1( #2 #1)}

\newcommand{\aas}{a.a.s.}
\newcommand{\uar}{u.a.r.}

\newcommand{\eps}{\varepsilon}
\newcommand{\eqdef}{:=}
\newcommand{\eqdefinv}{=:}

\newcommand{\ds}{\ensuremath{\mathbf{d}}}
\newcommand{\ts}{\ensuremath{\mathbf{t}}}
\newcommand{\twos}{\ensuremath{\mathbf{2}}}
\newcommand{\threes}{\ensuremath{\mathbf{3}}}
\newcommand{\Ys}{\ensuremath{\mathbf{Y}}}

\newcommand{\Dcal}{\ensuremath{\mathcal{D}}}
\newcommand{\Ecal}{\ensuremath{\mathcal{E}}}

\newcommand{\Gcal}{\ensuremath{\mathcal{G}}}

\newcommand{\Jcal}{\ensuremath{\mathcal{J}}}
\newcommand{\Kcal}{\ensuremath{\mathcal{K}}}

\newcommand{\Pcal}{\ensuremath{\mathcal{P}}}

\newcommand{\Scal}{\ensuremath{\mathcal{S}}}

\newcommand{\prekernel}{pre-kernel}
\newcommand{\Prekernel}{Pre-kernel}
\newcommand{\prekernels}{pre-kernels}

\DeclareMathOperator{\funcprob}{\mathbb{P}}
\newcommand{\prob}[2][]{\funcprob #1(#2#1)}

\newcommand{\probcond}[3][]{\funcprob#1(#2\,#1|\,#3#1)}

\newcommand{\mean}[2][]{\mathbb{E}\,#1(#2#1)}
\newcommand{\meancond}[3][]{\mathbb{E}\,#1(#2 #1| #3 #1)}

\newcommand{\var}[2][]{\mathop{\rm Var}#1(#2#1)}

\newcommand{\lambdakc}[1][]{\lambda(k,{\ifthenelse{\equal{#1}{}}{c}{#1}})}
\newcommand{\lambdakcparam}[3][]{\lambda #1(#2,#3 #1)}
\newcommand{\fpo}[1]{f_{#1}}
\newcommand{\etac}[1][]{\eta_c}

\newcommand{\tpo}[3][]{\mathop{\rm Po}#1(#2, #3 #1)}
\newcommand{\tpoisson}[3][]{\mathop{\rm Po}#1(#2, #3 #1)}

\newcommand{\Bollobas}{Bollob\'as}

\newcommand{\Luczak}{\L uczak}

\newcommand{\Wormald}{Wormald}

\newcommand{\tD}{\ensuremath{\tilde\Dcal}}

\newcommand{\markwalk}[1][]{
  Z(k,{\ifthenelse{\equal{#1}{}}{c}{#1}})
}

\newcommand{\markstepwalk}[1][]{
  Z_{\ifthenelse{\equal{#1}{}}{j}{#1}}
}
\newcommand{\totalmarkstepwalk}[1][]{
  Y_{\ifthenelse{\equal{#1}{}}{j}{#1}}
}

\newcommand{\markstepwalkplus}[1][]{
  Z_{\ifthenelse{\equal{#1}{}}{j}{#1}}^{+}
}
\newcommand{\totalmarkstepwalkplus}[1][]{
  Y_{\ifthenelse{\equal{#1}{}}{j}{#1}}^{+}
}

\newcommand{\markstepwalkminus}[1][]{
  Z_{\ifthenelse{\equal{#1}{}}{j}{#1}}^{-}
}
\newcommand{\totalmarkstepwalkminus}[1][]{
  Y_{\ifthenelse{\equal{#1}{}}{j}{#1}}^{-}
}

\newcommand{\markstepdel}[1][]{
  Z_{\ifthenelse{\equal{#1}{}}{j}{#1}}
}
\newcommand{\totalmarkstepdel}[1][]{
  Y_{\ifthenelse{\equal{#1}{}}{j}{#1}}
}

\newcommand{\probsizekc}[1][]{
  q(k,
  {\ifthenelse{\equal{#1}{}}{c}{#1}})
}

\newcommand{\probsizekstep}[1][]{
  p_{\ifthenelse{\equal{#1}{}}{j}{#1}}
}

\newcommand{\probkillstep}[1][]{
  p_{\ifthenelse{\equal{#1}{}}{j}{#1}}'
}

\DeclareMathOperator{\multisub}{multi}
\newcommand{\multik}{G_{k}^{\multisub}}
\newcommand{\multiknm}[1][]{\multik(n,m)}
\newcommand{\simplek}{G_k}
\newcommand{\simpleknm}[1][]{\simplek(n,m)}
\newcommand{\cNM}{C(N,M)}

\newcommand{\mcacti}{{m'}}

\newcommand{\gcore}{g_{\textrm{core}}}
\newcommand{\wcore}{w_{\textrm{core}}}
\newcommand{\wpre}{w_{\textrm{pre}}}
\newcommand{\gcacti}{g_{\textrm{forest}}}
\newcommand{\gpre}{g_{\textrm{pre}}}

\newcommand{\nN}{\check n}
\newcommand{\nNopt}{\check n^*}
\newcommand{\nopt}{n^*}
\newcommand{\RN}{\check R}

\newcommand{\f}{f_1}
\newcommand{\ff}{f_2}
\newcommand{\fff}{f_3}

\newcommand{\FF}{\ff}
\newcommand{\fg}{g_1}
\newcommand{\fgg}{g_2}

\newcommand{\nn}{n_1}
\newcommand{\nncore}{\hat{n}_1}

\newcommand{\nncopt}{\hat{n}_1^*}
\newcommand{\nnopt}{n_1^*}
\newcommand{\lambdaopt}{\lambda^*}
\newcommand{\lambdaoptopt}{\lambda^{**}}
\newcommand{\mcore}{\hat{m}}
\newcommand{\hcore}{h_{n}}
\newcommand{\fcore}{f_{\textrm{core}}}

\newcommand{\fpre}{f_{\textrm{pre}}}
\newcommand{\hpre}{\hcore}

\newcommand{\Gcore}{\Gcal}

\newcommand{\nthree}{n_3}
\newcommand{\mthree}{m_3}
\newcommand{\coremthree}{\hat{m}_3}
\newcommand{\Qtwo}{Q_2}
\newcommand{\coreQtwo}{\hat{Q}_2}
\newcommand{\ntwo}{n_2}
\newcommand{\corentwo}{\hat{n}_2}
\newcommand{\ctwo}{c_2}
\newcommand{\etatwo}{\eta_2}
\newcommand{\corectwo}{{c}_2}
\newcommand{\coreetatwo}{{\eta}_2}
\newcommand{\cthree}{c_3}
\newcommand{\precthree}{\hat{c}_3}
\newcommand{\etathree}{\eta_3}
\newcommand{\mtwo}{m_2}
\newcommand{\Mtwo}{M_2}
\newcommand{\Rtwo}{R_2}
\newcommand{\Rthree}{R_3}

\newcommand{\indicator}{\mathbbm{1}}

\DeclareMathOperator{\dif}{d}

\newcommand{\ignore}[1]{}

\newcommand{\prenn}{\hat{n}_1}
\newcommand{\prennopt}{\hat{n}_1^*}

\newcommand{\prenthree}{\hat{n}_3}
\newcommand{\prenthreeopt}{\hat{n}_3^*}

\newcommand{\kzero}{k_0}
\newcommand{\prekzero}{\hat{k}_0}
\newcommand{\prekzeroopt}{\hat{k}_0^*}

\newcommand{\kone}{k_1}
\newcommand{\prekone}{\hat{k}_1}
\newcommand{\prekoneopt}{\hat{k}_1^*}

\newcommand{\ktwo}{k_2}
\newcommand{\prektwo}{\hat{k}_2}
\newcommand{\prektwoopt}{\hat{k}_2^*}

\newcommand{\ntwoequal}{n_{2}}

\newcommand{\Pthree}{P_3}
\newcommand{\prePthree}{\hat{P}_3}
\newcommand{\prePthreeopt}{\hat{P}_3^*}

\newcommand{\Ptwo}{P_2}
\newcommand{\prePtwo}{\hat{P}_2}
\newcommand{\prePtwoopt}{\hat{P}_2^*}
 
\newcommand{\Qthree}{Q_3}
\newcommand{\preQthree}{\hat{Q}_3}
\newcommand{\preQthreeopt}{\hat{Q}_3^*}

\newcommand{\Tthree}{T_3}
\newcommand{\preTthree}{\hat{T}_3}
\newcommand{\preTthreeopt}{\hat{T}_3^*}

\newcommand{\Ttwo}{T_2}
\newcommand{\preTtwo}{\hat{T}_2}
\newcommand{\preTtwoopt}{\hat{T}_2^*}

\newcommand{\premtwo}{\hat{m}_2}
\newcommand{\premtwoopt}{\hat{m}_2^*}

\newcommand{\premthree}{\hat{m}_3}
\newcommand{\premthreeopt}{\hat{m}_3^*}

\newcommand{\mtwoprime}{m_2^{-}}
\newcommand{\premtwoprime}{\hat{m}_2^{-}}
\newcommand{\premtwoprimeopt}{\hat{m}_2^{-*}}

\newcommand{\mpre}{\hat{m}}
\newcommand{\prex}{\hat{x}}
\newcommand{\corex}{\prex}
\newcommand{\preS}{\hat{S}}
\newcommand{\coreJ}{\hat{J}}

\newcommand{\preB}{\hat{B}}
\newcommand{\preC}{\hat{C}}
\newcommand{\xopt}{{x}^*}

\newcommand{\prexopt}{\hat{x}^*}

\newcommand{\Gpre}{\Pcal}
\title{Asymptotic enumeration of sparse connected 3-uniform
  hypergraphs}

\author{C.~M.~Sato\\University of Waterloo\\\texttt{cmsato@uwaterloo.ca}
\and N.~Wormald\thanks{This research was supported in part by the Canada Research Chairs Program,  and currently by an ARC Australian Laureate Fellowship.}
\\Monash University\\\texttt{nicholas.wormald@monash.edu}}
\date{}
\begin{document}

\maketitle
\begin{abstract}
\nickkb{We derive an asymptotic formula for the number of connected $3$-uniform
hypergraphs with vertex set $[N]$ and $M$ edges for $M=N/2+R$ as long as
$R$ satisfies $R = o(N)$ and $R=\omega(N^{1/3}\ln^{2} N)$. This almost completely fills the gap in the range of $M$ for which the formula is known. We approach the problem using an `inside-out' approach of an earlier paper of Pittel and the second author, for connected graphs. A key part of the method uses structural components of connected hypergraphs called   cores and kernels.  These are structural components of connected hypergraphs. Our results also give information on the numbers of them with a given number of vertices and edges, and hence their  typical size in random connected $3$-uniform hypergraphs with $N$ vertices and $M$ edges, for the range of $M$ we consider. }
\end{abstract}
\section{Introduction}
 The problem of counting connected graphs with given number of vertices
and edges has been intensively studied throughout the years. \nickk{One} of the best results is an
asymptotic formula by Bender, Canfield and McKay~\cite{BCM} that works
when \nickk{the {\em excess}} $m-n\to \infty$ as $n\to\infty$, where $m=m(n)$ is the number of
edges and $n$ is the number of vertices. Pittel and
\Wormald~\cite{PWb} rederived this formula with improved
 error bounds for some ranges.  \nickk{Significantly} less is known about
connected hypergraphs. Karo{\'n}ski and \Luczak~\cite{KL} derived an
asymptotic formula for the number of connected $k$-uniform \nickkb{hyper}graphs on
$[N]$ with $M$ hyperedges for  \nickk{the range of small excess where}  $M = N/(k-1)+o(\ln N/\ln \ln
N)$, which is a range with small excess. This was later extended by
Andriamampianina and Ravelomanana~\cite{AR} for $M=N/(k-1) +
o(N^{1/3})$, which still has very small excess. \nickk{Regarding results for  denser \nickkb{hyper}graphs},
Behrisch, Coja-Oghlan and Kang~\cite{BCK} provided an asymptotic
formula for the case $M = N/(k-1) + \Theta(N)$. Thus, there is a gap
between the  case $M-N/(k-1) = o(N^{1/3})$ and the linear
case $M-N/(k-1) = \Omega(N)$ in which no asymptotic formulae were
found. The case $M-N/(k-1) = \omega(N)$ is also open.

In this \nickk{paper, we  obtain} an asymptotic formula for the number of connected $3$-uniform
\nickkb{hyper}graphs with vertex set $[N]$ and $M$ edges for $M=N/2+R$ as long as
$R$ satisfies $R = o(N)$ and $R=\omega(N^{1/3}\ln^{2} N)$. This leaves \nickk{only a tiny remaining  gap, between}   $M-N/2 = o(N^{1/3})$
and $M-N/2 = \omega(N^{1/3}\ln^{2} N)$. Our technique is based \nickk{on an}
approach that Pittel and Wormald~\cite{PWb} used to the enumerate
connected graphs. With this technique, we also obtain information on the sizes of a kind of core and kernel in the graphs being counted. \nickk{We restrict ourselves to the 3-uniform case because the complexity of this approach increases for the more general case.} The results in this \nickk{paper are contained in the PhD thesis~\cite{Sato13} of the first author.}

% Giant
Behrisch, Coja-Oghlan and Kang~\cite{BCK} obtained their enumeration
result by precisely estimating the joint distribution of the number of
vertices and the number of edges in the giant component of the random
hypergraph. We remark that the distribution of the number of vertices
and edges has \nickk{recently been described independently} by \Bollobas\ and
Riordan~\cite{BR12a}, but that result does not provide point
probabilities, which would allow the enumeration result to be deduced.

\nickkb{From our results, it will be possible to determine the joint distribution of the giant component's excess, core size and kernel size in the random $k$-uniform hypergraph  with $n$ vertices and $m$ random edges,  for the range of density that we have covered in this paper. This should be straightforward along the lines of the analogous argument for the case $k=2$ in~\cite{PWb}. }

% Our result 

\section{Main result}

A \textdef{hypergraph}\index{hypergraph} is a pair $(V,\Ecal)$, where
$V$ is a finite set and $\Ecal$ is a subset of nonempty sets in $2^V$,
which is the set of all subsets of~$V$. The elements in~$V$ are called
\textdef{vertices} and the elements in $\Ecal$ are called
\textdef{hyperedges}\index{hypergraph!hyperedges}.  For any integer
$k\geq 2$, a \textdef{$k$-uniform
  hypergraph}\index{hypergraph!kuniform@$k$-uniform} is a hypergraph
where each hyperedge has size~$k$. For any hypergraph~$G$, a
\textdef{path}\index{hypergraph!path} is a (finite) sequence $v_1E_1v_2E_2\dotsc v_k$, where
$v_1,\dotsc, v_k$ are distinct vertices and $E_1,\dotsc, E_{k-1}$ are
distinct hyperedges such that $v_i, v_{i+1}\in E_i$ for all
$i\in[k-1]$. We say that a hypergraph is \textdef{connected}\index{hypergraph!connected} if, for
any vertices $u$ and $v$, there exists a path from $u$ to~$v$.

An $(N,M,k)$-hypergraph is a $k$-uniform hypergraph with $V = [N]$ and
$M$ edges. Let $\cNM$ denote the number of connected
$(N,M,3)$-hypergraphs.  Our main result is an asymptotic formula for
$\cNM$ for a sparse range of~$M$. For $k\geq 0$, define $ g_k(\lambda) = \exp(\lambda)+k$ and recall
that $\fpo{k}(\lambda) = 
\exp(\lambda)-\sum_{i=0}^{k-1} {\lambda^i}/{i!}$.\label{glo:fpo}
\begin{thm} Let $M = M(N) = N/2 + R$ be such that $R=o(N)$ and $R =
  \omega(N^{1/3}\ln^{2} N)$. Then\label{glo:thm-main}
  \label{thm:main-hyper}
  \begin{equation*}
    \cNM\sim
    \sqrt{\frac{3}{\pi  N}}
    \exp\paren[\Big]{N \phi(\nNopt)+ N\ln N-N},
  \end{equation*}
  where
  \begin{align*}
     \phi(x) =
      &-\frac{(1-x)}{2} \ln(1-x) +\frac{1-x}{2}\\
      &+\frac{2R}{N}\ln(N)
      -\paren[\big]{\ln(2)+2}\frac{R}{N}
      - \frac{1}{2}\ln(2)x
      \\
      &+\frac{R}{N}\ln\left(\frac{
           g_1(\lambdaoptopt)}{\lambdaoptopt f_1(\lambdaoptopt) }
      \right)+ \frac{1}{2}x\ln\left(\frac{f_1(\lambdaoptopt) g_1(\lambdaoptopt)}{\lambdaoptopt}\right),
     \end{align*}
          \begin{equation*}
       \nNopt = \frac{\FF(2\lambdaoptopt)}{\fpo{1}(\lambdaoptopt) g_1(\lambdaoptopt)},
     \end{equation*}
     and $\lambdaoptopt$ is the unique positive solution of
     \begin{equation}
       \label{eq:lambdaoptopt-thm}
       \lambda
       \frac{e^{2\lambda}+e^{\lambda}+1}
       {\fpo{1}(\lambda) g_1(\lambda)}
       =
       \frac{3M}{N}.
     \end{equation}
\end{thm}
Our proof basically follows \nickk{one of the two approaches}  that Pittel and
Wormald~\cite{PWb} use to the enumerate connected graphs in the
sparser range. \nickk{This involves decomposing} a connected
graph into two parts: a cyclic structure and an acyclic structure. The
cyclic structure is a pre-kernel, which is a $2$-core without isolated
cycles. The acyclic structure is a rooted forest where the roots are
the vertices of the pre-kernel. A rooted forest with roots
$r_1,\dotsc, r_t$ (that are vertices in the forest) simply is a forest
such that each component contains exactly one of the roots. The graph
can then be obtained by `gluing' these two structures together. Pittel
and Wormald obtain an asymptotic formula for the number of the cyclic structures and
combine it with a known formula for the acyclic parts to obtain an
asymptotic formula for the number of connected graphs with given
number of vertices and edges.

We will also decompose a connected $3$-uniform hypergraph into two
parts: a cyclic structure (which we will also call pre-kernel) and an
acyclic structure (a forest rooted on the vertices of the
pre-kernel). We will also obtain asymptotic formulae for these
structures and then combine them to obtain an asymptotic formula for
the number of connected $(N, M, 3)$-hypergraphs.

From now on, we will deal with $3$-uniform hypergraphs most of the
time\nickk{, so henceforth,} for convenience, we will use the word `graph'
to denote $3$-uniform hypergraphs. When we want to refer to graphs in
the usual sense, we will call them `$2$-uniform hypergraphs'.  We will
also use the word `edge' instead of `hyperedge'.

\nickk{ We often give the results of routine algebraic manipulation
  and expansions  for which we have used Maple. The interested reader
  can see~\cite[Appendix A]{Sato13} for more details of these computations.} \nickkb{Please note that we give a glossary of notation at the end of this article.}

\section{Relation to a known formula}
\label{sec:rel-hyper}
As we mentioned before, Behrisch, Coja-Oghlan and Kang~\cite{BCK}
provided an asymptotic formula for the number of connected $(N,M,
k)$-hypergraphs for the range $M = N/(k-1) + \Omega(N)$. In this
section, we show that, for $k=3$, their formula is asymptotic to ours
when $R= M-N= o(N)$. Behrisch, Coja-Oghlan and Kang obtained their
result by computing the probability that the random hypergraph $H_k(N,
M)$ with uniform distribution on all $(N,M,k)$-hypergraphs is
connected.
\begin{thm}[{\cite[Theorem 5]{BCK}}]
  Let $k\geq 2$ be a fixed integer. For any compact set $\Jcal \subset
  (k(k-1)^{-1},\infty)$, and for any $\delta > 0$ there exists $N_0
  > 0$ such that the following holds. Let $M = M(N)$ be a
  sequence of integers such that $\zeta = \zeta(N) = kM/N
  \in\Jcal$ for all $N$. Then there exists a unique number $0 < r
  =r(N) < 1$ such that
  \begin{equation}
    \label{eq:CO-r}
    r
    =
    \exp\left(
      -\zeta\frac{(1-r)(1-r^{k-1})}{1-r^k}
      \right).
  \end{equation}
  Let $\Phi(\zeta) = r^{r/(1-r)} (1-r)^{1-\zeta}
  (1-r^k)^{\zeta/k}$. Furthermore, let
  \begin{align*}
    R_2(N, M) &= \frac{1+r-\zeta r}{\sqrt{(1-r)^2 -2\zeta r}} 
    \exp\left(\frac{2\zeta r + \zeta^2 r}{2(1+r)}\right) \Phi(\zeta)^N,\quad
    \text{and set}\\
    R_k(N,M) &=
    \frac{1-r^k-(1-r)\zeta(k-1)r^{k-1}}{\sqrt{(1-r^{k}+\zeta(k-1)(r-r^{k-1}))
      (1-r^k)-\zeta k r (1-r^{k-1})^2}}\\
    &\quad\times
    \exp \left( 
      \frac{\zeta(k-1)(r-2r^k+r^{k-1})}{2(1-r^k)}
    \right) \Phi(\zeta)^N,\quad\text{ if }k>2.
  \end{align*}
  For $N>N_0$, the probability that $H_k(N, M)$ is connected
  is in $((1-\delta)R_k(N,M), (1+\delta)R_k(N,M))$.
\end{thm}
From this theorem, it is immediate that the number of connected
$(N,M,k)$-graphs is asymptotic to
\begin{equation*}
 \binom{\binom{N}{k}}{M}R_k(N,M)\eqdefinv D(N,M,k)
\end{equation*}
when $R= M-N/2 = \Omega(N)$. Next we assume $R/N=o(1)$ and do some
simplifications in $D(N,M,3)$. So suppose $R=M-N/2=o(N)$. First we
compare $r$ in~\eqref{eq:CO-r} with $\lambdaoptopt$
in~\eqref{eq:lambdaoptopt-thm}:
\begin{equation*}
  r
  =
  \exp\left(
    -\frac{3M}{N}\frac{(1-r)(1-r^2)}{1-r^3}
  \right)
  \quad\text{and}\quad
  \lambdaoptopt
  \frac{e^{2\lambdaoptopt}+e^{\lambdaoptopt}+1}
  {e^{2\lambdaoptopt}-1}
  =
  \frac{3M}{N}.
\end{equation*}
By taking the logs in both sides of the definition of $r$, it is
obvious that $r = \exp(-\lambdaoptopt)$. As we will see later,
$\lambdaoptopt \to 0$ and so $r\to 1$. \nickk{Then by expanding we find   that}
\begin{equation*}
\lim_{r\to 1}
\frac{1-r^3-2(1-r)\zeta r^{2}}{\sqrt{(1-r^{3}+2\zeta(r-r^{2}))
    (1-r^3)-3\zeta r (1-r^{2})^2}}
=
\sqrt{3}
\end{equation*}
and
\begin{equation*}
\lim_{r\to 1} \frac{\zeta (r-2r^3+r^{2})}{1-r^3}=3/2.
\end{equation*}
Thus,
\begin{equation*}
  \begin{split}
    D(N,M,k) 
    &\sim \binom{\binom{N}{3}}{M}\sqrt{3}\exp(3/2)\Phi(3M/N)^N
    \\
    &\sim
    \frac{\displaystyle\sqrt{2\pi\binom{N}{3}}\left(\frac{\binom{N}{3}}{e}\right)^{\binom{N}{3}}}
      {\displaystyle\sqrt{2\pi M}\left(\frac{M}{e}\right)^M
        \sqrt{2\pi\left(\binom{N}{3}-M\right)}\left(\frac{\binom{N}{3}-M}{e}\right)^{\binom{N}{3}-M}}
      \sqrt{3} \exp(3/2)\Phi(3M/N)^N,
  \end{split}
\end{equation*}
by Stirling's approximation. Thus, using $M= N/2+R$ and $R= o(N)$,
\begin{equation*}
  \begin{split}
    &D(N,M,k)\\
    &\sim
    \sqrt{\frac{3}{\pi N}}
    \exp\left(
    3/2+
    \binom{N}{3}\ln \binom{N}{3}
    -M\ln M
    -\left(\binom{N}{3}-M\right)\ln\left( \binom{N}{3}-M\right)
    +N\Phi(3M/N)
  \right)
  \\
  &=
    \sqrt{\frac{3}{\pi N}}
    \exp\left(
    3/2
    -\left(\binom{N}{3}-M\right)\ln\left( 1 -\frac{M}{\binom{N}{3}}\right)
    -M\ln M+M\ln \binom{N}{3}
    +N\Phi(3M/N)
  \right)
 \\
 & =
  \sqrt{\frac{3}{\pi N}}
    \exp\left(
    3/2
    -\left(\binom{N}{3}-M\right)\left(-\frac{M}{\binom{N}{3}} 
      + O\left(\frac{M^2}{\binom{N}{3}^2}\right)\right)
    -M\ln M+M\ln \binom{N}{3}
    +N\Phi(3M/N)
  \right).
\end{split}
\end{equation*}
Thus,
\begin{equation*}
  \begin{split}
&D(N,M,k)=
\sqrt{\frac{3}{\pi N}}
    \exp\left(
    3/2
    +M
    -M\ln M+M\ln \binom{N}{3}
    +N\Phi(3M/N)
    +o(1)
  \right)
\\
&=
\sqrt{\frac{3}{\pi N}}
    \exp\left(
    3/2
    +M
    -M\ln M+ M\ln \frac{N^3}{6}
    + M\ln \frac{N(N-1)(N-2)}{N^3}
    +N\Phi(3M/N)
    +o(1)
  \right)
\\
&=
\sqrt{\frac{3}{\pi N}}
    \exp\left(
    3/2
    +M
    -M\ln M+ 3M\ln N - M\ln 6
    + M\ln \left(1 - \frac{3N-2}{N^2}\right)
    +N\Phi(3M/N)
    +o(1)
  \right)
\\
&=
\sqrt{\frac{3}{\pi N}}
    \exp\left(
    3/2
    +M
    -M\ln M+ 3M\ln N - M\ln 6
    - M\frac{3N}{N^2}
    +N\Phi(3M/N)
    +o(1)
  \right)
\\
&\sim
\sqrt{\frac{3}{\pi N}}
    \exp\left(
      M
    -M\ln M+ 3M\ln N - M\ln 6
    +N\Phi(3M/N)
  \right),
\end{split}
\end{equation*}
which is exactly the same as our formula in
Theorem~\ref{thm:main-hyper} after a series of \nickk{routine algebraic simplifications.}

\section{Basic definitions and results for hypergraphs}
\label{sec:def-hyper}

In this section, we present some basic definitions for hypergraphs and
show how to decompose a hypergraph into a cyclic structure and an
acyclic structure. 

A \textdef{cycle}\index{hypergraph!cycle} in a hypergraph $G = (V,
\Ecal)$ is a (finite) sequence $(v_0,E_0,\dotsc, v_k, E_k)$ such that
$v_1,\dotsc, v_k\in V$ are distinct vertices, $E_1,\dotsc,
E_k\in\Ecal$ are distinct edges with $v_i\in E_i$ and $v_{i+1}\in
E_{i}$ for every $0\leq i\leq k$ (operations in the indices are in
$\setZ_{k+1}$). A \textdef{tree}\index{hypergraph!tree} is an acyclic
connected hypergraph and a \textdef{forest}\index{hypergraph!forest}
is an acyclic hypergraph. A \textdef{rooted
  forest}\index{hypergraph!rooted forest} $G=(V,\Ecal)$ with
set of roots $S\subseteq V$ is a forest such that each component of
the forest has exactly one vertex in $S$. See
Figure~\ref{fig:forest-hyper} for a rooted forest.

\begin{figure}
  \centering
  \begin{tikzpicture}
    \def \initialg {0};
    \def \step {1};
    \def \rad {2pt};
    \def \diag {2pt};
    \def \eps {.3};
    \def \epsm {.15};
    \def \round {5pt};    

    \coordinate (r1) at ($(\initialg,\initialg)$);
    \coordinate (r2) at ($(r1)+2*\step*(1,0)$);
    \coordinate (r3) at ($(r1)+4*\step*(1,0)$);;
    \coordinate (r4) at ($(r1)+6*\step*(1,0)$);;

    \coordinate (v1) at ($(r1)+\step*({cos(60)},{sin(60)})$);
    \coordinate (v2) at ($(r1)+\step*({cos(120)},{sin(120)})$);

    \coordinate (v3) at ($(r2)+\step*({cos(60)},{sin(60)})$);
    \coordinate (v4) at ($(r2)+\step*({cos(120)},{sin(120)})$);

    \coordinate (v5) at ($(v3)+\step*({cos(60)},{sin(60)})$);
    \coordinate (v6) at ($(v3)+\step*({cos(120)},{sin(120)})$);

    \coordinate (v7) at ($(v6)+\step*({cos(90)},{sin(90)})$);
    \coordinate (v8) at ($(v6)+\step*({cos(150)},{sin(150)})$);

    \coordinate (v9) at ($(v5)+\step*({cos(90)},{sin(90)})$);
    \coordinate (v10) at ($(v5)+\step*({cos(30)},{sin(30)})$);

    \coordinate (v11) at ($(r3)+\step*({cos(60)},{sin(60)})$);
    \coordinate (v12) at ($(r3)+\step*({cos(120)},{sin(120)})$);

    \coordinate (v13) at ($(v11)+\step*({cos(80)},{sin(80)})$);
    \coordinate (v14) at ($(v11)+\step*({cos(20)},{sin(20)})$);

    \coordinate (v15) at ($(v11)+\step*({cos(100)},{sin(100)})$);
    \coordinate (v16) at ($(v11)+\step*({cos(160)},{sin(160)})$);

    \foreach \point in {r1,r2,r3,r4} \fill [black] 
    ($(\point)-(\diag,\diag)$) rectangle ($(\point)+(\diag,\diag)$);

    \foreach \point in {v1,v2} \fill [black] (\point) circle (\rad);

    \foreach \point in {v3,v4,v5,v6} \fill [black] (\point) circle
    (\rad);
    \foreach \point in {v7,v8,v9,v10} \fill [black] (\point) circle (\rad);
    \foreach \point in {v11,v12} \fill [black] (\point) circle (\rad);
    \foreach \point in {v13,v14,v15,v16} \fill [black] (\point) circle (\rad);

    \coordinate (b1) at ($1/3*(r1)+1/3*(v1)+1/3*(v2)$);
    \coordinate (a1) at ($-0.4*(b1)+1.4*(r1)$);
    \coordinate (a2) at ($-0.4*(b1)+1.4*(v1)$);
    \coordinate (a3) at ($-0.4*(b1)+1.4*(v2)$);
    %\foreach \point in {b1,a1,a2,a3} \fill [red] (\point) circle (\rad);
    \draw[rounded corners=\round] (a1)--(a2)--(a3)--cycle; 

    \coordinate (b2) at ($1/3*(r2)+1/3*(v3)+1/3*(v4)$);
    \coordinate (c1) at ($-0.4*(b2)+1.4*(r2)$);
    \coordinate (c2) at ($-0.4*(b2)+1.4*(v3)$);
    \coordinate (c3) at ($-0.4*(b2)+1.4*(v4)$);
    %\foreach \point in {b1,a1,a2,a3} \fill [red] (\point) circle (\rad);
    \draw[rounded corners=\round] (c1)--(c2)--(c3)--cycle; 

    \coordinate (b3) at ($1/3*(v3)+1/3*(v5)+1/3*(v6)$);
    \coordinate (d1) at ($-0.4*(b3)+1.4*(v3)$);
    \coordinate (d2) at ($-0.4*(b3)+1.4*(v5)$);
    \coordinate (d3) at ($-0.4*(b3)+1.4*(v6)$);
    %\foreach \point in {b1,a1,a2,a3} \fill [red] (\point) circle (\rad);
    \draw[rounded corners=\round] (d1)--(d2)--(d3)--cycle;

    \coordinate (b4) at ($1/3*(v6)+1/3*(v7)+1/3*(v8)$);
    \coordinate (e1) at ($-0.4*(b4)+1.4*(v6)$);
    \coordinate (e2) at ($-0.4*(b4)+1.4*(v7)$);
    \coordinate (e3) at ($-0.4*(b4)+1.4*(v8)$);
    %\foreach \point in {b1,a1,a2,a3} \fill [red] (\point) circle (\rad);
    \draw[rounded corners=\round] (e1)--(e2)--(e3)--cycle;

    \coordinate (b5) at ($1/3*(v5)+1/3*(v9)+1/3*(v10)$);
    \coordinate (f1) at ($-0.4*(b5)+1.4*(v5)$);
    \coordinate (f2) at ($-0.4*(b5)+1.4*(v9)$);
    \coordinate (f3) at ($-0.4*(b5)+1.4*(v10)$);
    %\foreach \point in {b1,a1,a2,a3} \fill [red] (\point) circle (\rad);
    \draw[rounded corners=\round] (f1)--(f2)--(f3)--cycle;

    \coordinate (b6) at ($1/3*(r3)+1/3*(v11)+1/3*(v12)$);
    \coordinate (g1) at ($-0.4*(b6)+1.4*(r3)$);
    \coordinate (g2) at ($-0.4*(b6)+1.4*(v11)$);
    \coordinate (g3) at ($-0.4*(b6)+1.4*(v12)$);
    %\foreach \point in {b1,a1,a2,a3} \fill [red] (\point) circle (\rad);
    \draw[rounded corners=\round] (g1)--(g2)--(g3)--cycle;

    \coordinate (b7) at ($1/3*(v11)+1/3*(v13)+1/3*(v14)$);
    \coordinate (h1) at ($-0.4*(b7)+1.4*(v11)$);
    \coordinate (h2) at ($-0.4*(b7)+1.4*(v13)$);
    \coordinate (h3) at ($-0.4*(b7)+1.4*(v14)$);
    %\foreach \point in {b1,a1,a2,a3} \fill [red] (\point) circle (\rad);
    \draw[rounded corners=\round] (h1)--(h2)--(h3)--cycle;

    \coordinate (b8) at ($1/3*(v11)+1/3*(v15)+1/3*(v16)$);
    \coordinate (i1) at ($-0.4*(b8)+1.4*(v11)$);
    \coordinate (i2) at ($-0.4*(b8)+1.4*(v15)$);
    \coordinate (i3) at ($-0.4*(b8)+1.4*(v16)$);
    %\foreach \point in {b1,a1,a2,a3} \fill [red] (\point) circle (\rad);
    \draw[rounded corners=\round] (i1)--(i2)--(i3)--cycle;
  \end{tikzpicture}
  \caption{A rooted forest. The roots are the vertices represented by
    squares.}
  \label{fig:forest-hyper}
\end{figure}
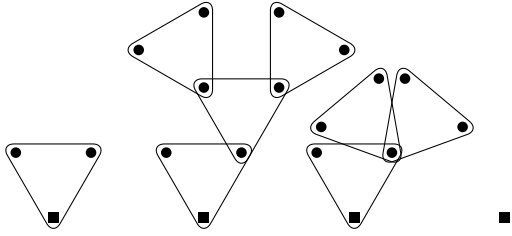

The \textdef{degree}\index{hypergraph!degree} of a vertex $v$ in a
hypergraph $G$ is the number of edges in $G$ containing~$v$. Recall
that we use the word `graph' to denote $3$-uniform hypergraphs. The
\textdef{core}\index{hypergraph!core} of a graph is its maximal
induced subgraph such that every edge contains at least two distinct
vertices of degree at least~$2$. To see that the core of a hypergraph
is unique, it suffices to notice that the union of two cores would
also be a core.  We remark that the $k$-core of a graph is usually
defined as the maximal subgraph such that every vertex has degree at
least~$k$. The core we defined contains the $2$-core of the hypergraph
and it allows some vertices of degree~$1$. We chose this definition of
core since otherwise the structure we would have to combine with the
$2$-core would not necessarily be acyclic. We also say that a graph is
a core, when its core is the graph itself. For an example of an core
see Figure~\ref{fig:core-isolated-cycles}.

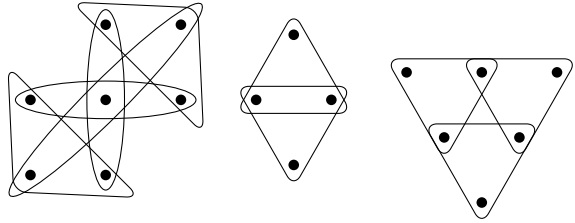
\begin{figure}
  \centering
  \begin{tikzpicture}
    \def \initialg {0};
    \def \step {1};
    \def \rad {2pt};
    \def \eps {.3};
    \def \round {5pt};

    \coordinate (1) at (\initialg,\initialg);
    \coordinate (1ul) at ($(1)+(-\eps,1.5*\eps)$);
    \coordinate (2) at (\initialg+\step,\initialg);
    \coordinate (3) at (\initialg+2*\step,\initialg);
    \coordinate (3dr) at ($(3)+(\eps,-1.5*\eps)$);
    \coordinate (4) at (\initialg+\step,\initialg+\step);
    \coordinate (4ul) at ($(4)+(-1.5*\eps,\eps)$);
    \coordinate (5) at (\initialg+\step,\initialg-1*\step);
    \coordinate (5dr) at ($(5)+(1.5*\eps,-\eps)$);
    \coordinate (6) at (\initialg,\initialg-1*\step);
    \coordinate (6dl) at ($(6)+(-0.75*\eps,-0.75*\eps)$);
    \coordinate (7) at (\initialg+2*\step,\initialg+\step);
    \coordinate (7ur) at ($(7)+(0.75*\eps,0.75*\eps)$);
    \coordinate (b1) at ($1/3*(1)+1/3*(5)+1/3*(6)$);

    \coordinate (1b) at ($(3)+(1*\step,0)$);
    \coordinate (1bdl) at ($(1b)+1.2*\eps*(-.75,-1/2)$);
    \coordinate (1bul) at ($(1b)+1.2*\eps*(-.75,1/2)$);
    \coordinate (2b) at ($(1b)+(\step,0)$);
    \coordinate (2bdr) at ($(2b)+1.2*\eps*(.75,-1/2)$);
    \coordinate (2bur) at ($(2b)+1.2*\eps*(.75,1/2)$);
    \coordinate (3b) at ($(1b)+ \step*({sin(30)},{cos(30)})$ );
    \coordinate (3bu) at ($(3b)+\eps*(0,1)$);
    \coordinate (4b) at ($(1b)+ \step*({sin(30)},-{cos(30)})$ );
    \coordinate (4bd) at ($(4b)+ \eps*(0,-1)$);

    \coordinate (1c) at ($(2b)+(1.5*\step,-.5*\step)$);
    \coordinate (1cul) at ($(1c)+1.2*\eps*(-.75,1/2)$);
    \coordinate (1cd) at ($(1c)+\eps*(0,-1)$);
    \coordinate (2c) at ($(1c)+(\step,0)$);
    \coordinate (2cur) at ($(2c)+1.2*\eps*(.75,1/2)$);
    \coordinate (2cd) at ($(2c)-(0,\eps)$);
    \coordinate (3c) at ($(1c)+ \step*({sin(30)},{cos(30)})$ );
    \coordinate (3cur) at ($(3c)+1.2*\eps*(.75,1/2)$);
    \coordinate (3cul) at ($(3c)+1.2*\eps*(-.75,1/2)$);

    \coordinate (4c) at ($(1c)+ \step*({sin(-30)},{cos(-30)})$);
    \coordinate (4cul) at ($(4c)+1.2*\eps*(-.75,1/2)$);
    \coordinate (5c) at ($(2c)+ \step*({sin(30)},{cos(30)})$ );
    \coordinate (5cur) at ($(5c)+1.2*\eps*(.75,1/2)$);
    \coordinate (6c) at ($(2c)+ \step*({sin(210)},{cos(210)})$ );
    \coordinate (6cd) at ($(6c)+\eps*(0,-1)$);

    \foreach \point in {1,2,3,4,5,6,7} \fill [black] (\point) circle (\rad); 
    \foreach \point in {1b,2b} \fill [black] (\point) circle (\rad); 
    \foreach \point in {3b,4b} \fill [black] (\point) circle (\rad); 
    \foreach \point in {1c,2c,3c} \fill [black] (\point) circle (\rad); 
    \foreach \point in {4c,5c,6c} \fill [black] (\point) circle (\rad); 
    % \foreach \point in {1cul, 6cd, 2cur} \fill [red] (\point) circle (\rad);
       
    \draw (2) ellipse ({\step*1.2} and {\step/4}); %123
    \draw (2) ellipse ({\step/4} and {\step*1.2}); %425
    \draw[rotate=-45] (2) ellipse ({\step/4} and {\step*1.8}); %627
    \draw[rounded corners=\round] (1ul)--(5dr)--(6dl)--cycle;
    \draw[rounded corners=\round] (4ul)--(3dr)--(7ur)--cycle;
    \draw[rounded corners=\round] (1bdl)--(2bdr)--(3bu)--cycle;
    \draw[rounded corners=\round] (1bul)--(2bur)--(4bd)--cycle;
    \draw[rounded corners=\round] (1cd)--(4cul)--(3cur)--cycle;
    \draw[rounded corners=\round] (2cd)--(5cur)--(3cul)--cycle;
    \draw[rounded corners=\round] (6cd)--(2cur)--(1cul)--cycle;
    
  \end{tikzpicture}
  \caption{A core with two isolated cycles. The leftmost component is
    a \prekernel.}
  \label{fig:core-isolated-cycles}
\end{figure}

Every edge in a core has either one vertex of degree~$1$, or none. We
say that an edge is a
\textdef{$2$-edge}\index{hypergraph!twoedge@$2$-edge} if it has a
vertex of degree~$1$ and that it is a
\textdef{$3$-edge}\index{hypergraph!threeedge@$3$-edge} if it has no
vertex of degree~$1$. It is easy to see that the core of graph can be
obtained by iteratively removing edges that are not $2$-edges nor
$3$-edges until all edges are $2$-edges or $3$-edges, and then
deleting all vertices of degree~$0$. See
Figure~\ref{fig:core-proc-hyper} for an example of this procedure.

\begin{figure}
  \centering
    \begin{tikzpicture}
    \def \initialg {0};
    \def \step {1};
    \def \rad {2pt};
    \def \eps {.3};
    \def \epsm {.15};
    \def \round {5pt};

    \coordinate (1) at (\initialg,\initialg);
    \coordinate (1ul) at ($(1)+(-\eps,1.5*\eps)$);
    \coordinate (2) at (\initialg+\step,\initialg);
    \coordinate (mid12) at ($.5*(2)+.5*(1)$);
    \coordinate (3) at (\initialg+2*\step,\initialg);
    \coordinate (3dr) at ($(3)+(\eps,-1.5*\eps)$);
    \coordinate (4) at (\initialg+\step,\initialg+\step);
    \coordinate (4ul) at ($(4)+(-1.5*\eps,\eps)$);
    \coordinate (5) at (\initialg+\step,\initialg-1*\step);
    \coordinate (mid52) at ($.5*(2)+.5*(5)$);
    \coordinate (5dr) at ($(5)+(1.5*\eps,-\eps)$);
    \coordinate (6) at (\initialg,\initialg-1*\step);
    \coordinate (6dl) at ($(6)+(-.75*\eps,-.75*\eps)$);
    \coordinate (7) at (\initialg+2*\step,\initialg+\step);
    \coordinate (7ur) at ($(7)+(0.75*\eps,0.75*\eps)$);
    \coordinate (b1) at ($1/3*(1)+1/3*(5)+1/3*(6)$);

    \coordinate (8) at ($(7)+\step*({cos(30)},{sin(30)})$);
    \coordinate (9) at ($(8)+\step*(0,-1)$);

   \coordinate%[label=left:\footnotesize{$10$}] 
   (10) at ($(9)+\step*({cos(60)},{sin(60)})$);
   \coordinate%[label=left:\footnotesize{$11$}] 
   (11) at ($(9)+\step*({cos(0)},{sin(0)})$);

   \coordinate%[label=left:\footnotesize{$12$}] 
   (12) at ($(9)+\step*({cos(-30)},{sin(-30)})$);
   \coordinate%[label=left:\footnotesize{$13$}] 
   (13) at ($(9)+\step*({cos(-90)},{sin(-90)})$);

    \coordinate%[label=left:\footnotesize{$14$}] 
    (14) at ($(1)+\step*({cos(150)},{sin(150)})$);
    \coordinate%[label=left:\footnotesize{$15$}] 
    (15) at ($(14)+\step*(0,-1)$);

    \coordinate%[label=left:\footnotesize{$16$}] 
    (16) at ($(15)+\step*({cos(150)},{sin(150)})$);
    \coordinate%[label=left:\footnotesize{$17$}] 
    (17) at ($(16)+\step*(0,-1)$);

     \foreach \point in {1,2,3,4,5,6,7} \fill [black] (\point) circle
    (\rad);
     \foreach \point in {8,9,10,11,12,13} \fill [black] (\point) circle
     (\rad);

     \foreach \point in {14,15,16,17} \fill [black] (\point) circle
     (\rad);

    \draw (2) ellipse ({\step*1.2} and {\step/4}); %123
    \draw (2) ellipse ({\step/4} and {\step*1.2}); %425
     \draw[rounded corners=\round] (1ul)--(5dr)--(6dl)--cycle;
     \draw[rounded corners=\round] (4ul)--(3dr)--(7ur)--cycle;
     
    \coordinate (b1) at ($1/3*(7)+1/3*(8)+1/3*(9)$);
    \coordinate (a1) at ($-0.4*(b1)+1.4*(7)$);
    \coordinate (a2) at ($-0.4*(b1)+1.4*(8)$);
    \coordinate (a3) at ($-0.4*(b1)+1.4*(9)$);
    %\foreach \point in {b1,a1,a2,a3} \fill [red] (\point) circle (\rad);
    \draw[rounded corners=\round] (a1)--(a2)--(a3)--cycle;

    \coordinate (b2) at ($1/3*(9)+1/3*(10)+1/3*(11)$);
    \coordinate (c1) at ($-0.4*(b2)+1.4*(9)$);
    \coordinate (c2) at ($-0.4*(b2)+1.4*(10)$);
    \coordinate (c3) at ($-0.4*(b2)+1.4*(11)$);
    %\foreach \point in {b1,a1,a2,a3} \fill [red] (\point) circle (\rad);
    \draw[rounded corners=\round] (c1)--(c2)--(c3)--cycle;

    \coordinate (b3) at ($1/3*(9)+1/3*(12)+1/3*(13)$);
    \coordinate (d1) at ($-0.4*(b3)+1.4*(9)$);
    \coordinate (d2) at ($-0.4*(b3)+1.4*(12)$);
    \coordinate (d3) at ($-0.4*(b3)+1.4*(13)$);
    %\foreach \point in {b1,a1,a2,a3} \fill [red] (\point) circle (\rad);
    \draw[rounded corners=\round] (d1)--(d2)--(d3)--cycle;

    \coordinate (b4) at ($1/3*(1)+1/3*(14)+1/3*(15)$);
    \coordinate (e1) at ($-0.4*(b4)+1.4*(1)$);
    \coordinate (e2) at ($-0.4*(b4)+1.4*(14)$);
    \coordinate (e3) at ($-0.4*(b4)+1.4*(15)$);
    %\foreach \point in {b1,a1,a2,a3} \fill [red] (\point) circle (\rad);
    \draw[rounded corners=\round] (e1)--(e2)--(e3)--cycle;

    \coordinate (b5) at ($1/3*(15)+1/3*(16)+1/3*(17)$);
    \coordinate (f1) at ($-0.4*(b5)+1.4*(15)$);
    \coordinate (f2) at ($-0.4*(b5)+1.4*(16)$);
    \coordinate (f3) at ($-0.4*(b5)+1.4*(17)$);
    %\foreach \point in {b1,a1,a2,a3} \fill [red] (\point) circle (\rad);
    \draw[rounded corners=\round] (f1)--(f2)--(f3)--cycle;

    \coordinate  (startarrow) at ($(12)+(\step,0)$);
    \coordinate  (endarrow) at ($(startarrow)+(\step,0)$);
    \coordinate
    (midarrow) at ($.5*(startarrow)+.5*(endarrow)$);
    \draw [->] (startarrow)--(endarrow);    

    \coordinate (1c) at ($(endarrow)+(\step,0)$);
    \coordinate (1ulc) at ($(1c)+(-\eps,1.5*\eps)$);
    \coordinate (2c) at ($(1c)+\step*(1,0)$);
    \coordinate (mid12c) at ($.5*(2c)+.5*(1c)$);
    \coordinate (3c) at ($(1c)+(2*\step,0)$);
    \coordinate (3drc) at ($(3c)+(\eps,-1.5*\eps)$);
    \coordinate (4c) at ($(1c)+(\step,\step)$);
    \coordinate (4ulc) at ($(4c)+(-1.5*\eps,\eps)$);
    \coordinate (5c) at ($(1c)+\step*(1,-1)$);
    \coordinate (mid52c) at ($.5*(2c)+.5*(5c)$);
    \coordinate (5drc) at ($(5c)+(1.5*\eps,-\eps)$);
    \coordinate (6c) at  ($(1c)+(0,-\step)$);
    \coordinate (6dlc) at ($(6c)+(-.75*\eps,-.75*\eps)$);
    \coordinate (7c) at  ($(1c)+(2*\step,\step)$);
    \coordinate (7urc) at ($(7c)+(0.75*\eps,0.75*\eps)$);
    \coordinate (b1c) at ($1/3*(1c)+1/3*(5c)+1/3*(6c)$);

     \foreach \point in {1c,2c,3c,4c,5c,6c,7c} \fill [black] (\point) circle
    (\rad);

    \draw (2c) ellipse ({\step*1.2} and {\step/4}); %123
    \draw (2c) ellipse ({\step/4} and {\step*1.2}); %425
     \draw[rounded corners=\round] (1ulc)--(5drc)--(6dlc)--cycle;
     \draw[rounded corners=\round] (4ulc)--(3drc)--(7urc)--cycle;
   \end{tikzpicture}
   \caption{Obtaining the core.}
   \label{fig:core-proc-hyper}
\end{figure}
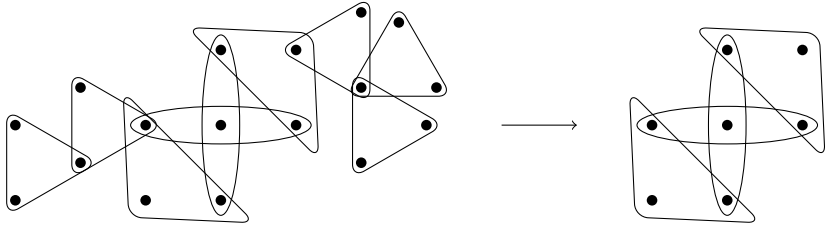

We also define cycles as graphs. We say that a graph $G=(V,\Ecal)$ is
a \textdef{cycle}\index{hypergraph!cycle} if there is an ordering $(v_0,\dotsc, v_k)$ of a
subset of $V$ and an ordering $(E_0,\dotsc,E_k)$ of $\Ecal$ such that
$(v_0,E_0,\dotsc, v_k,E_k)$ is a cycle in $G$ and every $v\in V$ is in
some $E\in\Ecal$.  Note that, if a graph is a cycle, then all edges are
actually $2$-edges.  An \textdef{isolated cycle}\index{hypergraph!cycle!isolated} in a graph is a
component that is a cycle.  A \textdef{pre-kernel}\index{hypergraph!prekernel@pre-kernel} is a core without
isolated cycles (see Figure~\ref{fig:core-isolated-cycles}). So, every
  connected core that is not just a cycle is also a pre-kernel.

The following proposition explains how to decompose a graph into its
core and a rooted forest.
\begin{prop}
  \label{p:decomposition-hyper}
  Let $G$ be a connected graph with a nonempty core.  The graph
  obtained from $G$ by deleting the edges of the core of $G$ and by
  setting all vertices in the core as roots is a rooted forest with
  $(N-n)/2$ edges, where $N$ is the number of vertices in $G$ and $n$
  is the number of vertices in the core. Moreover,  the core of $G$
  is connected.
\end{prop}
\begin{proof}
  As we already mentioned, the core of $G$ can be obtained by
  iteratively deleting edges that contain at most one vertex of degree
  at least~$2$. More precisely, start with $G' = G$ and while there is
  an edge in $G'$ containing less than $2$ vertices of degree at
  least~$2$ in $G'$, redefine $G'$ by deleting one such edge. When
  this procedure stops, $G'$ is the core of~$G$.  Let $F$ be the graph
  with vertex set $[N]$ with the deleted edges as its set of edges.
  Suppose for a contradiction that $F$ has a cycle. Such a cycle is a
  cycle in $G$ too. Let $E$ be the first edge of the cycle that was
  deleted by the procedure described above. All other edges in the
  cycle were still present in the graph $G'$ when $E$ was
  deleted. Thus, since $E$ was in the cycle, it had a least $2$
  vertices of degree at least~$2$. Hence, $E$~could not have been
  deleted at this point, which shows that $F$ has no cycles.

  Suppose for \nickc{deleted 'a'} contradiction that the core of $G$ is not
  connected. Then it has at least $2$ components that are joined by a
  path in~$G$ with all edges in~$F$ since $G$ is connected. The union
  of these $2$-components and the path is a $2$-core, which is a
  contradiction.  Thus, the core is connected. This argument also
  shows that that each component of $F$ has at most one vertex in the
  core. Every component of $F$ must have one vertex in the core,
  otherwise it is disconnected from the core and so $G$ would not be
  connected.

  Now we determine the number of edges in~$F$. As we discussed above,
  each component of $F$ has exactly one vertex in the core.  In the
  deletion procedure, for the initial $G'$ (that is, $G$), every edge
  has at least one vertex of degree at least $2$ since otherwise $G$
  would not be connected. We claim that the deletion procedure will
  only delete edges that contain exactly one vertex of degree~$2$ in
  the current $G'$.  If not, let $E$ be an edge that contained no
  vertex of degree at least~$2$ in~$G'$ in the moment it was deleted.
  Let $v_0$ be the vertex of the core in the same component of $E$
  in~$F$. Then there is a path $(v_0,E_0,\dotsc,E_{k-1} v_k)$ in $F$,
  where $E_{k-1}=E$. The edge $E_0$ cannot be $E$ since the vertex
  $v_0$ must have degree at least~$2$ the moment $E_0$ is deleted.  A
  trivial induction proof then shows that the deletion procedure
  cannot delete any of the edges $E_0,\dotsc, E_{k-2}$ before
  deleting~$E_{k-1}$, which shows that the moment $E$ was deleted the
  vertex $v_{k-1}$ still was in $2$ edges: $E$ and~$E_{k-2}$. This is
  a contradiction. Thus, the moment any edge is deleted is has exactly
  one vertex of degree at least~$2$.  This means that, for every
  deleted edge, we also delete exactly $2$ vertices that are not in
  the core. Since there $N-n$ vertices to be deleted, the number of
  edges in $F$ is $(N-n)/2$.
\end{proof}

For any graph $G$ with $N$ vertices and $M$ edges such that its core
has $n$ vertices and $m$ edges, we have that
\begin{equation}
  \label{eq:excess-transfer-hyper}
  m-n/2=
  M-(N-n)/2-n/2
  =
  M-N/2
\end{equation}
since $m=M-(N-n)/2$ by
Proposition~\ref{p:decomposition-hyper}. Intuitively speaking, this
says that the `excess' of edges $(M-N/2)$ in the graph is transferred
to its core.

Let $\gcacti(N,n)$ denote the number of forests with vertex set $[N]$
and $[n]$ as its set of roots. Let $\gpre(n,m)$ denote the number of
connected pre-kernels with vertex set $[n]$ and $m$ edges. Next, we show
how to write $\cNM$ using $\gcacti$ and $\gpre$.
\begin{prop}
  For $M = M(N)$ such that $R\eqdef M-N/2 \to \infty$, we have that
  \begin{equation}
    \label{eq:sum-decomposition-hyper}
    \cNM =  \sum_{\substack{1\leq n\leq N\\(N-n)\in 2\setZ}} \gcacti(N,n) \gpre(n,M-(N-n)/2),
  \end{equation}    
  for $N$ sufficiently large.
\end{prop}
\begin{proof}
  In view of Proposition~\ref{p:decomposition-hyper}, it suffices to
  show that, for any connected graph $G$ with $N$ vertices and $M$
  edges, the core of $G$ is a pre-kernel. If it is not, either the
  core is empty or it is a cycle. If the core is empty, then the graph
  $G$ is a forest and so $M < N/2$, which is impossible since $M =
  N/2+R$ with $R\to\infty$. If the core is a cycle, then $3
  m=2(n-m)+m$ since each edge in the core has two vertices of degree
  $2$ and one of degree $1$. Thus, in this case, we have that $m=n/2$,
  which is impossible since $m-n/2=M-N/2=R \to \infty$
  by~\eqref{eq:excess-transfer-hyper}.
\end{proof}

Basically, our approach to compute an asymptotic formula for $\cNM$
will be to analyse the summation
in~\eqref{eq:sum-decomposition-hyper}.

We will work with random graphs. More precisely, we will work with
random multihypergraphs and then deduce results for simple graphs.  A
\textdef{$k$-uniform
  multihypergraph}\index{multihypergraph!kuniform@$k$-uniform} is a
triple $G=(V, \Ecal, \Phi)$, where $V$ and $\Ecal$ are finite sets and
$\Phi: \Ecal\times [k] \to V$.  We say that $V$ is the \textdef{vertex
  set} of $G$ and $\Ecal$ is the \textdef{edge set} of $G$. From now
on, we will use the word `multigraph' to denote $3$-uniform
multihypergraphs.

Given a multigraph $G=(V,\Ecal, \Phi)$, a
\textdef{loop}\index{multihypergraph!loop} is an edge $E\in\Ecal$ such
that there exist distinct $j,j'\in\set{1,2,3}$ such that
$\Phi(E,j)=\Phi(E,j')$, a pair of \textdef{double
  edges}\index{multihypergraph!double edges} is a pair $(E,E')$ of
distinct edges in $\Ecal$ such that the collection
$\set{\Phi(E,1),\Phi(E,2),\Phi(E,3)}$ is the same as the collection
$\set{\Phi(E',1),\Phi(E',2),\Phi(E',3)}$. A multigraph $G$ with no
loops nor double edges corresponds naturally to a graph because each
edge corresponds to a unique subset of $V$ of size~$3$. In this case
we say that the multigraph is
\textdef{simple}\index{multihypergraph!simple}. Let $\Scal(n,m)$
denote the set of simple multigraphs with vertex set $[n]$ and edge
set $[m]$. We have the following relation between simple multigraphs
and graphs:
\begin{lem}
  \label{lem:number-graph-multigraph-hyper}
  For any $G = ([n],[m],\Phi) \in \Scal(n, m)$, let $s(G)$ be the
  graph with vertex set $[n]$ obtained by including one edge for each
  $i\in[m]$ incident to the vertices $\Phi(i,1)$, $\Phi(i,2)$ and
  $\Phi(i,3)$.  Let $G'$ be a graph with vertex set $[n]$ with $m$
  edges. Then $|s^{-1}(G')| = m!6^m$, that is, each graph corresponds
  to $m!6^m$ simple multigraphs.
\end{lem}
\begin{proof}
  Let $G = ([n],[m],\Phi) \in \Scal(n, m)$ be such that $s(G)=G'$. For
  any permutation $g$ of $[m]$, the multigraph $G_{g}\eqdef
  ([n],[m],\Phi')$ satisfies $s(G_g) = G'$, where $\Phi'(i,j) =
  \Phi(g(i),j)$ for each $i\in[m]$ and $j\in\set{1,2,3}$. (That is, any
  permutation of the label of the edges generates the same graph.)
  Moreover, for each $i\in[m]$ and permutation $g_i$ of $[3]$, the
  function $\Phi''(i,j) = \Phi''(i,g(j))$ satisfies $s([n],[m],\Phi'')
  = G'$. Since there are $m!$ permutations on $[m]$ and $3!$
  permutations of $[3]$, the number of graphs $G\in\Scal(n,m)$ with
  $s(G) = G'$ is $m! 3!^{m}$.
\end{proof}

We extend the definitions of path and connectedness for
multihypergraphs.  For any multihypergraph~$G=(V,\Ecal,\Phi)$, a
\textdef{path}\index{multihypergraph!path} is a (finite) sequence $v_1E_1v_2E_2\dotsc v_k$, where
$v_1,\dotsc, v_k$ are distinct vertices and $E_1,\dotsc, E_{k-1}$ are
distinct hyperedges such that $v_i, v_{i+1}\in \textrm{Im}(\Phi(E_i,
\cdot))$ for all $i\in[k-1]$. We say that a multihypergraph is
\textdef{connected}\index{multihypergraph!connected|} if, for any vertices $u$ and $v$, there exists a
path from $u$ to~$v$.
\section{Overview of proof}
\label{sec:overview-hyper}
In this section, we give an overview of our proof of the asymptotic
formula for $\cNM$ in Theorem~\ref{thm:main-hyper}.  Recall that $R=
M-N/2 =o(N)$ and $R=\omega(N^{1/3}\log^2 N)$. Our approach is to analyse $\cNM$ by
using~\eqref{eq:sum-decomposition-hyper}, which shows how to obtain
$\cNM$ from formulae for the number of rooted forests $\gcacti$ and
the number of pre-kernels $\gpre$.  The proof consists of the
following steps.
\begin{enumerate}
\item We obtain an exact formula $\gcacti(N,n)$ for the number of of
  rooted forests with set of roots $[n]$ and vertex set $[N]$. We show
  that, for even $N-n$,
  \begin{equation*}
  \gcacti(N,n)=
  {\displaystyle \frac{n}{N}\cdot\frac{(N-n)! N^{{(N-n)}/{2}}}
    {\displaystyle\left(\paren{N-n}/{2}\right)! 2^{\paren{N-n}/{2}}}},
\end{equation*}
and, for odd $N-n$, $\gcacti(N,n)=0$. The proof is in
Section~\ref{sec:cacti-hyper} and is a simple proof by induction.
\item \New{We show that the number of cores with vertex set $[n]$ and
    $m$ edges is at most the following function of $n$ and $m
$
    \begin{equation*}
      \label{eq:gcore-ineq}
      \gcore(n,m) := \alpha n\sqrt{m} \cdot n! \exp(n\fcore(\nncopt)), \text{
      for }m-n/2\to\infty,
    \end{equation*}
    where $\alpha$ is a constant, the function $\fcore$ is defined in
    Section~\ref{sec:core-hyper} and $\lambdaopt$ is the unique
    positive solution for ${\lambda
      \f(\lambda)\fgg(\lambda)}/{\FF(2\lambda)} = 3m/n$ and $\nncopt =
    {3m}/{(n\fgg(\lambdaopt))}$. The proof is in
    Section~\ref{sec:core-hyper}. } \nickcb{THis footnote should be deleted!} \Old{Let $\gcore(n,m)$
    denote the number of (simple) cores on $[n]$ with $m$ edges. We
    analyse $\gcore$ by writing it as follows:
  \begin{equation*}
    \gcore(n,m)
    =\sum_{\nn,\ds}
    \gcore(n,m,\nn,\ds),
  \end{equation*}
  where $\gcore(n,m,\nn,\ds)$ is the number of (simple) cores with $n$
  vertices, $m=n/2+R$ edges, $\nn$ vertices of degree~$1$, and degree
  sequence $\ds$ for the vertices of degree at least $2$. We use $r$
  to denote $R/n$. We show that there is a constant $\alpha$ such
  that 
  \begin{align}
    &\gcore(n,m)\leq \alpha n\sqrt{m} \cdot n! \exp(n\fcore(\nncopt)), \text{
      for }R\to\infty,\text{ and }
    \\
    &\gcore(n,m)
    \sim
    \frac{1}{2\pi n \sqrt{r}}\cdot
    n!\exp(n\fcore(\nncopt)), \text{
      for }R\to\infty \text{ and }R=o(n),
  \end{align}
  where the function $\fcore$ is defined in
  Section~\ref{sec:core-hyper} and $\lambdaopt$ is the unique positive
  solution for ${\lambda \f(\lambda)\fgg(\lambda)}/{\FF(2\lambda)}
  = 3m/n$ and $\nncopt = {3m}/{(n\fgg(\lambdaopt))}$. The proof is in
  Section~\ref{sec:core-hyper}.
}

\item We obtain an asymptotic formula for the number $\gpre(n,m)$ of
  simple connected pre-kernels with $n \to\infty$ vertices and
  $m=n/2+rn$ edges when $R=o(n)$ and $R=\omega(n^{1/2}\log^{3/2}n)$. We show that
  \begin{equation}
    \label{eq:gpre-formula-overview}
    \gpre(n,m)\sim
      \frac{\sqrt{3}}{\pi n}
      n!\exp(n\fpre(\prexopt)),
  \end{equation}
  where $\fpre$ is defined in Section~\ref{sec:pre-hyper} and
  $\prexopt\in\setR^4$ will be determined using $\lambdaopt$ as
  defined in the previous step.
\item  \New{ We define a set $I\subseteq \setZ$ such
    that~\eqref{eq:gpre-formula-overview} holds for every $n\in I$
    with $m = M-(N-n)/2$. Using~\eqref{eq:gpre-formula-overview}, we
    show that, for $n\in I$ and $m = M-(N-n)/2$,
    \begin{equation}
      \label{eq:gpre-simpler-overview}
      \gpre(n,m)\sim
      \frac{\sqrt{3}}{\pi n}\cdot
    n!\exp(n\fcore(\nncopt)),
    \end{equation}
    where $\nncopt$ is defined in Step 2.
    Using~\eqref{eq:gpre-simpler-overview}, we then show
    that \begin{equation}
      \label{eq:formula-prek-cacti-summary}
      \begin{split}
        \sum_{n\in I} \binom{N}{n}\gcacti(N,n)\gpre(n,m)
        &\sim
        \sqrt{\frac{3}{\pi  N}}\exp\paren[\big]{N t(\nNopt)+ N\ln N-N}
      \end{split}
    \end{equation}
    where $t$ is defined in Section~\ref{sec:core-cacti-hyper},
  $\lambdaoptopt$ is the unique positive solution of the equation $\lambda
  \paren{e^{2\lambda}+e^{\lambda}+1}/\paren{\fpo{1}(\lambda) g_1(\lambda)} =
  {3M}/{N}$ and $\nNopt =
  {\FF(2\lambdaoptopt)}\paren{\fpo{1}(\lambdaoptopt) g_1(\lambdaoptopt)}$.
  }
  \Old{We will at first work with cores since the
  function obtained for them is simpler than the formula for
  pre-kernels. We will find relations between the two formulae that
  justify why it is relevant to analyse the function for cores. We
  find a set $I$ for $n$ in which $n = \Theta(\sqrt{RN})$ such that
  \begin{equation}
    \label{eq:formula-core-cacti-summary}
    \sum_{n\in I} \binom{N}{n}\gcacti(N,n)\gcore(n,m)
    \sim
    \frac{\sqrt{3}}{\sqrt{\pi \lambdaoptopt N }}
    \exp\paren[\big]{N t(\nNopt)+ N\ln N-N},
  \end{equation}
  where $t$ is defined in Section~\ref{sec:core-cacti-hyper},
  $\lambdaoptopt$ is the unique positive solution of the equation $\lambda
  \paren{e^{2\lambda}+e^{\lambda}+1}/\paren{\fpo{1}(\lambda) g_1(\lambda)} =
  {3M}/{N}$ and $\nNopt =
  {\FF(2\lambdaoptopt)}\paren{\fpo{1}(\lambdaoptopt) g_1(\lambdaoptopt)}$.
  We then show that the contribution to the summation for $n$ outside
  $I$ is insignificant by using~\eqref{eq:gcore-ineq}:
  \begin{equation}
    \label{eq:insignificant-core-cacti-summary}
    \sum_{n\in [N]\setminus I} \binom{N}{n}\gcacti(N,n)\gcore(n,m)
    =
    o\left(\frac{1}{\sqrt{\pi  N}}
      \exp\paren[\big]{N t(\nNopt)+ N\ln N-N}\right).
  \end{equation}}
\item \New{Since every pre-kernel is a core and $\gcore(n,m)$ is an
    upper bound for the number of cores with vertex set $[n]$ and $m$
    edges, we have that $\gpre(n,m)\leq\gcore(n,m)$. Using this
    relation together with~\eqref{eq:gcore-ineq},
    we show that
  \begin{equation*}
    \begin{split}
      \sum_{n\in [N]\setminus I} \binom{N}{n}\gcacti(N,n)\gpre(n,m)
      &\leq
      \sum_{n\in [N]\setminus I} \binom{N}{n}\gcacti(N,n)\gcore(n,m)
      \\
      &=
      o\left(\frac{1}{\sqrt{\pi  N}}
      \exp\paren[\big]{N t(\nNopt)+ N\ln N-N}\right).
  \end{split}
  \end{equation*}
  Hence, together with~\eqref{eq:formula-prek-cacti-summary}, we have
  that
  \begin{equation*}
    \cNM
    \sim
    \frac{\sqrt{3}}{\sqrt{\pi \lambdaoptopt N }}
    \exp\paren[\big]{N t(\nNopt)+ N\ln N-N}.
  \end{equation*}
}

\Old{We use Step~2, Step~3 and~\eqref{eq:formula-core-cacti-summary}
  to show that
 \begin{equation*}
   \begin{split}
     \sum_{n\in I} \binom{N}{n}\gcacti(N,n)\gpre(n,m)
     &\sim
     2\sqrt{3r} \sum_{n\in I} \binom{N}{n}\gcacti(N,n)\gcore(n,m)\\
     &\sim
     \sqrt{\frac{3}{\pi  N}}\exp\paren[\big]{N t(\nNopt)+ N\ln N-N}
  \end{split}
  \end{equation*}
  and using the relation $\gpre(n,m)\leq\gcore(n,m)$ (since every
  pre-kernel is a core)
  and~\eqref{eq:insignificant-core-cacti-summary},
  \begin{equation*}
    \sum_{n\in [N]\setminus I} \binom{N}{n}\gcacti(N,n)\gpre(n,m)
    =
    o\left(\sum_{n\in I} \binom{N}{n}\gcacti(N,n)\gpre(n,m)\right).
  \end{equation*}}
\item The conclusion is then easily obtained by simplifying  $t(\nNopt)$.
\end{enumerate}

\section{Counting forests}
\label{sec:cacti-hyper}
In this section we prove an exact formula for rooted forests. In this
section we consider $k$-uniform hypergraphs, for any $k\geq 2$.  We
remark that this formula has also been proved in a note by
Lavault~\cite{Lavault} around the same time we obtained it. Lavault
shows a one-to-one correspondence between rooted forests and a set of
tuples whose size can be easily computed.

Recall that $\gcacti(N,n)$ is the number of rooted forests on $[N]$
with set of roots $[n]$. (See Figure~\ref{fig:forest-hyper} for a
rooted forest.)

\begin{thm} For integers $N\geq n\geq 0$ and any integer $k\geq 2$,
  \label{thm:cacti-formula}
  \begin{equation*}
    \gcacti(N,n)=
    \begin{cases}
      {\displaystyle \frac{n (N-n)! N^{\mcacti-1}}{\mcacti! (k-1)!^\mcacti}},& 
      \text{if } {\displaystyle \mcacti = \frac{N-n}{k-1}} 
      \text{ is a nonnegative integer;}\\
      0, &\text{ otherwise.}
    \end{cases}
  \end{equation*}
\end{thm}
\begin{proof}
  A connected $k$-uniform hypergraph is a tree if and only if, by
  iteratively deleting edges that have at least $k-1$ vertices of
  degree $1$, we delete all edges. It then is obvious that $\mcacti$
  is the number of edges in the forest.  We remark that the a tree can
  be seen as a $2$-uniform hypergraph where each block is a clique on
  $[n-1]$ vertices (which is known as a clique tree).

  The proof is by induction on $N$.  We have that $\gcacti(1,1) = 1 =
  1(1-1)! 1^{0-1} / (0! (k-1)!^0) = 1$ and the formula also works for
  $\gcacti(N,0) = 0$. So assume that $N > 1$ and $n\geq 1$. We will
  show how to obtain a recurrence relation for $\gcacti(N,n)$. Suppose
  that the vertex $1$ is in $j$ edges, where $0\leq j\leq m'$. We
  choose $(k-1)j$ other vertices to be in these $j$ edges. There are
  $\binom{N-n}{(k-1)j}$ ways to choose these vertices. The number of
  ways we can split the vertices into the edges  is
\begin{equation*}
  \binom{(k-1)j}{k-1}
  \binom{(k-1)j-(k-1)}{k-1}
  \cdots
  \binom{k-1}{k-1}
  \frac{1}{j!}
  =
  \frac{\paren[\big]{(k-1)j}!}{(k-1)!^j j!}.
\end{equation*}
We can build the rooted forest by first choosing the edges containing
$1$ and then deleting $1$ and considering the other $(k-1)j$ vertices
in these edges as new roots. This gives us the following recurrence:
\begin{equation*}
  \gcacti(N,n)
  =
  \sum_{j=0}^\mcacti
  \binom{N-n}{(k-1)j} \frac{\paren[\big]{(k-1)j}!}{(k-1)!^j j!}
  \gcacti\paren[\big]{N-1,n-1+(k-1)j}.
\end{equation*}
Note that $0 \leq n-1+(k-1)j\leq N-1$ since $j\in[0, \mcacti]$.  The new
number of edges is $m'' = \frac{1}{k-1}((N-1) - (n - 1 +(k-1)j))
= \mcacti - j$. So, by induction hypothesis,
\begin{equation*}
  \begin{split}
    &\gcacti(N,n)
    =
    \sum_{j=0}^\mcacti
    \binom{N-n}{(k-1)j} 
    \frac{\paren[\big]{(k-1)j}!}{(k-1)!^j j!}\cdot
    \frac{(n-1+(k-1)j) (N-n-(k-1)j)! (N-1)^{\mcacti-j-1}}{(\mcacti-j)! (k-1)!^{\mcacti-j}}
    \\
    &=
    \frac{(N-n)!}{(N-1)(k-1)!^\mcacti}
    \sum_{j=0}^\mcacti
    \frac{(n-1+(k-1)j)(N-1)^{\mcacti-j}}{j!(\mcacti-j)!}
    \\
    &=
    \frac{(N-n)!}{\mcacti!(N-1)(k-1)!^\mcacti}
    \sum_{j=0}^\mcacti
    \binom{\mcacti}{j}(n-1+(k-1)j)(N-1)^{\mcacti-j}
    \\
    &=
   \frac{(N-n)!}{\mcacti!(N-1)(k-1)!^\mcacti}
   \left(
     (n-1)  \sum_{j=0}^\mcacti\binom{\mcacti}{j}(N-1)^{\mcacti-j}
     +
     (k-1) \sum_{j=0}^\mcacti \binom{\mcacti}{j}j(N-1)^{\mcacti-j}
   \right).
 \end{split}
\end{equation*}
Using the Binomial Theorem,
\begin{equation*}
  \sum_{j=0}^\mcacti\binom{\mcacti}{j}(N-1)^{\mcacti-j} = N^\mcacti  
\end{equation*}
and by differentiating both sides with respect to $N$,
\begin{equation*}
  \sum_{j=0}^\mcacti\binom{\mcacti}{j}(\mcacti-j)(N-1)^{\mcacti-j-1} = \mcacti N^{\mcacti-1},  
\end{equation*}
and so
\begin{equation*}
  \begin{split}
    \sum_{j=0}^\mcacti\binom{\mcacti}{j}j(N-1)^{\mcacti-j} 
    &=
    \mcacti\sum_{j=0}^\mcacti\binom{\mcacti}{j}(N-1)^{\mcacti-j} -
    \mcacti N^{\mcacti-1}(N-1)
    \\
    &=
    \mcacti N^\mcacti - \mcacti N^{\mcacti-1}(N-1)  =
    \mcacti N^{\mcacti-1}.  
  \end{split}
\end{equation*}
Hence,
\begin{equation*}
  \begin{split}
    \gcacti(N,n)
    &=
    \frac{(N-n)!}{\mcacti!(N-1)(k-1)!^\mcacti}
    \left(
      (n-1)N^\mcacti
      +
      (k-1)\mcacti N^{\mcacti-1}
    \right)
    \\
    &=\frac{(N-n)! N^{\mcacti-1}}{\mcacti!(N-1)(k-1)!^\mcacti}
    \left(
      N(n-1)
      +
      N-n
    \right)
    \\
    &=
    \frac{n(N-n)!N^{\mcacti-1}}{\mcacti!(k-1)!^\mcacti},
  \end{split}
\end{equation*}
and we are done.
\end{proof}

%%%%%%%%%%%%%%%%%%%%%%%%%%%%%%%%%%%%%%%%%%%%%%%%%%%%%%%%%%%%%%%%
% Tools
%%%%%%%%%%%%%%%%%%%%%%%%%%%%%%%%%%%%%%%%%%%%%%%%%%%%%%%%%%%%%%%%
\section{Tools}
In this section, we include some \New{definitions and }computations
that will be used a number of times throughout the proofs.

\New{
Given a nonnegative real number $\lambda$ and a nonnegative integer
$k$, we say that a random variable $Y$ is a \textdef{truncated
  Poisson}\index{truncated Poisson} with parameters~$(k,\lambda)$ if,
for every $j\in\setN$,
\begin{equation}
\label{eq:truncated-poisson-prelim}
\prob{Y=j} =
\begin{cases}
{\displaystyle \frac{\lambda^j}{j! \fpo{k}(\lambda)}}, &\text{ if }j\geq k;\\
0,&\text{ otherwise}.
\end{cases}
\end{equation}
where 
\begin{equation}
\label{eq:fk-def-prelim}
  \fpo{k}(\lambda) \eqdef
  e^{\lambda}-\sum_{i=0}^{k-1} \frac{\lambda^i}{i!}.
\end{equation}
We use $\tpo{k}{\lambda}$ to denote the distribution of a truncated
Poisson random variable with parameters $(k,\lambda)$. Throughout this
paper, we often use properties of truncated Poisson random variables
proved by Pittel and Wormald~\cite{PWa}.}

\New{Pittel and Wormald~\cite[Lemma 1]{PWa} showed that, for every $c
  > k$, there exists a unique positive real $\lambda$ such that
  \begin{equation}
    \label{eq:def:lambdakc-prelim}
    \frac{\lambda \fpo{k-1}(\lambda)}{\fpo{k}(\lambda)}
    = 
    c.
  \end{equation}
  Note that this implies that, for any $c > k$, there exists $\lambda
  >0$ such that the expectation of a random variable with distribution
  $\tpoisson{k}{\lambda}$ is $c$.}

\New{The first derivative of $\lambda
  \fpo{k-1}(\lambda)/\fpo{k}(\lambda)$ is obviously a continuous
  function and it is positive for $\lambda >0$ (as compute in the
  proof of~\cite[Lemma 1]{PWa}). From this, one obtains the following
  lemma:
\begin{lem}
  \label{lem:lambda-close-hyper}
  Let $\gamma$ and $k$ be positive integer constants with $\gamma >
  k$.  Let $\alpha(n), \beta(n)$ be function such that $k< \alpha(n)<
  \beta(n)< \gamma$ and $|\alpha(n)-\beta(n)| = o(\phi)$ where
  $\phi=o(1)$. Then $\abs{\lambdakcparam{k}{\alpha}
    -\lambdakcparam{k}{\beta}} = o(\phi)$.
\end{lem}}

Let $k$ be a positive integer. Let $c: \setR\to \setR$ so that $c(y) >
k$ for all $y\in\setR$. Let $\lambda(y)$ be defined by
\begin{equation*}
  \frac{\lambda(y) \fpo{k-1}(\lambda(y))}{\fpo{k}(\lambda(y))}
  =c(y).
\end{equation*}
The existence and uniqueness of $\lambda(y)$ follow
from~\textred{\cite[Lemma 1]{PWa}}. We compute the derivative
$\lambda'$ of $\lambda(y)$ by implicit differentiation. Assuming that
$c$ is differentiable with derivative $c'$:
\begin{equation}
  \label{eq:lambda-diff-hyper}
  \lambda'
  \frac{\fpo{k-1}(\lambda(y))}{\fpo{k}(\lambda(y)))}
  \left(1
    +\frac{\lambda(y) \fpo{k-2}(\lambda(y))}{\fpo{k-1}(\lambda(y))}
  -\frac{\lambda(y) \fpo{k-1}(\lambda(y))}{\fpo{k}(\lambda(y))}\right)
  =
  c'.
\end{equation}

Let $T, t: \setR\to \setR$ be differentiable functions be such that
$T(y)/t(y)>k$ for all $y\in \setR$. Let $t'$ and $T'$ denote the derivatives of $t$ and
$T$, resp. We will compute the derivative of $t(y)\log
\fpo{k}(\lambda(y)) - T(y)\log(\lambda(y))$. For $c(y) =T(y)/t(y)=
{\lambda(y) \fpo{k-1}(\lambda(y))}/{\fpo{k}(\lambda(y))}$ and
$\eta(y) = {\lambda(y)
  \fpo{k-2}(\lambda(y))}/{\fpo{k-1}(\lambda(y))}$, and
using~\eqref{eq:lambda-diff-hyper},
\begin{equation}
  \label{eq:difdeg}
  \begin{split}
    &\frac{\dif\paren[\Big]{t(y)\log f_k(\lambda(y)) - T(y)\log(\lambda(y))}}{\dif y}=
     \\
     &=
     t'\log f_{k}(\lambda(y)) 
     +
     \lambda'
     \frac{t(y)f_{k-1}(\lambda(y))}{f_{k}(\lambda(y))}
     -
     T' \log\lambda(y)
     - \lambda'\frac{T(y)}{\lambda(y)}
     \\
     &=
     t' \log f_{k}(\lambda(y))
     +
     t(\lambda)\frac{f_{k-1}(\lambda(y))}{f_{k}(\lambda(y))}
     \frac{\lambda(y) c'}
     {c(y)(1+\eta(y)-c(y))}
     \\
     &\quad- 
     T' \log\lambda(y) 
     - \frac{T(y)}{\lambda(y)}
     \frac{\lambda(y) c'}
     {c(y)(1+\eta(y)-c(y))}
     \\
     &=
     t'\log f_{k}(\lambda(y))
     +
     \frac{t(y)c'}{1+\eta(y)-c(y)}
     -
     T'\log\lambda(y) -
     \frac{t(y)c'}{1+\eta(y)-c(y)}
     \\
     &=
     t' \log f_{k}(\lambda(y))
     - T' \log\lambda(y).
   \end{split}
 \end{equation}

 The following lemma is an application of standard results concerning
 Gaussian functions and the definition of Riemann integral.
 \begin{lem}
   \label{lem:integral-tools}
   Let $\phi(n) \to 0$, $\psi(n) \to 0$, $T_n \to \infty$ and $s_n \to
   \infty$.  Let $f_n = \exp(-\alpha x^2 + \beta x + \phi x^2 +\psi
   x)$ with constants $\alpha > 0$ and $\beta$. Let $\Pcal_n =
   z+\setZ$, where $z\in\setR$. Then
 \begin{equation*}
   \frac{1}{s_n}\sum_{\substack{x\in \Pcal_n/s_n\\ |x|\leq T_n}}
     f_n(x)
     \sim
     \exp\paren[\Big]{\frac{\beta^2}{4\alpha}}\sqrt{\frac{\pi}{\alpha}}.
   \end{equation*}
 \end{lem}
 \begin{proof}
   Let $\eps \in (0,\min(\alpha,\beta))$ and let $f^+(x) = \exp(-\alpha x^2 + \beta x +\eps
   x^2+\eps x)$ and $f^-(x) = \exp(-\alpha x^2 + \beta x -\eps x^2-\eps x)$. Since
   $\phi = o(1)$ and $\psi=o(1)$, we may assume $f^-(x)\leq f_n(x) \leq
   f^+(x)$. We will show that
   \begin{equation}
     \label{eq:upper-riemann-tools}
     \frac{1}{s_n}\sum_{\substack{x\in \Pcal_n/s_n\\ |x|\leq T_n}}
     f^+(x)
     \sim
     \exp\paren[\Big]{\frac{(\beta+\eps)^2}{4(\alpha+\eps)}}\sqrt{\frac{\pi}{\alpha+\eps}}
   \end{equation}
   and
   \begin{equation}
     \label{eq:lower-riemann-tools}
     \frac{1}{s_n}\sum_{\substack{x\in \Pcal_n/s_n\\ |x|\leq T_n}}
     f^-(x)
     \sim
     \exp\paren[\Big]{\frac{(\beta-\eps)^2}{4(\alpha-\eps)}}\sqrt{\frac{\pi}{\alpha-\eps}}.
   \end{equation}
   Since we can choose $\eps$ arbitrarily close to zero, this proves
   the lemma. We will only show the proof
   for~\eqref{eq:upper-riemann-tools} since the proof
   for~\eqref{eq:lower-riemann-tools} is very similar.  We have that
   \begin{equation*}
     \int_{-\infty}^{\infty} f^+(x) dx
     =
     \lim_{C\to\infty}  \int_{-C}^{C} f^+(x) dx
   \end{equation*}
   and
   \begin{equation*}
     \lim_{-\infty}^{\infty} f^+(x) dx
     =
     e^{\frac{(\beta+\eps)^2}{4(\alpha+\eps)}}\sqrt{\frac{\pi}{\alpha+\eps}}.
   \end{equation*}
   So it suffices to show that, 
   \begin{equation*}
     {\Bigg|}
     \lim_{n\to\infty}  \frac{1}{s_n}\sum_{\substack{x\in \Pcal_n/s_n\\ |x|\leq T_n}}
        f^+(x)
        -
        \lim_{C\to\infty}  \int_{-C}^{C} f^+(x) dx
      {\Bigg|}
      = 0.
   \end{equation*}
   We have that
   \begin{equation*}
     \begin{split}
       &{\Bigg|}
     \lim_{n\to\infty}  \frac{1}{s_n}\sum_{\substack{x\in \Pcal_n/s_n\\ |x|\leq T_n}}
        f^+(x)
        -
        \lim_{C\to\infty}  \int_{-C}^{C} f^+(x) dx
      {\Bigg|}
      \leq
      \\
      &\qquad{\Bigg|}
      \lim_{n\to\infty}  \frac{1}{s_n}\sum_{\substack{x\in \Pcal_n/s_n\\ |x|\leq T_n}}
        f^+(x)
        -
        \lim_{C\to\infty}  
        \lim_{n\to\infty}  \frac{1}{s_n}\sum_{\substack{x\in \Pcal_n/s_n\\ |x|\leq C}}
        f^+(x)
      {\Bigg|}
      \\
      & \qquad+
      {\Bigg|}
        \lim_{C\to\infty}  
        \lim_{n\to\infty}  \frac{1}{s_n}\sum_{\substack{x\in \Pcal_n/s_n\\ |x|\leq C}}
        f^+(x)
        -
        \lim_{C\to\infty}  \int_{-C}^{C} f^+(x) dx
      {\Bigg|}
     \end{split}
   \end{equation*}
   the last term goes to is zero by the definition of Riemann
   integral. It is known that the tail for Gaussian functions is very
   small. More precisely, for each $\eps' >0$ there exists $n_0$ such
   that, for each $n\geq n_0$,
   \begin{equation*}
     {\Bigg|}
     \frac{1}{s_n}\sum_{\substack{x\in \Pcal_n/s_n\\ |x|\leq T_n}}
     f^+(x)
     -
     \lim_{n\to\infty}  \frac{1}{s_n}\sum_{\substack{x\in \Pcal_n/s_n\\ |x|\leq T_{n_0}}}
     f^+(x)
     {\Bigg|}
\leq \eps'.
   \end{equation*}
   Since $C\to\infty$, $C$ is eventually bigger than $T_{n_0}$. And we
   are done since we can choose $\eps'>0$ arbitrarily small.
 \end{proof}

%%%%%%%%%%%%%%%%%%%%%%%%%%%%%%%%%%%%%%%%%%%%%%%%%%%%%%%%%%%%%%%%
% Cores
%%%%%%%%%%%%%%%%%%%%%%%%%%%%%%%%%%%%%%%%%%%%%%%%%%%%%%%%%%%%%%%%
\section{Counting cores}
\label{sec:core-hyper}

\Old{In this section we obtain an asymptotic formula for the number of
  cores (not necessarily connected) with vertex set $[n]$ and
  $m=n/2+R$ edges, when $R= \omega(\log n)$ and $R=o(n)$. We also
  obtain an upper bound for the number of such cores when $R\to
  \infty$.} \New{In this section we obtain an upper bound for the
  number of cores with vertex set $[n]$ and $m=n/2+R$ edges, when
  $R\to \infty$.} We remark that the asymptotics in this section are
for $n\to \infty$. We will always use $r$ to denote $R/n$.

For $\nn\in\setR$, define
\begin{align*}
  &\ntwo(\nn) = n -\nn,
  % \quad&\corentwo(n) = \ntwo(n)/n;
  \\
  &\mthree(\nn) = m -\nn,
  % \quad&\coremthree(\nncore) = \mthree(\nn)/n;
  \\
  &\Qtwo(\nn) = 3m-\nn,
  %\quad&\coreQtwo(\nncore) = \Qtwo(\nn)/n;
  \\
  &\ctwo(\nn) = \Qtwo(\nn)/\ntwo(\nn) = (3m-\nn)/(n-\nn).
\end{align*}
For any symbol $y$ in this section, we use $\hat y$ to denote $y/n$.

We will use $\nn$ as the number of vertices of degree~$1$ in the
core. Then  $\ntwo(\nn)$ is the number of vertices of degree at
least~$2$, $\mthree(\nn)$ is the number of $3$-edges, $\Qtwo(\nn)$ is
the sum of degrees of vertices of degree at least~$2$, and
$\ctwo(\nn)$ is the average degree of the vertices of degree at
least~$2$. We omit the argument $\nn$ when it is obvious from the
context.

Let $J_m$ denote the set of reals $\nn$ such that
$\max\set{0,2n-3m}\leq \nn\leq \min\set{n,m}$.  The lower bound
$2n-3m$ is used to ensure that $\ctwo(\nn) \geq 2$ for $\nn\in J_m$.
Let $\coreJ_m = \set{x/n: x\in J_m}$, that is, $\coreJ_m$ is a scaled
version of~$J_m$.  For $\nn\in J_m\setminus\set{2n-3m}$, let
$\lambda_{\nn}$ be the unique positive solution of
\begin{equation}
  \label{eq:lambdacore-hyper}
  \frac{\lambda \f(\lambda)}{\ff(\lambda)}
  =
  \ctwo(\nn).
\end{equation}
Such a solution exists and is unique since $\ctwo(\nn) =
(3m-\nn)/(n-\nn)$ and $\nn> 2n-3m$ ensures that $3m-\nn> 2(n-\nn)$
(see~\textred{\cite[Lemma 1]{PWa}}). By continuity reasons, we define
$\lambda_{2n-3m} = 0$.

Let
\begin{equation}
  \label{eq:etatwo-def-hyper}
  \etatwo(\nn)=
  \frac{\lambda_{\nn}\exp(\lambda_{\nn})}{\f(\lambda_{\nn}).}  
\end{equation}

Let $\hcore(x) = x\ln(xn)-x$ and define, for $\nncore$ in the interior
of $\coreJ_m$,
\begin{equation}
  \label{eq:fcore-hyper}
  \begin{split}
    \fcore(\nncore)
    =
    &\hcore(\coreQtwo(\nn))
    - \hcore(\corentwo(\nn))
    - \hcore(\nncore)
    - \hcore(\coremthree(\nn))
    \\
    &- \nncore \ln(2)
    - \coremthree(\nn)\ln(6)
    \\
    &+ \corentwo(\nn)\ln(\ff(\lambda_{\nn}))
    -\coreQtwo(\nn)\ln(\lambda_{\nn}),
  \end{split}
\end{equation}
We extend the definition $\fcore$ to $\coreJ_m$ by setting the
$\fcore(\nncore)$ to be the limit of $\fcore(x)$ as $x\to \nncore$,
for the points $\nncore\in \coreJ_m\cap \set{0,1,\mcore,
  2-3\mcore}$. For all points in
$\coreJ_m\cap\set{0,1,\mcore,2-3\mcore}$ except $2-3\mcore$, this only
means that $0\log 0$ should be interpreted as~$1$. For $\nncore =
2-3\mcore$, as we already mentioned, $\lambda_{2n-3m}=0$ by continuity
reasons. But then $\corentwo(\nn)\ln(\ff(\lambda_{\nn}))
-\coreQtwo(\nn)\ln(\lambda_{\nn})$ is not defined (and note that
$\corentwo(\nn)\ln(\ff(\lambda_{\nn}))$ and
$\coreQtwo(\nn)\ln(\lambda_{\nn})$ appear in the definition of
$\fcore$). For $\nncore = 2-3\mcore$,
\begin{equation*}
  \begin{split}
    &\lim_{\lambda\to 0}
    \left( \corentwo(\nn)\ln(\ff(\lambda))
      -\coreQtwo(\nn)\ln(\lambda)\right)
    =
    \corentwo(\nn)\lim_{\lambda\to 0}
    \left(\ln(\ff(\lambda))
      -2\ln(\lambda)\right)
    \\
    &=
    \corentwo(\nn)
    \lim_{\lambda\to 0}
    \left(\ln\left(\frac{\exp(\lambda)-1-\lambda}{\lambda^2}\right)\right)
    =
    \corentwo(\nn)\ln \left(\frac{1}{2}\right).
  \end{split}
\end{equation*}
Thus,
\begin{equation}
  \label{eq:val-extreme-fcore-hyper}
  \fcore(2-3\mcore)
  =
    \hcore(\coreQtwo)
    - \hcore(\corentwo)
    - \hcore(\nncore)
    - \hcore(\coremthree)
    - \nncore \ln(2)
    - \coremthree\ln(6)
    - \corentwo \ln 2.
\end{equation}

We will show that the $n! \exp(n\fcore(\nncore))$ approximates the
exponential part of the number of cores with $\nn$ vertices of degree
$1$. \Old{Recall that $\gcore(n,m)$ is the number of cores with vertex set
$[n]$ and $m$ edges. We obtain the following result for $\gcore$:}
\New{We obtain an upper bound for the number of cores with
  vertex set $[n]$ and $m$ edges.}
\begin{thm}
  \label{thm:corevalue-hyper}
  Let $m(n) = n/2+R$ with $R \to \infty$. There exists a constant
  $\alpha$ such that, for $n\geq 1$, \New{the number of cores with vertex
  set  $[n]$ and $m$ edges is at most}
  \begin{equation}
    \label{eq:upper-formula-core-hyper}
    \gcore(n,m) \Old{\leq}\New{:=}
    \alpha n\sqrt{m}\cdot n!
    \exp\paren[\Big]{n\fcore(\nncopt)}.
  \end{equation}
\Old{ If $R = o(n)$ and $R =\omega(\log n)$, we have that
  \begin{equation}
    \label{eq:formula-core-hyper}
    \gcore(n,m) \sim
          \frac{1}{2\pi n \sqrt{r}}
          \cdot n!
      \exp\paren[\Big]{
        n\fcore(\nncopt)},
  \end{equation}}
  where $\nncopt = {3m}/\paren{n\fgg(\lambdaopt)}$ and $\lambdaopt$ is
  the unique positive solution of
  \begin{align}
    \label{eq:m-sol-core}&
    \frac{\lambda \f(\lambda)\fgg(\lambda)}{\FF(2\lambda)}
    = 
    \frac{3m}{n}.
  \end{align}
\end{thm}
 We will show that the
point $\nncopt$ maximizes $\fcore$ in $\coreJ_m$. \Old{The result
in~\eqref{eq:formula-core-hyper} will then obtained by expanding the
summation around $\nncopt$.}

For all subsections of this section, let $\Sigma_{\nn}$ denote the event
that a random vector $\Ys=(Y_1,\dotsc,Y_{n-\nn})$ satisfies
$\sum_{i} Y_i = 3m-\nn$, where the $Y_i$'s are independent random
variables with truncated Poisson distribution
$\tpoisson{2}{\lambda_{\nn}}$. Also, whenever symbols $y$ and
$\hat y$ appear in the same computation, $\hat y$ denotes $y/n$.

\subsection{Random cores}
\label{sec:random-core-hyper}
\Old{Recall that our aim in Section~\ref{sec:core-hyper} is to find an
  asymptotic formula for $\gcore(n,m)$.}\New{Recall that our aim in
  Section~\ref{sec:core-hyper} is to obtain the upper bound
  $\gcore(n,m)$ for the number of cores with vertex set $[n]$ and $m$
  edges.} Note that up
to this point there is no random graph involved in the problem.
However, \Old{similarly to Chapter~\ref{chap:2c}
(Section~\ref{sec:magic-2c}),} we show how to reduce the asymptotic
enumeration problem to approximating the expectation, in a probability
space of random sequences $\Ys$, of the probability that a certain
type of random multigraph with given degree sequence $\Ys$ is simple.

For integer $\nn \in J_{m}$, let $\Dcal_{\nn}$ be the set of all
$\ds\in(\setN\setminus\set{0,1})^{n-\nn}$ with $\sum_{i=1}^{n-\nn} d_i
= 3m-\nn$. For $\nn \in J_m\cap \setZ$ and $\ds\in\Dcal_{\nn}$, let
$\Gcore (\nn,\ds) = \Gcore_{n,m}(\nn, \ds)$ be the multigraph obtained
by the following procedure.\label{glo:gcore} We will start by creating
for each edge one bin/set with $3$ points inside it. These bins are
called \textdef{edge-bins}. We also create one bin for each vertex
with the number of points inside it equal to the degree of the
vertex. These bins are called \textdef{vertex-bins}. Each point in a
vertex-bin will be matched to a point in an edge-bin with some
constraints.  The multigraph can then be obtained by creating one edge
for each edge-bin $i$ such that the vertices incident to the edge are
the vertices with points matched to the edge-bin of $i$. We describe
the procedure in detail now. In the following, in each step, every
choice is made \uar\ among all possible choices satisfying the stated
constraints:
\begin{enumerate}
\item\textit{(Edge-bins)} For each $i\in [m]$, create an edge-bin $i$
  with $3$ points labelled $1,2$ and~$3$.
\item\textit{(Vertex-bins)} Choose a set $V_1$ of $\nn$ vertices in
  $[n]$ to be the vertices of degree~$1$. For each $v\in V_1$, create
  one vertex-bin $v$ with one point inside each.  Let $v_1<\dotsm<
  v_{n-\nn}$ be an enumeration of the vertices in $[n]\setminus
  V_1$. For each $i\in [n-\nn]$ create a vertex-bin $v_i$ with $d_i$
  points.
\item\textit{(Matching)} Match the points from the vertex-bins to the
  points in edge-bins so that each edge-bin has at most one point
  being matched to a point in a vertex-bin of size~$1$. This matching
  is called a \textdef{configuration}.
\item{\textit{(Multigraph)}} $\Gcore(\nn,\ds) = ([n],[m],\Phi)$, where
  $\Phi(i,j) = v$, where $v$ is the vertex-bin containing the point
  matched to~$j$.
\end{enumerate}
See Figure~\ref{fig:core-bins-hyper}, for an example for the procedure
described above.

\begin{figure}
  \centering
  \begin{tikzpicture}
    \def \initialg {0};
    \def \step {1};
    \def \rad {2pt};
    \def \eps {.3};
    \def \epsm {.15};
    \def \round {5pt};

    \coordinate[label={[label distance=1.5pt]-90:\footnotesize{$5$}}] (1) at (\initialg,\initialg);
    \coordinate (1ul) at ($(1)+(-\eps,1.5*\eps)$);
    \coordinate[label={[label distance=-1.5pt]0:\footnotesize{$7$}}]
    (2) at (\initialg+\step,\initialg);
    \coordinate[label={[label distance=4pt]90:\footnotesize{\textit{3}}}]
    (mid12) at ($.5*(2)+.5*(1)$);
    \coordinate[label={[label distance=1.5pt]90:\footnotesize{$3$}}] (3) at (\initialg+2*\step,\initialg);
    \coordinate (3dr) at ($(3)+(\eps,-1.5*\eps)$);
    \coordinate[label= {[label distance=1.5pt]0:\footnotesize{$4$}}] (4) at (\initialg+\step,\initialg+\step);
    \coordinate (4ul) at ($(4)+(-1.5*\eps,\eps)$);
    \coordinate[label={[label distance=1.5pt]180:\footnotesize{$6$}}]
    (5) at (\initialg+\step,\initialg-1*\step);
    \coordinate[label={[label distance=3.5pt]0:\footnotesize{\textit{4}}}]
    (mid52) at ($.5*(2)+.5*(5)$);
    \coordinate (5dr) at ($(5)+(1.5*\eps,-\eps)$);
    \coordinate[label=right:\footnotesize{$2$}] (6) at (\initialg,\initialg-1*\step);
    \coordinate[label=below:\footnotesize{\textit{2}}] (6dl) at ($(6)+(-.75*\eps,-.75*\eps)$);
    \coordinate[label=left:\footnotesize{$1$}] (7) at (\initialg+2*\step,\initialg+\step);
    \coordinate[label=above:\footnotesize{\textit{1}}] (7ur) at ($(7)+(0.75*\eps,0.75*\eps)$);
    \coordinate (b1) at ($1/3*(1)+1/3*(5)+1/3*(6)$);

     \foreach \point in {1,2,3,4,5,6,7} \fill [black] (\point) circle
    (\rad);

    \draw (2) ellipse ({\step*1.2} and {\step/4}); %123
    \draw (2) ellipse ({\step/4} and {\step*1.2}); %425
     \draw[rounded corners=\round] (1ul)--(5dr)--(6dl)--cycle;
     \draw[rounded corners=\round] (4ul)--(3dr)--(7ur)--cycle;
     
    \coordinate  (startarrow) at ($(1)+(-2*\step,0)$);
    \coordinate  (endarrow) at ($(1)+(-\step,0)$);
    \draw [->] (startarrow)--(endarrow);

    \coordinate (edgeorigin) at ($(startarrow)-(\step,0)$);

    \coordinate (1ed1) at ($(edgeorigin)-(0,2*\step)$);
    \coordinate (2ed1) at ($(1ed1)+(0,0.3*\step)$);
    \coordinate (3ed1) at ($(2ed1)+(0,0.3*\step)$);

    \coordinate (1ed2) at ($(3ed1)+(0,0.6*\step)$);
    \coordinate (2ed2) at ($(1ed2)+(0,0.3*\step)$);
    \coordinate (3ed2) at ($(2ed2)+(0,0.3*\step)$);

    \coordinate (1ed3) at ($(3ed2)+(0,0.6*\step)$);
    \coordinate (2ed3) at ($(1ed3)+(0,0.3*\step)$);
    \coordinate (3ed3) at ($(2ed3)+(0,0.3*\step)$);

    \coordinate (1ed4) at ($(3ed3)+(0,0.6*\step)$);
    \coordinate (2ed4) at ($(1ed4)+(0,0.3*\step)$);
    \coordinate (3ed4) at ($(2ed4)+(0,0.3*\step)$);

    \coordinate (1v1) at ($(3ed4)-(1.3\step,0)$);

    \coordinate (1v2) at ($(1v1)-(0,0.6*\step)$);

    \coordinate (1v3) at ($(1v2)-(0,0.6*\step)$);
    \coordinate (2v3) at ($(1v3)-(0,0.3*\step)$);

    \coordinate (1v4) at ($(2v3)-(0,0.6*\step)$);
    \coordinate (2v4) at ($(1v4)-(0,0.3*\step)$);

    \coordinate (1v5) at ($(2v4)-(0,0.6*\step)$);
    \coordinate (2v5) at ($(1v5)-(0,0.3*\step)$);

    \coordinate (1v6) at ($(2v5)-(0,0.6*\step)$);
    \coordinate (2v6) at ($(1v6)-(0,0.3*\step)$);

    \coordinate (1v7) at ($(2v6)-(0,0.6*\step)$);
    \coordinate (2v7) at ($(1v7)-(0,0.3*\step)$);
  
%Matching
\draw (1v1)--(3ed4);
\draw (1v2)--(3ed3);
\draw (1v3)--(2ed4);
\draw (1v4)--(1ed4);
\draw (1v5)--(2ed3);
\draw (1v6)--(1ed3);
\draw (2v3)--(3ed2);
\draw (2v4)--(3ed1);
\draw (2v5)--(2ed2);
\draw (2v6)--(2ed1);
\draw (1v7)--(1ed2);
\draw (2v7)--(1ed1);

\foreach \point in {1ed1, 2ed1, 3ed1} \fill [black] (\point) circle
(\rad);
\foreach \point in {1ed2, 2ed2, 3ed2} \fill [black] (\point) circle
(\rad);

\foreach \point in {1ed3, 2ed3, 3ed3} \fill [black] (\point) circle
(\rad);
\foreach \point in {1ed4, 2ed4, 3ed4} \fill [black] (\point) circle
(\rad);

\foreach \point in {1v1,1v2,1v3,2v3,1v4,2v4,1v5,2v5,1v6,2v6,1v7,2v7} \fill [black] (\point) circle
(\rad);
       
\draw (1v1) ellipse ({\epsm} and {\epsm}); 
\draw (1v2) ellipse ({\epsm} and {\epsm}); 

\draw ($.5*(1v3)+.5*(2v3)$) ellipse ({\epsm} and {0.15+\epsm}); 
\draw ($.5*(1v4)+.5*(2v4)$) ellipse ({\epsm} and {0.15+\epsm}); 
\draw ($.5*(1v5)+.5*(2v5)$) ellipse ({\epsm} and {0.15+\epsm}); 
\draw ($.5*(1v6)+.5*(2v6)$) ellipse ({\epsm} and {0.15+\epsm}); 
\draw ($.5*(1v7)+.5*(2v7)$) ellipse ({\epsm} and {0.15+\epsm}); 

\draw (2ed1) ellipse ({\epsm} and {0.3+\epsm}); 
\draw (2ed2) ellipse ({\epsm} and {0.3+\epsm}); 
\draw (2ed3) ellipse ({\epsm} and {0.3+\epsm}); 
\draw (2ed4) ellipse ({\epsm} and {0.3+\epsm}); 

% labels
\coordinate[label=left:\footnotesize{vertex-bins}] (vbins) at ($(1v1)+\step*(0,.5)$);
\coordinate[label=right:\footnotesize{edge-bins}] (ebins) at ($(3ed4)+\step*(0,.5)$);

\coordinate[label=left:\footnotesize{$1$}] (1v1l) at ($(1v1)+\step*(-0.2,0)$);
\coordinate[label=left:\footnotesize{$2$}] (1v2l) at ($(1v2)+\step*(-0.2,0)$);
\coordinate [label=left:\footnotesize{$3$}] (v3l) at ($.5*(1v3)+.5*(2v3)-\step*(0.2,0)$);
\coordinate [label=left:\footnotesize{$4$}] (v4l) at ($.5*(1v4)+.5*(2v4)-\step*(0.2,0)$);
\coordinate [label=left:\footnotesize{$5$}] (v5l) at ($.5*(1v5)+.5*(2v5)-\step*(0.2,0)$);
\coordinate [label=left:\footnotesize{$6$}] (v6l) at ($.5*(1v6)+.5*(2v6)-\step*(0.2,0)$);
\coordinate [label=left:\footnotesize{$7$}] (v7l) at ($.5*(1v7)+.5*(2v7)-\step*(0.2,0)$);

\coordinate[label=right:\footnotesize{\textit{4}}] (e1l) at ($(2ed1) +\step*(0.2,0)$);
\coordinate[label=right:\footnotesize{\textit{3}}] (e2l) at ($(2ed2)+\step*(0.2,0)$);
\coordinate [label=right:\footnotesize{\textit{2}}] (e3l) at ($(2ed3)+\step*(0.2,0)$);
\coordinate [label=right:\footnotesize{\textit{1}}] (e4l) at ($(2ed4)+\step*(0.2,0)$);

  \end{tikzpicture}
  \caption{A core generated with vertex-bins and edge-bins}
\label{fig:core-bins-hyper}
\end{figure}
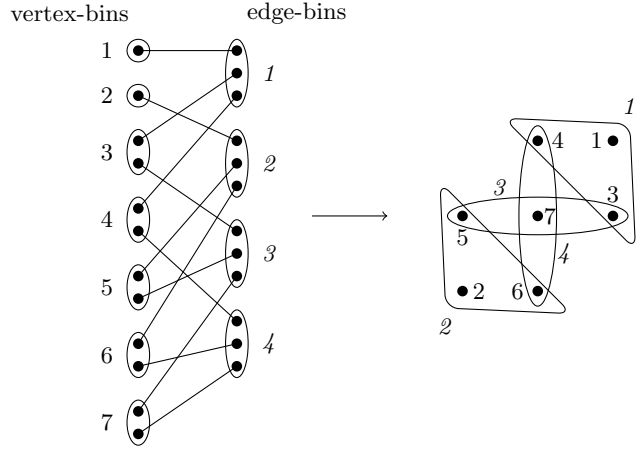

Let $\gcore(n,m,\nn)$ denote the number of cores with vertex-set $[n]$
with $m$ edges and $\nn$ vertices of degree~$1$, and let
$\gcore(n,m,\nn,\ds)$ denote the number of such cores with the
additional constraint that $\ds\in\setN^{n-\nn}$ is such that, given
the set $V_1$ of vertices of degree~$1$ and an enumeration
$v_1<\dotsc< v_{n-\nn}$ of the vertices in $[n]\setminus V_1$, the
degree of $v_i$ is $d_i$. We say that $\ds$ is the degree sequence for
the vertices of degree at least~$2$, although $\ds$ is not indexed by
the set of vertices of degree at least~$2$.  The following proposition
relates $\gcore(n,m,\nn)$ and $\gcore(n,m,\nn,\ds)$ to
$\Gcore_{n,m}(\nn, \ds)$ and $\Ys$. Recall that
$\Scal(n,m)$ is defined in Section~\ref{sec:def-hyper} as the set of
multigraphs with vertex set $[n]$ and $m$ edges corresponding to
simple graphs.  Let $U(\nn,\ds)$ denote the probability that
$\Gcore_{n,m}(\nn,\ds)\in\Scal(n,m)$.
\begin{prop}
  \label{prop:magic-core-hyper}
    We have that, for any integer $\nn\in J_m$
    \begin{equation}
      \label{eq:gcore-prob-hyper}
      \gcore(n,m,\nn,\ds)
      =
      n! \frac{\Qtwo(\nn)!}
      {\ntwo(\nn)!\nn!\mthree(\nn)! 2^{\nn} 6^{\mthree(\nn)}}
     \frac{1}{\prod_i d_i!}
      \prob[\big]{U(\nn,\ds)},
    \end{equation}
    and, for any integer $\nn\in J_m\setminus\set{2n-3m}$,
    \begin{equation}
      \label{eq:magic-hyper}
      \gcore (n,m,\nn) = n!  \frac{\Qtwo(\nn)!
        \ff(\lambda_{\nn})^{\ntwo(\nn)}} {\ntwo(\nn)!\nn!\mthree(\nn)!
        2^{\nn} 6^{\mthree(\nn)}\lambda_{\nn}^{\Qtwo(\nn)}}
      \meancond[\Big]{U(\nn,\Ys)}{\Sigma_{\nn}}
      \prob[\Big]{\Sigma_{\nn}},
    \end{equation}
    where $\Sigma_{\nn}$ is the event that a random vector
    $\Ys=(Y_1,\dotsc,Y_{n-\nn})$ satisfies $\sum_{i} Y_i = 3m-\nn$
    and the $Y_i$'s are independent random variables with truncated
    Poisson distribution $\tpoisson{2}{\lambda_{\nn}}$.
\end{prop}
\begin{proof}
  First we compute the total number of configurations that can be
  generated. There are $\binom{n}{\nn}$ ways of choosing the vertices
  of degree~$1$ in Step~2. We can split Step~3 by first choosing the
  $\nn$ edge-bins and one point in each of these edge-bins to be matched to
  the points inside vertex-bins of size~$1$.  There are
  $\binom{m}{\nn}3^{\nn}$ possible choices for these edge-bins and the
  points inside them. There are $\nn!$ ways of matching these points
  to the points in vertex-bins of size~$1$ and there are $\Qtwo(\nn)!$
  ways of matching the remaining points in the edge-bins to the
  vertex-bins of size at least $2$. Thus, the total number of
  configurations is
\begin{equation}
  \label{eq:total-config-core-hyper}
  \binom{n}{\nn}
 \binom{m}{\nn}3^{\nn}\nn!
 \Qtwo(\nn)!\eqdefinv \beta.
\end{equation}
It is straightforward to see that every multigraph with degree
sequence $\ds$ for the vertices of degree at least~$2$ is generated by
$\prod_{i=1}^{n-\nn} d_i!$ configurations. Together with
Lemma~\ref{lem:number-graph-multigraph-hyper}, this implies that every
graph with degree sequence $\ds$ for the vertices of degree at
least~$2$ is generated by
\begin{equation}
  \label{eq:prop-magic-aux-2-hyper}
 \alpha = m!6^m \prod_{i=1}^{n-\nn} d_i!
\end{equation}
configurations. Thus, since each configuration is generated with the
same probability, 
\begin{equation}
  \label{eq:prop-magic-aux-hyper}
  \gcore(n,m,\nn,\ds)
  =\frac{\beta}{\alpha}
  U(\nn,\ds).
\end{equation}
Together with~\eqref{eq:total-config-core-hyper}
and~\eqref{eq:prop-magic-aux-2-hyper}, and trivial simplifications,
this implies~\eqref{eq:gcore-prob-hyper}.

We now prove~\eqref{eq:magic-hyper}. \Old{The
proof is very similar to the proofs of
Propositions~\ref{prop:enum-pairing-2c}
and~\ref{prop:enum-kernel-config-2c} in Chapter~\ref{chap:2c} (which
in turn are very similar to the proof
of~\cite[Equation~(13)]{PWa}).}\New{The proof is very similar to~\cite[Equation~(13)]{PWa}.}
Recall that $\Dcal_{\nn}$ be the set of all
$\ds\in(\setN\setminus\set{0,1})^{\nthree(\nn)}$ with $\sum_{i} d_i =
\Qtwo(\nn)$. We have that
\begin{equation*}
  \begin{split}
    \gcore(n,m,\nn)
    &\eqdef
    \sum_{\ds\in\Dcal_{\nn}}
    \gcore(n,m,\nn,\ds)
    \\
    &=
    n! \sum_{\ds\in\Dcal_{\nn}} 
     \frac{\Qtwo(\nn)!}
      {\ntwo(\nn)!\nn!\mthree(\nn)! 2^{\nn} 6^{\mthree(\nn)}}
     \frac{1}{\prod_{i=1}^{\nthree(\nn)} d_i!}
     U(\nn,\ds)
    \\
    &=
    n! 
     \frac{\Qtwo(\nn)!}
      {\ntwo(\nn)!\nn!\mthree(\nn)! 2^{\nn} 6^{\mthree(\nn)}}
      \frac{\ff(\lambda_{\nn})^{\ntwo(\nn)}}{\lambda_{\nn}^{\Qtwo(\nn)}}
      \sum_{\ds\in\Dcal_{\nn}} U(\nn,\ds) 
      \prod_{i=1}^{\nthree(\nn)} \frac{ \lambda_{\nn}^{d_i}}
      {d_i! \ff(\lambda_{\nn})}
   \\
   &=
   n! 
   \frac{\Qtwo(\nn)!}
   {\ntwo(\nn)!\nn!\mthree(\nn)! 2^{\nn} 6^{\mthree(\nn)}}
   \frac{\ff(\lambda_{\nn})^{\ntwo(\nn)}}{\lambda_{\nn}^{\Qtwo(\nn)}}
   \sum_{\ds\in\Dcal_{\nn}} U(\nn,\ds)\prob{\Ys=\ds}
   \\
   &=
   n! 
   \frac{\Qtwo(\nn)!}
   {\ntwo(\nn)!\nn!\mthree(\nn)! 2^{\nn} 6^{\mthree(\nn)}}
   \frac{\ff(\lambda_{\nn})^{\ntwo(\nn)}}{\lambda_{\nn}^{\Qtwo(\nn)}}    
   \meancond{U(\nn,\Ys)}{\Sigma_{\nn}}\prob{\Sigma_{\nn}},
  \end{split}
\end{equation*}
which proves~\eqref{eq:magic-hyper}. We remark that the only reason
why the above proof does not work for $\nn = 2n-3m$ (and so for the
whole $J_m\cap \setZ$) is that $\lambda_{2n-3m}=0$ (by continuity), which would
cause a division by zero in~\eqref{eq:magic-hyper}.
\end{proof}

\Old{The next lemma gives conditions which are sufficient for the
expectation in~\eqref{eq:prop-magic-aux-hyper} to be asymptotic
to~$1$.
\begin{lem}
  \label{lem:simple-core-hyper}
  Let $m(n) = n/2+R$ with $R\to \infty$ and $R = o(n)$. Uniformly for
  $\nn=\nn(n)\in J_m\cap \setZ$, we have that 
  \begin{equation*}
    \meancond[\Big]{U(\nn,\Ys)}{\Sigma_{\nn}}\sim
    1.
  \end{equation*}
\end{lem}
\begin{proof}
  Recall that $\Dcal_{\nn}$ is the set of all
  $\ds\in(\setN\setminus\set{0,1})^{\nthree(\nn)}$ with $\sum_{i} d_i
  = \Qtwo(\nn)$. For a constant $\eps\in (0,1/6)$, let $\Dcal_{\nn}'$
  be the subset of $\Dcal_{\nn}$ such that $\ds\in\Dcal_{\nn}'$ if
  $d_i\leq n^{\eps}$ for every $i$.  We will show that $\probcond{\Ys
    \in\Dcal_{\nn}'}{\Sigma_{\nn}}=1+o(1)$, that is, $\Dcal_{\nn}'$
  contains all `typical' degree sequences. We then show that the
  degree sequences not in $\Dcal_{\nn}'$ have no significant
  contribution to the expectation.
% Let  $\Dcal''=\Dcal_{\nn}'' = \Dcal_{\nn}\setminus \Dcal_{\nn}'$.

  First we consider the case $\nn = 2n-3m = n/2-3R$. In this case the
  only degree sequence in $\Dcal_{\nn}$ is the sequence of all $2$'s,
  which is in $\Dcal_{\nn}'$. So suppose $\nn > 2n-3m$.  Since $R=o(n)$ and
  $\nn\leq m = n/2+R$, we have that $3m-\nn\sim 2(n-\nn)$ and so
  $\lambda_{\nn}\to 0$, Thus, for any $j \to \infty$, by computing the
  series of $\ff(\lambda_{\nn})$ with $\lambda_{\nn}\to 0$,
  \begin{equation*}
    \prob{Y_i\geq j}
    =
    \frac{\lambda_{\nn}^j}{j! \ff(\lambda)}
    (1+O(\lambda_{\nn}))
    =
    \frac{2\lambda_{\nn}^{j-2}}{j!}(1+O(\lambda_{\nn}))
    =
    o\paren[\Big]{
      \exp\paren[\big]{- \alpha j\log j}
    },
  \end{equation*}
  for a positive constant $\alpha$. Thus, by the union bound
  \begin{equation*}
    \prob{\max_i Y_i \geq n^\eps}
    \leq
    \ntwo
    \cdot 
    o\paren[\Big]{
      \exp\paren[\big]{- \alpha\eps n^\eps\log n}
    }
    =
    o\paren[\Big]{
      \exp\paren[\big]{- \alpha' n^\eps\log n}
    },
  \end{equation*}
  for a positive constant $\alpha'$. 

  We estimate the probability of $\Sigma_{\nn}$. Let $\Rtwo =
  (3m-\nn)-2(n-\nn) = \nn-(2n-3m)$. If $\Rtwo = o({n_2}^{1/3})$,
  then, by~\textred{\cite[Theorem 4]{PWa}},
  \begin{equation*}
    \prob{\Sigma_{\nn}}
    \sim
    \frac{e^{-\Rtwo} \Rtwo^{\Rtwo}}{\Rtwo!}
    =
    \Omega\left(\frac{1}{\sqrt{\Rtwo}}\right)
  \end{equation*}
  by Stirling's approximation\crisc{Removed reference to appendix}. If
  $\Rtwo \to \infty$, by~\textred{\cite[Theorem 4]{PWa}},
  \begin{equation*}
    \prob{\Sigma_{\nn}}
    \sim
    \frac{1}{\sqrt{2\pi \ntwo \ctwo (1+\etatwo-\ctwo)}}
    =
    \Omega\left(
      \frac{1}{\sqrt{\ntwo}}
    \right)
  \end{equation*}
  where $\etatwo = \lambda_{\nn}
  \exp(\lambda_{\nn})/\f(\lambda_{\nn})$ and since $\ctwo
  (1+\etatwo-\ctwo)\sim \ctwo-2$ by~\textred{\cite[Lemma 2]{PWa}}. Thus, for $\Dcal_{\nn}'' \eqdef
  \Dcal_{\nn}\setminus \Dcal_{\nn}'$.
\begin{equation}
\label{eq:prob-typical-core-hyper}
  \probcond{\Ys\in\Dcal_{\nn}''}{\Sigma_{\nn}}
  \leq
  \frac{\prob{\Ys\in\Dcal_{\nn}''}}{\prob{\Sigma_{\nn}}}
  =
  O\left(\sqrt{n} \exp\paren[\big]{- \alpha' n^\eps\log n}\right)
  =
  o(1).
\end{equation}
Now we show that
\begin{equation}
  \label{eq:probsimple_core-hyper} 
 \prob{\Gcore(\nn,\ds) \text{ simple}}=1+o(1),
\end{equation}
for $\ds\in\Dcal_{\nn}$. We have to compute the probability that there
are no loops and no double edges in $\Gcore(\nn,\ds)$.

A loop arises from the edge-bins that have at least two points matched
to points in the same vertex-bin.  The expected number of loops is
\begin{equation*}
  \sum_{i}\binom{d_i}{2}\frac{(\Qtwo-2)!}{\Qtwo!}
  =
  O\left(
    \frac{mn^{\eps}}{n^2}
  \right)
  = o(1)
\end{equation*}
since $\Qtwo = 3m-\nn \geq n/2+3R = \Omega(n)$ and
$\ds\in\Dcal_{\nn}'$ (and $\eps < 1/6)$.  Double edges arise from
pairs edge-bins that the points in each of them are mapped to the same
vertex-bins with the same multiplicities.  Using that $\max_i d_i \leq
n^\eps$ for $\ds\in\Dcal_{\nn}'$, we have that the expected number of
double edges (not involving loops) is at most
\begin{equation*}
  \binom{\mthree}{2}
  \ntwo^3 (n^\eps)^6
  6!
  \frac{(\Qtwo-6)!}{\Qtwo!}
  =
  O\left(
    \frac{n^{5+6\eps}}{n^6}
  \right)
  =
  o(1),
\end{equation*}
since $\eps < 1/6$. (Note that $2$-edges cannot be double edges
because of the vertices of degree~$1$.) Thus, Markov's inequality
implies~\eqref{eq:probsimple_core-hyper}.

Since $\Dcal_{\nn}'$ is a finite set for each $n$
and~\eqref{eq:probsimple_core-hyper}, by Lemma~\ref{lem:unif-deg-seq},
we have that there exists $f(n)=1-o(1)$ such that
$\prob{\Gcore(\nn,\ds) \text{ simple}}\geq f$ for all
$\ds\in\Dcal_{\nn}'$. Thus,
\begin{equation*}
  \begin{split}
      \meancond{U(\Ys)}{\Sigma_{\nn}}
      &\leq
      \sum_{\ds\in\Dcal_{\nn}'} \prob{\Gcore(\nn,\ds)\text{ simple
        }}\prob{\Ys=\ds}
      + \probcond{\Ys\in\Dcal_{\nn}''}{\Sigma_{\nn}}
      \\
      &\geq
      \probcond{\Ys\in\Dcal_{\nn}'}{\Sigma_{\nn}}f(n)
      \probcond{\Ys\in\Dcal_{\nn}''}{\Sigma_{\nn}}
      =1-o(1),
  \end{split}
\end{equation*}
by~\eqref{eq:prob-typical-core-hyper}.
\end{proof}
}

\subsection{Proof of Theorem~\ref{thm:corevalue-hyper}}
\label{sec:thm-proof-cores-hyper}

\Old{In this section we present the proof of the asymptotic formula in
Theorem~\ref{thm:corevalue-hyper} for the number $\gcore(n,m)$ of cores
(not necessarily connected) with vertex set $[n]$ with $m=n/2+R$
edges, when $R= \omega(\log n)$ and $R=o(n)$: 
\begin{equation}  
  \label{eq:formula-core-aux-hyper}
  \gcore(n,m) \sim
  \frac{1}{2\pi n \sqrt{r}}
  \cdot n!
  \exp\paren[\Big]{
    n\fcore(\nncopt)},
\end{equation}
where $\nncopt = {3\mpre}/{\fgg(\lambdaopt)}$ and $\lambdaopt$ is
the unique positive solution for
\begin{equation}
  \label{eq:lambdaopt-def-aux-hyper}
  \frac{\lambda \f(\lambda)\fgg(\lambda)}{\FF(2\lambda)}
  = 
  \frac{3m}{n}.
\end{equation}
We also show the upper bound in Theorem~\ref{thm:corevalue-hyper} for
$\gcore(n,m)$ that holds as long as $R\to \infty$.} \New{In this
section we present the proof that the number of cores with vertex set
$[n]$ and $m$ edges is at most $\gcore(n,m)$, which is defined in
Theorem~\ref{thm:corevalue-hyper} as $\alpha n\sqrt{m}\cdot n!
\exp\paren[\Big]{n\fcore(\nncopt)}$, where $\alpha$ is a constant,
$\nncopt = {3m}/\paren{n\fgg(\lambdaopt)}$ and $\lambdaopt$ is the
unique positive solution of
\begin{equation}
  \label{eq:lambdaopt-def-aux-hyper}
  \frac{\lambda \f(\lambda)\fgg(\lambda)}{\FF(2\lambda)}
  = 
  \frac{3m}{n}.
\end{equation}
} First we show that
$\lambdaopt$ is well-defined.
 \begin{lem}
  \label{lem:unique-pre-tools-hyper}
  The equation ${\lambda \f(\lambda) \fgg(\lambda)}/{\FF(2\lambda)}
  = \alpha$ has a unique positive solution $\lambdaopt_\alpha$
  for any $\alpha>3/2$. Moreover, for any positive constant $\eps$,
  there exists a positive constant $\eps'$, such that, if $\alpha,\beta\in
  (0,\eps)$, then $|\lambdaopt_\alpha-\lambdaopt_\beta| \leq
  \eps'|\alpha-\beta|$.
\end{lem}
\begin{proof}
  It suffices to show that $f(\lambda) \eqdef {\lambda \f(\lambda)
    \fgg(\lambda)}/{\FF(2\lambda)}$ is a strictly increasing function
  of $\lambda$ with $\lambda >0$ and $\lim_{\lambda \to 0^+}
  f(\lambda) = 3/2$. \nickca{deleted: See Section~
  for a Maple spreadsheet.} By computing the series of $f(\lambda)$
  with $\lambda\to 0$, we obtain
  \begin{equation*}
    f(\lambda)= \frac{3}{2} +\frac{\lambda}{4} + O(\lambda^2).
  \end{equation*}
  The derivative of $f$ is
  \begin{equation*}
    \frac{\dif f(\lambda)}{\dif \lambda}
    =
    \frac{
      2+2e^{2\lambda}\lambda
      -e^{\lambda}\lambda
      -4e^{2\lambda}\lambda^2
      -e^{3\lambda}\lambda
      -2e^{\lambda}\lambda^2
      +e^{4\lambda}
      +e^{3\lambda}
      -3e^{2\lambda}
      -e^{\lambda}}
    {\FF(2\lambda)^2},
  \end{equation*}
  which we want to show that is positive for any $\lambda >0$. Let
  $F(\lambda)$ denote the numerator in the above. It suffices to show
  that $F(\lambda)$ is positive for $\lambda > 0$. Let $F^{(0)}= F$.
  We will use the following strategy: starting with $i=1$, we check
  that $F^{(i-1)}(0)\geq 0$ and compute the derivative $F^{(i)}$ of
  $F^{(i-1)}$. If for some $i$ we can show that $F^{(i)}(\lambda) > 0$
  for any $\lambda>0$, then we obtain $F(\lambda) > 0$ for
  $\lambda>0$. Otherwise, we try to simplify the derivative.  If
  $\exp(\lambda)$ appears in every term of $F^{(i)}$, we redefine
  $F^{(i)}$ by dividing it by $\exp(\lambda)$. Eventually, we obtain
  \begin{equation*}
    216 e^{2\lambda}
    -24 \lambda e^{\lambda}
    -44 e^{\lambda}
    -16\lambda
    -52,
  \end{equation*}
  which is trivially positive since $\exp(2x)\geq \exp(x) \geq 1+x$
  for $x\geq 0$ and the sum of the coefficients of the negative terms
  is less than $216$. 

  The proof of the second statement in the lemma follows trivially
  from the fact that the first derivative is always positive and, with
  $\lambda\to0$,
  \begin{equation*}
    \frac{\dif f(\lambda)}{\dif \lambda}
    =
    \frac{\lambda^4 + O(\lambda^5)}{4\lambda^4 + O(\lambda^5)}
    \to
    \frac{1}{4}>0.
  \end{equation*}
  \end{proof}
  Since $3\mpre = 3/2+3r$, for $r=o(1)$ we have that $\lambdaopt$ is
  well-defined and $\lambdaopt \to 0$, by
  Lemma~\ref{lem:unique-pre-tools-hyper}.
  
  Let
  \begin{equation}
    \label{eq:wcore-def-hyper}
    \wcore(\nn)
    =
    \begin{cases}
      n! {\displaystyle\frac{\Qtwo(\nn)! \ff(\lambda_{\nn})^{\ntwo(\nn)}}
        {\ntwo(\nn)!\nn!\mthree(\nn)! 2^{\nn}
          6^{\mthree(\nn)}\lambda_{\nn}^{\Qtwo(\nn)}}},& \text{if
      }\nn\in J_m\setminus\set{2n-3m};\\
      n! {\displaystyle\frac{\Qtwo(\nn)!}
       {\ntwo(\nn)!\nn!\mthree(\nn)! 2^{n}
         6^{\mthree(\nn)}}},&  \text{if }\nn = 2n-3m\in J_m.
  \end{cases}
\end{equation}
Then, Proposition~\ref{prop:magic-core-hyper} implies that
\begin{equation*}
  \begin{split}
    \gcore(n,m)
    &=
    \sum_{\nn \in J_m\setminus\set{2n-3m}}
    \wcore(\nn)
    \meancond{\Gcore(n,m,\nn,\Ys)\text{ simple}}{\Sigma_{\nn}}
    \prob{\Sigma_{\nn}}
    \\
    &+ \indicator_{2n-3m\in J_m} \wcore(2n-3m)
    \prob{\Gcore(n,m,\nn,\twos)\text{ simple}},
  \end{split}
\end{equation*}
where the last term comes from $\nn=2n-3m$ and
$\Dcal_{2n-3m}=\set{\twos}$.

Recall that $\hcore(x) = x\ln(xn)-x$ and 
\begin{equation*}
  \begin{split}
\fcore(\nncore)
    =
    &\hcore(\coreQtwo)
    - \hcore(\corentwo)
    - \hcore(\nncore)
    - \hcore(\coremthree)
    \\
    &- \nncore \ln(2)
    - \coremthree\ln(6)
    \\
    &+ \corentwo\ln(\ff(\lambda_{\nn}))
    -\coreQtwo\ln(\lambda_{\nn}).
  \end{split}
\end{equation*}
The function $n!\exp(n \fcore(\nncore))$ is an approximation for the
exponential part of $\wcore(\nn)$. We will analyse $\fcore$ and use it
to draw conclusions about $\wcore$. It will be useful to know the
asymptotic values of $\nncopt$ and some functions of it. In
Equation~\eqref{eq:lambdaopt-def-aux-hyper}, the RHS is $3m/n=3/2+3r$
and so we can write $r$ in terms of $\lambdaopt$.  Since $\nncopt$ is
defined as ${3\mpre}/{\fgg(\lambdaopt)}$, we can also write it in
terms of $\lambdaopt$ and so we can write $\Qtwo(\nnopt)$,
$\ntwo(\nnopt)$ and $\mthree(\nnopt)$ in terms of $\lambdaopt$ (and
$n$). As we have mentioned before, by
Lemma~\ref{lem:unique-pre-tools-hyper}, we have that $\lambdaopt\to
0$. By computing the series with $\lambdaopt \to 0$, we have that
\begin{equation}
  \label{eq:optcore-r-hyper}
  \begin{split}
    &\lambdaopt = 12r + O(r^2);
    \\
    &\nncopt = 1/2-r+O(r^2);
    \\
    &\Qtwo(\nnopt) = 3m-\nnopt = n+4R+o(R);
    \\
    &\ntwo(\nnopt) = n-\nnopt = n/2 + R +o(R);
    \\
    &\mthree(\nnopt) = m-\nnopt = 2R+o(R).
  \end{split}
\end{equation}

Next, we state the main lemmas for the proof of
Theorem~\ref{thm:corevalue-hyper}.\Old{We defer their proofs to
Section~\ref{sec:lemmas-core-hyper}.} First we show that $\nncopt$
achieves the maximum value for $\fcore$ in $\coreJ_m$.

\begin{lem}
  \label{lem:max-core} 
  The point $\nncopt$ is the unique maximum of the function
  $\fcore(\nncore)$ for $\nncore\in \coreJ_m$. Moreover, we have
  that $\fcore'(\nncopt) = 0$, and $\fcore'(\nncore)>0$ for $\nncore< \nncopt$
  and $\fcore'(\nncore) <0$ for $\nncore > \nncopt$.
\end{lem}
\New{
\begin{proof}[Proof of Lemma~\ref{lem:max-core}]
Using~\eqref{eq:difdeg} with $T = \coreQtwo$ and $t = \corentwo$, the
derivative of $\fcore(\nncore)$ is
\begin{equation}
  \label{eq:diffcore}
  -\ln(\coreQtwo)+\ln(\corentwo)-\ln(\nncore)+\ln(\coremthree)
  +\ln(3)- \ln \ff(\lambda) +\ln \lambda.
\end{equation}
The second derivative is
\begin{equation}
  \label{eq:ddiffcore}
  \frac{1}{\coreQtwo}-\frac{1}{\corentwo}-\frac{1}{\nncore}
  -\frac{1}{\coremthree}- \frac{(1-\corectwo)^2}{\coreQtwo(1+\coreetatwo-\corectwo)} < 0,
\end{equation} 
because $1/\coreQtwo < 1/\coremthree$.

  By setting the derivative of $\fcore$ in~\eqref{eq:diffcore} to $0$
  and using the definition of $\lambda_{\nn}$
  in~\eqref{eq:lambdacore-hyper}, we obtain the
  Equation~\eqref{eq:m-sol-core}, which has a unique positive solution
  $\lambdaopt$ by Lemma~\ref{lem:unique-pre-tools-hyper}
  The second derivative computation in~\eqref{eq:ddiffcore} implies
  that $\fcore$ is strictly concave and so $\lambdaopt$ is the unique
  maximum. 
\end{proof} 

}

\Old{Then we expand the summation around the maximum and approximate it
by an integral, which we compute, obtaining the following: 
\begin{lem}
  \label{lem:approx-core-hyper}
  Suppose that $\delta=o(r)$ and $\delta^2 =\omega(r/n)$, with
  $r=o(1)$.  We have that
  \begin{equation*}
    \sum_{\substack{x\in[-\delta n,\delta n]\\\nnopt +x\in\setZ}} \exp(n\fcore(\nncopt+\corex)
    \sim
    \sqrt{2\pi rn} \exp(n\fcore(\nncopt)).
  \end{equation*}
\end{lem}
 
Finally, we show that points far from the maximum do not contribute
significantly to the summation:
\begin{lem}
  \label{lem:tail-core-hyper}
  Suppose that $\delta=o(r)$ and $\delta^2 = \omega(r\log n/n)$ with
  $r=o(1)$. Then
  \begin{equation*}
    \sum_{\substack{\nn \in J_m\\ |\nn-\nnopt|>\delta n}} \wcore(\nn)
    =
    o\left(\frac{n!}{n\sqrt{r}}\exp(n\fcore(\nncopt))\right).
  \end{equation*}
\end{lem}
}

We are now ready to prove Theorem~\ref{thm:corevalue-hyper}. \Old{First we
prove~\eqref{eq:upper-formula-core-hyper}.} We discuss the relation of
$\wcore$ and $\fcore$ more precisely here.  The function
$n!\exp(n\fcore(\nncore))$ can be obtained from the definition of
$\wcore(\nn)$ in~\eqref{eq:wcore-def-hyper} as follows: replace
$\Qtwo(n_1)!$ by $\exp(\hcore(\coreQtwo))$, and do the same for
$\nn!$, $\ntwo(\nn)!$, and $\mthree(\nn)!$. That is,
$n!\exp(n\fcore(\nncore))$ can be obtained from $\wcore(\nn)$ be
replacing each factorial involving $\nn$ by its Stirling approximation
(but ignoring the polynomials terms).  By Stirling's
approximation\crisc{Removed reference to appendix}, there exists constants
$\alpha_1$ and $\alpha_2$ such that, for every $x\in\setN$, 
\begin{equation}
\label{eq:error-Stirling-core-hyper}
  \alpha_1\sqrt{x} \paren[\bigg]{\frac{x}{e}}^x 
  \leq x!\leq 
  \alpha_2\sqrt{x} \paren[\bigg]{\frac{x}{e}}^x,
\end{equation}
and so, there exists a constant $\alpha$ such that
\begin{equation*}
  \wcore(\nn)
  \leq
  \alpha\sqrt{m}
  \exp(n\fcore(\nncore)).
\end{equation*}
Together with Lemma~\ref{lem:max-core}, this immediately
implies~\eqref{eq:upper-formula-core-hyper}.

\Old{Now we will prove~\eqref{eq:formula-core-hyper}. So assume that $R =
o(n)$. In order to use Lemma~\ref{lem:approx-core-hyper} and
Lemma~\ref{lem:tail-core-hyper}, we need to choose $\delta =
\delta(n)$ that satisfies $\delta = o(r)$ and $\delta^2 = \omega(r
\log n / n)$. This is possible if and only if $r^2 = \omega(r\log n/
n)$. That is, if and only if $R=\omega(\log n)$, which is one of the
hypotheses of the theorem. Thus, let $\delta$ be such that $\delta =
o(r)$ and $\delta^2 = \omega(r\log n /n)$.

Let $J(\delta) = [\nnopt-\delta n,\nnopt+\delta n]\cap\setZ$. We have
that $2n-3m = n/2-3R$ is not in $J(\delta)$ because $\nnopt =
n/2-R+o(R)$ by~\eqref{eq:optcore-r-hyper} and $\delta n = o(R)$. By
Proposition~\ref{prop:magic-core-hyper} and
Lemma~\ref{lem:simple-core-hyper}, for $\nn(n)\in J(\delta)$,
\begin{equation}
  \label{eq:corevalue-aux1-hyper}
  \gcore(n,m,\nn)
  \sim
  \wcore(\nn)
  \prob{\Sigma_{\nn}}. 
\end{equation}
For any $\nn(n)\in J(\delta)$, we have that $\Qtwo(\nn)$, $\mthree(\nn)$,
$\ntwo(\nn)$ are all $\Omega(R)$ by~\eqref{eq:optcore-r-hyper} and
$\delta n = o(R)$. Thus, by~\eqref{eq:corevalue-aux1-hyper},
Stirling's approximation and the definition of $\fcore$, for $\nn(n)\in J(\delta)$,
\begin{equation*}
  \gcore(n,m,\nn)
  =
  \frac{1}{2\pi}
  \sqrt{\frac{\Qtwo(\nn)}
              {\ntwo(\nn) \nn \mthree(\nn)}
   }
   \cdot \prob{\Sigma_{\nn}}\cdot 
  n!\exp\paren[\Big]{n\fcore(\nncore)};
\end{equation*}
 Using~\eqref{eq:optcore-r-hyper} and
$\delta n = o(R)$, we obtain
\begin{equation}
  \frac{1}{2\pi}
  \sqrt{\frac{\Qtwo(\nn)}
    {\ntwo(\nn) \nn \mthree(\nn)}
  }
  =
  \frac{1}{2\pi}
  \sqrt{\frac{n+4R+o(R)}
    {(n/2+R+o(R)) (n/2-R+o(R)) (2R+o(R))}
  }
  \sim
\frac{1}{\pi n \sqrt{2 r}}
\end{equation}
and
\begin{equation*}
  \frac{\Qtwo(\nn)}{\ntwo(\nn)}
  =
  \frac{\Qtwo(\nnopt)}{\ntwo(\nnopt)}
  (1+o(r)),
\end{equation*}
for $\nn(n)\in J(\delta)$.  By Lemma~\ref{lem:lambda-close-hyper}, this
implies that $\lambda_{\nn} \sim \lambda_{\nnopt}$ for $\nn\in
J(\delta)$. Moreover, we have that $\Qtwo(\nn)-2\ntwo(\nn) = 2R+o(R)
\to \infty$ by~\eqref{eq:optcore-r-hyper} and so, by~\textred{\cite[Theorem 4]{PWa}},
\begin{equation}
  \begin{split}
    \prob{\Sigma_{\nn}}
    &\sim
    \left(2\pi \Qtwo(\nn) (1+\etatwo(\nn)-\ctwo(\nn))\right)^{-1/2}
    \\
    &=
    \left(2\pi (n+4R+o(R)) 
        \left(
          1  
          + \lambda_{\nn} \frac{\exp(\lambda_{\nn})}{\f(\lambda_{\nn})}
          - \lambda_{\nn} \frac{\f(\lambda_{\nn})}{\ff(\lambda_{\nn})}
        \right)\right)^{-1/2}
    \\
    &=
    \left(2\pi (n+4R+o(R)) 
        \left(
          1  
          +
          \frac{1+\lambda_{\nn}+O(\lambda_{\nn}^2)}
               {1+\lambda_{\nn}/2 + O(\lambda_{\nn}^2)}
          - \frac{1+\lambda_{\nn}/2 + O(\lambda_{\nn}^2)}
                 {1/2+\lambda_{\nn}/6 + O(\lambda_{\nn}^2)}
        \right)\right)^{-1/2}
   \\      
   &=
   \left(2\pi (n+4R+o(R)) 
     \left(\frac{\lambda_{\nn}}{6} + O(\lambda_{\nn}^2) \right)
   \right)^{-1/2}
   \\
   &=
   \left(2\pi (n+4R+o(R)) 
     \left(2r + O(r^2) \right)
   \right)^{-1/2}
   \sim
   \frac{1}{\sqrt{4\pi nr}},
    \end{split}
\end{equation}
 for $\nn(n)\in J(\delta)$. Thus, 
 \begin{equation}
   \label{eq:aux-main-core-hyper}
   \gcore(n,m,\nn) = (1+o(1))
   \frac{1}{\pi n \sqrt{2 r}}
    \cdot
    \frac{1}{\sqrt{4\pi nr}}\cdot
        n!\exp\paren[\Big]{n\fcore(\nncore)},
 \end{equation}
 for $\nn(n)\in J(\delta)$. Since $J(\delta)$ is a finite set for each
 $n$, by Lemma~\ref{lem:unif-deg-seq} there exists a function
 $q(n)=o(1)$ such that the $o(1)$ in~\eqref{eq:aux-main-core-hyper} is
 bounded in absolute value by $q(n)$. Thus, by
 Lemma~\ref{lem:approx-core-hyper},
\begin{equation*}
  \begin{split}
    \gcore(n,m)&=
    \sum_{\nn\in J(\delta)} \gcore(n,m,\nn)
    \sim
    \frac{1}{\pi n \sqrt{2 r}}
    \cdot
    \frac{1}{\sqrt{4\pi nr}}
    \sum_{\nn\in J(\delta)} 
    n!\exp\paren[\Big]{n\fcore(\nncore)}
    \\
    &\sim
    \frac{1}{\pi n \sqrt{2 r}}
    \cdot
    \frac{1}{\sqrt{4\pi nr}}
    \cdot
    \sqrt{2\pi rn} \exp(n\fcore(\nncopt))
    \\
    &\sim
    \frac{1}{2\pi n \sqrt{r}}
    \exp(n\fcore(\nncopt)),
\end{split}
\end{equation*}
which together with Lemma~\ref{lem:tail-core-hyper} finishes the proof
of Theorem~\ref{thm:corevalue-hyper}.
}

\Old{\subsection{Proof of lemmas in Section~\ref{sec:thm-proof-cores-hyper}}
\label{sec:lemmas-core-hyper}

In this section, we prove
Lemmas~\ref{lem:max-core},~\ref{lem:approx-core-hyper},
and~\ref{lem:tail-core-hyper}. \nickca{Deleted: See Section~ for a Maple
spreadsheet with some computations in this section.} \nickka{Maple
was used for several computations in this section.}

Using~\eqref{eq:difdeg} with $T = \coreQtwo$ and $t = \corentwo$, the
derivative of $\fcore(\nncore)$ is
\begin{equation}
  \label{eq:diffcore}
  -\ln(\coreQtwo)+\ln(\corentwo)-\ln(\nncore)+\ln(\coremthree)
  +\ln(3)- \ln \ff(\lambda) +\ln \lambda.
\end{equation}
The second derivative is
\begin{equation}
  \label{eq:ddiffcore}
  \frac{1}{\coreQtwo}-\frac{1}{\corentwo}-\frac{1}{\nncore}
  -\frac{1}{\coremthree}- \frac{(1-\corectwo)^2}{\coreQtwo(1+\coreetatwo-\corectwo)} < 0,
\end{equation} 
because $1/\coreQtwo < 1/\coremthree$. The third derivative is
\begin{equation}
  \begin{split}
    \label{eq:dddiffcore}
    &\frac{1}{\coreQtwo^2}-\frac{1}{\corentwo^2}+\frac{1}{\nncore^2}
    -\frac{1}{\coremthree^2}
    \\
    &-\frac{\dif\frac{(1-\corectwo)^2}{\coreQtwo}}
    {\dif \nncore}\frac{1}{(1+\coreetatwo-\corectwo)}
    +\frac{(1-\corectwo)^2}{\coreQtwo(1+\coreetatwo-\corectwo)^2}
    \frac{\dif(1+\coreetatwo-\corectwo)}{\dif \nncore}.
  \end{split}
\end{equation}

In order to approximate the value of $\fcore$ around the maximum by
using Taylor's approximation, we will bound the third derivative.
\begin{lem}
  \label{third-derivative-core-hyper} 
  Let $\delta = o(r)$ and $\nncore \in
  [\nncopt-\delta,\nncopt+\delta]$. Then the third derivative of
  $\fcore$ at $\nncore$ is $O(1/r^2)$.
\end{lem}
\begin{proof}
  We will bound each term in~\eqref{eq:dddiffcore}.
  By~\eqref{eq:optcore-r-hyper}, 
 \begin{equation*}
    \begin{split}
      &\frac{1}{\coreQtwo^2} =    
      \frac{1}{(1+4r+o(r))^2}\sim 1
      \\
      &\frac{1}{\corentwo^2}=\frac{1}{(1/2+O(r))^2}\sim 4 
      \\
      &\frac{1}{\nncore^2}=\frac{1}{(1/2+O(r))^2}\sim 4 
      \\
      &\frac{1}{\coremthree^2}
      =\frac{1}{(2r+o(r))^2}\sim \frac{1}{4r^2}.
   \end{split}
  \end{equation*}
  We have that $\lambda\eqdef \lambda_{\nn} = \lambdaopt + o(r) = o(1)$ by
  Lemma~\ref{lem:lambda-close-hyper}, and so $1+\coreetatwo-\corectwo
  = \lambda/6 + O(\lambda^2)\sim \lambdaopt/6\sim 2r$. Thus,
  by~\eqref{eq:optcore-r-hyper},
  \begin{equation*}
    \frac{\dif\frac{(1-\corectwo)^2}{\corectwo}}
    {\dif \nncore}\frac{1}{(1+\coreetatwo-\corectwo)}
    =
    \frac{(3\mcore-1)^2(6\mcore-3\nncore+1)}{\corentwo^3\coreQtwo^2}
    \frac{1}{(1+\coreetatwo-\corectwo)}
    = \Theta\left(\frac{1}{r}\right).
  \end{equation*}

  We now bound the last term in the third derivative.  The previous
  computations imply
  \begin{equation*}
    \frac{(1-\corectwo)^2}{\coreQtwo(1+\coreetatwo-\corectwo)^2}
    \sim
    \frac{1}{4r^2},
  \end{equation*}
  since $1-\corectwo =\coreQtwo/\corentwo \sim 2$
  by~\eqref{eq:optcore-r-hyper}.  So we need to bound
\begin{equation*}
  \frac{\dif(1+\coreetatwo-\corectwo)}{\dif \nncore}.
\end{equation*}

We have that 
\begin{equation*}
  \frac{\dif \corectwo}{\dif \nncore}=\frac{3\mcore-1}{(1-\nncore)^2}=\Theta(1)
\end{equation*}
and, by using~\eqref{eq:lambda-diff-hyper},
\begin{equation*}
  \frac{\dif\coreetatwo}{\dif \nncore}
=
\lambda
\frac{\dif \corectwo}{\dif \nncore}
\frac{1}{\corectwo(1+\coreetatwo-\corectwo)}
\left(\frac{\exp(\lambda)}{\fpo{1}(\lambda)}
  +\frac{\lambda\exp(\lambda)}{\fpo{1}(\lambda)}
- \frac{\lambda\exp(\lambda)^2}{\fpo{1}(\lambda)^2}\right)
=
\Theta(1),
\end{equation*}
because
\begin{equation*}
  \lambda \left(\frac{\exp(\lambda)}{\fpo{1}(\lambda)}
  +\frac{\lambda\exp(\lambda)}{\fpo{1}(\lambda)}
- \frac{\lambda\exp(\lambda)^2}{\fpo{1}(\lambda)^2}\right)
=
\frac{\lambda}{2}+O(\lambda^2).
\end{equation*}
Thus the last term has contribution $O(1/r^2)$ to the third
derivative.
\end{proof}

We now present the proofs for Lemmas~\ref{lem:max-core},~\ref{lem:approx-core-hyper},
and~\ref{lem:tail-core-hyper}.
\begin{proof}[Proof of Lemma~\ref{lem:max-core}]
  By setting the derivative of $\fcore$ in~\eqref{eq:diffcore} to $0$
  and using the definition of $\lambda_{\nn}$
  in~\eqref{eq:lambdacore-hyper}, we obtain the
  Equation~\eqref{eq:m-sol-core}, which has a unique positive solution
  $\lambdaopt$ by Lemma~\ref{lem:unique-pre-tools-hyper}
  The second derivative computation in~\eqref{eq:ddiffcore} implies
  that $\fcore$ is strictly concave and so $\lambdaopt$ is the unique
  maximum. 
\end{proof} 

\begin{proof}[Proof of Lemma~\ref{lem:approx-core-hyper}]
  Using Taylor's approximation and
  Lemma~\ref{third-derivative-core-hyper}, for any
  $\nncore\in[\nncopt-\delta, \nncopt+\delta]$,
  \begin{equation}
    \label{eq:taylor-core-hyper}
  \exp\left(n\fcore(\nncore)\right)  
  =
  \exp\left(n\fcore(\nncopt)+\frac{n\fcore''(\nncopt)|\nncopt-\nncore|^2}{2}+O(\delta^3/r^2)
    \right).
\end{equation}
Since $\delta^3/r^2=o(r^3/r^2)=o(1)$, this implies that
\begin{equation*}
  \sum_{\substack{x\in[-\delta n,\delta n]\\\nncopt +x\in\setZ}}
  \exp\paren[\Big]{n\fcore(\nncopt+\corex)}
  \sim
  \sum_{\substack{x\in[-\delta n,\delta n]\\\nncopt +x\in\setZ}}
    \exp\left(n\fcore(\nncopt)+\frac{\fcore''(\nncopt)x^2}{2n}
    \right).
\end{equation*}
By changing the variable in summation below to
$y=\sqrt{\frac{|\fcore''(\nncopt)|}{n}}x$,
\begin{equation*}
  \sum_{\substack{x\in[-\delta n,\delta n]\\\nncopt +x\in\setZ}}
  \exp\left(\frac{\fcore''(\nncopt)x^2}{2n}
  \right)
  =
  \sum_{\substack{y\in[-T_n,T_n]\\ 
      y\in\Pcal_n/s_n}}
  \exp\left(-\frac{y^2}{2}
  \right),
\end{equation*}
where $T_n\eqdef\delta \sqrt{n|\fcore''(\nncopt)|}$, $\Pcal_n \eqdef -
\nnopt+\setZ$, and $s_n \eqdef \sqrt{n/\abs{\fcore(\nncopt)}}$. By
using~\eqref{eq:ddiffcore} and~\eqref{eq:optcore-r-hyper}, we 
 approximated $\fcore''(\nncopt)$:
\begin{equation}
  \label{eq:second-dif-core-hyper}
  \fcore''(\nncopt) \sim -\frac{1}{r},
\end{equation}
and, since $\delta^2 =\omega(r/n)$ and $rn=R\to\infty$, we have
that $T_n\to\infty$ and $s_n\to\infty$. Thus, by Lemma~\ref{lem:integral-tools},
\begin{equation*}
  \sum_{\substack{y\in[-T_n,T_n]
      \\y\in\Pcal_n/s_n}}
  \exp\left(-\frac{y^2}{2}
  \right)
  \sim
  \sqrt{2\pi}  s_n
  \sim
  \sqrt{2\pi}
  \sqrt{rn},
\end{equation*}
which finished the proof of Lemma~\ref{lem:approx-core-hyper}.
\begin{comment}
  By Riemann sums, we have that
  \begin{equation*}
  \sum_{\substack{x\in[-\delta n,\delta n]\\\nncopt +x\in\setZ}}
    \exp\left(\frac{\fcore''(\nncopt)x^2}{2n}
    \right)
    \sim
    \int_{x=-\delta n}^{\delta n}
    \exp\left (\frac{\fcore''(\nncopt)x^2}{2n}\right)
    \dif x.
\end{equation*}

By changing $x$ to $y=\sqrt{\frac{|\fcore''(\nncopt)|}{n}}x$, we get
\begin{equation*}
  \int_{x=-\delta n}^{\delta n}
  \exp\left (\frac{\fcore''(\nncopt)x^2}{2n}\right)
  \dif x
  =
  \sqrt{\frac{n}{|\fcore''(\nncopt)|}}
  \int_{y=-\delta \sqrt{n|\fcore''(\nncopt)|}}^{\delta \sqrt{n|\fcore''(\nncopt)|}}
  \exp\left (\frac{-y^2}{2}\right)
  \dif y.
\end{equation*}

By computing the series of $\fcore''(\nncopt)$, we have that
\begin{equation*}
  \fcore''(\nncopt) \sim -\frac{1}{r},
\end{equation*}
\end{comment}
\end{proof}

\begin{proof}[Proof of Lemma~\ref{lem:tail-core-hyper}]
  Instead of working directly with $\wcore$, we will prove an upper
  bound for $\wcore$ using $\fcore$ and then bound the summation using
  this upper bound.

  By Stirling's approximation \crisc{removed reference to appendix}and
  the definitions of $\wcore$ and $\fcore$, for $\nn\in J_m$,
\begin{equation}
  \label{eq:bound-stirling-core-hyper}
  \wcore(\nn)
  \leq n^\beta n! \exp\paren[\big]{n\fcore(\nncore)},
\end{equation}
for some constant $\beta$. First we bound the summation for the tail
$\nn\leq \nnopt-\delta n$. By
Lemma~\ref{lem:max-core},~\eqref{eq:taylor-core-hyper}
and~\eqref{eq:second-dif-core-hyper}
\begin{equation*}
  \begin{split}
    \sum_{ \substack{\nn \leq \nnopt-\delta n\\ \nn\in\setZ}}  &\wcore(\nn)
    \leq n^{\beta+1} n!\exp\paren[\Big]{n\fcore(\nncopt-\delta)}
    \\
    &\leq n!\exp\paren[\Big]{n\fcore(\nncopt)}
    \exp\paren[\Big]{n\fcore''(\nncopt)\delta^2/2+(\beta+1)\ln n +o(1)}
    \\
    &= O\left(\frac{n!\exp(n\fcore(\nncopt))}
      {\exp(n\delta^2/(2r)-(\beta+1)\ln n +o(1))}\right)
  \end{split}
\end{equation*}
and we are done since $n\delta^2/r=\omega(\ln n)$. The proof
for $\nncore\geq \nncopt+\delta$ is similar.
\end{proof}
}
%%%%%%%%%%%%%%%%%%%%%%%%%%%%%%%%%%%%%%%%%%%%%%%%%%%%%%%%%%%%%%%%
% Pre-kernels
%%%%%%%%%%%%%%%%%%%%%%%%%%%%%%%%%%%%%%%%%%%%%%%%%%%%%%%%%%%%%%%%
\section{Counting pre-kernels}
\label{sec:pre-hyper}

In this section we obtain an asymptotic formula for the number of
\prekernels\ with vertex set~$[n]$ with $m=n/2+R$ edges, when $R=
\omega(n^{1/2}\log^{3/2} n)$ and $R=o(n)$. We remark that the
asymptotics in this section are for $n\to \infty$. We will always use
$r$ to denote $R/n$.

% Define weak feasible region
For $x=(\nn,\kzero,\kone,\ktwo)\in\setR^4$, let
\begin{equation}
\label{eq:pre-param-def}
  \begin{split}
      &\ntwoequal(x) =\kzero +\kone +\ktwo,
%  \quad&\prentwoequal(\prex) = \ntwoequal(x)/n;
\\
  &\nthree(x) = n-\nn-\ntwoequal(x) = n-\nn-\kzero -\kone -\ktwo,
 % \quad&\prenthree(\prex) = \nthree(x)/n;
\\
  &\mtwo(x) = \nn,
  %\quad&\premtwo(\prex)=\mtwo(x)/n;
\\
  &\mtwoprime(x) = \nn-\kzero,
 % \quad&\premtwoprime(\prex)=\mtwoprime(x)/n;
\\
  &\Ptwo(x)   = 2\mtwoprime(x) = 2\nn-2\kzero,
 % \quad& \prePtwo(\prex) = \Ptwo(x)/n;
\\
  &\mthree(x) = m-\nn,
 % \quad& \premthree(\prex) = \mthree(x)/n;
\\
  &\Pthree(x) = 3\mthree(x) = 3m-3\nn,
 % \quad& \prePthree(\prex) = \Pthree(x)/n;
\\
  &\Qthree(x) = 3m-\nn-2\ntwoequal(x) = 3m-\nn-2\kzero-2\kone-2\ktwo,
%  \quad& \preQthree(\prex) = \Qthree(x)/n;
\\
  &\Tthree(x) = \Pthree(x)-\kone-2\ktwo = 3m-3\nn-\kone-2\ktwo,
%  \quad& \preTthree(\prex) = \Tthree(x)/n;
\\
&\Ttwo(x) = \Ptwo(x)-\kone = 2\nn-2\kzero-\kone,
 % \quad& \preTtwo(\prex) = \Ttwo(x)/n;\\
  \end{split}
\end{equation}
For any symbol $y$ in this section (and following subsections), we use
$\hat y$ to denote $y/n$.

We will have $\nn$ as the number of vertices of degree~$1$, $\kzero$
as the number of vertices of degree~$2$ such that the two edges
incident to it are $2$-edges, $\ktwo$ as the number of vertices of
degree~$2$ such that the two edges incident to it are $3$-edges and
$\kone$ as the remaining vertices of degree~$2$. Then it is clear that
$\ntwoequal$ is the number of vertices of degree~2, $\nthree$ is the
number of vertices of degree at least~$3$, $\Qthree$ is the sum of
degrees of vertices of degree at least~$3$, $\mthree$ is the number of
$3$-edges, $\mtwo$~is the number of $2$-edges, and $\mtwoprime$ is the
number of $2$-edges that contain exactly two vertices of
degree~$2$. We omit the argument $x$ when it is obvious from the
context.

For $x\in \setR^4$, let 
\begin{equation*}
  \cthree(x) = \frac{\Qthree(x)}{\nthree(x)} 
  = \frac{3m-\nn-2\ntwoequal(x)}{n-\nn-\ntwoequal(x)},  
\end{equation*}
that is, $\cthree$ is the average degree of the vertices of degree at
least~$3$. Note that $\cthree(x) = \preQthree(x)/\prenthree(x) =
\precthree(x)$. For $x\in \setR^4$ such that $\preQthree(x) > 3
\prenthree(x) > 0$, let $\lambda = \lambda(x)$ be the unique positive
solution~of
\begin{equation}
  \label{eq:lambda-def-pre-hyper}  \frac{\lambda \ff(\lambda)}{\fff(\lambda)} 
  = \cthree(x).
\end{equation}
Such $\lambda(x)$ always exists and is unique by
Lemma~\textred{\cite[Lemma 1]{PWa}}. By continuity reasons, we define
$\lambda(x)=0$ when $\cthree(x)=3.$

Let $S_m$ be the region of $\setR^4$ such that
$x=(\nn,\kzero,\kone,\ktwo)\in S_m$ if all of the following conditions
hold:
\begin{itemize}
\item $\nn,\kzero,\kone,\ktwo \in [0,n]$;
\item $\Qthree(x) \geq 3\nthree(x) \geq 0$, and $\Qthree(x)=0$
  whenever $\nthree(x) = 0$;
\item $\mthree(x), \mtwo(x),\mtwoprime(x), \Tthree(x),\Ttwo(x) \geq 0$.
\end{itemize}
We will work with \prekernels\ with $\nn$ vertices of degree~$1$ and
$k_i$ vertices of degree~$2$ incident to exactly~$i$ $3$-edges, for
$i=0,1,2$. We say that such \prekernels\ have parameters
$(\nn,\kzero,\kone,\ktwo)$.  The region $S_m$ is defined so that all
tuples $(\nn,\kzero,\kone,\ktwo)$ for which it there exists a
\prekernel\ with such parameters are included. Let $\preS_m =
\set{x/n: x\in S_m}$ denote the scaled version of $S_m$. The set $S_m$
is not closed because $\Qthree(x)=0$ whenever $\nthree(x) = 0$. This
constraint is added because $\Qthree(x)$ should be the sum of the
degrees of vertices of degree at least~$3$ and $\nthree(x)$ should be
the number of vertices of degree at least~$3$.

% Introducing fpre
For $\prex = (\prenn,\prekzero,\prekone,\prektwo)$  in
the interior of $\preS_m$, define
\begin{equation}
\label{eq:fpre-def-hyper}
  \begin{split}
    \fpre(\prenn,\prekzero,\prekone,\prektwo)
    =
    &
    \ \hpre(\prePthree)
    +\hpre(\prePtwo)
    +\hpre(\preQthree)
    +\hpre(\mtwo)
    \\
    &-\hpre(\prekzero)
    -\hpre(\prekone)
    -\hpre(\prektwo)
    -\hpre(\prenthree)
    -\hpre(\premthree)
    \\
    &-\hpre(\prePthree-\prekone-2\prektwo)
    -\hpre(\prePtwo-\prekone)
    -2\hpre(\premtwoprime)
    \\
    &
    -\prektwo \ln 2
    - \premtwoprime \ln 2
    - \premthree \ln 6
    \\
    &+\prenthree\ln f_3(\lambda)
    -\preQthree\ln\lambda,
  \end{split}
\end{equation}
where $\lambda = \lambda(x)$. As we will see later, for
$x=(\nn,\kzero,\kone,\ktwo)\in S_m\cap Z^4$, we have that $n!
\exp(n\fpre(\prex))$ approximates the exponential part of the number
of pre-kernels with parameters $(\nn,\kzero,\kone,\ktwo)$.

We extend the definition of $\fpre$ for points $\prex\in \preS_m$ that
are in the boundary of $\preS_m$ as the limit of $\fpre(x^{(i)})$ on
any sequence of points $(x^{(i)})_{i\in \setN}$ in the interior of
$\preS_m$ with $x^{(i)}\to \prex$.  One of the reasons the points $x$
with $\Qthree(x) > \cthree(x) = 0$ are not allowed is that
$\fpre(x_i)$ does not necessarily converge on a sequence of points
$(x^{(i)})_{i\in \setN}$ converging to~$x$.  For the points in the
boundary where $\Qthree(x) > \cthree(x)$, this only means that $0\log
0$ should be interpreted as~$1$.  For $\prex\in \preS_m$ such that
$\preQthree(\prex) = 3\prenthree(\prex)$, we have that
$\lambda(x)=0$. This means that $\prenthree(x)\ln f_3(\lambda(x))
-\preQthree(x)\ln\lambda(x)$ is not defined (and note that
$\prenthree(x)\ln f_3(\lambda(x))$ and $-\preQthree(x)\ln\lambda(x)$
are the last two terms in the definition of $\fpre(\prex)$). We
compute $\lim_{\lambda\to 0} (\prenthree(x)\ln f_3(\lambda)
-\preQthree(x)\ln\lambda)$. We have that $\lim_{\lambda\to 0}{(\ln
  f_3(\lambda) -3\ln\lambda)} = -\ln 6$. Thus, $\lim_{\lambda\to 0}
(\prenthree(x)\ln f_3(\lambda) -\preQthree(x)\ln\lambda) =
-\prenthree(x)\ln 6$ and
  \begin{equation}
    \label{eq:fpre-extreme}
    \begin{split}
      \fpre(\prex)=
      &\hpre(\prePthree)
    +\hpre(\prePtwo)
    +\hpre(\preQthree)
    +\hpre(\mtwo)
    \\
    &-\hpre(\prekzero)
    -\hpre(\prekone)
    -\hpre(\prektwo)
    -\hpre(\prenthree)
    -\hpre(\premthree)
    \\
    &-\hpre(\prePthree-\prekone-2\prektwo)
    -\hpre(\prePtwo-\prekone)
    -2\hpre(\premtwoprime)
    \\
    &
    -\prektwo \ln 2
    - \premtwoprime \ln 2
    - \premthree \ln 6
    \\
    & -\prenthree\ln 6.
  \end{split}
  \end{equation}

% Main result
We obtain the following asymptotic formula for the number of
\prekernels\ with $n$ vertices and $m = m(n)$ edges.
\begin{thm}
  \label{thm:formula-pre-hyper}
  Let $m=m(n)=n/2+R$ such that $R=o(n)$ and
  $R=\omega(n^{1/2}\log^{3/2}n)$. Then
  \begin{equation*}
    \gpre(n,m)\sim
      \frac{\sqrt{3}}{\pi n}
      n!\exp(n\fpre(\prexopt)),
  \end{equation*}
   where $\prexopt$ is defined as
  $(\prennopt,\prekzeroopt,\prekoneopt,\prektwoopt)$ with
\begin{equation}
  \label{eq:pre-param-opt}
  \begin{split}
    \prennopt &= \frac{3\mpre}{g_2(\lambdaopt)},\\
  \prekzeroopt &= \frac{3\mpre}{g_2(\lambdaopt)}
                  \frac{2\lambdaopt}{f_1(\lambdaopt)g_1(\lambdaopt)},\\
  \prekoneopt  &= \frac{3\mpre}{g_2(\lambdaopt)}
                  \frac{2\lambdaopt}{g_1(\lambdaopt)},\\
  \prektwoopt &=\frac{3\mpre}{g_2(\lambdaopt)}
                \frac{\lambdaopt f_1(\lambdaopt)}{2g_1(\lambdaopt)},
  \end{split}
\end{equation}
and $\lambdaopt = \lambdaopt(n)$ is the unique nonnegative solution for
the equation
\begin{equation}
  \label{pre-m-opt}
  \frac{\lambda \f(\lambda) \fgg(\lambda)}{\FF(2\lambda)}
  =
  3\mpre.
\end{equation}
\end{thm}

We discussed the existence and uniqueness of $\lambdaopt$ in
Section~\ref{sec:core-hyper}. Also, note that~\eqref{pre-m-opt}, implies
\begin{equation}
  \label{eq:pre-r-opt}
    r = \frac{1}{3}\frac{\lambdaopt 
    f_1(\lambdaopt) g_2(\lambdaopt)}{\FF(2\lambdaopt)}
  -\frac{1}{2}.
\end{equation}

We will show that the point $\prexopt$ maximizes $\fpre$ in a region
that contains all points $(\prenn,\prekzero,\prekone,\prektwo)$ for
which there exists a pre-kernel with parameters
$(\nn,\kzero,\kone,\ktwo)$. The result is then obtained basically by
expanding the summation around $\prexopt$ in a region such that each
term in~\eqref{eq:pre-param-def} are nonnegative and $\cthree\geq
3$. This approach is similar to the one in
Section~\ref{sec:core-hyper} in which we analyse cores, but it will
require much more work since we are now dealing with a $4$-dimensional
space.  We remark that $\lambda^* = \lambda(\xopt)$, that is,
$\lambdaopt \fpo{2}(\lambdaopt)/\fpo{3}(\lambdaopt) = \cthree(\xopt)$.

Similarly to Section~\ref{sec:core-hyper} that deals with cores, it
will be useful to know approximations for some parameters at the
 point $\prexopt = (\prennopt, \prekzeroopt,\prekoneopt,
\prektwoopt)$ that achieves the maximum. For $r = o(1)$, we proved in Lemma~\ref{lem:max-core}
that $\lambdaopt =o(1)$. From~\eqref{pre-m-opt}, we can write $r$ in
terms of $\lambdaopt$ and so we can write $\prennopt$, $\prekzeroopt$,
$\prekoneopt$ and $\prektwoopt$ in terms of $\lambdaopt$. Thus,
using~\eqref{eq:pre-param-opt}, and computing the series of each
function in~\eqref{eq:pre-param-def} as $\lambdaopt\to 0$, we have
\begin{equation}
  \label{eq:optpre-rel-hyper}
  \begin{array}{rlcrl}
    \vspace{5pt}
    r 
    &= \frac{1}{12}\lambdaopt
    +\frac{1}{36}(\lambdaopt)^2
    +O((\lambdaopt)^3)
    &\quad&
    \preQthreeopt 
    &= \frac{1}{2} \lambdaopt
    +\frac{1}{12} (\lambdaopt)^2
    +O((\lambdaopt)^3)
    \\ \vspace{5pt}
    \prennopt 
    &=
    \frac{1}{2}
    -\frac{1}{12}\lambdaopt
    -\frac{1}{36}(\lambdaopt)^2
    +O((\lambdaopt)^3)
    &\quad&
    \premthreeopt 
    &= 
    \frac{1}{6} \lambdaopt
    +\frac{1}{18} (\lambdaopt)^2
    +O((\lambdaopt)^3)
    \\ \vspace{5pt}
    \prekzeroopt
    & =\frac{1}{2}
    -\frac{7}{12} \lambdaopt
    +\frac{2}{9} (\lambdaopt)^2
    +O((\lambdaopt)^3)
    &\quad&
    \premtwoprimeopt 
    &= \frac{1}{2} \lambdaopt
    -\frac{1}{4} (\lambdaopt)^2
    +O((\lambdaopt)^3)
    \\ \vspace{5pt}
    \prekoneopt 
    &= \frac{1}{2} \lambdaopt
    -\frac{1}{3} (\lambdaopt)^2
    +O((\lambdaopt)^3)
    &\quad&
    \preTtwoopt 
    &= \frac{1}{2} \lambdaopt
    -\frac{1}{6} (\lambdaopt)^2
    +O((\lambdaopt)^3)    
    \\\vspace{5pt}
    \prektwoopt 
    &= 
    \frac{1}{8} (\lambdaopt)^2
    +O((\lambdaopt)^3)
    &\quad&
    \preTthreeopt 
    &= \frac{1}{4} (\lambdaopt)^2
    +O((\lambdaopt)^3)
    \\ \vspace{5pt}
    \prenthreeopt 
    &= \frac{1}{6} \lambdaopt
    +\frac{1}{72} (\lambdaopt)^2
    +O((\lambdaopt)^3).
    &\quad& &\\
  \end{array}
\end{equation}

In the following subsections, we will use $\Ys = (Y_1,\dotsc,
Y_{\nthree})$ to denote a vector of independent random variables
$Y_1,\dotsc, Y_{\nthree}$ such that each $Y_i$ has truncated Poisson
distribution with parameters $(3, \lambda(x))$ and $\Sigma(x)$ to
denote the event $\sum_i Y_i= \Qthree$.

\subsection{Kernels}
In this section, we define the notion of kernels of \prekernels, which
will be useful to study properties of \prekernels\ and to generate
random \prekernels.

Recall that the \prekernel\ is a core with no isolated cycles. Let the
\textdef{kernel}\index{multihypergraph!kernel}\index{hypergraph!loop}
of a pre-kernel $G$ be the multihypergraph obtained as follows. Start
by obtaining $G'$ from $G$ by deleting all vertices of degree~$1$ and
replacing each edge containing a vertex of degree $1$ by a new edge of
size $2$ incident to the other two vertices (and note that the
multihypergraph is not necessarily uniform anymore).  While there is a
vertex $v$ of degree $2$ in~$G'$ such that the two edges incident to
$v$ have size~$2$, update $G'$ by deleting both edges, and adding a
new edge of size $2$ containing the vertices other than~$v$ that were
in the deleted edges. When this procedure is finished, delete all
vertices of degree less than~$2$. The final multihypergraph is the
kernel of $G$. This procedure obviously produces a unique
multihypergraph (disregarding edge labels). See
Figure~\ref{fig:prek-proc-hyper} for an example of the procedure
above.

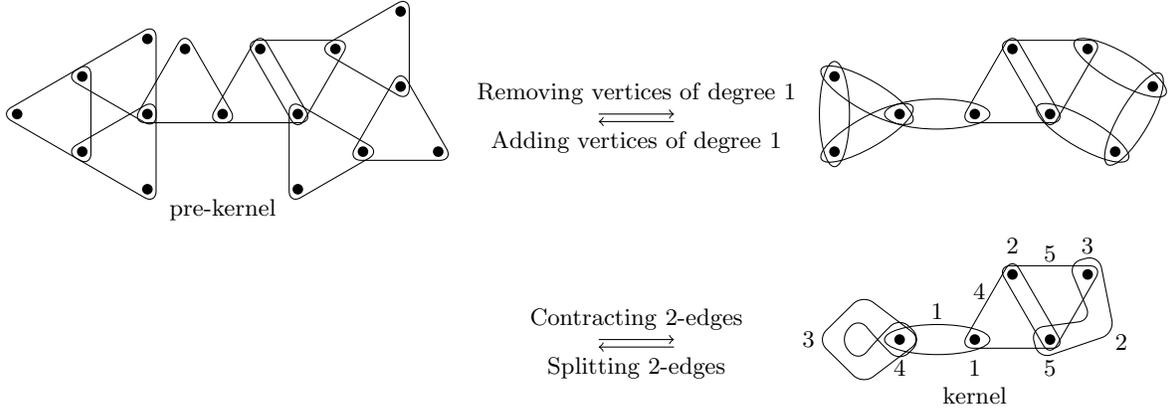
\begin{figure}
  \centering
  \begin{tikzpicture}
    \def \initialg {0};
    \def \step {1};
    \def \rad {2pt};
    \def \eps {.3};
    \def \epsm {.15};
    \def \epsg {.2};
    \def \round {5pt};
    \def \scalep {0.3}

    % Kernel
    % Vertices
    \coordinate[label={[label distance=4pt]-90:\footnotesize{$4$}}]
    (4) at ($(\initialg,\initialg)$);
    \coordinate[label={[label distance=4pt]-90:\footnotesize{$1$}}]
    (1) at ($(4)+\step*(1,0)$);
    \coordinate[label={[label distance=4pt]-90:\footnotesize{$5$}}]
    (5) at ($(1)+\step*(1,0)$);
    \coordinate[label={[label distance=4pt]90:\footnotesize{$2$}}]
    (2) at ($(1)+\step*({sin(30)},{cos(30)})$);
    \coordinate[label={[label distance=4pt]90:\footnotesize{$3$}}]
    (3) at ($(2)+\step*(1,0)$);

    \coordinate[label={[label distance=0pt]-90:\footnotesize{kernel}}]
    (ker) at ($(1)-(0,0.5*\step)$);
    \foreach \point in {1,2,3,4,5} 
    \fill [black] (\point) circle (\rad);

    % Edges
    \coordinate[label={[label distance=4pt]90:\footnotesize{$1$}}] 
    (b0) at ($.5*(1)+.5*(4)$);
    \draw (b0) ellipse ({.5*\step+\epsg} and {\epsg}); 
  
    \coordinate[label={[label distance=-1pt]above left:\footnotesize{$4$}}] 
    (b12) at ($.5*(1)+.5*(2)$);
    \coordinate (b1) at ($1/3*(1)+1/3*(2)+1/3*(5)$);
    \coordinate (a1) at ($-0.4*(b1)+1.4*(2)$);
    \coordinate (a2) at ($-0.4*(b1)+1.4*(1)$);
    \coordinate (a3) at ($-0.4*(b1)+1.4*(5)$);
    %\foreach \point in {b1,a1,a2,a3} \fill [red] (\point) circle (\rad);
    \draw[rounded corners=\round] (a1)--(a2)--(a3)--cycle;
    
    \coordinate[label={[label distance=1pt]above:\footnotesize{$5$}}] 
    (b23) at ($.5*(3)+.5*(2)$);
    \coordinate (b2) at ($1/3*(3)+1/3*(2)+1/3*(5)$);
    \coordinate (c1) at ($-0.4*(b2)+1.4*(2)$);
    \coordinate (c2) at ($-0.4*(b2)+1.4*(3)$);
    \coordinate (c3) at ($-0.4*(b2)+1.4*(5)$);
    %\foreach \point in {b2,c1,c2,c3} \fill [red] (\point) circle (\rad);
    \draw[rounded corners=\round] (c1)--(c2)--(c3)--cycle;

    \coordinate (b3) at ($.5*(5)+.5*(3)$);
    \coordinate (d1) at ($(b3)+.35*({sin(120)},{cos(120)})$);
    \coordinate[label={[label distance=-5pt]below right:\footnotesize{$2$}}]  
    (d2) at ($(b3)+.7*({sin(120)},{cos(120)})$);
    \coordinate (d3) at ($(3)+.3*({sin(30)},{cos(30)})$);
    \coordinate (d4) at ($(5)-.3*({sin(30)},{cos(30)})$);
    \coordinate (d5) at ($(3)+.3*({sin(-60)},{cos(-60)})$);
    \coordinate (d6) at ($(5)+.3*({sin(-60)},{cos(-60)})$);
    %\foreach \point in {b3,d1,d2,d3,d4,d5,d6} \fill [red] (\point) circle (\rad);
    \draw[rounded corners=\round]
    (d1)--(d5)--(d3)--(d2)--(d4)--(d6)--cycle;

    \coordinate (b4) at (4);
    \coordinate (e1) at ($(4)+0.3*(0,1)$);
    \coordinate (e2) at ($(4)+0.3*(1,0)$);
    \coordinate (e3) at ($(4)-0.8*(0,1)$);
    \coordinate[label={[label distance=-3pt]180:\footnotesize{$3$}}]
    (e4) at ($(b4)-1.1*(1,0)$);
    \coordinate (e7) at ($(4)-0.3*(0,1)$);
    \coordinate (e8) at ($(4)-0.3*(1,0)$);
    \coordinate (e10) at ($(e4)+0.3*(1,0)$);    
    \coordinate (e9) at ($0.5*(e8)+0.5*(e10)+0.3*(0,1)$);
    \coordinate (e11) at ($0.5*(e8)+0.5*(e10)-0.3*(0,1)$);    
    \coordinate (e5) at ($(e9)+0.3*(0.2,1)$);
    \coordinate (e3) at ($(e11)-0.3*(-0.2,1)$);

    %\foreach \point in {e1,e2,e3,e4,e5,e7,e8,e9,e10,e11}
    %\fill [red] (\point) circle (\rad);
    \draw[rounded corners=\round]
    (e1)--(e2)--(e3)--(e4)--(e5)--(e2)--(e7)--(e8)--(e9)--(e10)--(e11)--(e8)--cycle;

    \coordinate  (startarrow) at ($(4)-(4*\step,0)$);
    \coordinate  (endarrow) at ($(startarrow)+(\step,0)$);
    \coordinate  (startarrowb) at ($(startarrow)+(0,-0.1)$);
    \coordinate  (endarrowb) at ($(startarrowb)+(\step,0)$);
    
    \coordinate[label=above:\footnotesize{Contracting $2$-edges}]
    (midarrow) at ($.5*(startarrow)+.5*(endarrow)$);
    \draw [->] (startarrow)--(endarrow);
    \coordinate[label=below:\footnotesize{Splitting $2$-edges}]
    (midarrowb) at ($.5*(startarrowb)+.5*(endarrowb)$);
    \draw [<-] (startarrowb)--(endarrowb);
    
    % Split Kernel
    % Vertices
    \coordinate (4s) at ($(4)+3*\step*(0,1)$);
    \coordinate (1s) at ($(4s)+\step*(1,0)$);
    \coordinate (5s) at ($(1s)+\step*(1,0)$);
    \coordinate (2s) at ($(1s)+\step*({sin(30)},{cos(30)})$);
    \coordinate (3s) at ($(2s)+\step*(1,0)$);
    \coordinate (6s) at ($(4s)+\step*({sin(-120)},{cos(-120)})$);
    \coordinate (7s) at ($(6s)+\step*({cos(90)},{sin(90)})$);
    \coordinate (8s) at ($(3s)+\step*({cos(-30)},{sin(-30)})$);
    \coordinate (9s) at ($(5s)+\step*({cos(-30)},{sin(-30)})$);

    \foreach \point in {1s,2s,3s,4s,5s,6s,7s,8s,9s} 
    \fill [black] (\point) circle (\rad);
    % Edges
    \coordinate (b0s) at ($.5*(1s)+.5*(4s)$);
    \draw (b0s) ellipse ({.5*\step+\epsg} and {\epsg}); 
  
    \coordinate (b12s) at ($.5*(1s)+.5*(2s)$);
    \coordinate (b1s) at ($1/3*(1s)+1/3*(2s)+1/3*(5s)$);
    \coordinate (a1s) at ($-0.4*(b1s)+1.4*(2s)$);
    \coordinate (a2s) at ($-0.4*(b1s)+1.4*(1s)$);
    \coordinate (a3s) at ($-0.4*(b1s)+1.4*(5s)$);
    %\foreach \point in {b1,a1,a2,a3} \fill [red] (\point) circle (\rad);
    \draw[rounded corners=\round] (a1s)--(a2s)--(a3s)--cycle;
    
    \coordinate (b23s) at ($.5*(3s)+.5*(2s)$);
    \coordinate (b2s) at ($1/3*(3s)+1/3*(2s)+1/3*(5s)$);
    \coordinate (c1s) at ($-0.4*(b2s)+1.4*(2s)$);
    \coordinate (c2s) at ($-0.4*(b2s)+1.4*(3s)$);
    \coordinate (c3s) at ($-0.4*(b2s)+1.4*(5s)$);
    %\foreach \point in {b2,c1,c2,c3} \fill [red] (\point) circle (\rad);
    \draw[rounded corners=\round] (c1s)--(c2s)--(c3s)--cycle;

    \coordinate (b67s) at ($.5*(6s)+.5*(7s)$);
    \draw (b67s) ellipse ({\epsg} and {.5*\step+\epsg}); 

    \coordinate (b46s) at ($.5*(6s)+.5*(4s)$);
    \draw[rotate=120] (b46s) ellipse ({\epsg} and {.5*\step+\epsg}); 

    \coordinate (b47s) at ($.5*(4s)+.5*(7s)$);
    \draw[rotate=-120] (b47s) ellipse ({\epsg} and {.5*\step+\epsg}); 

    \coordinate (b38s) at ($.5*(3s)+.5*(8s)$);
    \draw[rotate=-120] (b38s) ellipse ({\epsg} and {.5*\step+\epsg}); 

    \coordinate (b59s) at ($.5*(5s)+.5*(9s)$);
    \draw[rotate=-120] (b59s) ellipse ({\epsg} and {.5*\step+\epsg}); 

    \coordinate (b89s) at ($.5*(8s)+.5*(9s)$);
    \draw[rotate=-30] (b89s) ellipse ({\epsg} and {.5*\step+\epsg}); 

    % Prekernel
    % Vertices
    \coordinate (4p) at ($(4s)-10*\step*(1,0)$);
    \coordinate (1p) at ($(4p)+\step*(1,0)$);
    \coordinate (5p) at ($(1p)+\step*(1,0)$);
    \coordinate (2p) at ($(1p)+\step*({sin(30)},{cos(30)})$);
    \coordinate (3p) at ($(2p)+\step*(1,0)$);
   
    \coordinate (6p) at ($(4p)+\step*({sin(-120)},{cos(-120)})$);
    \coordinate (7p) at ($(6p)+\step*({cos(90)},{sin(90)})$);
    
    \coordinate (8p) at ($(3p)+\step*({cos(-30)},{sin(-30)})$);
    \coordinate (9p) at ($(5p)+\step*({cos(-30)},{sin(-30)})$);
    
    \coordinate (10p) at ($(3p)+\step*({cos(30)},{sin(30)})$);
    \coordinate (11p) at ($(5p)+\step*({cos(-90)},{sin(-90)})$);
    \coordinate (12p) at ($(9p)+\step*({cos(0)},{sin(0)})$);

    \coordinate (13p) at ($(4p)+\step*({cos(60)},{sin(60)})$);

    \coordinate (14p) at ($(4p)+\step*({cos(90)},{sin(90)})$);
    \coordinate (15p) at ($(7p)+\step*({cos(-150)},{sin(-150)})$);
    \coordinate (16p) at ($(4p)+\step*({cos(-90)},{sin(-90)})$);

    \coordinate[label={[label distance=0pt]-90:\footnotesize{pre-kernel}}]
    (prek) at ($(1p)-(0,1*\step)$);

    \foreach \point in {1p,2p,3p,4p,5p,6p,7p,8p,9p,10p,11p,12p,13p,14p,15p,16p} 
    \fill [black] (\point) circle (\rad);

    % edges
    \coordinate (b1p) at ($1/3*(1p)+1/3*(2p)+1/3*(5p)$);
    \coordinate (a1p) at ($-0.4*(b1p)+1.4*(2p)$);
    \coordinate (a2p) at ($-0.4*(b1p)+1.4*(1p)$);
    \coordinate (a3p) at ($-0.4*(b1p)+1.4*(5p)$);
    %\foreach \point in {b1,a1,a2,a3} \fill [red] (\point) circle (\rad);
    \draw[rounded corners=\round] (a1p)--(a2p)--(a3p)--cycle;
    
    \coordinate (b2p) at ($1/3*(3p)+1/3*(2p)+1/3*(5p)$);
    \coordinate (c1p) at ($-0.4*(b2p)+1.4*(2p)$);
    \coordinate (c2p) at ($-0.4*(b2p)+1.4*(3p)$);
    \coordinate (c3p) at ($-0.4*(b2p)+1.4*(5p)$);
    %\foreach \point in {b2,c1,c2,c3} \fill [red] (\point) circle (\rad);
    \draw[rounded corners=\round] (c1p)--(c2p)--(c3p)--cycle;

    \coordinate (b3p) at ($1/3*(3p)+1/3*(8p)+1/3*(10p)$);
    \coordinate (d1p) at ($-0.4*(b3p)+1.4*(8p)$);
    \coordinate (d2p) at ($-0.4*(b3p)+1.4*(3p)$);
    \coordinate (d3p) at ($-0.4*(b3p)+1.4*(10p)$);
    %\foreach \point in {b2,c1,c2,c3} \fill [red] (\point) circle (\rad);
    \draw[rounded corners=\round] (d1p)--(d2p)--(d3p)--cycle;

    \coordinate (b4p) at ($1/3*(5p)+1/3*(9p)+1/3*(11p)$);
    \coordinate (e1p) at ($-0.4*(b4p)+1.4*(9p)$);
    \coordinate (e2p) at ($-0.4*(b4p)+1.4*(5p)$);
    \coordinate (e3p) at ($-0.4*(b4p)+1.4*(11p)$);
    %\foreach \point in {b2,c1,c2,c3} \fill [red] (\point) circle (\rad);
    \draw[rounded corners=\round] (e1p)--(e2p)--(e3p)--cycle;

    \coordinate (b5p) at ($1/3*(9p)+1/3*(8p)+1/3*(12p)$);
    \coordinate (f1p) at ($-0.4*(b5p)+1.4*(8p)$);
    \coordinate (f2p) at ($-0.4*(b5p)+1.4*(9p)$);
    \coordinate (f3p) at ($-0.4*(b5p)+1.4*(12p)$);
    %\foreach \point in {b2,c1,c2,c3} \fill [red] (\point) circle (\rad);
    \draw[rounded corners=\round] (f1p)--(f2p)--(f3p)--cycle;

    \coordinate (b6p) at ($1/3*(4p)+1/3*(1p)+1/3*(13p)$);
    \coordinate (g1p) at ($-0.4*(b6p)+1.4*(1p)$);
    \coordinate (g2p) at ($-0.4*(b6p)+1.4*(4p)$);
    \coordinate (g3p) at ($-0.4*(b6p)+1.4*(13p)$);
    %\foreach \point in {b2,c1,c2,c3} \fill [red] (\point) circle (\rad);
    \draw[rounded corners=\round] (g1p)--(g2p)--(g3p)--cycle;

    \coordinate (b7p) at ($1/3*(4p)+1/3*(7p)+1/3*(14p)$);
    \coordinate (h1p) at ($-0.4*(b7p)+1.4*(7p)$);
    \coordinate (h2p) at ($-0.4*(b7p)+1.4*(4p)$);
    \coordinate (h3p) at ($-0.4*(b7p)+1.4*(14p)$);
    %\foreach \point in {b2,c1,c2,c3} \fill [red] (\point) circle (\rad);
    \draw[rounded corners=\round] (h1p)--(h2p)--(h3p)--cycle;

    \coordinate (b8p) at ($1/3*(4p)+1/3*(6p)+1/3*(16p)$);
    \coordinate (i1p) at ($-0.4*(b8p)+1.4*(6p)$);
    \coordinate (i2p) at ($-0.4*(b8p)+1.4*(4p)$);
    \coordinate (i3p) at ($-0.4*(b8p)+1.4*(16p)$);
    %\foreach \point in {b2,c1,c2,c3} \fill [red] (\point) circle (\rad);
    \draw[rounded corners=\round] (i1p)--(i2p)--(i3p)--cycle;

    \coordinate (b9p) at ($1/3*(7p)+1/3*(6p)+1/3*(15p)$);
    \coordinate (j1p) at ($-0.4*(b9p)+1.4*(6p)$);
    \coordinate (j2p) at ($-0.4*(b9p)+1.4*(7p)$);
    \coordinate (j3p) at ($-0.4*(b9p)+1.4*(15p)$);
    %\foreach \point in {b2,c1,c2,c3} \fill [red] (\point) circle (\rad);
    \draw[rounded corners=\round] (j1p)--(j2p)--(j3p)--cycle;

    \coordinate  (startarrow2) at ($(4s)+(-4*\step,0)$);
    \coordinate  (endarrow2) at ($(startarrow2)+(\step,0)$);
    \coordinate  (startarrow2b) at ($(startarrow2)+(0,-0.1)$);
    \coordinate  (endarrow2b) at ($(startarrow2b)+(\step,0)$);
    \coordinate[label=above:\footnotesize{Removing vertices of degree
    $1$}]
    (midarrow2) at ($.5*(startarrow2)+.5*(endarrow2)$);
    \coordinate[label=below:\footnotesize{Adding vertices of degree
      $1$}]
    (midarrow2b) at ($.5*(startarrow2b)+.5*(endarrow2b)$);
    \draw [->] (startarrow2)--(endarrow2);
    \draw [<-] (startarrow2b)--(endarrow2b);
      \end{tikzpicture}
  \caption{Obtaining a kernel from a pre-kernel, and vice-versa.}
  \label{fig:prek-proc-hyper}
\end{figure}

This procedure is similar to the one for obtaining kernels of
$2$-uniform hypergraphs described by Pittel and Wormald~\cite{PWb}:
the kernel of a pre-kernel is obtained by repeatedly replacing edges
$uv$ and~$vw$, where $v$ is a vertex of degree~$2$, by a new edge $uw$
until no vertex of degree~$2$ remains, and then deleting all isolated
vertices.  Note that in our procedure there may be vertices of
degree~$2$ in the kernel while there is no vertex of degree~$2$ in the
kernel of a $2$-uniform hypergraph.

In the kernel all edges have size $2$ or~$3$. We call these edges
$2$-edges and $3$-edges in the kernel, resp.  It is trivial from the
description above that in the kernel every degree~$2$ vertex is
contained in at least one $3$-edge. We say that any multihypergraph in
which all edges have size~$2$ or $3$, there are no vertices of
degree~$1$, and every vertex of degree~$2$ is in at least one edge of
size~$3$, is a kernel.  The reason for this is that given such a
multihypergraph, one can create a \prekernel\ following the procedure
we discuss next. Consider the following operation: we \textdef{split}
one $2$-edge with vertices $u$ and $v$ by deleting the edge and adding
a new vertex $w$ and two new $2$-edges, one containing $u$ and $w$ and
the other containing $v$ and $w$.  Given the
kernel of a \prekernel, one can split edges from the kernel in a way
that it reverses the steps of the procedure for finding the
kernel. After including the vertices of degree $1$ in the $2$-edges,
the resulting graph is the \prekernel. Note that, by replacing
$2$-edges in the kernel by splitting $2$-edges and adding new vertices
(of degree~$1$) to all $2$-edges, the resulting multigraph does not
have any isolated cycle. Thus, whenever the resulting multigraph is
simple, it is a \prekernel.

\subsection{Random kernels and pre-kernels}
\label{sec:random-prek-hyper}

Recall that our aim in Section~\ref{sec:pre-hyper} is to find an
asymptotic formula for $\gpre(n,m)$, the number of connected
\prekernels\ with vertex set $[n]$ and $m$ edges. Similarly to
Section~\ref{sec:random-core-hyper} about random cores, we show how to
reduce the enumeration problem for \prekernels\ to approximating the
expectation, in a probability space of random degree sequences, of the
probability that a random graph with given degree sequence is
connected and simple.

We will describe a procedure to generate pre-kernels. 
% Describe random model
For $x=(\nn,\kzero,\kone,\ktwo)\in S_m\cap\setZ^4$, let
$\Dcal(x)\subseteq \setN^{\nthree(x)}$\label{Dcalx} be such that $\ds\in\Dcal(x)$
if $d_i\geq 3$ for all $i$ and $\sum_{i=1}^{\nthree(x)} d_i =
\Qthree(x)$.  Our strategy to generate a random \prekernel\ is the
following. We start by choosing the vertices and $3$-edges that will
be in the kernel. We then generate a random kernel with degree
sequence $\ds$ for the vertices of degree at least~$3$, $k_1$+$k_2$
vertices of degree~$2$, $\mtwoprime$ $2$-edges and $\mthree$ $3$-edges
so that $k_i$ vertices of degree $2$ are contained in exactly $i$
$3$-edges for $i=1,2$. The \prekernel\ is then obtained by splitting
$2$-edges $\kzero$ times and assigning the vertices of degree~$1$.

The kernel is generated in a way similar to the random cores in
Section~\ref{sec:core-hyper} but with different restrictions. For each
vertex, we create a bin/set with the number of points inside it equal
to the degree of the vertex. These bins are called
\textdef{vertex-bins}. For each edge, we create one bin/set with $2$
or $3$ points inside it, depending on whether it is a $2$-edge or a
$3$-edge. These bins are called \textdef{edge-bins}. Each point in a
vertex-bin will be matched to a point in an edge-bin with some
constraints.  The kernel can then be obtained by creating one edge for
each edge-bin $i$ such that the vertices incident to it are the
vertices with points matched to point in the edge-bin~$i$. We describe
how to generate a random kernel $\Kcal(V,M_3,\kone,\ktwo,\ds)$\label{glo:kernel} where
$V\subseteq[n]$ is a set of size $\kone+\ktwo+\nthree$ and
$M_3\subseteq[m]$ is a set of size $\mthree$. In each step, every
choice is made \uar\ among all possible choices satisfying the stated
conditions:
\begin{enumerate}
\item\textit{(Vertex-bins)} Choose a set $V_3$ of $\nthree$ vertices
  in $V$ to be the vertices of degree at least~$3$. Let $v_1<\dotsm <
  v_{\nthree}$ be an enumeration of $V_3$. For each $i\in [\nthree]$,
  create a vertex-bin $v_i$ with points \textred{labelled} $1,\dotsc, d_i$ inside~it. For
  each $v\in V\setminus V_3$, create a vertex-bin $v$ with points
  \textred{labelled} $1$
  and $2$ inside it\label{mod6}.
  \item\textit{(Edge-bins)} For each $i\in M_3$, create one edge-bin
    with points \textred{labelled} $1$, $2$, and $3$ inside it. Let $M_2 = \set{(i,0):
      i\in[\mtwoprime]}$. For each $i\in M_2$, create one edge-bin
    with points \textred{labelled} $1$ and $2$ inside it.
  \item\textit{(Matching)} Match the points from the vertex-bins to
    the points in edge-bins so that, for $i=1,2$, $k_i$ vertex-bins
    with two points have exactly $i$ points being matched to an
    edge-bin of size~$3$. This matching is called a
    \textdef{kernel-configuration} with parameters
    $(V,M_3,\kone,\ktwo,\ds)$.
  \item\textit{(Kernel)} The kernel $\Kcal(V,M_2,\kone,\ktwo)=
    (V,M_2\cup M_3,\Phi)$ is the multihypergraph such that for each
    $E\in M_2\cup M_3$, we have that $\Phi(E,i) = v$, where $v$ is the
    vertex corresponding to the vertex-bin containing $j$ and $j$ is
    the point matched to point $i$ in the edge-bin $E$ in the previous
    step.
  \end{enumerate}
  See Figure~\ref{fig:kernel-proc-hyper} for an example of this
  procedure. The constraints in Step~3 ensures that each vertex of
  degree~$2$ is contained by at least one $3$-edge and so the
  procedure above always generates a kernel. It is also trivial that
  all kernels (with edges $M_3 \cup M_2$) can be generated by this
  procedure.
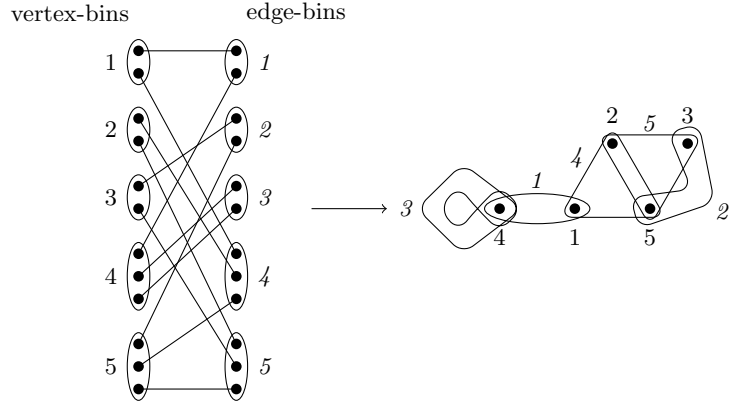
\begin{figure}
  \centering
  \begin{tikzpicture}
    \def \initialg {0};
    \def \step {1};
    \def \rad {2pt};
    \def \eps {.3};
    \def \epsm {.15};
    \def \epsg {.2};
    \def \round {5pt};
    \def \scalep {0.3}

    % Bins
    % Vertex-bins
    \coordinate (1v1) at ($(\initialg,\initialg)$);
    \coordinate (1v2) at ($(1v1)+\scalep*\step*(0,-1)$);

    \coordinate (2v1) at ($(1v2)+\scalep*\step*2*(0,-1)$);
    \coordinate (2v2) at ($(2v1)+\scalep*\step*(0,-1)$);

    \coordinate (3v1) at ($(2v2)+\scalep*\step*2*(0,-1)$);
    \coordinate (3v2) at ($(3v1)+\scalep*\step*(0,-1)$);

    \coordinate (4v1) at ($(3v2)+\scalep*\step*2*(0,-1)$);
    \coordinate (4v2) at ($(4v1)+\scalep*\step*(0,-1)$);
    \coordinate (4v3) at ($(4v2)+\scalep*\step*(0,-1)$);

    \coordinate (5v1) at ($(4v3)+\scalep*\step*2*(0,-1)$);
    \coordinate (5v2) at ($(5v1)+\scalep*\step*(0,-1)$);
    \coordinate (5v3) at ($(5v2)+\scalep*\step*(0,-1)$);
       
    \foreach \point in {1v1, 1v2} \fill [black] (\point) circle (\rad);
    \foreach \point in {2v1, 2v2} \fill [black] (\point) circle (\rad);
    \foreach \point in {3v1, 3v2} \fill [black] (\point) circle (\rad);
    \foreach \point in {4v1, 4v2, 4v3} \fill [black] (\point) circle (\rad);
    \foreach \point in {5v1, 5v2, 5v3} \fill [black] (\point) circle (\rad);

    % Edge-bins
    \coordinate (1e1) at ($(1v1)+1.3*\step*(1,0)$);
    \coordinate (1e2) at ($(1e1)+\scalep*\step*(0,-1)$);

    \coordinate (2e1) at ($(1e2)+\scalep*\step*2*(0,-1)$);
    \coordinate (2e2) at ($(2e1)+\scalep*\step*(0,-1)$);

    \coordinate (3e1) at ($(2e2)+\scalep*\step*2*(0,-1)$);
    \coordinate (3e2) at ($(3e1)+\scalep*\step*(0,-1)$);

    \coordinate (4e1) at ($(3e2)+\scalep*\step*2*(0,-1)$);
    \coordinate (4e2) at ($(4e1)+\scalep*\step*(0,-1)$);
    \coordinate (4e3) at ($(4e2)+\scalep*\step*(0,-1)$);

    \coordinate (5e1) at ($(4e3)+\scalep*\step*2*(0,-1)$);
    \coordinate (5e2) at ($(5e1)+\scalep*\step*(0,-1)$);
    \coordinate (5e3) at ($(5e2)+\scalep*\step*(0,-1)$);

    \foreach \point in {1e1, 1e2} \fill [black] (\point) circle (\rad);
    \foreach \point in {2e1, 2e2} \fill [black] (\point) circle (\rad);
    \foreach \point in {3e1, 3e2} \fill [black] (\point) circle (\rad);
    \foreach \point in {4e1, 4e2, 4e3} \fill [black] (\point) circle (\rad);
    \foreach \point in {5e1, 5e2, 5e3} \fill [black] (\point) circle (\rad);

    \coordinate[label=left:\footnotesize{vertex-bins}] (vbins) at ($(1v1)+\step*(0,.5)$);
    \coordinate[label=right:\footnotesize{edge-bins}] (ebins) at ($(1e1)+\step*(0,.5)$);
    % Vertex-bins ellipse
    \draw ($.5*(1v1)+.5*(1v2)$) ellipse ({\epsm} and {0.15+\epsm}); 
    \draw ($.5*(2v1)+.5*(2v2)$) ellipse ({\epsm} and {0.15+\epsm}); 
    \draw ($.5*(3v1)+.5*(3v2)$) ellipse ({\epsm} and {0.15+\epsm}); 
    \draw (4v2) ellipse ({\epsm} and {0.3+\epsm}); 
    \draw (5v2) ellipse ({\epsm} and {0.3+\epsm}); 

    % Edge-bins ellipse
    \draw ($.5*(1e1)+.5*(1e2)$) ellipse ({\epsm} and {0.15+\epsm}); 
    \draw ($.5*(2e1)+.5*(2e2)$) ellipse ({\epsm} and {0.15+\epsm}); 
    \draw ($.5*(3e1)+.5*(3e2)$) ellipse ({\epsm} and {0.15+\epsm}); 
    \draw (4e2) ellipse ({\epsm} and {0.3+\epsm}); 
    \draw (5e2) ellipse ({\epsm} and {0.3+\epsm}); 

    % Vertex labels
    \coordinate[label={left:\footnotesize{$1$}}] 
    (1vL) at ($.5*(1v1)+.5*(1v2)-(\epsm,0)$);
    \coordinate[label={left:\footnotesize{$2$}}] 
    (1vL) at ($.5*(2v1)+.5*(2v2)-(\epsm,0)$);
    \coordinate[label={left:\footnotesize{$3$}}] 
    (1vL) at ($.5*(3v1)+.5*(3v2)-(\epsm,0)$);
    \coordinate[label={left:\footnotesize{$4$}}] 
    (4vL) at ($(4v2)-(\epsm,0)$);
    \coordinate[label={left:\footnotesize{$5$}}] 
    (5vL) at ($(5v2)-(\epsm,0)$);

    % Edge labels
    \coordinate[label={right:\footnotesize{\textit{1}}}] 
    (1eL) at ($.5*(1e1)+.5*(1e2)+(\epsm,0)$);
    \coordinate[label={right:\footnotesize{\textit{2}}}] 
    (1eL) at ($.5*(2e1)+.5*(2e2)+(\epsm,0)$);
    \coordinate[label={right:\footnotesize{\textit{3}}}] 
    (1eL) at ($.5*(3e1)+.5*(3e2)+(\epsm,0)$);
    \coordinate[label={right:\footnotesize{\textit{4}}}] 
    (4eL) at ($(4e2)+(\epsm,0)$);
    \coordinate[label={right:\footnotesize{\textit{5}}}] 
    (5eL) at ($(5e2)+(\epsm,0)$);

    % Matching
    \draw (1v1)--(1e1);
    \draw (1v2)--(4e1);
    \draw (2v1)--(4e2);
    \draw (2v2)--(5e1);
    \draw (3v1)--(2e1);
    \draw (3v2)--(5e2);
    \draw (4v1)--(1e2);
    \draw (4v2)--(3e1);
    \draw (4v3)--(3e2);
    \draw (5v1)--(2e2);
    \draw (5v2)--(4e3);
    \draw (5v3)--(5e3);

    \coordinate  (startarrow) at ($(3e2)+(\step,0)$);
    \coordinate  (endarrow) at ($(startarrow)+(\step,0)$);
    \draw [->] (startarrow)--(endarrow);

    % Kernel
    % Vertices
    \coordinate[label={[label distance=4pt]-90:\footnotesize{$4$}}]
    (4) at ($(endarrow)+\step*(1.5,0)$);
    \coordinate[label={[label distance=4pt]-90:\footnotesize{$1$}}]
    (1) at ($(4)+\step*(1,0)$);
    \coordinate[label={[label distance=4pt]-90:\footnotesize{$5$}}]
    (5) at ($(1)+\step*(1,0)$);
    \coordinate[label={[label distance=4pt]90:\footnotesize{$2$}}]
    (2) at ($(1)+\step*({sin(30)},{cos(30)})$);
    \coordinate[label={[label distance=4pt]90:\footnotesize{$3$}}]
    (3) at ($(2)+\step*(1,0)$);

    \foreach \point in {1,2,3,4,5} 
    \fill [black] (\point) circle (\rad);

    % Edges
    \coordinate[label={[label distance=4pt]90:\footnotesize{\textit{1}}}] 
    (b0) at ($.5*(1)+.5*(4)$);
    \draw (b0) ellipse ({.5*\step+\epsg} and {\epsg}); 
  
    \coordinate[label={[label distance=-1pt]above left:\footnotesize{\textit{4}}}] 
    (b12) at ($.5*(1)+.5*(2)$);
    \coordinate (b1) at ($1/3*(1)+1/3*(2)+1/3*(5)$);
    \coordinate (a1) at ($-0.4*(b1)+1.4*(2)$);
    \coordinate (a2) at ($-0.4*(b1)+1.4*(1)$);
    \coordinate (a3) at ($-0.4*(b1)+1.4*(5)$);
    %\foreach \point in {b1,a1,a2,a3} \fill [red] (\point) circle (\rad);
    \draw[rounded corners=\round] (a1)--(a2)--(a3)--cycle;
    
    \coordinate[label={[label distance=1pt]above:\footnotesize{\textit{5}}}] 
    (b23) at ($.5*(3)+.5*(2)$);
    \coordinate (b2) at ($1/3*(3)+1/3*(2)+1/3*(5)$);
    \coordinate (c1) at ($-0.4*(b2)+1.4*(2)$);
    \coordinate (c2) at ($-0.4*(b2)+1.4*(3)$);
    \coordinate (c3) at ($-0.4*(b2)+1.4*(5)$);
    %\foreach \point in {b2,c1,c2,c3} \fill [red] (\point) circle (\rad);
    \draw[rounded corners=\round] (c1)--(c2)--(c3)--cycle;

    \coordinate (b3) at ($.5*(5)+.5*(3)$);
    \coordinate (d1) at ($(b3)+.35*({sin(120)},{cos(120)})$);
    \coordinate[label={[label distance=-5pt]below right:\footnotesize{\textit{2}}}]  
    (d2) at ($(b3)+.7*({sin(120)},{cos(120)})$);
    \coordinate (d3) at ($(3)+.3*({sin(30)},{cos(30)})$);
    \coordinate (d4) at ($(5)-.3*({sin(30)},{cos(30)})$);
    \coordinate (d5) at ($(3)+.3*({sin(-60)},{cos(-60)})$);
    \coordinate (d6) at ($(5)+.3*({sin(-60)},{cos(-60)})$);
    %\foreach \point in {b3,d1,d2,d3,d4,d5,d6} \fill [red] (\point) circle (\rad);
    \draw[rounded corners=\round]
    (d1)--(d5)--(d3)--(d2)--(d4)--(d6)--cycle;

    \coordinate (b4) at (4);
    \coordinate (e1) at ($(4)+0.3*(0,1)$);
    \coordinate (e2) at ($(4)+0.3*(1,0)$);
    \coordinate (e3) at ($(4)-0.8*(0,1)$);
    \coordinate[label={[label distance=-3pt]180:\footnotesize{\textit{3}}}]
    (e4) at ($(b4)-1.1*(1,0)$);
    \coordinate (e7) at ($(4)-0.3*(0,1)$);
    \coordinate (e8) at ($(4)-0.3*(1,0)$);
    \coordinate (e10) at ($(e4)+0.3*(1,0)$);    
    \coordinate (e9) at ($0.5*(e8)+0.5*(e10)+0.3*(0,1)$);
    \coordinate (e11) at ($0.5*(e8)+0.5*(e10)-0.3*(0,1)$);    
    \coordinate (e5) at ($(e9)+0.3*(0.2,1)$);
    \coordinate (e3) at ($(e11)-0.3*(-0.2,1)$);

    %\foreach \point in {e1,e2,e3,e4,e5,e7,e8,e9,e10,e11}
    %\fill [red] (\point) circle (\rad);
    \draw[rounded corners=\round]
    (e1)--(e2)--(e3)--(e4)--(e5)--(e2)--(e7)--(e8)--(e9)--(e10)--(e11)--(e8)--cycle;
  \end{tikzpicture}
  \caption{A kernel generated with vertex and edge-bins}
  \label{fig:kernel-proc-hyper}
\end{figure}

We now describe the \prekernel\ model precisely. For
$x=(\nn,\kzero,\kone,\ktwo)\in S_m\cap\setZ^4$ and $\ds\in\Dcal(x)$,
let $\Gpre(x,\ds)=\Gpre_{n,m}(x,\ds)$ be the random graph generated as
follows. In each step, every choice is made \uar\ among all possible
choices satisfying the stated conditions:\label{glo:gpre}

 \begin{enumerate}
 \item\textit{(Kernel)} Let $V$ be a subset of $[n]$ of size
   $n-\nn-\kzero$ and $M_3$ be a subset of $[m]$ of size
   $\mthree(x)$. Let $\Kcal=(V,M_\Kcal,\Phi_\Kcal)$ be the random kernel
   $\Kcal(V,M_3,\kone,\ktwo,\ds)$.
 \item\textit{(Splitting edges)}Let $V_{\kzero}$ be a subset of
   $[n]\setminus$ of size $\kzero$. This set will be the set of
   vertices of degree at $2$ contained by two $2$-edges. Let
   $v_1,\dotsc, v_{\kzero}$ be an enumeration of the vertices
   in~$V_{\kzero}$.  Let $P = \Kcal$. For $i=1$ to $\kzero$, do the
   following operation: split a $2$-edge of $P$ with new vertex $v_i$
   and update $P$.
 \item\textit{(Assigning $2$-edges and vertices of degree $1$)} Let
   $V_1$ be a subset of size $\nn$ in $[n]\setminus V$. These will be
   the vertices of degree~$1$ in the multigraph. Assign for each
   $2$-edge $E$ of $P$ a (unique) edge~$E'$ from $[m]\setminus M_3$
   and a (unique) vertex $u$ in $V_1$. Place a perfect matching $M_{E'}$
   between the collection $\set{\Phi_\Kcal(E,1),\Phi_\Kcal(E,2), u}$
   and $\set{1,2,3}$.  We call this matching together with the
   sequence of splittings in the previous step a
   \textdef{splitting-configuration}.
 \item\textit{(\Prekernel)} Let $\Gpre(\nn,\kzero,\kone,\ktwo,\ds) =
   ([n],[m],\Phi)$, where $\Phi(E,\cdot) = \Phi_\Kcal(E,\cdot)$ if
   $E\in M_3$ and, otherwise, $\Phi(E,i)=v$, where $v$ is the vertex
   matched to~$i$ in~$M_E$.
\end{enumerate}

When the procedure above results in a (simple) graph, it is a
\prekernel\ since it is obtained by splitting the $2$-edges of a
kernel and assigning vertices of degree~$1$ to the $2$-edges. It is
trivial all \prekernels\ are generated since all kernels and the ways
of splitting the edges are considered. 

% Describe magic trick
For $(\nn,\kzero,\kone,\ktwo)\in S_m$, let
$\gpre(n,m,\nn,\kzero,\kone,\ktwo)$ denote the number of connected
(simple) pre-kernels with vertex set $[n]$ and $m$ edges such that $\nn$
vertices have degree~$1$, and $\kzero+\kone+\ktwo$ vertices have
degree~2 so that $k_i$ of the degree~$2$ vertices are incident to
exactly $i$ $3$-edges for $i=0,1,2$. For
$\ds\in\Dcal(\nn,\kzero,\kone,\ktwo)$, let
$\gpre(n,m,\nn,\kzero,\kone,\ktwo,\ds)$ denote the number of such
\prekernels\ with the additional constraint that $\ds$ is the degree
sequence of the vertices of degree at least~$3$.

In order to analyse $\gpre(n,m,\nn,\kzero,\kone,\ktwo,\ds)$ it will be
useful to know the number of kernel-configurations.
\begin{lem}
\label{lem:number-kernel-hyper}
Let $x=(\nn,\kzero,\kone,\ktwo)\in S_m\cap\setZ^4$ and $\ds\in\Dcal(x)$. The
number of kernel-configurations with parameters
$(V,M_3,\kone,\ktwo,\ds)$, where $V$ is a set of size
$\kone+\ktwo+\nthree$ and $M_3$ is a set of size $\mthree$, is
\begin{equation*}
  \begin{split}
    \binom{\kone+\ktwo+\nthree}{\nthree}
    \binom{\kone+\ktwo}{\kone}2^{\kone}\binom{\Pthree}{\kone+2\ktwo}(\kone+2\ktwo)!\binom{\Ptwo}{\kone}\kone!\Qthree!=
    \frac{(\kone+\ktwo+\nthree)! \Pthree!\Ptwo!\Qthree!
      2^{\kone}}{\nthree!\kone!\ktwo!\Tthree!\Ttwo!}.
\end{split}
\end{equation*}
Moreover, each kernel with parameters
  $(V,M_3,\kone,\ktwo,\ds)$ is generated by
  exactly  $2^{\kone+\ktwo}\prod_{i=1}^{\nthree}d_i!$ kernel-configurations.
\end{lem}
\begin{proof}
   There are $\binom{\kone+\ktwo+\nthree}{\nthree}$ ways of
  choosing the vertices of degree at least~$3$ in the first step. The
  step where the kernel-configuration is created can be described in
  the following more detailed way:
  \begin{enumerate}
  \item Choose $\kone$ vertex-bins of size $2$. Let $U$ be a set
    containing exactly one point of each of these vertex-bins and let
    $D$ be the set consisting of all points in  vertex-bins of size~$2$
    that are not in~$U$.
  \item Choose $\kone+2\ktwo$ points inside edges-bins of size~$3$ and
    match them to points in $D$.
  \item Choose $\kone$ points inside edges-bins of size~$2$ and match
    them to points in $U$.
  \item Match the remaining unmatched $\Qthree$ points from the vertex-bins to
    the unmatched points in the edge-bins.
  \end{enumerate}
  In Step~1, there are $\binom{\kone+\ktwo}{\kone}$ choices for the
   vertex-bins of size~$2$ and $2^{\kone}$ choices for $U$.  There are
  $\binom{\Pthree}{\kone+2\ktwo}(\kone+2\ktwo)!$ choices for Step~2,
  $\binom{\Ptwo}{\kone}\kone!$ for Step~3 and $\Qthree!$ choices for
  Step~4. The first part of the lemma then follows trivially.

  Each kernel with parameters $(V,M_3,\kone,\ktwo,\ds)$ is generated
  by $2^{\kone+\ktwo}\prod_{i=1}^{\nthree}d_i!$ distinct
  kernel-configurations, because any permutation of the points inside
  vertex-bins can be done without changing the resulting kernel.
\end{proof}

The following proposition relates
$\gpre(n,m,\nn,\kzero,\kone,\ktwo,\ds)$ and
$\gpre(n,m,\nn,\kzero,\kone,\ktwo)$ to the random \prekernels\
$\Pcal(x,\ds)$ and random degree sequences. The proof is similar to
the proof of Proposition~\ref{prop:magic-core-hyper}. We include it
here for completeness.
\begin{prop}
  \label{prop:magic-pre-hyper}
 For $x=(\nn,\kzero,\kone,\ktwo)\in S_m\cap\setZ^4$ and $\ds\in\Dcal(x)$,
\begin{equation*}
  \begin{split}
    &\gpre(n,m,\nn,\kzero,\kone,\ktwo,\ds)
    \\
    &=
    n!
    \frac{\Pthree!\Ptwo!\Qthree!(\mtwo-1)! \prob{\Gpre(\nn,
        \kzero,\kone,\ktwo,\ds)\text{ simple and connected}}}
    {\kzero!\kone!\ktwo!\nthree!\mthree!
      \Tthree!\Ttwo! 
      (\mtwoprime-1)!\mtwoprime!2^{\ktwo}  2^{\mtwoprime}
      6^{\mthree}
      \prod_{i} d_i!}
\end{split}
\end{equation*}  
and, if $\Qthree(x) > \nthree(x)$, then
\begin{equation}
  \label{eq:magic-pre-hyper}
  \begin{split}
    &\gpre(n,m,\nn,\kzero,\kone,\ktwo)
    \\
    &=
    n!
    \frac{\Pthree!\Ptwo!\Qthree!(\mtwo-1)! }
    {\kzero!\kone!\ktwo!\nthree!\mthree!
      (\Pthree-\kone-2\ktwo)!(\Ptwo- \kone)! 
      (\mtwoprime-1)!\mtwoprime!2^{\ktwo}  2^{\mtwoprime}
      6^{\mthree}
    }
    \\
    &\qquad\cdot
    \frac{f_3(\lambda)^{\nthree}}{\lambda^{\Qthree}}
    \meancond[\Big]{\prob[\big]{\Gpre(\nn,
        \kzero,\kone,\ktwo,\Ys)\text{ simple and
          connected}}}{\Sigma(x)} 
    \prob[\big]{\Sigma(x)},
\end{split}
\end{equation}
where $\Ys = (Y_1,\dotsc, Y_{\nthree})$ is a vector of independent
random variables $Y_1,\dotsc, Y_{\nthree}$ such that each $Y_i$ has
truncated Poisson distribution with parameters $(3, \lambda(x))$ and
$\Sigma(x)$ denotes the event $\sum_i Y_i= \Qthree$.
\end{prop}
\begin{proof}
   Any multigraph obtained by
  the process for $\Pcal(x,\ds)$ is generated by
  $2^{\kone+\ktwo} (\prod_{i=1}^{\nthree}d_i!) \mtwoprime!
  2^{\mtwoprime}$ combinations of kernel-configurations and
  splitting-configuration.  This is because each kernel is generated
  by $2^{\kone+\ktwo} \prod_{i=1}^{\nthree}d_i!$ kernel-configurations
  by Lemma~\ref{lem:number-kernel-hyper} and permuting the labels and
  points inside each of the $2$-edges in the kernel do not change the
  resulting multigraph. Thus, by
  Lemma~\ref{lem:number-graph-multigraph-hyper}, each \prekernel\ with
  parameters $(x,\ds)$ is generated by
  \begin{equation}
  \alpha\eqdef   
  2^{\kone+\ktwo}\paren[\Big]{\prod_{i=1}^{\nthree}d_i!} 
  \mtwoprime! 2^{\mtwoprime} 
  m! 6^m 
\end{equation}
combinations of kernel-configurations and
splitting-configurations. Next we compute the total number of such
combinations. In Step~1 in which we generate the kernel, there are
$\binom{n}{\kone+\ktwo+\nthree}$ ways of choosing $V$ and
$\binom{m}{\mthree}$ ways of choosing $M_3$. The number of ways of
generating the kernel is
\begin{equation*}
  \begin{split}
    \frac{(\kone+\ktwo+\nthree)! \Pthree!\Ptwo!\Qthree!
      2^{\kone}}{\nthree!\kone!\ktwo!\Tthree!\Ttwo!}
\end{split}
\end{equation*}
by Lemma~\ref{lem:number-kernel-hyper}.  In Step~2, there are
$\binom{\nn+\kzero}{\kzero}$ ways of choosing $V_{\kzero}$ and
$\mtwoprime(\mtwoprime+1)\cdots(\mtwoprime+\kzero-1) =
(\mtwo-1)!/(\mtwoprime-1)!$ ways of splitting the edges. In Step~$3$,
that are $(\mtwo!)^2$ ways of assigning the $2$-edges and vertices of
degree~$1$ and $6^{\mtwo}$ ways of placing the matchings. Thus, the
total number of combinations of kernel-configurations and
splitting-configurations is
\begin{equation*}
  \binom{n}{\kone+\ktwo+\nthree}
  \binom{m}{\mthree} 
  \frac{(\kone+\ktwo+\nthree)! \Pthree!\Ptwo!\Qthree!
    2^{\kone}}{\nthree!\kone!\ktwo!\Tthree!\Ttwo!}
  \binom{\nn+\kzero}{\kzero}\frac{(\mtwo-1)!}{(\mtwoprime-1)!}
  (\mtwo!)^2 6^{\mtwo}\eqdefinv \beta.
\end{equation*}
Hence, since each combination is
generated with the same probability, we have that
\begin{equation}
  \label{eq:prop-magic-pre-aux-hyper}
  \gpre(\nn,\kzero,\kone,\ktwo,\ds)
  =\frac{\beta}{\alpha}
  \prob{\Gpre(\nn,\kzero,\kone,\ktwo,\ds)\text{ simple and connected}},
\end{equation}
where $\beta$ is the total number of configurations.
which together with~\eqref{eq:prop-magic-pre-aux-hyper} and trivial
simplifications implies~\eqref{prop:magic-pre-hyper}.

We now prove~\eqref{eq:magic-pre-hyper}. Again, the proof is very
similar to\Old{ the proofs of Proposition~\ref{prop:enum-pairing-2c} and
Proposition~\ref{prop:enum-kernel-config-2c} in
Chapter~\ref{chap:2c}}\New{~\cite[Equation~(13)]{PWa}}. For $x=(\nn,\kzero,\kone,\ktwo)$, let $U(x,\ds)$
denote the probability that $\Gpre_{n,m}(\nn,\kzero,\kone,\ktwo,\ds)$
is simple and connected. For $x=(\nn,\kzero,\kone,\ktwo)$,
\begin{equation*}
  \begin{split}
    &\gpre(n,m,\nn,\kzero,\kone,\ktwo)
    \eqdef
    \sum_{\ds\in\Dcal(x)}
    \gpre(n,m,\nn,\kzero,\kone,\ktwo,\ds)
    \\
    &=
    n! \sum_{\ds\in\Dcal(x)}
    \frac{\Pthree!\Ptwo!\Qthree!(\mtwo-1)! }
    {\kzero!\kone!\ktwo!\nthree!\mthree!
      \Tthree!\Ttwo! 
      (\mtwoprime-1)!\mtwoprime!2^{\ktwo}  2^{\mtwoprime}
      6^{\mthree}
      \prod_{i} d_i!}
     U(\nn,\kzero,\kone,\ktwo,\ds)
     \\
     &=
     n! \frac{\Pthree!\Ptwo!\Qthree!(\mtwo-1)!}
     {\kzero!\kone!\ktwo!\nthree!\mthree!
       \Tthree!\Ttwo! 
      (\mtwoprime-1)!\mtwoprime!2^{\ktwo}  2^{\mtwoprime}
      6^{\mthree}}
    \frac{\fff(\lambda(x))^{\nthree}}{\lambda(x)^{\Qthree}}
       \sum_{\ds\in\Dcal(x)}
       \prod_{i} \frac{ \lambda(x)^{d_i}}
     {d_i! \fff(\lambda(x))}
     U(x,\ds) 
     \\
     &=
     n! \frac{\Pthree!\Ptwo!\Qthree!(\mtwo-1)!}
     {\kzero!\kone!\ktwo!\nthree!\mthree!
       \Tthree!\Ttwo! 
      (\mtwoprime-1)!\mtwoprime!2^{\ktwo}  2^{\mtwoprime}
      6^{\mthree}}
    \frac{\fff(\lambda(x))^{\nthree}}{\lambda(x)^{\Qthree}}
       \sum_{\ds\in\Dcal(x)}
       U(x,\ds) 
       \prob{\Ys=\ds}
     \\
     &=
   n! \frac{\Pthree!\Ptwo!\Qthree!(\mtwo-1)!}
     {\kzero!\kone!\ktwo!\nthree!\mthree!
       \Tthree!\Ttwo! 
      (\mtwoprime-1)!\mtwoprime!2^{\ktwo}  2^{\mtwoprime}
      6^{\mthree}}
    \frac{\fff(\lambda(x))^{\nthree}}{\lambda(x)^{\Qthree}}
    \meancond{ U(x,\Ys) }{\Sigma(x)}
    \prob{\Sigma(x)}
  \end{split}
\end{equation*}
which proves~\eqref{eq:magic-pre-hyper}. 
\end{proof}

The goal of the next lemmas is to show that the expectation
in~\eqref{eq:magic-pre-hyper} goes to~$1$ for points $x$ close to
$\xopt$. For $x\in S_m\cap\setZ^4$ and $\phi = \phi(n) > 0$, let
\begin{equation*}
  \tD_{\phi}(x)
  =
  \set[\big]{\ds\in\Dcal(x): |\eta(\ds)-\mean{\eta(\Ys)}| \leq R\phi}
\end{equation*}
where $\eta(\ds) \eqdef \sum_{i=1}^{\nthree} {d_i(d_i-1)}/\paren{2m}$,
and recall that $R = m-n/2$. We will show that, for some function
$\phi = o(1)$, conditioned upon~$\Sigma(x)$, the probability that
$\Ys$ is in $\tD_{\phi}(x)$ goes to~$1$. Intuitively, this means that
the set $\tD_{\phi}(x)$ contains all `typical' degree sequences for
points $x\in S$ that are close to~$\xopt$. For $\psi = \psi(n) =
o(1)$, let
\begin{equation}
\label{eq:Spsi}
  \begin{split}
    S_{\psi}^*  
    =
    {\bigg\{}
    x &= (\nn,\kzero,\kone,\ktwo) \in S:
    \abs[\Big]{\prenn - \frac{1}{2}} \leq \psi r;\\
    &\abs[\Big]{\prekzero - \frac{1}{2}} \leq \psi r;\
    \abs[\Big]{\prekone - 6r} \leq \psi r;\\
    &\abs[\Big]{\prektwo - 18 r^2} \leq \psi r^2;\
    \abs[\Big]{\prenthree - 2r}\leq \psi r;\\
    &\abs[\Big]{\preQthree - 6r}\leq \psi r;\
    \abs[\Big]{\premthree - 2r} \leq \psi r;\\
    &\abs[\Big]{\premtwoprime - 6r}\leq \psi r;\
    \abs[\Big]{\preTtwo - 6r}\leq \psi r;\\
    &\abs[\Big]{\preTthree - 36 r^2}\leq \psi r^2
    {\bigg \} }.
  \end{split}
\end{equation}
We define $S_{\psi}^*$ this way so that all points in it are close to
$\xopt$, where we are using~\eqref{eq:optpre-rel-hyper} to find around
which values each of the functions in the definition of $S_{\psi}^*$
should be concentrated. The idea is to define $\psi$ later in a way
that it is small enough so that we can approximate the summation of
$n!\exp(n\fpre(\prex))$ in the integer points $x$ in $S_{\psi}^*$, but
large enough so that what is not included do not significant effect in
the summation $\sum_{x\in S_m\cap\setZ^4} n!\exp(n\fpre(\prex))$.

% Lemma for typical degree sequences
Next we show that for points in $x\in S_{\psi}^*$ with $\psi=o(1)$ the set
$\tD_{\phi}(x)$ (for some $\phi=o(1)$) is a set of `typical' degree sequences.
\begin{lem}
\label{lem:typical-pre-hyper}
Let $\psi = o(1)$. There exists $\phi = o(1)$ such that, for every
integer point $x\in S_{\psi}^*$, we have that
$\probcond{\Ys\in\tD_{\phi(x)}}{\Sigma(x)} = 1-o(1)$.
\end{lem}

We then show that for $x=x(n)\in S_{\psi}^*\cap\setZ^4$ and $\ds\in \tD_{\phi(x)}$,
the random \prekernel\ $\Pcal(x,\ds)$ is connected and simple \aas
% Lemma for prob of simple
\begin{lem}
  \label{lem:simple-pre-hyper}
  Assume $R=o(n)$. Let $\psi, \phi = o(1)$. Let $x=x(n)\in S_{\psi}^*$
  be an integer point and
  $\ds=\ds(n)\in\tD_{\phi}(x)$. Then $\Gpre(x, \ds)$ is simple \aas
\end{lem}

% Lemma for prob of connected
\begin{lem}
  \label{lem:connected-pre-hyper}
  Assume $R=o(n)$. Let $\psi, \phi = o(1)$. Let $x\in S_{\psi}^*$ be
  an integer point and
  $\ds = \ds(n)\in\tD_{\phi}(x)$. Then $\Gpre(x, \ds)$ is connected
  \aas
\end{lem}

The proofs for
Lemmas~\ref{lem:typical-pre-hyper},~\ref{lem:simple-pre-hyper},
and~\ref{lem:connected-pre-hyper} are presented in
Sections~\ref{sec:typical-pre-hyper}, ~\ref{sec:simple-pre-hyper},
and~\ref{sec:connected-pre-hyper}, respectively. We now show how to
prove that the expectation in~\eqref{eq:magic-pre-hyper} goes to $1$
assuming Lemmas~\ref{lem:typical-pre-hyper},~\ref{lem:simple-pre-hyper},
and~\ref{lem:connected-pre-hyper}.
% Corollary of lemmas
\begin{cor}
  \label{cor:expectation-pre-hyper}
  Let $\psi=o(1)$ and let $x=(\nn,\kzero,\kone,\ktwo)\in S_{\psi}^*\cap\setZ^4$.
  Then
  \begin{equation*}
    \meancond[\Big]{\prob[\big]{\Gpre(\nn,
        \kzero,\kone,\ktwo,\Ys)\text{ simple and
          connected}}}{\Sigma(x)}
    \sim 1.
  \end{equation*}
\end{cor}
\begin{proof}
  Let $U(\Ys)$ denote the probability that $\Gpre(x,\Ys)$ is connected and
  simple. Let $\phi = o(1)$ be given by
  Lemma~\ref{lem:typical-pre-hyper}. We have that
  \begin{equation*}
    \begin{split}
      \meancond[\Big]{U(\Ys)}{\Sigma(x)}
      &\geq
      \sum_{\ds\in\tD_{\phi}(x)}
      \prob{U(\ds)}
      \probcond{\Ys=\ds}{\Sigma(x)}.
    \end{split}
  \end{equation*}
  By Lemmas~\ref{lem:simple-pre-hyper}
  and~\ref{lem:connected-pre-hyper}, we have that
  $\prob{U(\ds)}=1-o(1)$ for every $\ds=\ds(n)\in\tD_{\phi}(x)$. Since
  $\tD_{\phi}(x)$ is a finite set for each $n$, this implies that
  there exists a function $q(n) = o(1)$ such that $\prob{U(\ds)}\geq
  1-q(n)$ for every $\ds\in\tD_{\phi}(x)$\crisc{Removed reference to a
    uniformity lemma in the appendix}. Thus,
  \begin{equation*}
    \begin{split}
      \meancond[\Big]{u(\Ys)}{\Sigma(x)}
      \geq
      (1-q(n))\prob{\Ys\in\tD_{\phi}(x)}
      = 1-o(1).
    \end{split}
  \end{equation*}
  by Lemma~\ref{lem:typical-pre-hyper}.
\end{proof}

\subsection{Typical degree sequences}
\label{sec:typical-pre-hyper}
In this section, given an integer point $x\in S_m$ `close' to the
point $\xopt$ (more precisely $x\in S_{\psi}^*$ and $\psi=o(1)$), we
show that, for a random vector of $\Ys = (Y_1,\dotsc, Y_{\nthree(x)})$
of independent truncated Poisson random variables with parameters
$(3,\lambda(x))$ conditioned upon the event $\Sigma(x)$ that
$\sum_{i=1}^{\nthree(x)} Y_i = \Qthree(x)$, the value of
$\sum_{i=1}^{\nthree(x)}\binom{Y_i}{2}$ is concentrated around its
expected value. More specifically, we present the proof for
Lemma~\ref{lem:typical-pre-hyper}. Recall that
\begin{equation*}
  \tD_{\phi}(x)
  =
  \set[\big]{\ds\in\Dcal(x): |\eta(\ds)-\mean{\eta(\Ys)}| \leq R\phi}
\end{equation*}
where $\eta(\ds) = \sum_{i=1}^{\nthree} {d_i(d_i-1)}/\paren{2m}$. We
want to show that, given $x\in S_{\psi}^*\cap\setZ^4$ with $\psi=o(1)$, there
exists $\phi = o(1)$ such that
$\probcond{\Ys\in\tD_{\phi(x)}}{\Sigma(x)} > 1-\phi$, where $\Ys =
(Y_1,\dotsc,Y_{\nthree})$ is a vector of independent random variables
with distribution $\tpoisson{3}{\lambda(x)}$.

Recall that $\nthree \sim 2rn = 2R \to\infty$, and $\Qthree/\nthree
\sim 6r/(2r) = 3$ for $x\in S_{\psi}^*$. Thus, by the definition of
$\lambda(x)$ (in~\eqref{eq:lambda-def-pre-hyper}) and~\textred{\cite[Lemma
1]{PWa}}, we must have $\lambda(x) = o(1)$. Then by~\textred{\cite[Lemma 2]{PWa}}, $\var{Y_i(Y_i-1)} = \Theta(\lambda)$.
Thus, by Chebyshev's inequality,
    \begin{equation*}
      \prob[\Big]{|\eta(\Ys)-\mean{\eta(\Ys)}|\geq R\phi}
      \leq
      \frac{\var{\eta(\Ys)}}{R^2\phi^2}
      =
      \frac{\nthree\Theta(\lambda)}{R^2\phi^2}
      =
      o\left(\frac{\nthree}{R^2\phi^2}\right).
    \end{equation*}
    If $\Rthree\eqdef \Qthree-3\nthree \leq \log \nthree$, by~\textred{\cite[Theorem 4]{PWa}}
     and Stirling's approximation
    \crisc{removed reference to appendix}
    \begin{equation*}
      \prob{\Sigma(x)} = (1+o(1))
      e^{-\Rthree}\frac{\Rthree^{\Rthree}}{\Rthree!}
      =
      \Omega\left(
      \frac{1}{\sqrt{\Rthree}}\right)
      =
      \Omega\left(
      \frac{1}{\sqrt{\log \nthree}}\right).
    \end{equation*}
    If $\Qthree-3\nthree \geq \log \nthree$, by~\textred{\cite[Theorem 4]{PWa}},
    %Theorem~\ref{thm:thm4-PWa-prelim},
  \begin{equation*}
    \prob{\Sigma(x)}
    \sim
    \frac{1}{\sqrt{2\pi \nthree \cthree(1+\etathree-\cthree)}}
    =
    \Omega\left(\frac{1}{\sqrt{\nthree}}\right),
  \end{equation*}
  where $\cthree = \Qthree/\nthree$ and $\etathree =
  \lambda(x)\f(\lambda(x))/\ff(\lambda(x))$, and we
  used~\textred{\cite[Lemma 2]{PWa}}. Thus,
  \begin{equation*}
    \prob[\Big]{|\eta(\Ys)-\mean{\eta(\Ys)}|\geq R\phi | \Sigma}
    =
    O\left(\frac{\nthree}{R^2\phi^2}\sqrt{\nthree}\right)
    =
    O\left(\frac{1}{R^{1/2}\phi^2}\right)
 \end{equation*}
 since $\nthree \sim  2R$ and so it is suffices to choose
 $\phi^2 = \omega(\sqrt{1/R})$. This finishes the proof of
 Lemma~\ref{lem:typical-pre-hyper}.

\subsection{Simple pre-kernels}
\label{sec:simple-pre-hyper} 
In this section, given an integer point $x\in S_m$ `close' to the
point $\xopt$ and $\ds\in\setN^{\nthree}$ with some constraints (more
precisely $x\in S_{\psi}^*$ and $\ds\in\tilde\Dcal_{\phi}(x)$ with
$\psi,\phi=o(1)$), we show that the random multigraph $\Pcal(x,\ds)$
defined in Section~\ref{sec:random-prek-hyper} is simple \aas, thus
proving Lemma~\ref{lem:simple-pre-hyper}. Recall that a multigraph is
simple if it has no loops and no double edges (as defined in
Section~\ref{sec:def-hyper}). Any loop (or double edge) involving only
$3$-edges in the kernel remains a loop (or double edge) in the
\prekernel. Any double edge involving $2$-edges in the kernel will not
be a double edge in the \prekernel, because each $2$-edge will be
assigned a unique vertex of degree~$1$ in the procedure that creates
the \prekernel\ from the kernel. A loop in the kernel that is an
$2$-edge will cease to be a loop in the \prekernel\ if it is split at
least once. Note that, if a $2$-edge that is a loop in the kernel is
split exactly once, the two $2$-edges created will not form a double
edge in the final multigraph since the assignment of vertices of
degree $1$ to the $2$-edges eliminates all double edges involving
$2$-edges. It is clear that no other loops or double edges can be
created. We rewrite these conditions for the kernel-configuration: the
pre-kernel $\Pcal = \Pcal(x,\ds)$ is simple if and only if
  \begin{enumerate}
  \item[(A)] (\textit{No loops in $3$-edges}) No edge-bin of size~$3$
    has at least $2$ points matched to points from the same
    vertex-bin.
  \item[(B)] (\textit{No double $3$-edges}) Assuming no loops in
    $3$-edges, no pair of edges-bins of size~$3$ has their points
    matched to points in the same $3$ vertices.
  \item[(C)] (\textit{No loops in $2$-edges}) For every edge-bin of
    size~$2$, its points are matched to points from distinct
    vertex-bins or the $2$-edge corresponding to this edge-bin is
    split at least once in the process that obtains the \prekernel\
    from the kernel.
  \end{enumerate}

  We will show that, for $x\in S_{\psi}^*$ and
  $\ds\in\tilde\Dcal_{\phi}(x)$ with $\psi,\phi=o(1)$, the random
  multigraph $\Pcal(x,\ds)$ is simple \aas, which proves
  Lemma~\ref{lem:simple-pre-hyper}. We need to show that each of the
  conditions (A), (B) and (C) holds \aas{} We will use the detailed
  procedure for obtaining kernel-configurations described in the proof
  of Lemma~\ref{lem:number-kernel-hyper}.  We work in the probability
  space conditioned upon the vertices of degree~$3$ and the points in
  $U$ being already chosen, since the particular choices of these
  vertices and points do not affect the probability of loops or double
  edges in the kernel.

  First we prove (A) holds \aas{} Consider the case that the loop is
  on a vertex of degree $2$. There are $\ktwo$ possible choices for
  the vertex-bin. There are $\mthree$ choices for the edge-bin of
  size~$3$ and $3\cdot 2$ choices for the points inside of the
  edge-bin to be matched to the points in the vertex-bin of
  size~$2$. Thus, we have $6\ktwo \mthree$ choices.  Following the
  proof of Lemma~\ref{lem:number-kernel-hyper}, after the vertices of
  degree~$3$ and $U$ are chosen, there are
  \begin{equation}
    \label{eq:kernel-comp-hyper}
    \binom{\Pthree}{\kone+2\ktwo}({\kone+2\ktwo})!
    \binom{\Ptwo}{\kone}\kone!
    \Qthree!
  \end{equation}
  ways of completing the kernel-configuration. The number of
  completions of kernel-configurations containing a given matching
  that matches $2$ points in a vertex-bin of size~$2$ to $2$ points in
  an edge-bin of size~$3$ is then
  \begin{equation*}
    \binom{\Pthree-2}{\kone+2\ktwo-2}(\kone+2\ktwo-2)!
    \binom{\Ptwo}{\kone}\kone!
    \Qthree!
  \end{equation*}
  Thus, using the definition of $S_{\phi}^*$, the probability that
  there is a loop on a vertex of degree $2$ in a $3$-edge is at most
  \begin{equation*}
    6\ktwo\mthree 
    \frac{1}{\Pthree (\Pthree-1)}
    =
    O\left(\frac{\ktwo\mthree}{\Pthree^2}\right)
    =
    O\left(\frac{(r^2n)(r n)}{(r n)^2}\right)
    =O(r) = o(1).
  \end{equation*}
  Now consider the case that the loop is on a vertex of degree at least~$3$.
  There are $\sum_{i=1}^{\nthree} \binom{d_i}{2} = \eta(\ds)$ possible
  choices for the vertex-bin and $2$ points inside it. Since
  $\ds\in\Dcal(x)$ and $\mean{\eta(\Ys)} = \nthree \mean{Y_1(Y_1-1)}
  \sim 6 \nthree = \Theta(R)$,
  \begin{equation*}
    \eta(\ds) = 
    \Theta(\nthree).    
  \end{equation*}
  There are $\mthree$ choices for the edge-bin of size~$3$ and $3\cdot
  2$ choices for the points inside of the edge-bin to be matched to
  the chosen points in the vertex-bin. Thus, we have
  $O(\nthree\mthree)$ choices. The number of completions of
  kernel-configurations containing one given matching that matches $2$
  points in a vertex-bin of size at least~$3$ and $2$ points in a
  edge-bin of size~$3$ is
  \begin{equation*}
    \binom{\Pthree-2}{\kone+2\ktwo}({\kone+2\ktwo})!
    \binom{\Ptwo}{\kone}\kone!
    (\Qthree-2)!
  \end{equation*}
  Thus, using~\eqref{eq:kernel-comp-hyper}, the probability that there
  is a loop on a vertex-bin of size at least $3$ in edge-bin of
  size~$3$ is
  \begin{equation*}
    O\left(\nthree\mthree
      \cdot\frac{(\Pthree-2)!}{\Pthree!}
      \frac{(\Qthree-2)!}{\Qthree!}
      \frac{(\Tthree)!}{(\Tthree-2)!}
    \right)
    =
    O\left(\frac{\nthree\mthree\Tthree^2}{\Pthree^2\Qthree^2}\right)
    =
    O\left(\frac{(r n)(r n)(r^2 n)^2}{(r n)^2(r n)^2}\right)
    =O(r^2) = o(1).
  \end{equation*}
  This finishes the proof that Condition (A) holds \aas{} Now we prove
  that Condition (B) holds \aas\ We consider $4$ cases:

\begin{itemize}
\item[(B1)] The edge-bins corresponding to the double edge have their
  points matched to points in $3$ vertex-bins all of size $2$.
\item[(B2)] The edge-bins corresponding to the double edge have their
  points matched to points in $2$ vertex-bins of size $2$ and $1$
  vertex-bin of size at least~$3$.
\item[(B3)] The edge-bins corresponding to the double edge have their
  points matched to points in $1$ vertex-bin of size $2$ and $2$
  vertex-bins of size at least~$3$.
\item[(B4)] The edge-bins corresponding to the double edge have none of their
  points matched to points in vertex-bins of size $2$.
\end{itemize}
Let us start with (B1).  We have $O(\ktwo^3\mthree^2)$ choices for the
$3$ vertex-bins and $2$ edge-bins involved. There are $O(1)$ matchings
between the points $6$ in these vertex-bins and the $6$ points in
these edge-bins that creates a double edge.  The number of completions
for the kernel-configurations containing a giving matching creating
such a double edge is
\begin{equation*}
  \binom{\Pthree-6}{\kone+2\ktwo-6}({\kone+2\ktwo-6})!
  \binom{\Ptwo}{\kone}\kone!
  \Qthree!,
\end{equation*}
where we are following the proof of
Lemma~\ref{lem:number-kernel-hyper}, after the vertices of degree~$3$
and $U$ are chosen.  Thus, using~\eqref{eq:kernel-comp-hyper}, the 
expected number of double edges as in (B1) is at most
\begin{equation*}
  O\left(\ktwo^3 \mthree^2\right)
  \frac{(\Pthree-6)!}{\Pthree!}   
  =
  O\left(\frac{\ktwo^3 \mthree^2}{\Pthree^6}\right)
  =
  O\left(\frac{(r^2n)^3 (r n)^2}{(r n)^6}\right)
  = 
  O\left(\frac{r^2}{n}\right).
\end{equation*}

Now let us consider (B2). We have $O(\ktwo^2 \eta(\ds) \mthree^2 )$
choices for the 2 vertex-bins of size $2$ and the points inside them,
the vertex-bin of size at least $3$ and the points inside them, and the
$2$ edge-bins of size~$3$ involved in the double edge. We match $4$
points from the $2$ vertex-bins of size $2$ to the $4$ points in the
edge-bins of size~$3$ and $2$ points from the vertex-bin of size at least
$3$ to $2$ points in the edges-bins of size~$3$.  The number of
completions for the kernel-configurations containing a giving matching
creating such a double edge is
\begin{equation*}
 \binom{\Pthree-6}{\kone+2\ktwo-4}({\kone+2\ktwo-4})!
  \binom{\Ptwo}{\kone}\kone!
  (\Qthree-2)!
\end{equation*}
Thus, using~\eqref{eq:kernel-comp-hyper} and the definition of
$S_{\psi}^*$,  the 
expected number of double edges as in (B2) is at most
\begin{equation*}
  \begin{split}
    O\left({\ktwo^2 \nthree \mthree^2}\right)
  \frac{(\Pthree-6)!}{\Pthree!}
  \frac{(\Qthree-2)!}{\Qthree!}
  \frac{\Tthree!}{(\Tthree-2)!}
  &=
  O\left(\frac{\ktwo^2 \nthree \mthree^2 \Tthree^2}{\Pthree^6 \Qthree^2}\right)
  \\
  &=
  O\left(\frac{(r^2 n)^2(r n)(r n)^2 (r^2 n)^2}
    {(r n)^6 (r n)^2} \right)
  = O\left(\frac{r^3}{n}\right).
  \end{split}
\end{equation*}
We analyse (B3) now. There are $2$ vertex-bins of size at least~$3$ involved.  We
have $O(\ktwo \eta(\ds)^2 \mthree^2)$ choices for the vertex-bin of
size $2$, the $2$ vertex-bins of size at least $3$ and the points
inside them, and the $2$ edge-bins involved. There are $O(1)$ matchings
between the $6$ points in the vertex-bins ($2$ in the vertex-bin of
size $2$ and $4$ in the other vertex-bins) and the $6$ points in the
edge-bins creating a double edge. The number of completions
for the kernel-configurations containing a giving matching creating
such a double edge is
\begin{equation*}
 \binom{\Pthree-6}{\kone+2\ktwo-2}({\kone+2\ktwo-2})!
  \binom{\Ptwo}{\kone}\kone!
  (\Qthree-4)!
\end{equation*}
Thus, using~\eqref{eq:kernel-comp-hyper} and the definition of
$S_{\psi}^*$,  the 
expected number of double edges as in (B3) is at most
\begin{equation*}
  \begin{split}
    O\left({\ktwo \nthree^{2} \mthree^2}\right)
    \frac{(\Pthree-6)!}{\Pthree!}
    \frac{(\Qthree-4)!}{\Qthree!}
    \frac{\Tthree!}{(\Tthree-4)!}
    &=
    O\left(\frac{\ktwo \nthree^{2} \mthree^2\Tthree^4}{\Pthree^6 \Qthree^4}\right)
    \\
    &=
    O\left(\frac{(r^2 n) (r n)^{2} (r n)^2(r^2 n)^4}
      {(r n)^6 (r n)^4}\right)
    = O\left(\frac{r^4}{n}\right).
  \end{split}
\end{equation*}
We analyse (B4) now.  We have $O(\eta(\ds)^3 \mthree^2)$ choices for
the $3$ vertex-bins of size at least $3$ and the points inside them
and the $2$ edge-bins involved. There are $O(1)$ matchings between the
$6$ points in the vertex-bins and the $6$ points in the edge-bins
creating a double edge. The number of completions for the
kernel-configurations containing a giving matching creating such a
double edge is
\begin{equation*}
 \binom{\Pthree-6}{\kone+2\ktwo}({\kone+2\ktwo})!
  \binom{\Ptwo}{\kone}\kone!
  (\Qthree-6)!
\end{equation*}
Thus, using~\eqref{eq:kernel-comp-hyper} and the definition of
$S_{\psi}^*$, the expected number of double edges as in (B4) is at most
\begin{equation*}
  \begin{split}
    O\left({\nthree^3 \mthree^2}\right)
    \frac{(\Pthree-6)!}{\Pthree!}
    \frac{(\Qthree -6)!}{\Qthree!}
    \frac{\Tthree!}{(\Tthree -6)!}
    &=
    O\left(\frac{\nthree^3 \mthree^2 \Tthree^6}{\Pthree^6\Qthree^6}\right)
    \\
    &=
    O\left(\frac{(r n)^3 (r n)^2 (r^2 n)^6}
      {(r n)^6(r n)^6}\right) 
    = O\left(\frac{r^5}{n}\right).
  \end{split}
\end{equation*}
This finishes the proof of that Condition (B) holds \aas

Now consider the event in case (C). First we will bound the expected
number of edge-bins of size $2$ with points matched to points from the
same vertex-bin (and so corresponding to loops in the kernel). Since
every vertex-bin of size $2$ has at least one point being matched to a
point in an edge-bin of size~$3$, if an edge-bin of size~$2$ has
points matched to the same vertex-bin, such vertex-bin must have size
at least~$3$.  Thus, we have $\eta(\ds)=\Theta(\nthree)$ choices for
such vertex-bin and the two points inside it that will be matched to
the points in the $2$-edge, and $\mtwoprime$ choices for the edge-bin
of size~$2$ (and $2$ choices for the matching of these points).  The
number of completions for the kernel-configurations containing a
giving matching creating such a loop is
\begin{equation*}
 \binom{\Pthree}{\kone+2\ktwo}({\kone+2\ktwo})!
 \binom{\Ptwo-2}{\kone}\kone!
 (\Qthree-2)!
\end{equation*}
Thus, using~\eqref{eq:kernel-comp-hyper} and the definition of
$S_{\psi}^*$, the expected number of loops as in (C) is at most
\begin{equation*}
  \begin{split}
    \frac{(\Ptwo-2)!}{\Ptwo!}
    \frac{(\Qthree -2)!}{\Qthree!}
    \frac{\Ttwo!}{(\Ttwo -2)!}
    O\left(
      {\nthree\mtwoprime}
    \right)
    &=
    O\left(
      \frac{\nthree\mtwoprime \Ttwo^2}{\Ptwo^2\Qthree^2}
    \right)
    \\
    &=
    O\left(
      \frac{(r n)(r n) (r n)^2}{(r n)^2(r n)^2}
    \right)
    = O(1).(C)
  \end{split}
\end{equation*}
So let $\alpha(n)\to \infty$ such that $\alpha r \to 0$. Then the
number of edge-bins corresponding to $2$-edges that are loops in the
kernel is less than $\alpha$ \aas\ For any $2$-edge in the kernel, let
$A_i$ be the event that it is not split by the $i$-th splitting
operation performed when creating the \prekernel\ from the
kernel. Then
\begin{equation*}
  \begin{split}
    \prob[\bigg]{\bigcap_{i=1}^{\kzero} A_i}
    &=
    \prod_{i=1}^{\kzero}
    \probcond[\bigg]{A_i}{\bigcap_{j=1}^{j-1} A_j}
    =
    \frac{\mtwoprime -1}{\mtwoprime}
    \frac{\mtwoprime}{\mtwoprime +1}
    \cdot\cdots\cdot
    \frac{\mtwoprime +\kzero -2}{\mtwoprime +\kzero-1}
    \\
    &=\frac{\nn-\kzero-1}{\nn-1}
    \sim\frac{6r n}{(1/2)n}
    \sim{12 r}.
  \end{split}
\end{equation*}
This together with the fact the expected number of $2$-edges that are
loops in the kernel is less than $\alpha$ \aas\ implies that the
probability there is a $2$-edge that is a loop in the pre-kernel is $O(\alpha r)+o(1)
= o(1)$.  This finishes the proof of Lemma~\ref{lem:simple-pre-hyper}.

\subsection{Connected pre-kernels}
\label{sec:connected-pre-hyper}
  In
this section, we analyse the probability that the random multigraph
$\Pcal(x,\ds)$ is connected for $x$ `close' to $\xopt$ and
$\ds\in\setN^{\nthree}$ with some constraints (more precisely $x\in
S_{\psi}^*$ and $\ds\in\tilde\Dcal_{\phi}(x)$ with
$\psi,\phi=o(1)$). We will show that $\Pcal(x,\ds)$ is connected \aas,
proving Lemma~\ref{lem:connected-pre-hyper}. Our strategy has
  some similarities with the proof by \Luczak\cite{LucA} for connected
  random $2$-uniform hypergraphs with given degree sequence and
  minimum degree at least~$3$. The main difference is that, in our
  case, we have some vertices of degree~$2$ and the matching on the
  set of points in the bins has some constraints because of these
  vertices. This makes it more difficult to compute the probability of
  connectedness.

A \prekernel\ is connected if and only of its kernel is connected,
since the \prekernel\ is obtained by splitting $2$-edges of the kernel
and assigning vertices of degree~$1$. Thus, we only need to analyse
the connectivity of the kernel. Let $\ds$ denote the degree sequence
of the vertices of degree at least~$3$, $k_i$ the number of vertices
of degree $2$ that are in exactly $i$ $3$-edges (for $i=1,2$),
$\mtwoprime$ the number of $2$-edges and $\mthree$ the number of
$3$-edges.

We say that a kernel-configuration is \textdef{connected} if the
$2$-uniform multigraph described as follows is connected: contract
each vertex-bin and each edge-bin into a single vertex and add one
edge $uv$ for each edge $ij$ of the matching in the
kernel-configuration such that $i$ is in the bin corresponding to~$u$
and $j$ is in the bin corresponding to $v$.  Given a
kernel-configuration with matching $M$, perform the following
operations:
\begin{enumerate}
\item For each vertex-bin $v$ with more than $6$ points, partition the
  points of $v$ into new vertex-bins so that each of the new
  vertex-bins has $3$, $4$ or $5$ points. Delete $v$ and keep $M$
  unchanged. See Figure~\ref{fig:splitting-vertex-hyper}.
\item For each edge-bin $e$ of size~$2$ such that exactly one of its
  points, say $p_e$, is matched to a point, say $p_v$, in a vertex-bin $v$ of
  size~$2$, do the following. Let $p_e'$ be the point in $e$ other
  than $p_e$ and let $p_v'$ be the point in $v$ other than $p_v$. Let
  $i$ be the point matched to $p_e'$ in $M$ and let $j$ be the point
  matched to $p_v'$ in $M$. Delete $v$ and $e$ from the
  kernel-configuration. Add a new edge to $M$ connecting $i$ and~$j$.
  See Figure~\ref{fig:accross-hyper}.
\item For each edge-bin $e$ of size~$2$ such that both of its points
  $p_e$ and $p_e'$ are matched to points $p_v$ and $p_{w}$ in
  vertex-bins $v$ and $w$ of size~$2$, do the following. Let $p_v'$ be
  the point in $v$ other than $p_v$ and let $p_w'$ be the point in $w$
  other than $p_w$. Let $i$ be the point matched to $p_v'$ in $M$ and
  let $j$ be the point matched to $p_w'$ in $M$. Delete $v$, $w$ and
  $e$ from the kernel-configuration. Create a new vertex-bin of size
  $2$ with points $p_v'$ and $p_w'$ and add the edges $p_v'i$ and
  $p_w'j$ to~$M$. See Figure~\ref{fig:left-bins-hyper}.
\end{enumerate}
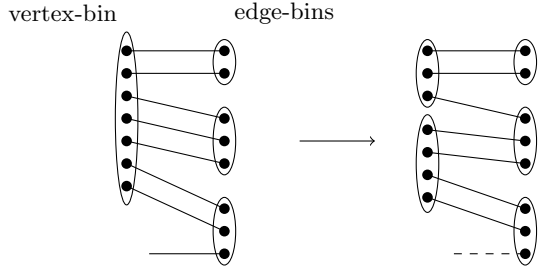
\begin{figure}
  \centering
  \begin{tikzpicture}
    \def \initialg {0};
    \def \step {1};
    \def \rad {2pt};
    \def \eps {.3};
    \def \epsm {.15};
    \def \epsg {.2};
    \def \round {5pt};
    \def \scalep {0.3}

    % Bins
    % Vertex-bins
    \coordinate (1v1) at ($(\initialg,\initialg)$);
    \coordinate (1v2) at ($(1v1)+\scalep*\step*(0,-1)$);
    \coordinate (1v3) at ($(1v2)+\scalep*\step*(0,-1)$);
    \coordinate (1v4) at ($(1v3)+\scalep*\step*(0,-1)$);
    \coordinate (1v5) at ($(1v4)+\scalep*\step*(0,-1)$);
    \coordinate (1v6) at ($(1v5)+\scalep*\step*(0,-1)$);
    \coordinate (1v7) at ($(1v6)+\scalep*\step*(0,-1)$);

    \foreach \point in {1v1, 1v2,1v3,1v4,1v5,1v6,1v7} \fill [black] (\point) circle (\rad);
    
    % Edge-bins
    \coordinate (1e1) at ($(1v1)+1.3*\step*(1,0)$);
    \coordinate (1e2) at ($(1e1)+\scalep*\step*(0,-1)$);

    \coordinate (2e1) at ($(1e2)+\scalep*\step*2*(0,-1)$);
    \coordinate (2e2) at ($(2e1)+\scalep*\step*(0,-1)$);
    \coordinate (2e3) at ($(2e2)+\scalep*\step*(0,-1)$);

    \coordinate (3e1) at ($(2e3)+\scalep*\step*2*(0,-1)$);
    \coordinate (3e2) at ($(3e1)+\scalep*\step*(0,-1)$);
    \coordinate (3e3) at ($(3e2)+\scalep*\step*(0,-1)$);

    \foreach \point in {1e1, 1e2} \fill [black] (\point) circle (\rad);
    \foreach \point in {2e1, 2e2, 2e3} \fill [black] (\point) circle (\rad);
    \foreach \point in {3e1, 3e2, 3e3} \fill [black] (\point) circle (\rad);
    
    \coordinate[label=left:\footnotesize{vertex-bin}] (vbins) at ($(1v1)+\step*(0,.5)$);
    \coordinate[label=right:\footnotesize{edge-bins}] (ebins) at ($(1e1)+\step*(0,.5)$);
    % Vertex-bins ellipse
    \draw (1v4) ellipse ({\epsm} and {\step+\epsm}); 

    % Edge-bins ellipse
    \draw ($.5*(1e1)+.5*(1e2)$) ellipse ({\epsm} and {0.15+\epsm}); 
    \draw (2e2) ellipse ({\epsm} and {0.3+\epsm}); 
    \draw (3e2) ellipse ({\epsm} and {0.3+\epsm}); 

    % Matching
    \draw (1v1)--(1e1);
    \draw (1v2)--(1e2);
    \draw (1v3)--(2e1);
    \draw (1v4)--(2e2);
    \draw (1v5)--(2e3);
    \draw (1v6)--(3e1);
    \draw (1v7)--(3e2);
    \coordinate (p3e3) at ($(3e3)+\step*({cos(180)},{sin(180)})$);
    \draw (3e3)--(p3e3);

    \coordinate  (startarrow) at ($(2e2)+(\step,0)$);
    \coordinate  (endarrow) at ($(startarrow)+(\step,0)$);
    \draw [->] (startarrow)--(endarrow);

    % Bins
    % Vertex-bins
    \coordinate (1v1) at ($(\initialg+4*\step,\initialg)$);
    \coordinate (1v2) at ($(1v1)+\scalep*\step*(0,-1)$);
    \coordinate (1v3) at ($(1v2)+\scalep*\step*(0,-1)$);

    \coordinate (2v1) at ($(1v2)+\scalep*\step*2.5*(0,-1)$);
    \coordinate (2v2) at ($(2v1)+\scalep*\step*(0,-1)$);
    \coordinate (2v3) at ($(2v2)+\scalep*\step*(0,-1)$);
    \coordinate (2v4) at ($(2v3)+\scalep*\step*(0,-1)$);

    \foreach \point in {1v1, 1v2,1v3} \fill [black] (\point) circle (\rad);
    \foreach \point in {2v1, 2v2, 2v3, 2v4} \fill [black] (\point) circle (\rad);
    
    % Edge-bins
    \coordinate (1e1) at ($(1v1)+1.3*\step*(1,0)$);
    \coordinate (1e2) at ($(1e1)+\scalep*\step*(0,-1)$);

    \coordinate (2e1) at ($(1e2)+\scalep*\step*2*(0,-1)$);
    \coordinate (2e2) at ($(2e1)+\scalep*\step*(0,-1)$);
    \coordinate (2e3) at ($(2e2)+\scalep*\step*(0,-1)$);

    \coordinate (3e1) at ($(2e3)+\scalep*\step*2*(0,-1)$);
    \coordinate (3e2) at ($(3e1)+\scalep*\step*(0,-1)$);
    \coordinate (3e3) at ($(3e2)+\scalep*\step*(0,-1)$);

    \foreach \point in {1e1, 1e2} \fill [black] (\point) circle (\rad);
    \foreach \point in {2e1, 2e2, 2e3} \fill [black] (\point) circle (\rad);
    \foreach \point in {3e1, 3e2, 3e3} \fill [black] (\point) circle (\rad);
    
    % Vertex-bins ellipse
    \draw (1v2) ellipse ({\epsm} and {0.3*\step+\epsm}); 
    \draw ($1/2*(2v2)+1/2*(2v3)$) ellipse ({\epsm} and {0.5*\step+\epsm}); 

    % Edge-bins ellipse
    \draw ($.5*(1e1)+.5*(1e2)$) ellipse ({\epsm} and {0.15+\epsm}); 
    \draw (2e2) ellipse ({\epsm} and {0.3+\epsm}); 
    \draw (3e2) ellipse ({\epsm} and {0.3+\epsm}); 

    % Matching
    \draw (1v1)--(1e1);
    \draw (1v2)--(1e2);
    \draw (1v3)--(2e1);
    \draw (2v1)--(2e2);
    \draw (2v2)--(2e3);
    \draw (2v3)--(3e1);
    \draw (2v4)--(3e2);
    \coordinate (p3e3) at ($(3e3)+\step*({cos(180)},{sin(180)})$);
    \draw[dashed] (3e3)--(p3e3);

  \end{tikzpicture}
  \caption{Breaking a vertex-bin into smaller pieces.}
  \label{fig:splitting-vertex-hyper}
\end{figure}

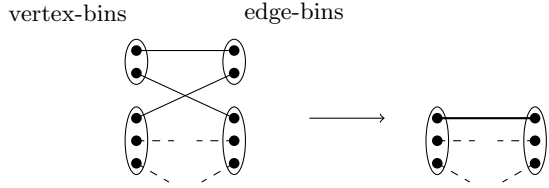
\begin{figure}
  \centering
  \begin{tikzpicture}
    \def \initialg {0};
    \def \step {1};
    \def \rad {2pt};
    \def \eps {.3};
    \def \epsm {.15};
    \def \epsg {.2};
    \def \round {5pt};
    \def \scalep {0.3}

    % Bins
    % Vertex-bins
    \coordinate (1v1) at ($(\initialg,\initialg)$);
    \coordinate (1v2) at ($(1v1)+\scalep*\step*(0,-1)$);
    
    \coordinate (2v1) at ($(1v2)+\scalep*\step*2*(0,-1)$);
    \coordinate (2v2) at ($(2v1)+\scalep*\step*(0,-1)$);
    \coordinate (2v3) at ($(2v2)+\scalep*\step*(0,-1)$);
    
    \foreach \point in {1v1, 1v2} \fill [black] (\point) circle (\rad);
    \foreach \point in {2v1, 2v2, 2v3} \fill [black] (\point) circle (\rad);
    
    % Edge-bins
    \coordinate (1e1) at ($(1v1)+1.3*\step*(1,0)$);
    \coordinate (1e2) at ($(1e1)+\scalep*\step*(0,-1)$);

    \coordinate (2e1) at ($(1e2)+\scalep*\step*2*(0,-1)$);
    \coordinate (2e2) at ($(2e1)+\scalep*\step*(0,-1)$);
    \coordinate (2e3) at ($(2e2)+\scalep*\step*(0,-1)$);

    \foreach \point in {1e1, 1e2} \fill [black] (\point) circle (\rad);
    \foreach \point in {2e1, 2e2, 2e3} \fill [black] (\point) circle (\rad);
    
    \coordinate[label=left:\footnotesize{vertex-bins}] (vbins) at ($(1v1)+\step*(0,.5)$);
    \coordinate[label=right:\footnotesize{edge-bins}] (ebins) at ($(1e1)+\step*(0,.5)$);
    % Vertex-bins ellipse
    \draw ($.5*(1v1)+.5*(1v2)$) ellipse ({\epsm} and {0.15+\epsm}); 
    \draw (2v2) ellipse ({\epsm} and {0.3+\epsm}); 

    % Edge-bins ellipse
    \draw ($.5*(1e1)+.5*(1e2)$) ellipse ({\epsm} and {0.15+\epsm}); 
    \draw (2e2) ellipse ({\epsm} and {0.3+\epsm}); 
    
    % Matching
    \draw (1v1)--(1e1);
    \draw (1v2)--(2e1);
    \draw (2v1)--(1e2);
     \coordinate (p2v2) at ($(2v2)+\step*0.5*({cos(0)},{sin(0)})$);
    \draw[dashed] (2v2)--(p2v2);
     \coordinate (p2v3) at ($(2v3)+\step*0.5*({cos(-30)},{sin(-30)})$);
    \draw[dashed] (2v3)--(p2v3);
     \coordinate (p2e2) at ($(2e2)+\step*0.5*({cos(180)},{sin(180)})$);
    \draw[dashed] (2e2)--(p2e2);
     \coordinate (p2e3) at ($(2e3)+\step*0.5*({cos(210)},{sin(210)})$);
    \draw[dashed] (2e3)--(p2e3);

    \coordinate  (startarrow) at ($(2e1)+(\step,0)$);
    \coordinate  (endarrow) at ($(startarrow)+(\step,0)$);
    \draw [->] (startarrow)--(endarrow);

    % Bins
    % Vertex-bins
    \coordinate (1v1) at ($(\initialg+\step*4,\initialg)$);
    \coordinate (1v2) at ($(1v1)+\scalep*\step*(0,-1)$);
    
    \coordinate (2v1) at ($(1v2)+\scalep*\step*2*(0,-1)$);
    \coordinate (2v2) at ($(2v1)+\scalep*\step*(0,-1)$);
    \coordinate (2v3) at ($(2v2)+\scalep*\step*(0,-1)$);
    
    %\foreach \point in {1v1, 1v2} \fill [black] (\point) circle (\rad);
    \foreach \point in {2v1, 2v2, 2v3} \fill [black] (\point) circle (\rad);
    
    % Edge-bins
    \coordinate (1e1) at ($(1v1)+1.3*\step*(1,0)$);
    \coordinate (1e2) at ($(1e1)+\scalep*\step*(0,-1)$);

    \coordinate (2e1) at ($(1e2)+\scalep*\step*2*(0,-1)$);
    \coordinate (2e2) at ($(2e1)+\scalep*\step*(0,-1)$);
    \coordinate (2e3) at ($(2e2)+\scalep*\step*(0,-1)$);

    %\foreach \point in {1e1, 1e2} \fill [black] (\point) circle (\rad);
    \foreach \point in {2e1, 2e2, 2e3} \fill [black] (\point) circle (\rad);
    
    % Vertex-bins ellipse
    %\draw ($.5*(1v1)+.5*(1v2)$) ellipse ({\epsm} and {0.15+\epsm}); 
    \draw (2v2) ellipse ({\epsm} and {0.3+\epsm}); 

    % Edge-bins ellipse
    %\draw ($.5*(1e1)+.5*(1e2)$) ellipse ({\epsm} and {0.15+\epsm}); 
    \draw (2e2) ellipse ({\epsm} and {0.3+\epsm}); 
    
    % Matching
    %\draw (1v1)--(1e1);
    %\draw (1v2)--(2e1);
    %\draw (2v1)--(1e2);
    \draw[thick] (2v1)--(2e1);
     \coordinate (p2v2) at ($(2v2)+\step*0.5*({cos(0)},{sin(0)})$);
    \draw[dashed] (2v2)--(p2v2);
     \coordinate (p2v3) at ($(2v3)+\step*0.5*({cos(-30)},{sin(-30)})$);
    \draw[dashed] (2v3)--(p2v3);
     \coordinate (p2e2) at ($(2e2)+\step*0.5*({cos(180)},{sin(180)})$);
    \draw[dashed] (2e2)--(p2e2);
     \coordinate (p2e3) at ($(2e3)+\step*0.5*({cos(210)},{sin(210)})$);
    \draw[dashed] (2e3)--(p2e3);
  \end{tikzpicture}
  \caption{Transforming an edge-bin of size~$2$ matched to a
    vertex-bin of size~$2$ into an edge of the matching}
  \label{fig:accross-hyper}
\end{figure}

\begin{figure}
  \centering
  \begin{tikzpicture}
    \def \initialg {0};
    \def \step {1};
    \def \rad {2pt};
    \def \eps {.3};
    \def \epsm {.15};
    \def \epsg {.2};
    \def \round {5pt};
    \def \scalep {0.3}

    % Bins
    % Vertex-bins
    \coordinate (1v1) at ($(\initialg,\initialg)$);
    \coordinate (1v2) at ($(1v1)+\scalep*\step*(0,-1)$);
    
    \coordinate (2v1) at ($(1v2)+\scalep*\step*2*(0,-1)$);
    \coordinate (2v2) at ($(2v1)+\scalep*\step*(0,-1)$);
        
    \foreach \point in {1v1, 1v2} \fill [black] (\point) circle (\rad);
    \foreach \point in {2v1, 2v2} \fill [black] (\point) circle (\rad);
    
    % Edge-bins
    \coordinate (1e1) at ($(1v1)+1.3*\step*(1,0)$);
    \coordinate (1e2) at ($(1e1)+\scalep*\step*(0,-1)$);

    \coordinate (2e1) at ($(1e2)+\scalep*\step*2*(0,-1)$);
    \coordinate (2e2) at ($(2e1)+\scalep*\step*(0,-1)$);
    \coordinate (2e3) at ($(2e2)+\scalep*\step*(0,-1)$);

    \coordinate (3e1) at ($(2e3)+\scalep*\step*2*(0,-1)$);
    \coordinate (3e2) at ($(3e1)+\scalep*\step*(0,-1)$);
    \coordinate (3e3) at ($(3e2)+\scalep*\step*(0,-1)$);

    \foreach \point in {1e1, 1e2} \fill [black] (\point) circle (\rad);
    \foreach \point in {2e1, 2e2, 2e3} \fill [black] (\point) circle (\rad);
    \foreach \point in {3e1, 3e2, 3e3} \fill [black] (\point) circle (\rad);
    
    \coordinate[label=left:\footnotesize{vertex-bins}] (vbins) at ($(1v1)+\step*(0,.5)$);
    \coordinate[label=right:\footnotesize{edge-bins}] (ebins) at ($(1e1)+\step*(0,.5)$);
    % Vertex-bins ellipse
    \draw ($.5*(1v1)+.5*(1v2)$) ellipse ({\epsm} and {0.15+\epsm}); 
    \draw ($.5*(2v1)+.5*(2v2)$) ellipse ({\epsm} and {0.15+\epsm}); 
    
    % Edge-bins ellipse
    \draw ($.5*(1e1)+.5*(1e2)$) ellipse ({\epsm} and {0.15+\epsm}); 
    \draw (2e2) ellipse ({\epsm} and {0.3+\epsm}); 
    \draw (3e2) ellipse ({\epsm} and {0.3+\epsm}); 
    
    % Matching
    \draw (1v1)--(1e1);
    \draw (1v2)--(2e1);
    \draw (2v1)--(1e2);
    \draw (2v2)--(3e1);
    \coordinate (p2e2) at ($(2e2)+\step*0.5*({cos(180)},{sin(180)})$);
    \draw[dashed] (2e2)--(p2e2);
    \coordinate (p2e3) at ($(2e3)+\step*0.5*({cos(180)},{sin(180)})$);
    \draw[dashed] (2e3)--(p2e3);
    \coordinate (p3e2) at ($(3e2)+\step*0.5*({cos(180)},{sin(180)})$);
    \draw[dashed] (3e2)--(p3e2);
    \coordinate (p3e3) at ($(3e3)+\step*0.5*({cos(180)},{sin(180)})$);
    \draw[dashed] (3e3)--(p3e3);

    \coordinate  (startarrow) at ($(2e1)+(\step,0)$);
    \coordinate  (endarrow) at ($(startarrow)+(\step,0)$);
    \draw [->] (startarrow)--(endarrow);

    % Bins
    % Vertex-bins
    \coordinate (1v1) at ($(\initialg+4*\step,\initialg-\step)$);
    \coordinate (1v2) at ($(1v1)+\scalep*\step*(0,-1)$);
        
    \foreach \point in {1v1, 1v2} \fill [black] (\point) circle (\rad);
    \foreach \point in {2v1, 2v2} \fill [black] (\point) circle (\rad);
    
    % Edge-bins
    \coordinate (1e1) at ($(1v1)+(0,\step)+1.3*\step*(1,0)$);
    \coordinate (1e2) at ($(1e1)+\scalep*\step*(0,-1)$);

    \coordinate (2e1) at ($(1e2)+\scalep*\step*2*(0,-1)$);
    \coordinate (2e2) at ($(2e1)+\scalep*\step*(0,-1)$);
    \coordinate (2e3) at ($(2e2)+\scalep*\step*(0,-1)$);

    \coordinate (3e1) at ($(2e3)+\scalep*\step*2*(0,-1)$);
    \coordinate (3e2) at ($(3e1)+\scalep*\step*(0,-1)$);
    \coordinate (3e3) at ($(3e2)+\scalep*\step*(0,-1)$);

    %\foreach \point in {1e1, 1e2} \fill [black] (\point) circle (\rad);
    \foreach \point in {2e1, 2e2, 2e3} \fill [black] (\point) circle (\rad);
    \foreach \point in {3e1, 3e2, 3e3} \fill [black] (\point) circle (\rad);
    
    % Vertex-bins ellipse
    \draw[thick] ($.5*(1v1)+.5*(1v2)$) ellipse ({\epsm} and {0.15+\epsm}); 
        
    % Edge-bins ellipse
    %\draw ($.5*(1e1)+.5*(1e2)$) ellipse ({\epsm} and {0.15+\epsm}); 
    \draw (2e2) ellipse ({\epsm} and {0.3+\epsm}); 
    \draw (3e2) ellipse ({\epsm} and {0.3+\epsm}); 
    
    % Matching
    \draw[thick] (1v1)--(2e1);
    \draw[thick] (1v2)--(3e1);
    \coordinate (p2e2) at ($(2e2)+\step*0.5*({cos(180)},{sin(180)})$);
    \draw[dashed] (2e2)--(p2e2);
    \coordinate (p2e3) at ($(2e3)+\step*0.5*({cos(180)},{sin(180)})$);
    \draw[dashed] (2e3)--(p2e3);
    \coordinate (p3e2) at ($(3e2)+\step*0.5*({cos(180)},{sin(180)})$);
    \draw[dashed] (3e2)--(p3e2);
    \coordinate (p3e3) at ($(3e3)+\step*0.5*({cos(180)},{sin(180)})$);
    \draw[dashed] (3e3)--(p3e3);
  \end{tikzpicture}
  \caption{Transforming an edge-bin of size~$2$ matched to two
    vertex-bins of size~$2$ into a vertex-bin of size~$2$}
  \label{fig:left-bins-hyper}
\end{figure}
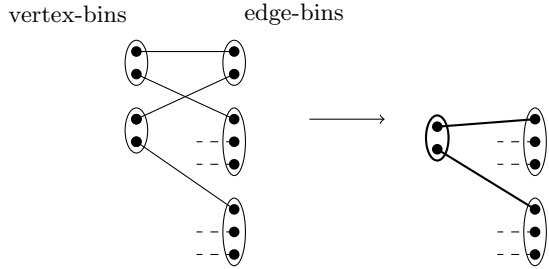

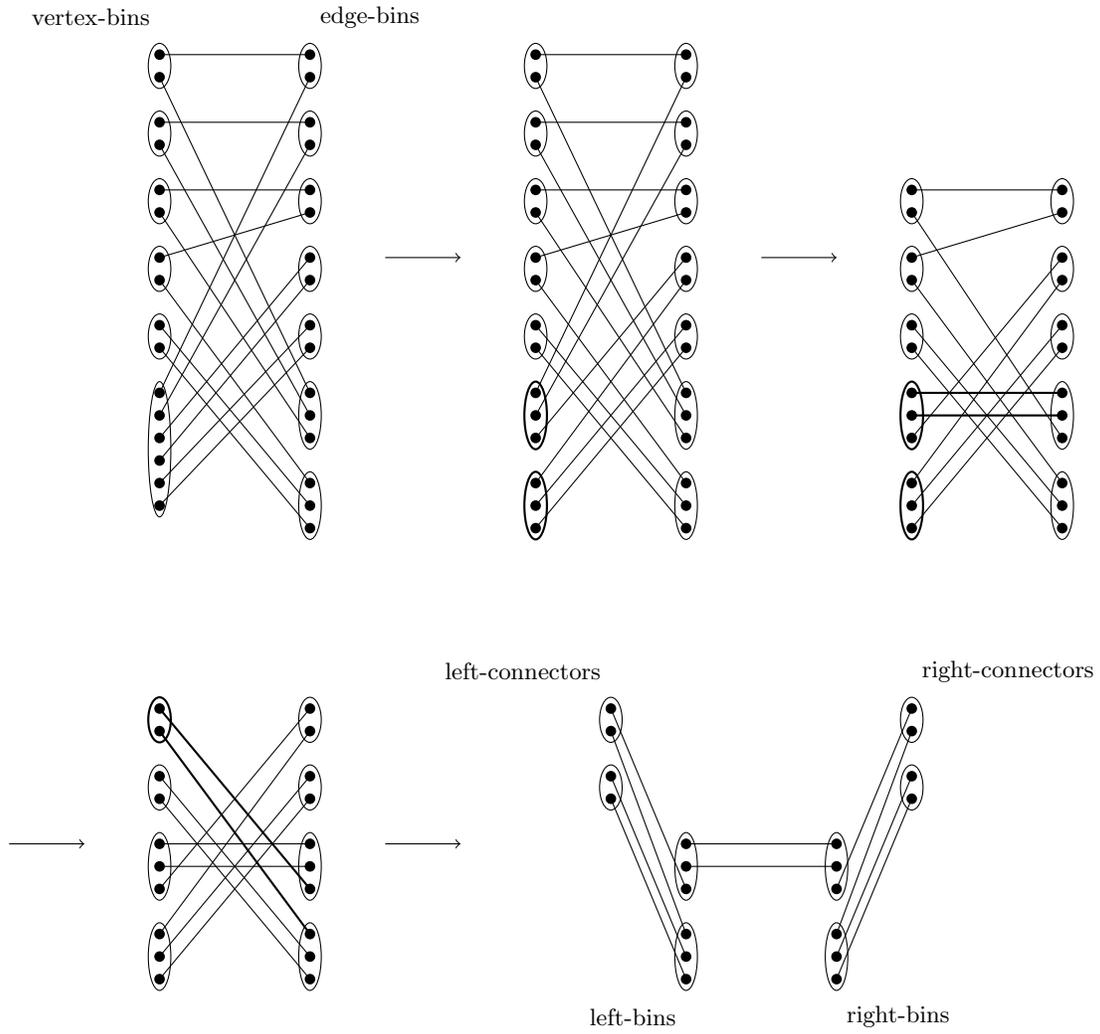
\begin{figure}
  \centering
  \begin{tikzpicture}
    \def \initialg {0};
    \def \step {1};
    \def \rad {2pt};
    \def \eps {.3};
    \def \epsm {.15};
    \def \epsg {.2};
    \def \round {5pt};
    \def \scalep {0.3}

    % Bins
    % Vertex-bins
    \coordinate (1v1) at ($(\initialg,\initialg)$);
    \coordinate (1v2) at ($(1v1)+\scalep*\step*(0,-1)$);
    
    \coordinate (2v1) at ($(1v2)+\scalep*\step*2*(0,-1)$);
    \coordinate (2v2) at ($(2v1)+\scalep*\step*(0,-1)$);

    \coordinate (3v1) at ($(2v2)+\scalep*\step*2*(0,-1)$);
    \coordinate (3v2) at ($(3v1)+\scalep*\step*(0,-1)$);

    \coordinate (4v1) at ($(3v2)+\scalep*\step*2*(0,-1)$);
    \coordinate (4v2) at ($(4v1)+\scalep*\step*(0,-1)$);

    \coordinate (5v1) at ($(4v2)+\scalep*\step*2*(0,-1)$);
    \coordinate (5v2) at ($(5v1)+\scalep*\step*(0,-1)$);

    \coordinate (6v1) at ($(5v2)+\scalep*\step*2*(0,-1)$);
    \coordinate (6v2) at ($(6v1)+\scalep*\step*(0,-1)$);
    \coordinate (6v3) at ($(6v2)+\scalep*\step*(0,-1)$);
    \coordinate (6v4) at ($(6v3)+\scalep*\step*(0,-1)$);
    \coordinate (6v5) at ($(6v4)+\scalep*\step*(0,-1)$);
    \coordinate (6v6) at ($(6v5)+\scalep*\step*(0,-1)$);

    \foreach \point in {1v1, 1v2} \fill [black] (\point) circle (\rad);
    \foreach \point in {2v1, 2v2} \fill [black] (\point) circle (\rad);
    \foreach \point in {3v1, 3v2} \fill [black] (\point) circle (\rad);
    \foreach \point in {4v1, 4v2} \fill [black] (\point) circle (\rad);
    \foreach \point in {5v1, 5v2} \fill [black] (\point) circle (\rad);
    \foreach \point in {6v1, 6v2, 6v3, 6v4, 6v5, 6v6} \fill [black] (\point) circle (\rad);
    
    % Edge-bins
    \coordinate (1e1) at ($(1v1)+(2*\step,0)$);
    \coordinate (1e2) at ($(1e1)+\scalep*\step*(0,-1)$);
    
    \coordinate (2e1) at ($(1e2)+\scalep*\step*2*(0,-1)$);
    \coordinate (2e2) at ($(2e1)+\scalep*\step*(0,-1)$);

    \coordinate (3e1) at ($(2e2)+\scalep*\step*2*(0,-1)$);
    \coordinate (3e2) at ($(3e1)+\scalep*\step*(0,-1)$);

    \coordinate (4e1) at ($(3e2)+\scalep*\step*2*(0,-1)$);
    \coordinate (4e2) at ($(4e1)+\scalep*\step*(0,-1)$);

    \coordinate (5e1) at ($(4e2)+\scalep*\step*2*(0,-1)$);
    \coordinate (5e2) at ($(5e1)+\scalep*\step*(0,-1)$);

    \coordinate (6e1) at ($(5e2)+\scalep*\step*2*(0,-1)$);
    \coordinate (6e2) at ($(6e1)+\scalep*\step*(0,-1)$);
    \coordinate (6e3) at ($(6e2)+\scalep*\step*(0,-1)$);

    \coordinate (7e1) at ($(6e3)+\scalep*\step*2*(0,-1)$);
    \coordinate (7e2) at ($(7e1)+\scalep*\step*(0,-1)$);
    \coordinate (7e3) at ($(7e2)+\scalep*\step*(0,-1)$);

    \foreach \point in {1e1, 1e2} \fill [black] (\point) circle (\rad);
    \foreach \point in {2e1, 2e2} \fill [black] (\point) circle (\rad);
    \foreach \point in {3e1, 3e2} \fill [black] (\point) circle (\rad);
    \foreach \point in {4e1, 4e2} \fill [black] (\point) circle (\rad);
    \foreach \point in {5e1, 5e2} \fill [black] (\point) circle (\rad);
    \foreach \point in {6e1, 6e2, 6e3} \fill [black] (\point) circle (\rad);
    \foreach \point in {7e1, 7e2, 7e3} \fill [black] (\point) circle (\rad);
    
    \coordinate[label=left:\footnotesize{vertex-bins}] (vbins) at ($(1v1)+\step*(0,.5)$);
    \coordinate[label=right:\footnotesize{edge-bins}] (ebins) at ($(1e1)+\step*(0,.5)$);
    % Vertex-bins ellipse
    \draw ($.5*(1v1)+.5*(1v2)$) ellipse ({\epsm} and {0.15+\epsm}); 
    \draw ($.5*(2v1)+.5*(2v2)$) ellipse ({\epsm} and {0.15+\epsm}); 
    \draw ($.5*(3v1)+.5*(3v2)$) ellipse ({\epsm} and {0.15+\epsm}); 
    \draw ($.5*(4v1)+.5*(4v2)$) ellipse ({\epsm} and {0.15+\epsm}); 
    \draw ($.5*(5v1)+.5*(5v2)$) ellipse ({\epsm} and {0.15+\epsm}); 
    \draw ($.5*(6v3)+.5*(6v4)$) ellipse ({\epsm} and {0.75+\epsm}); 
    
    % Edge-bins ellipse
    \draw ($.5*(1e1)+.5*(1e2)$) ellipse ({\epsm} and {0.15+\epsm}); 
    \draw ($.5*(2e1)+.5*(2e2)$) ellipse ({\epsm} and {0.15+\epsm}); 
    \draw ($.5*(3e1)+.5*(3e2)$) ellipse ({\epsm} and {0.15+\epsm}); 
    \draw ($.5*(4e1)+.5*(4e2)$) ellipse ({\epsm} and {0.15+\epsm}); 
    \draw ($.5*(5e1)+.5*(5e2)$) ellipse ({\epsm} and {0.15+\epsm}); 
    \draw (6e2) ellipse ({\epsm} and {0.3+\epsm}); 
    \draw (7e2) ellipse ({\epsm} and {0.3+\epsm}); 
    
    % Matching
    \draw (1v1)--(1e1);
    \draw (1v2)--(6e1);
    \draw (2v1)--(2e1);
    \draw (2v2)--(6e2);
    \draw (3v1)--(3e1);
    \draw (3v2)--(6e3);
    \draw (4v1)--(3e2);
    \draw (4v2)--(7e1);
    \draw (5v1)--(7e2);
    \draw (5v2)--(7e3);
    \draw (6v1)--(1e2);
    \draw (6v2)--(2e2);
    \draw (6v3)--(4e1);
    \draw (6v4)--(4e2);
    \draw (6v5)--(5e1);
    \draw (6v6)--(5e2);

    \coordinate  (startarrow) at ($(4e1)+(\step,0)$);
    \coordinate  (endarrow) at ($(startarrow)+(\step,0)$);
    \draw [->] (startarrow)--(endarrow);

    % Bins
    % Vertex-bins
    \coordinate (1v1) at ($(\initialg+5*\step,\initialg)$);
    \coordinate (1v2) at ($(1v1)+\scalep*\step*(0,-1)$);
    
    \coordinate (2v1) at ($(1v2)+\scalep*\step*2*(0,-1)$);
    \coordinate (2v2) at ($(2v1)+\scalep*\step*(0,-1)$);

    \coordinate (3v1) at ($(2v2)+\scalep*\step*2*(0,-1)$);
    \coordinate (3v2) at ($(3v1)+\scalep*\step*(0,-1)$);

    \coordinate (4v1) at ($(3v2)+\scalep*\step*2*(0,-1)$);
    \coordinate (4v2) at ($(4v1)+\scalep*\step*(0,-1)$);

    \coordinate (5v1) at ($(4v2)+\scalep*\step*2*(0,-1)$);
    \coordinate (5v2) at ($(5v1)+\scalep*\step*(0,-1)$);

    \coordinate (6v1) at ($(5v2)+\scalep*\step*2*(0,-1)$);
    \coordinate (6v2) at ($(6v1)+\scalep*\step*(0,-1)$);
    \coordinate (6v3) at ($(6v2)+\scalep*\step*(0,-1)$);

    \coordinate (6v4) at ($(6v3)+\scalep*\step*2*(0,-1)$);
    \coordinate (6v5) at ($(6v4)+\scalep*\step*(0,-1)$);
    \coordinate (6v6) at ($(6v5)+\scalep*\step*(0,-1)$);

    \foreach \point in {1v1, 1v2} \fill [black] (\point) circle (\rad);
    \foreach \point in {2v1, 2v2} \fill [black] (\point) circle (\rad);
    \foreach \point in {3v1, 3v2} \fill [black] (\point) circle (\rad);
    \foreach \point in {4v1, 4v2} \fill [black] (\point) circle (\rad);
    \foreach \point in {5v1, 5v2} \fill [black] (\point) circle (\rad);
    \foreach \point in {6v1, 6v2, 6v3, 6v4, 6v5, 6v6} \fill [black] (\point) circle (\rad);
    
    % Edge-bins
    \coordinate (1e1) at ($(1v1)+(2*\step,0)$);
    \coordinate (1e2) at ($(1e1)+\scalep*\step*(0,-1)$);
    
    \coordinate (2e1) at ($(1e2)+\scalep*\step*2*(0,-1)$);
    \coordinate (2e2) at ($(2e1)+\scalep*\step*(0,-1)$);

    \coordinate (3e1) at ($(2e2)+\scalep*\step*2*(0,-1)$);
    \coordinate (3e2) at ($(3e1)+\scalep*\step*(0,-1)$);

    \coordinate (4e1) at ($(3e2)+\scalep*\step*2*(0,-1)$);
    \coordinate (4e2) at ($(4e1)+\scalep*\step*(0,-1)$);

    \coordinate (5e1) at ($(4e2)+\scalep*\step*2*(0,-1)$);
    \coordinate (5e2) at ($(5e1)+\scalep*\step*(0,-1)$);

    \coordinate (6e1) at ($(5e2)+\scalep*\step*2*(0,-1)$);
    \coordinate (6e2) at ($(6e1)+\scalep*\step*(0,-1)$);
    \coordinate (6e3) at ($(6e2)+\scalep*\step*(0,-1)$);

    \coordinate (7e1) at ($(6e3)+\scalep*\step*2*(0,-1)$);
    \coordinate (7e2) at ($(7e1)+\scalep*\step*(0,-1)$);
    \coordinate (7e3) at ($(7e2)+\scalep*\step*(0,-1)$);

    \foreach \point in {1e1, 1e2} \fill [black] (\point) circle (\rad);
    \foreach \point in {2e1, 2e2} \fill [black] (\point) circle (\rad);
    \foreach \point in {3e1, 3e2} \fill [black] (\point) circle (\rad);
    \foreach \point in {4e1, 4e2} \fill [black] (\point) circle (\rad);
    \foreach \point in {5e1, 5e2} \fill [black] (\point) circle (\rad);
    \foreach \point in {6e1, 6e2, 6e3} \fill [black] (\point) circle (\rad);
    \foreach \point in {7e1, 7e2, 7e3} \fill [black] (\point) circle (\rad);
    
    % Vertex-bins ellipse
    \draw ($.5*(1v1)+.5*(1v2)$) ellipse ({\epsm} and {0.15+\epsm}); 
    \draw ($.5*(2v1)+.5*(2v2)$) ellipse ({\epsm} and {0.15+\epsm}); 
    \draw ($.5*(3v1)+.5*(3v2)$) ellipse ({\epsm} and {0.15+\epsm}); 
    \draw ($.5*(4v1)+.5*(4v2)$) ellipse ({\epsm} and {0.15+\epsm}); 
    \draw ($.5*(5v1)+.5*(5v2)$) ellipse ({\epsm} and {0.15+\epsm}); 
    \draw[thick] (6v2) ellipse ({\epsm} and {0.3+\epsm}); 
    \draw[thick] (6v5) ellipse ({\epsm} and {0.3+\epsm}); 
    
    % Edge-bins ellipse
    \draw ($.5*(1e1)+.5*(1e2)$) ellipse ({\epsm} and {0.15+\epsm}); 
    \draw ($.5*(2e1)+.5*(2e2)$) ellipse ({\epsm} and {0.15+\epsm}); 
    \draw ($.5*(3e1)+.5*(3e2)$) ellipse ({\epsm} and {0.15+\epsm}); 
    \draw ($.5*(4e1)+.5*(4e2)$) ellipse ({\epsm} and {0.15+\epsm}); 
    \draw ($.5*(5e1)+.5*(5e2)$) ellipse ({\epsm} and {0.15+\epsm}); 
    \draw (6e2) ellipse ({\epsm} and {0.3+\epsm}); 
    \draw (7e2) ellipse ({\epsm} and {0.3+\epsm}); 
    
    % Matching
    \draw (1v1)--(1e1);
    \draw (1v2)--(6e1);
    \draw (2v1)--(2e1);
    \draw (2v2)--(6e2);
    \draw (3v1)--(3e1);
    \draw (3v2)--(6e3);
    \draw (4v1)--(3e2);
    \draw (4v2)--(7e1);
    \draw (5v1)--(7e2);
    \draw (5v2)--(7e3);
    \draw (6v1)--(1e2);
    \draw (6v2)--(2e2);
    \draw (6v3)--(4e1);
    \draw (6v4)--(4e2);
    \draw (6v5)--(5e1);
    \draw (6v6)--(5e2);

    \coordinate  (startarrow) at ($(4e1)+(\step,0)$);
    \coordinate  (endarrow) at ($(startarrow)+(\step,0)$);
    \draw [->] (startarrow)--(endarrow);

    % Bins
    % Vertex-bins
    \coordinate (1v1) at ($(1v1)+(5*\step,0)$);
    \coordinate (1v2) at ($(1v1)+\scalep*\step*(0,-1)$);
    
    \coordinate (2v1) at ($(1v2)+\scalep*\step*2*(0,-1)$);
    \coordinate (2v2) at ($(2v1)+\scalep*\step*(0,-1)$);

    \coordinate (3v1) at ($(2v2)+\scalep*\step*2*(0,-1)$);
    \coordinate (3v2) at ($(3v1)+\scalep*\step*(0,-1)$);

    \coordinate (4v1) at ($(3v2)+\scalep*\step*2*(0,-1)$);
    \coordinate (4v2) at ($(4v1)+\scalep*\step*(0,-1)$);

    \coordinate (5v1) at ($(4v2)+\scalep*\step*2*(0,-1)$);
    \coordinate (5v2) at ($(5v1)+\scalep*\step*(0,-1)$);

    \coordinate (6v1) at ($(5v2)+\scalep*\step*2*(0,-1)$);
    \coordinate (6v2) at ($(6v1)+\scalep*\step*(0,-1)$);
    \coordinate (6v3) at ($(6v2)+\scalep*\step*(0,-1)$);

    \coordinate (6v4) at ($(6v3)+\scalep*\step*2*(0,-1)$);
    \coordinate (6v5) at ($(6v4)+\scalep*\step*(0,-1)$);
    \coordinate (6v6) at ($(6v5)+\scalep*\step*(0,-1)$);

    %\foreach \point in {1v1, 1v2} \fill [black] (\point) circle (\rad);
    %\foreach \point in {2v1, 2v2} \fill [black] (\point) circle (\rad);
    \foreach \point in {3v1, 3v2} \fill [black] (\point) circle (\rad);
    \foreach \point in {4v1, 4v2} \fill [black] (\point) circle (\rad);
    \foreach \point in {5v1, 5v2} \fill [black] (\point) circle (\rad);
    \foreach \point in {6v1, 6v2, 6v3, 6v4, 6v5, 6v6} \fill [black] (\point) circle (\rad);
    
    % Edge-bins
    \coordinate (1e1) at ($(1v1)+(2*\step,0)$);
    \coordinate (1e2) at ($(1e1)+\scalep*\step*(0,-1)$);
    
    \coordinate (2e1) at ($(1e2)+\scalep*\step*2*(0,-1)$);
    \coordinate (2e2) at ($(2e1)+\scalep*\step*(0,-1)$);

    \coordinate (3e1) at ($(2e2)+\scalep*\step*2*(0,-1)$);
    \coordinate (3e2) at ($(3e1)+\scalep*\step*(0,-1)$);

    \coordinate (4e1) at ($(3e2)+\scalep*\step*2*(0,-1)$);
    \coordinate (4e2) at ($(4e1)+\scalep*\step*(0,-1)$);

    \coordinate (5e1) at ($(4e2)+\scalep*\step*2*(0,-1)$);
    \coordinate (5e2) at ($(5e1)+\scalep*\step*(0,-1)$);

    \coordinate (6e1) at ($(5e2)+\scalep*\step*2*(0,-1)$);
    \coordinate (6e2) at ($(6e1)+\scalep*\step*(0,-1)$);
    \coordinate (6e3) at ($(6e2)+\scalep*\step*(0,-1)$);

    \coordinate (7e1) at ($(6e3)+\scalep*\step*2*(0,-1)$);
    \coordinate (7e2) at ($(7e1)+\scalep*\step*(0,-1)$);
    \coordinate (7e3) at ($(7e2)+\scalep*\step*(0,-1)$);

    %\foreach \point in {1e1, 1e2} \fill [black] (\point) circle (\rad);
    %\foreach \point in {2e1, 2e2} \fill [black] (\point) circle (\rad);
    \foreach \point in {3e1, 3e2} \fill [black] (\point) circle (\rad);
    \foreach \point in {4e1, 4e2} \fill [black] (\point) circle (\rad);
    \foreach \point in {5e1, 5e2} \fill [black] (\point) circle (\rad);
    \foreach \point in {6e1, 6e2, 6e3} \fill [black] (\point) circle (\rad);
    \foreach \point in {7e1, 7e2, 7e3} \fill [black] (\point) circle (\rad);
    
    % Vertex-bins ellipse
    %\draw ($.5*(1v1)+.5*(1v2)$) ellipse ({\epsm} and {0.15+\epsm}); 
    %\draw ($.5*(2v1)+.5*(2v2)$) ellipse ({\epsm} and {0.15+\epsm}); 
    \draw ($.5*(3v1)+.5*(3v2)$) ellipse ({\epsm} and {0.15+\epsm}); 
    \draw ($.5*(4v1)+.5*(4v2)$) ellipse ({\epsm} and {0.15+\epsm}); 
    \draw ($.5*(5v1)+.5*(5v2)$) ellipse ({\epsm} and {0.15+\epsm}); 
    \draw[thick] (6v2) ellipse ({\epsm} and {0.3+\epsm}); 
    \draw[thick] (6v5) ellipse ({\epsm} and {0.3+\epsm}); 
    
    % Edge-bins ellipse
    %\draw ($.5*(1e1)+.5*(1e2)$) ellipse ({\epsm} and {0.15+\epsm}); 
    %\draw ($.5*(2e1)+.5*(2e2)$) ellipse ({\epsm} and {0.15+\epsm}); 
    \draw ($.5*(3e1)+.5*(3e2)$) ellipse ({\epsm} and {0.15+\epsm}); 
    \draw ($.5*(4e1)+.5*(4e2)$) ellipse ({\epsm} and {0.15+\epsm}); 
    \draw ($.5*(5e1)+.5*(5e2)$) ellipse ({\epsm} and {0.15+\epsm}); 
    \draw (6e2) ellipse ({\epsm} and {0.3+\epsm}); 
    \draw (7e2) ellipse ({\epsm} and {0.3+\epsm}); 
    
    % Matching
    \draw[thick] (6e1)--(6v1);
    \draw[thick] (6v2)--(6e2);
    \draw (3v1)--(3e1);
    \draw (3v2)--(6e3);
    \draw (4v1)--(3e2);
    \draw (4v2)--(7e1);
    \draw (5v1)--(7e2);
    \draw (5v2)--(7e3);
    \draw (6v3)--(4e1);
    \draw (6v4)--(4e2);
    \draw (6v5)--(5e1);
    \draw (6v6)--(5e2);

    % Bins
    % Vertex-bins
    \coordinate (1v1) at ($(1v1)+(-10*\step,-6*\step)$);
    \coordinate (1v2) at ($(1v1)+\scalep*\step*(0,-1)$);
    
    \coordinate (2v1) at ($(1v2)+\scalep*\step*2*(0,-1)$);
    \coordinate (2v2) at ($(2v1)+\scalep*\step*(0,-1)$);

    \coordinate (3v1) at ($(2v2)+\scalep*\step*2*(0,-1)$);
    \coordinate (3v2) at ($(3v1)+\scalep*\step*(0,-1)$);

    \coordinate (4v1) at ($(3v2)+\scalep*\step*2*(0,-1)$);
    \coordinate (4v2) at ($(4v1)+\scalep*\step*(0,-1)$);

    \coordinate (newv1) at ($(3v2)+\scalep*\step*2*(0,-1)$);
    \coordinate (newv2) at ($(newv1)+\scalep*\step*(0,-1)$);

    \coordinate (5v1) at ($(4v2)+\scalep*\step*2*(0,-1)$);
    \coordinate (5v2) at ($(5v1)+\scalep*\step*(0,-1)$);

    \coordinate (6v1) at ($(5v2)+\scalep*\step*2*(0,-1)$);
    \coordinate (6v2) at ($(6v1)+\scalep*\step*(0,-1)$);
    \coordinate (6v3) at ($(6v2)+\scalep*\step*(0,-1)$);

    \coordinate (6v4) at ($(6v3)+\scalep*\step*2*(0,-1)$);
    \coordinate (6v5) at ($(6v4)+\scalep*\step*(0,-1)$);
    \coordinate (6v6) at ($(6v5)+\scalep*\step*(0,-1)$);

    \coordinate  (startarrow) at ($(6v1)+(-2*\step,0)$);
    \coordinate  (endarrow) at ($(startarrow)+(\step,0)$);
    \draw [->] (startarrow)--(endarrow);

    %\foreach \point in {1v1, 1v2} \fill [black] (\point) circle (\rad);
    %\foreach \point in {2v1, 2v2} \fill [black] (\point) circle (\rad);
    %\foreach \point in {3v1, 3v2} \fill [black] (\point) circle (\rad);
    %\foreach \point in {4v1, 4v2} \fill [black] (\point) circle
    %(\rad);
    \foreach \point in {newv1, newv2} \fill [black] (\point) circle (\rad);
    \foreach \point in {5v1, 5v2} \fill [black] (\point) circle (\rad);
    \foreach \point in {6v1, 6v2, 6v3, 6v4, 6v5, 6v6} \fill [black] (\point) circle (\rad);
    
    % Edge-bins
    \coordinate (1e1) at ($(1v1)+(2*\step,0)$);
    \coordinate (1e2) at ($(1e1)+\scalep*\step*(0,-1)$);
    
    \coordinate (2e1) at ($(1e2)+\scalep*\step*2*(0,-1)$);
    \coordinate (2e2) at ($(2e1)+\scalep*\step*(0,-1)$);

    \coordinate (3e1) at ($(2e2)+\scalep*\step*2*(0,-1)$);
    \coordinate (3e2) at ($(3e1)+\scalep*\step*(0,-1)$);

    \coordinate (4e1) at ($(3e2)+\scalep*\step*2*(0,-1)$);
    \coordinate (4e2) at ($(4e1)+\scalep*\step*(0,-1)$);

    \coordinate (5e1) at ($(4e2)+\scalep*\step*2*(0,-1)$);
    \coordinate (5e2) at ($(5e1)+\scalep*\step*(0,-1)$);

    \coordinate (6e1) at ($(5e2)+\scalep*\step*2*(0,-1)$);
    \coordinate (6e2) at ($(6e1)+\scalep*\step*(0,-1)$);
    \coordinate (6e3) at ($(6e2)+\scalep*\step*(0,-1)$);

    \coordinate (7e1) at ($(6e3)+\scalep*\step*2*(0,-1)$);
    \coordinate (7e2) at ($(7e1)+\scalep*\step*(0,-1)$);
    \coordinate (7e3) at ($(7e2)+\scalep*\step*(0,-1)$);

    %\foreach \point in {1e1, 1e2} \fill [black] (\point) circle (\rad);
    %\foreach \point in {2e1, 2e2} \fill [black] (\point) circle (\rad);
    %\foreach \point in {3e1, 3e2} \fill [black] (\point) circle (\rad);
    \foreach \point in {4e1, 4e2} \fill [black] (\point) circle (\rad);
    \foreach \point in {5e1, 5e2} \fill [black] (\point) circle (\rad);
    \foreach \point in {6e1, 6e2, 6e3} \fill [black] (\point) circle (\rad);
    \foreach \point in {7e1, 7e2, 7e3} \fill [black] (\point) circle (\rad);
    
    % Vertex-bins ellipse
    %\draw ($.5*(1v1)+.5*(1v2)$) ellipse ({\epsm} and {0.15+\epsm}); 
    %\draw ($.5*(2v1)+.5*(2v2)$) ellipse ({\epsm} and {0.15+\epsm}); 
    %\draw ($.5*(3v1)+.5*(3v2)$) ellipse ({\epsm} and {0.15+\epsm}); 
    \draw[thick] ($.5*(newv1)+.5*(newv2)$) ellipse ({\epsm} and {0.15+\epsm}); 
    \draw ($.5*(5v1)+.5*(5v2)$) ellipse ({\epsm} and {0.15+\epsm}); 
    \draw (6v2) ellipse ({\epsm} and {0.3+\epsm}); 
    \draw (6v5) ellipse ({\epsm} and {0.3+\epsm}); 
    
    % Edge-bins ellipse
    %\draw ($.5*(1e1)+.5*(1e2)$) ellipse ({\epsm} and {0.15+\epsm}); 
    %\draw ($.5*(2e1)+.5*(2e2)$) ellipse ({\epsm} and {0.15+\epsm}); 
    %\draw ($.5*(3e1)+.5*(3e2)$) ellipse ({\epsm} and {0.15+\epsm}); 
    \draw ($.5*(4e1)+.5*(4e2)$) ellipse ({\epsm} and {0.15+\epsm}); 
    \draw ($.5*(5e1)+.5*(5e2)$) ellipse ({\epsm} and {0.15+\epsm}); 
    \draw (6e2) ellipse ({\epsm} and {0.3+\epsm}); 
    \draw (7e2) ellipse ({\epsm} and {0.3+\epsm}); 
    
    % Matching
    \draw[thick] (newv1)--(6e3);
    \draw[thick] (newv2)--(7e1);
    \draw (5v1)--(7e2);
    \draw (5v2)--(7e3);
    \draw (6v3)--(4e1);
    \draw (6v4)--(4e2);
    \draw (6v5)--(5e1);
    \draw (6v6)--(5e2);
    \draw (6v1)--(6e1);
    \draw (6v2)--(6e2);

    % Vertex-bins
    \coordinate (1v1) at ($(1v1)+(7*\step,0)$);
    \coordinate (1v2) at ($(1v1)+\scalep*\step*(0,-1)$);
    
    \coordinate (2v1) at ($(1v2)+\scalep*\step*2*(0,-1)$);
    \coordinate (2v2) at ($(2v1)+\scalep*\step*(0,-1)$);

    \coordinate (3v1) at ($(2v2)+\scalep*\step*2*(0,-1)$);
    \coordinate (3v2) at ($(3v1)+\scalep*\step*(0,-1)$);

    \coordinate (4v1) at ($(3v2)+\scalep*\step*2*(0,-1)$);
    \coordinate (4v2) at ($(4v1)+\scalep*\step*(0,-1)$);

    \coordinate (4l1) at ($(4v1)+\step*(-1,0)$);
    \coordinate (4l2) at ($(4l1)+\scalep*\step*(0,-1)$);

    \coordinate (newv1) at ($(3v2)+\scalep*\step*2*(0,-1)$);
    \coordinate (newv2) at ($(newv1)+\scalep*\step*(0,-1)$);

    \coordinate (5v1) at ($(4v2)+\scalep*\step*2*(0,-1)$);
    \coordinate (5v2) at ($(5v1)+\scalep*\step*(0,-1)$);

    \coordinate (5l1) at ($(5v1)+\step*(-1,0)$);
    \coordinate (5l2) at ($(5l1)+\scalep*\step*(0,-1)$);

    \coordinate (6v1) at ($(5v2)+\scalep*\step*2*(0,-1)$);
    \coordinate (6v2) at ($(6v1)+\scalep*\step*(0,-1)$);
    \coordinate (6v3) at ($(6v2)+\scalep*\step*(0,-1)$);

    \coordinate (6v4) at ($(6v3)+\scalep*\step*2*(0,-1)$);
    \coordinate (6v5) at ($(6v4)+\scalep*\step*(0,-1)$);
    \coordinate (6v6) at ($(6v5)+\scalep*\step*(0,-1)$);

    %\foreach \point in {1v1, 1v2} \fill [black] (\point) circle (\rad);
    %\foreach \point in {2v1, 2v2} \fill [black] (\point) circle (\rad);
    %\foreach \point in {3v1, 3v2} \fill [black] (\point) circle (\rad);
    %\foreach \point in {4v1, 4v2} \fill [black] (\point) circle
    %(\rad);
    \foreach \point in {4l1,4l2} \fill [black] (\point) circle (\rad);
    \foreach \point in {5l1, 5l2} \fill [black] (\point) circle (\rad);
    \foreach \point in {6v1, 6v2, 6v3, 6v4, 6v5, 6v6} \fill [black] (\point) circle (\rad);
    
    % Edge-bins
    \coordinate (1e1) at ($(1v1)+(2*\step,0)$);
    \coordinate (1e2) at ($(1e1)+\scalep*\step*(0,-1)$);
    
    \coordinate (2e1) at ($(1e2)+\scalep*\step*2*(0,-1)$);
    \coordinate (2e2) at ($(2e1)+\scalep*\step*(0,-1)$);

    \coordinate (3e1) at ($(2e2)+\scalep*\step*2*(0,-1)$);
    \coordinate (3e2) at ($(3e1)+\scalep*\step*(0,-1)$);

    \coordinate (4e1) at ($(3e2)+\scalep*\step*2*(0,-1)$);
    \coordinate (4e2) at ($(4e1)+\scalep*\step*(0,-1)$);

    \coordinate (5e1) at ($(4e2)+\scalep*\step*2*(0,-1)$);
    \coordinate (5e2) at ($(5e1)+\scalep*\step*(0,-1)$);

    \coordinate (6e1) at ($(5e2)+\scalep*\step*2*(0,-1)$);
    \coordinate (6e2) at ($(6e1)+\scalep*\step*(0,-1)$);
    \coordinate (6e3) at ($(6e2)+\scalep*\step*(0,-1)$);

    \coordinate (7e1) at ($(6e3)+\scalep*\step*2*(0,-1)$);
    \coordinate (7e2) at ($(7e1)+\scalep*\step*(0,-1)$);
    \coordinate (7e3) at ($(7e2)+\scalep*\step*(0,-1)$);

    % Right connectors
    \coordinate (4c1) at ($(4e1)+\step*(1,0)$);
    \coordinate (4c2) at ($(4c1)+\scalep*\step*(0,-1)$);

    \coordinate (5c1) at ($(5e1)+\step*(1,0)$);
    \coordinate (5c2) at ($(5c1)+\scalep*\step*(0,-1)$);

    \coordinate  (startarrow) at ($(6v1)+(-4*\step,0)$);
    \coordinate  (endarrow) at ($(startarrow)+(\step,0)$);
    \draw [->] (startarrow)--(endarrow);

    %\foreach \point in {1e1, 1e2} \fill [black] (\point) circle (\rad);
    \foreach \point in {4c1, 4c2} \fill [black] (\point) circle (\rad);
    \foreach \point in {5c1, 5c2} \fill [black] (\point) circle (\rad);
    %\foreach \point in {3e1, 3e2} \fill [black] (\point) circle (\rad);
    %\foreach \point in {4e1, 4e2} \fill [black] (\point) circle (\rad);
    %\foreach \point in {5e1, 5e2} \fill [black] (\point) circle (\rad);
    \foreach \point in {6e1, 6e2, 6e3} \fill [black] (\point) circle (\rad);
    \foreach \point in {7e1, 7e2, 7e3} \fill [black] (\point) circle (\rad);
    
    % Vertex-bins ellipse
    %\draw ($.5*(1v1)+.5*(1v2)$) ellipse ({\epsm} and {0.15+\epsm}); 
    \draw ($.5*(5c1)+.5*(5c2)$) ellipse ({\epsm} and {0.15+\epsm}); 
    \draw ($.5*(4c1)+.5*(4c2)$) ellipse ({\epsm} and {0.15+\epsm}); 
    \draw ($.5*(5l1)+.5*(5l2)$) ellipse ({\epsm} and {0.15+\epsm}); 
    \draw ($.5*(4l1)+.5*(4l2)$) ellipse ({\epsm} and {0.15+\epsm}); 
    % \draw ($.5*(newv1)+.5*(newv2)$) ellipse ({\epsm} and {0.15+\epsm}); 
    %\draw ($.5*(5v1)+.5*(5v2)$) ellipse ({\epsm} and {0.15+\epsm}); 
    \draw (6v2) ellipse ({\epsm} and {0.3+\epsm}); 
    \draw (6v5) ellipse ({\epsm} and {0.3+\epsm}); 
    
    % Edge-bins ellipse
    %\draw ($.5*(1e1)+.5*(1e2)$) ellipse ({\epsm} and {0.15+\epsm}); 
    %\draw ($.5*(2e1)+.5*(2e2)$) ellipse ({\epsm} and {0.15+\epsm}); 
    %\draw ($.5*(3e1)+.5*(3e2)$) ellipse ({\epsm} and {0.15+\epsm}); 
   % \draw ($.5*(4e1)+.5*(4e2)$) ellipse ({\epsm} and {0.15+\epsm}); 
    %\draw ($.5*(5e1)+.5*(5e2)$) ellipse ({\epsm} and {0.15+\epsm}); 
    \draw (6e2) ellipse ({\epsm} and {0.3+\epsm}); 
    \draw (7e2) ellipse ({\epsm} and {0.3+\epsm}); 
    
    % Matching
    \draw (4c1)--(6e3);
    \draw (4c2)--(7e1);
    \draw (5c1)--(7e2);
    \draw (5c2)--(7e3);
    \draw (6v3)--(4l1);
    \draw (6v4)--(4l2);
    \draw (6v5)--(5l1);
    \draw (6v6)--(5l2);
    \draw (6v1)--(6e1);
    \draw (6v2)--(6e2);

    \coordinate[label=left:\footnotesize{left-connectors}] (lc) at
    ($(4l1)+(0,0.5*\step)$); 
    \coordinate[label=right:\footnotesize{right-connectors}] (rc) at
    ($(4c1)+(0,0.5*\step)$); 
    \coordinate[label=left:\footnotesize{left-bins}] (lb) at
    ($(6v6)+(0,-0.5*\step)$);
    \coordinate[label=right:\footnotesize{right-bins}] (rb) at
    ($(7e3)+(0,-0.5*\step)$);  
  \end{tikzpicture}
  \caption{Modifying a kernel-configuration}
  \label{fig:proc-left-right-hyper}
\end{figure}  

See Figure~\ref{fig:proc-left-right-hyper} for an example of the
procedure. If the kernel-configuration created in Step~1 is connected,
the original kernel-configuration was also connected, since splitting
vertex-bins cannot turn a disconnected kernel-configuration into a
connected one. We say that the structures in Step~2 and Step~3 are
connected if the $2$-uniform hypergraph obtained by contracting each
bin into a single vertex is connected. It is trivial that, if the
structure obtained is connected, then the original
kernel-configuration was connected. 

Recall that $M$ is chosen \uar\ from all possible matchings when
generating a random kernel as described in
Section~\ref{sec:random-prek-hyper}. This implies that, in the
structure obtained after Step~3, the resulting matching has uniform
distribution among the perfect matchings on the set of points in the
bins such that each point in an edge-bin of size $2$ is matched to a
point in a vertex-bin of size at least~$3$, each point in a vertex-bin
of size $2$ is matched to a point in an edge-bin of size~$3$, each
point in a vertex-bin of size at least~$3$ is matched to a point
in an edge-bin, and each point in an edge-bin of size $3$ is 
matched to a point in a vertex-bin.

%once we choose a set points $S_L$ in
%the vertex-bins of size at least~$3$ to be matched to the points in
%the edge-bins of size~$2$ and a set points $S_R$ in the edge-bins of
%size least~$3$ to be matched to the points in the vertex-bins of
%size~$2$ for the kernel-configuration, the matching from the points in
%the $S_L$ and the points in the edge-bins of size~$2$ has uniform
%distribution, the matching from the points in the $S_R$ and the points
%in the vertex-bins of size~$2$ has uniform distribution, the matching
%from the remaining points in the vertex-bins of size at least $3$ and
%the remaining points in the edge-bins of size~$3$ has uniform
%distribution. Also, in this last matching, the number of remaining
%points in the vertex-bins (and edge-bins) is $\Tthree+\mtwoprime(1)$,
%where $\mtwoprime(1)$ is the number of edge-bins  as described in
%Step~$2$.

Here we describe a new model to generate structures as the one
obtained by the process above. Let and $\ts \in\set{3,4,5}^{N}$ and let
$\ts'\in\set{3,4,5}^{N'}$. Let $L\leq \sum_{i}t_i/2$ and $L'\leq
\sum_{i}t_i'/2$ be such that $\sum_{i} t_i-2L = \sum_{i} t_i'-2L'
\eqdefinv K$. Let $B(\ts,\ts',L,L')$ be generated as follows. In each
step, every choice is made \uar:
\begin{enumerate}
\item(\textit{Left-bins}) For each $i\in[N]$, create one bin/set with $t_i$ points in
  it. We call these bins \textdef{left-bins}.
\item(\textit{Right-bins}) For each $i\in[N']$, create one bin/set with $t_i'$ points in
  it. We call these bins \textdef{right-bins}.
\item(\textit{Left-connectors}) Create $L$ bins with $2$ points inside
  each. We call these bins \textdef{left-connectors}.
\item(\textit{Right-connectors}) Create $L'$ bins with $2$ points
  inside each. We call these bins \textdef{right-connectors}.
\item(\textit{Matching}) Choose a perfect matching such that each
  point in a left-connector is matched to a point in a left-bin, each
  point in a right-connector is matched to a point in a right-bin,
  each point in a left-bin is either matched to a point in
  a left-connector or in a right-bin, and each point in a right-bin is
  either matched to a point in a right-connector or in a left-bin.  The
  edges in the matching from points in right-bins to points in
  left-bins are called \textdef{across-edges}.
\end{enumerate}
In the structure we obtained from the kernel-configuration,
vertex-bins of size at least~$3$ have the same role as the left-bins,
edge-bins of size~$3$ have the same role as the right-bins,
vertex-bins of size~$2$ have the same role as the right-connectors,
and edge-bins of size~$2$ have the same role as the left-connectors.
See Figure~\ref{fig:proc-left-right-hyper}.

We will prove that $B(\ts,\ts',L,L')$ with $K\to \infty$ is connected
\aas
\begin{lem}
  \label{lem:connected-new-hyper}
  Let and $\ts \in\set{3,4,5}^{N}$ and let
  $\ts'\in\set{3,4,5}^{N'}$. Let $L\leq \sum_{i}t_i/2$ and $L'\leq
  \sum_{i}t_i'/2$ be such that $\sum_{i} t_i-2L = \sum_{i} t_i-2L'
  \eqdefinv K$. If $K\to\infty$, then $B(\ts,\ts',L,L')$ is connected \aas
\end{lem}
Before presenting the proof for this lemma, we explain how to prove
Lemma~\ref{lem:connected-pre-hyper} assuming
Lemma~\ref{lem:connected-new-hyper} holds.  In the structure obtained
from the kernel-configuration, the number of points from vertex-bins
of size at least $3$ (which corresponds to left-bins) that are matched
to points in edge-bins of size $3$ (which corresponds to right-bins)
is $\Tthree+\mtwoprime(1)$, where $\mtwoprime(1)$ is the number of
edge-bins  as described in Step~$2$ of the procedure. In order to
use Lemma~\ref{lem:connected-new-hyper} to conclude that the
kernel-configuration is connected \aas\ (and thus proving
Lemma~\ref{lem:connected-pre-hyper}), it suffices to show that
$\mtwoprime(1)\to \infty$ \aas\ (which ensures that the condition 
$K\to\infty$
is satisfied).

Let $U$ be the set of points in vertex-bins of size~$2$ that will be
matched to points in edge-bins of size~$2$. (See Step~3 in the proof
of Lemma~\ref{lem:number-kernel-hyper}.)  There are
\begin{equation*}
  \binom{2\mtwoprime}{\kone}\kone!
\end{equation*}
ways of matching the points in $U$ to points in edge-bins of
size~$2$. For every edge-bin $i$ of size~$2$, let $X_i$ be the
indicator random for the event that $i$ has both of its points matched
to points in~$U$.  For $x\in S_{\psi}^*$, we have that $\mtwoprime\sim
\kone$ and so
\begin{align*}
  \prob{X_i = 1} =\frac{\displaystyle \binom{\kone}{2}2! 
    \binom{2\mtwoprime-2}{\kone-2}(\kone-2)!}
  {\displaystyle\binom{2\mtwoprime}{\kone}\kone!}\sim \frac{1}{4},\\
  \prob{X_i =1, X_j=1} =
\frac{\displaystyle \binom{\kone}{4}4! 
  \binom{2\mtwoprime-4}{\kone-4}(\kone-4)!}
{\displaystyle\binom{2\mtwoprime}{\kone}\kone!}\sim
\frac{1}{16},\text{ for }i\neq j,
\end{align*}
and so $\mean{\sum_i X_i}\sim \mtwoprime/4$ and $\var{\sum_i X_i} =
o(\mean{\sum_i X_i}^2)$. Thus, by Chebyshev's inequality, 
\begin{equation*}
  \prob[\Big]{{\textstyle\abs[\Big]{\sum_i X_i-\mean[\big]{\sum_i X_i}}
    }\geq {\textstyle t\mean[\big]{\sum_i X_i}}}=
  \frac{o(1)}{t^2}
\end{equation*}
and so we can choose $t$ going to $0$ sufficiently slowly so that
$\mtwoprime(2) = \sum_i X_i \sim \mtwoprime/4$ \aas{} Similarly,
$\mtwoprime(0) = \sum_i X_i \sim \mtwoprime/4$ \aas{} Thus,
\begin{equation*}
  \mtwoprime{(1)}\geq (1+o(1))\frac{\mtwoprime}{2}\to\infty
\end{equation*}
since $x\in S_{\psi}^*$.

We finish this section by presenting the proof for
Lemma~\ref{lem:connected-new-hyper}.
\begin{proof}[Proof of Lemma~\ref{lem:connected-new-hyper}]
  Let $Q = \sum_i t_i$ and let $Q' =\sum_{i} t_i'$. The number of
  choices for the matching in Step~5 is
\begin{equation*}
  \binom{Q}{2L}(2L)!
  \binom{Q'}{2L'}(2L')!
  K!
  = \frac{Q!Q'!}{K!}.
\end{equation*}
Let $A$ be a set of left-bins with $P$ points of which $S$ points are
matched to a set of left-connectors (covering all points in these
left-connectors). Similarly, let $A'$ be a set of right-bins with $P'$
points of which $S'$ points are matched to a set of right-connectors. Note
that $S$ and $S'$ must be even numbers. We compute the number of
configurations such that $A,A'$ form a connected component with
$r\eqdef P-S=P'-S'$ across-edges:
\begin{equation*}
  \begin{split}
    &\left( \binom{L}{S/2}\binom{P}{S} S! \binom{Q-P}{2L-S} (2L-S)! \right)\\
    \quad&\times
    r!(K-r)!\\
    \quad&\times
    \left( \binom{L'}{S'/2}\binom{P'}{S'} S'! \binom{Q'-P'}{2L'-S'} (2L'-S')! \right)
  \end{split}
\end{equation*}
Thus, the probability that $A,A'$ form a connected component (with
parameters $S,S'$) is exactly
\begin{equation*}
  \frac{{\displaystyle \binom{L}{S/2}\binom{L'}{S'/2}\binom{K}{r}}}
  {{\displaystyle\binom{Q}{P}\binom{Q'}{P'}}}.
\end{equation*}
So we want to bound the summation:
\begin{equation}
  \label{eq:sumprob_connected}
  \sum_{\substack{(P,S,n)\\(P',S',n')}}
  \sum_{(A,A')}
  \frac{{\displaystyle \binom{L}{S/2}\binom{L'}{S'/2}\binom{K}{r}}}
  {{\displaystyle\binom{Q}{P}\binom{Q'}{P'}}}
\end{equation}
where the second summation is over the pairs $(A,A')$ where $A$ is a
set of $n$ left-bins with $P$ points and $S$ points matched to
left-connectors and $A'$ is a set of $n'$ right-bins with $P'$ points
and $S'$ points matched to right-connectors; and $r = P-S =
P'-S'$.  Let $C$ be an integer constant to be determined later.

First consider the case where
\begin{align*}
  P\leq C &\text{ and }P'\leq C,\\
  &\text{or}\\
  Q-P\leq C &\text{ and }Q'-P'\leq C.
\end{align*}
We only need to check one of the options above because if $A\cup A'$
is disconnected from the rest of the graph the same is true for the
$\overline{A}\cup\overline{A}'$ where $\overline{A}$ is the complement
of $A$ in the set of left-bins and $\overline{A'}$ is the complement
of $A'$ in the set of right-bins. So let us assume $P\leq C$ and
$P'\leq C$. Then the number of choices for $(P,S,n)$ and $(P',S',n')$
is $O(1)$. Moreover, there are at most $\binom{N}{n}$ choices for $A$
and $\binom{N'}{n'}$ choices for $A'$, where $N$ is the number of
left-bins and $N'$ is the number of right-bins. Then the summation
in~\eqref{eq:sumprob_connected} for this case is at most
\begin{equation*}
  \begin{split}
    \frac{{\displaystyle \binom{L}{S/2}\binom{L'}{S'/2}
        \binom{K}{r}\binom{N}{n}\binom{N'}{n'}}}
    {{\displaystyle\binom{Q}{P}\binom{Q'}{P'}}}
    &=O\left(\frac{L^{S/2}(L')^{S'/2}K^rN^n (N')^{n'}}{Q^P(Q')^{P'}}
    \right)
    \\
    &=O\left(\frac{1}{Q^{P-S/2-r/2-n}(Q')^{P'-S'/2-r/2-n'}}
  \right)
  \\
  &=O\left(\frac{1}{Q^{P/6}(Q')^{P'/6}}
  \right)
  =
  o(1),
  \end{split}
\end{equation*}
since $P-S/2-r/2-n \geq P-S/2-(P-S)/2-P/3 = P/6$ (and similarly for
$P'-S'/2-r/2-n'$) and $P$ or $P'$ is at least $1$.

Now consider the case where
\begin{align*}
  P\leq C &\text{ and }Q'-P'\leq C,\\
  &\text{or}\\
  Q-P\leq C &\text{ and }P'\leq C.
\end{align*}
If $P \leq C$ and $Q'-P'\leq C$. Then $r=P-S\leq C$ and $r =
P'-S'\geq P'-2L' \geq Q-C-2L' =K-C$, which is impossible since
$K\to\infty$ and $C=O(1)$.

Finally consider the case
\begin{align*}
  P\geq C &\text{ and }P'\geq C,\\
  &\text{or}\\
  Q-P\geq C &\text{ and }Q'-P'\geq C.
\end{align*}
Using Stirling's approximation\crisc{removed reference to appendix},
there is a positive constant $\alpha$ such that
\begin{equation*}
  \frac{\displaystyle \binom{K}{r}}
  {\displaystyle
    \binom{\ceil{K/2}}{\ceil{r/2}}\binom{\floor{K/2}}{\floor{r/2}}}
  \leq
  \alpha\sqrt{K}.
\end{equation*}
Thus, for $P$ and $P'$ in this range, 
\begin{equation*}
  \begin{split}
    &\sum_{\substack{(P,S,n)\\(P',S',n')}}
    \sum_{(A,A')}
    \frac{{\displaystyle \binom{L}{S/2}\binom{L'}{S'/2}\binom{K}{r}}}
    {{\displaystyle\binom{Q}{P}\binom{Q'}{P'}}}
    \leq
    \alpha
    \sum_{\substack{(P,S,n)\\(P',S',n')}}
    \sum_{(A,A')}
  \frac{{\displaystyle \binom{L}{S/2}\binom{L'}{S'/2}
      \sqrt{K}\binom{\ceil{K/2}}{\ceil{r/2}}\binom{\floor{K/2}}{\floor{r/2}}}}
  {{\displaystyle\binom{Q}{P}\binom{Q'}{P'}}}\\
  \\
  &\leq
    \alpha
    \sum_{\substack{(P,S,n)\\(P',S',n')}}
  \frac{{\displaystyle \binom{N}{n}\binom{N'}{n'} \binom{L}{S/2}\binom{L'}{S'/2}
      \sqrt{K}\binom{\ceil{K/2}}{\ceil{r/2}}\binom{\floor{K/2}}{\floor{r/2}}}}
  {{\displaystyle\binom{Q}{P}\binom{Q'}{P'}}}\\
  &\leq
  \alpha 
  \sum_{\substack{(P,S,n)\\(P',S',n')}}
  \frac{\displaystyle\binom{N}{n}\binom{N'}{n'} \displaystyle  
    \sqrt{K}}
  {{\displaystyle\binom{Q-L-\ceil{K/2}}{P-S/2-\ceil{r/2}}
      \binom{Q'-L'-\floor{K/2}}{P'-S'/2-\floor{r/2}}}}\\
  &=
  \alpha 
  \sum_{\substack{(P,S,n)\\(P',S',n')}}
  \frac{\displaystyle  
    \binom{N}{n}\binom{N'}{n'} \sqrt{K}}
  {{\displaystyle\binom{Q/2-u(K)}{P/2-u(r)}
      \binom{Q'/2-d(K)}{P'/2-d(r)}}},
\end{split}
  \end{equation*}
  where $u(x) := \ceil{x/2}-x/2$ and $d(x) := x/2-\floor{x/2}$. Note
  that, for $P'\leq Q'/2$,
  \begin{equation*}
    \frac{\displaystyle \binom{N'}{n'}}
    {\displaystyle \binom{Q'/2-d(K)}{P'/2-d(r)}}
    \leq
    \frac{\displaystyle \binom{Q'/3}{P'/3}}
    {\displaystyle \binom{Q'/2-d(K)}{P'/2-d(r)}}
    \leq
    \frac{1}
    {\displaystyle \binom{Q'/6-d(K)}{P'/6-d(r)}}
    \leq 1,
  \end{equation*}
  and for $P'\geq Q'/2$
  \begin{equation*}
    \begin{split}
      \frac{\displaystyle \binom{N'}{n'}}
    {\displaystyle \binom{Q'/2-d(K)}{P'/2-d(r)}}
    &=
    \frac{\displaystyle \binom{N'}{N'-n'}}
    {\displaystyle \binom{Q'/2-d(K)}{Q'/2-P'/2-d(K)+d(r)}}
    \leq
    \frac{\displaystyle \binom{Q'/3}{Q'/3-P'/3}}
    {\displaystyle \binom{Q'/2-d(K)}{Q'/2-P'/2-d(K)+d(r)}}
    \\
    &\leq
    \frac{1}
    {\displaystyle \binom{Q'/6-d(K)}{Q'/6-P'/6-d(K)+d(r)}}
    \leq 1,
    \end{split}
  \end{equation*}
  Thus, for $P'\leq Q'$,
  \begin{equation}
    \label{eq:aux1-conn-hyper}
    \frac{\displaystyle \binom{N'}{n'}}
    {\displaystyle \binom{Q'/2-d(K)}{P'/2-d(r)}}
    \leq
    1.
  \end{equation}
  For $C\leq P\leq \beta \log Q$,
  \begin{equation*}
  \frac{\displaystyle  
    \binom{N}{n}}
  {{\displaystyle\binom{Q/2-u(K)}{P/2-d(r)}
      }}
  \leq
  \frac{\displaystyle  
    \binom{Q/3}{P/3}}
  {{\displaystyle\binom{Q/2-u(K)}{P/2-u(r)}
    }}
  \leq
    \binom{Q/6-u(K)}{P/6-u(r)}^{-1}
    =O\left( \frac{Q}{\beta\log Q}\right)^{-P/6+u(r)}
  \end{equation*}
  and so by choosing $C$ big enough and
  using~\eqref{eq:aux1-conn-hyper}
  \begin{equation*}
    \sum_{\substack{(P,S,n)\\(P',S',n')\\C\leq P\leq \beta \log Q}}
    \frac{\displaystyle  
      \binom{N}{n}\binom{N'}{n'} \sqrt{K}}
    {{\displaystyle\binom{Q/2-u(K)}{P/2-u(r)}
      \binom{Q'/2-d(K)}{P'/2-d(r)}}}
    \leq Q^{11/2}\log Q \cdot O\left(\frac{\beta \log Q}{Q}\right)^6
    = o(1).
    \end{equation*}  
  The range $Q-\beta\log Q\leq P\leq Q-C$ can be treated similarly.

There exists a constant $\gamma>0$ such that, for $\beta\log Q\leq P \leq Q/2$, 
    \begin{equation*}
  \frac{\displaystyle  
    \binom{N}{n}}
  {{\displaystyle\binom{Q/2-u(K)}{P/2-u(r)}
    }}
  \leq
  \frac{\displaystyle  
    \binom{Q/3}{P/3}}
  {{\displaystyle\binom{Q/2-u(K)}{P/2-u(r)}
    }}
  \leq
  \binom{Q/6-u(K)}{P/6-u(r)}^{-1}
      =O(\gamma^{P/6-u(r)}),
    \end{equation*}
    and so, by~\eqref{eq:aux1-conn-hyper},
    \begin{equation*}
      \sum_{\substack{(P,S,n)\\(P',S',n')\\\beta\log Q\leq P\leq Q/2}}
      \frac{\displaystyle  
        \binom{N}{n}\binom{N'}{n'} \sqrt{K}}
      {{\displaystyle\binom{Q/2-u(K)}{P/2-u(r)}
          \binom{Q'/2-d(K)}{P'/2-d(r)}}}
      \leq
      Q^{13/2} \cdot O\left(\gamma^{
          \beta\log N}\right)
      = o(1),
    \end{equation*}
    for sufficiently large constant $\beta$. The range $Q/2\leq P\leq
    Q-\beta\log Q$ can be treated similarly. The same argument works
    for $(P',S',n')$ and $Q'$. We are done because $P\leq Q-C$ or
    $P'\leq Q'-C$ (otherwise, it falls in a case that has already been
    treated).
\end{proof}

\subsection{Proof of Theorem~\ref{thm:formula-pre-hyper}}
  In this section we obtain an asymptotic formula
for the number of connected \prekernels\ with vertex set $[n]$ and
$m=n/2+R$ edges, when $R=\omega(n^{1/2} \log^{3/2} n)$ and
$R=o(n)$. The complete proof is contained in this section together
with Sections~\ref{sec:partial-prek-hyper},~\ref{sec:max-pre-hyper}
and~\ref{sec:approx-tail-pre-hyper}, in which we prove some lemmas we
state in this section. This proves
Theorem~\ref{thm:formula-pre-hyper}.

We rewrite the conditions defining $S_m\subseteq \setR^4$. We have
that $(\nn,\kzero,\kone,\ktwo)\in S_m$ if  all of the following
conditions are satisfied:
\begin{itemize}
\item[(C1)] $\nn,\kzero,\kone,\ktwo\geq 0$;
\item[(C2)] $\Ttwo \geq 0$ (equivalently, $2\nn-2\kzero-\kone\geq 0$);
\item[(C3)] $\Tthree \geq 0$; (equivalently, $3\nn+\kone+2\ktwo\leq 3m$);
\item[(C4)] $\Qthree \geq 3\nthree\geq 0$ (equivalently,
  $\kzero-\kone-\ktwo \leq 3m-n$ and $\nn-\kzero-\kone-\ktwo\leq n$);
\item[(C5)] $\Qthree = 0$ whenever $\nthree=0$.
\end{itemize}

For $x = (\nn, \kzero, \kone, \ktwo) \in S_m$, let
\begin{equation}
  \label{eq:wpre-def-hyper}
  \wpre(x)
  =
  \begin{cases}
    {\displaystyle\frac{\Pthree!\Ptwo!\Qthree!(\mtwo-1)!}
    {\kzero!\kone!\ktwo!
      \nthree!\mthree!
      \Tthree!\Ttwo!
      (\mtwoprime-1)!\mtwoprime!
      2^{\ktwo} 2^{\mtwoprime} 6^{\mthree}}
      \frac{\fff(\lambda)^{\nthree}}{\lambda^{\Qthree}}},&\text{if
      }\Qthree > 3\nthree;\\
      {\displaystyle\frac{\Pthree!\Ptwo!\Qthree!(\mtwo-1)!}
      {\kzero!\kone!\ktwo!
      \nthree!\mthree!
      \Tthree!\Ttwo!
      (\mtwoprime-1)!\mtwoprime!
      2^{\ktwo} 2^{\mtwoprime} 6^{\mthree}}
      \frac{1}{6^{\nthree}}},&\text{otherwise.}
  \end{cases}
\end{equation}

%Recall that $\fpre$ was defined only in points of $\preS$ such that
%$\prenn,\prekzero,\prekone, \prektwo, \preTtwo, \preTthree,
%\preQthree, \prenthree$ are strictly positive (and $\preQthree >
%3\prenthree$). We extend the definition of $\fpre$ for the other
%points of $\preS$ as the limit of the function of any sequence of points
%is $\preS$ converging to these points. We discuss later the value of
%$\fpre$ in these boundary points, why the limit is unique and how it
%relates to $\wpre$.

Recall that $\prexopt =
(\prennopt,\prekzeroopt,\prekoneopt,\prektwoopt)$ is defined as
\begin{align*}
  &\prennopt = \frac{3\mpre}{g_2(\lambdaopt)},
  &\prekzeroopt = \frac{3\mpre}{g_2(\lambdaopt)}
                  \frac{2\lambdaopt}{f_1(\lambdaopt)g_1(\lambdaopt)},\\
  &\prekoneopt  = \frac{3\mpre}{g_2(\lambdaopt)}
                   \frac{2\lambdaopt}{g_1(\lambdaopt)},
  &\prektwoopt =\frac{3\mpre}{g_2(\lambdaopt)}
                \frac{\lambdaopt f_1(\lambdaopt)}{2g_1(\lambdaopt)},\end{align*}
where $\lambdaopt = \lambdaopt(n)$ is the unique nonnegative solution of
the equation
\begin{equation*}
  \frac{\lambda \f(\lambda) \fgg(\lambda)}{\FF(2\lambda)}
  =
  3\mpre.
\end{equation*}
The existence and uniqueness of $\lambdaopt$ was discussed in
Lemma~\ref{lem:unique-pre-tools-hyper}.

We will show that $\prexopt$ is the unique point achieving the maximum
for $\fpre$ in the set $\preS_m$ and then we will expand the summation
around $\prexopt$. To determine the region where the summation will be
expanded we will analyse the Hessian of $\fpre$. Let
  \begin{equation}
    \label{eq:def-H0-T-pre-hyper}
  H_0
  =\frac{1}{36}
  \left(\begin{array}{cccc}
      33&12&15&18\\
      12&6&6&6\\
      15&6&7&8\\
      18&6&8&12
  \end{array}
\right)
\quad
\text{and}
\quad
T =
\frac{1}{30}
\left(
  \begin{array}{cccc}
    -47& -16& -11& -6\\ 
    -16&  22&  12&  2\\ 
    -11&  12& 31/3& -4/3\\
    -6&   2& -4/3& -4/3\\
  \end{array}
\right)
\end{equation}
Later we will see that the Hessian of $\fpre$ at $\prexopt$ is
$(-1/r^2) H_0 - (1/r) T + O(J)$, where $J$ denotes the $4\times 4$
matrix of all 1's. For two matrices $A,B$ of same dimensions, we say that
a matrix $A = O(B)$ if $A_{ij}= O(B_{ij})$ for all $i,j$.

Let $z_1 = (1,1,-3,0)$. Then $z_1$ is an eigenvector of
$H_0$ with eigenvalue~$0$. Let $e_i\in\setR^4$ be the vector such that
the $i$-th coordinate is~$1$ and all the others are~$0$. Let
\begin{equation*}
  B
  \eqdef
  \set[\Big]{
    x\in\setR^4 :
    x = \gamma_1 z_1 + \gamma_2 e_2 + \gamma_3 e_3 + \gamma_4 e_4,\
   |\gamma_1|\leq \delta_1 n
   \text{ and }
   |\gamma_i|\leq \delta n  \text{ for }i=2,3,4
  },
\end{equation*}
and let $\preB = \set{(\nn/n, \kzero/n, \kone/n, \ktwo/n): (\nn,
  \kzero, \kone, \ktwo) \in B}$, that is, $\preB$ is a scaled version
of~$B$.  We will choose $\delta_1$ and $\delta$ later. The set
$\xopt+B$ (this is the Minkowski sum of $\set{\xopt}$ and $B$) is the
region where we will approximate $\sum_{x} n!\exp(n\fpre(\prex))$ by
using Taylor's approximation. For this, we show that, for an
appropriate choice for $\delta_1$ and $\delta$, the set $\xopt+ B$ is
contained in $S_m$.
\begin{lem}
  \label{lem:B-in-S}
  Suppose that $\delta_1 = o(r)$ and that $\delta = o(r^2)$. Let $x\in
  B$.  For any function $F$ among $\nn(x+\xopt)$, $k_i(\xopt+x)$ for
  $i=0,1,2$, $\Qthree(\xopt+x)-3\nthree(\xopt+x)$, and the linear
  functions defined in~\eqref{eq:pre-param-def}, we have that
  $F(\xopt+x)\sim F(\xopt)$.  Moreover, $\lambda(x)\sim \lambda(\xopt)$.
\end{lem}
\begin{proof}
  Write $x$ as $x=\gamma_1 z_1 + \gamma_2 e_2 + \gamma_3 e_3 +
  \gamma_4 e_4$ with $|\gamma_1|\leq \delta_1$ and $|\gamma_i|\leq
  \delta$ for $i=2,3,4$.  We will show that $F(\gamma_1
  z_1)=o(F(\xopt))$ and $F(\gamma e_i)=o(F(\xopt))$ for
  $i=2,3,4$. Since $F$ is a linear function, this implies that
  $F(\xopt+x) = F(\xopt) + F(x) = F(\xopt)+o(F(\xopt))$, proving the
  first statement in the lemma.

  Using~\eqref{eq:optpre-rel-hyper}, we have that $F(\xopt) =
  \Omega(r^2n)$ for all the functions $F$ under consideration and so,
  for $i=2,3,4$, we have that $F(\gamma_i e_i) = o(r^2 n) =
  o(F(\xopt))$ since $|\gamma_i|\leq \delta n = o(r^2n)$.
  
  Using~\eqref{eq:optpre-rel-hyper}, we have that $F(\xopt) =
  \Omega(rn)$ for all $F$ under consideration except $\ktwo$,
  $\Tthree$ and $\Qthree-3\nthree$.  Since $|\gamma_1|\leq \delta_1n =
  o(rn)$, we have that $F(\gamma_1 z_1) = o(rn) = o(F(\xopt))$ for all
  $F$ under consideration, except $\ktwo$, $\Tthree$ and
  $\Qthree-3\nthree$. So let $F$ be one of the functions $\ktwo$,
  $\Tthree$ or $\Qthree-3\nthree$. Then, using $z_1 = (1,1,-3,0)$, we
  have that $F(z_1)=0$ and so
  $F(\xopt+x)=F(\xopt)$, finishing the proof of the first statement in
  the lemma.

  Since $\Qthree(\xopt+x)\sim \Qthree(\xopt)$ and
  $\nthree(\xopt+x)\sim \nthree(\xopt)$, we have that
  $\cthree(x+\xopt) = \Qthree(\xopt+x)/\nthree(\xopt+x) \sim
  \cthree(\xopt)$. Thus, since $\lambda(y)$ is defined as the unique
  solution of $\lambda \ff(\lambda)/\fff(\lambda) = \cthree(y)$, we
  have that $\lambda(x)\sim\lambda(\xopt)$ by
  Lemma~\ref{lem:lambda-close-hyper}.
\end{proof}

\begin{cor}
  \label{cor:B-in-S}
  Suppose that $\delta_1 = o(r)$ and that $\delta = o(r^2)$. Let $x\in
  B$. Then there exists $\psi = o(1)$ such that $\xopt+x\in
  S_{\psi}^*$ and $\xopt+x$ is in the interior of $S_m$.
\end{cor}
\begin{proof}
  Recall that $S_\psi^*$ is defined
  in~\eqref{eq:Spsi}. Lemma~\ref{lem:B-in-S} and the definition of
  $S_\psi^*$ immediately imply the first part of the conclusion.

  We check whether $\xopt+x$ satisfies the conditions (C1)--(C5)
  strictly. We have that $\xopt$ satisfies the constraints (C1)--(C4)
  with slack $\Omega(r^2n)$ by~\eqref{eq:optpre-rel-hyper} and recall
  that $r^2n\to \infty$. By Lemma~\ref{lem:B-in-S}, we have that
  $\xopt+x$ also satisfies all the constraints (C1)--(C4) with slack
  $\Omega(r^2n)$.

  It remains to check (C5). We have that $\nthree(\xopt+x)\sim
  \nthree(\xopt) = \Omega(rn)=\omega(1)$ and so (C5) is satisfied
  strictly. We conclude that $\xopt+x$ is in the interior of $S_m$.
\end{proof}

The following lemmas are the main steps in the proof of
Theorem~\ref{thm:formula-pre-hyper}. We show that $\prexopt$ is the
unique maximum for $\fpre$ in $\preS$ and compute a bound for any
other local maximum.
\begin{lem}
  \label{lem:max-pre-hyper}
  The point $\prexopt =
  (\prennopt,\prekzeroopt,\prekoneopt,\prektwoopt)$ is the unique
  maximum for $\fpre$ in $\preS_m$ and
  \begin{equation*}
    \begin{split}
      \fpre(\prexopt) =   
      2r\ln n -4r\ln r &+ 
      \left(
        -\frac{2}{3}\ln(2)-\frac{1}{3}\ln(3)+\frac{1}{3}\right)
      \lambdaopt
      \\
      &+
      \left(
        -\frac{2}{9}\ln(2)-\frac{1}{9}\ln(3)+\frac{7}{36}
      \right) (\lambdaopt)^2
      +O((\lambdaopt)^3).
    \end{split}
  \end{equation*}
  Moreover, there exists a constant $\beta <
  -(2/9)\ln(2)-(1/9)\ln(3)+(7/36)$ such that any other
  local maximum in $\preS_m$ has value at most
  \begin{equation*}
    2r\ln n -4r\ln r +  
      \left(
        -\frac{2}{3}\ln(2)-\frac{1}{3}\ln(3)+\frac{1}{3}\right)
      \lambdaopt
      +
      \beta (\lambdaopt)^2.
  \end{equation*}
\end{lem}
We then estimate the summation of $\exp\paren{ n \fpre(\prex+\prexopt)}$ over points $x\in B$
    such that $x+\xopt$ is integer.
\begin{lem}
  \label{lem:approx-pre-hyper} Suppose that $\delta_1^3=o(r/n)$ and
  $\delta_1^2=\omega(r/n)$, and $\delta^3 = o(r^4/n)$ and $\delta^2 =
  \omega(r^2/n)$. Then
  \begin{equation*}
    \sum_{\substack{x\in B\\x+\xopt \in\setZ^4}}
    \exp\paren[\Big]{ n \fpre(\prex+\prexopt)}
    \sim
    144\sqrt{3} \pi^2n^2r^{7/2}
    \exp\paren{ n \fpre(\prexopt)}.
\end{equation*}
\end{lem}
Finally, we bound the contribution from points far from the maximum.
\begin{lem}
\label{lem:tail-pre-hyper}
Suppose that $\delta_1^3=o(r/n)$ and $\delta_1^2 = \omega(r\ln n/n)$,
and $\delta^3 = o(r^4/n)$ and $\delta^2 = \omega (r^2\ln n/n)$. We
have that
\begin{equation*} 
  \sum_{\substack{x\in S\setminus (\xopt+B)\\ x\in \setZ^4}}  
  \wpre(x) =
  o\left({n!\exp(n\fpre(\prexopt))}\right).
\end{equation*}
\end{lem}

The proof of Lemma~\ref{lem:max-pre-hyper} is deferred to
Section~\ref{sec:max-pre-hyper}. The proofs of
Lemmas~\ref{lem:approx-pre-hyper} and~\ref{lem:tail-pre-hyper} are
presented in Section~\ref{sec:approx-tail-pre-hyper}. We are now ready
to prove Theorem~\ref{thm:formula-pre-hyper}.

% Choose \delta_1 and \delta
In order to use Lemmas~\ref{lem:approx-pre-hyper}
and~\ref{lem:tail-pre-hyper}, we need to check if there exists
$\delta_1$ such that $\delta_1^3=o(r/n)$ and $\delta_1^2 = \omega(r\ln
n/n)$, and $\delta$ such that $\delta^3 = o(r^4/n)$ and $\delta^2 =
\omega (r^2\ln n/n)$. There exists such $\delta_1$ if and only if
$(r/n)^2 = \omega((r \ln n/n)^3)$, which is true if and only if $n/r =
\omega(\ln^3 n)$, which is true since $r = o(1)$. There exists such
$\delta$ if and only if $(r^4/n)^2 = \omega((r^2\ln n/n)^3)$, which is
true if and only if $r^2 = \omega(\ln^3 n/n)$, which is one of the
hypotheses of the theorem.

% summation in B by magic trick
By Proposition~\ref{prop:magic-pre-hyper} and Lemma~\ref{lem:B-in-S},
we have that, for $x\in (\xopt+B)$,
\begin{equation*}
  \begin{split}
    \gpre(x)
    =
    \wpre(x)
    \meancond[\Big]{\prob[\big]{\Gpre(x,\Ys)\text{ simple and
          connected}}}{\Sigma(x)}
    \prob[\big]{\Sigma(x)},
  \end{split}
\end{equation*}
% Expectation
where $\Sigma(x)$ is the event that a random vector $\Ys =
(Y_1,\dotsc, Y_{\nthree(x)})$ of independent truncated Poisson random
variables with parameters $(3,\lambda(x))$ satisfy
$\sum_{i=1}^{\nthree(x)}=\Qthree(x)$. By
Corollary~\ref{cor:expectation-pre-hyper} and Lemma~\ref{lem:B-in-S},
\begin{equation}
  \meancond[\Big]{\prob[\big]{\Gpre(x,\Ys)\text{ simple and
        connected}}}{\Sigma(x)}\sim 1.
\end{equation}

% Stirling
By Stirling's approximation, the definition of $\fpre$
(in~\eqref{eq:fpre-def-hyper} and~\eqref{eq:fpre-extreme}), and
definition of $\wpre$ (in~\ref{eq:wpre-def-hyper}) , we have that
\begin{equation*}
  \wpre(x)
  \sim
  n!\frac{1}{(2\pi n)^{5/2}}
  \left(\frac{\prePthree\prePtwo\preQthree}
    {\prekzero\prekone\prektwo\prenthree\premthree
      \preTthree \preTtwo
      \premtwo
    }
  \right)^{1/2}
  \exp(n \fpre(\hat x)).
\end{equation*}
Since $x\in (\xopt+B)$, by Lemma~\ref{lem:B-in-S}, we have that
  \begin{equation*}
    \frac{\prePthree\prePtwo\preQthree}
    {\prekzero \prekone\prektwo\prenthree\premthree
      \preTthree\preTtwo
      \premtwo
    }
    \sim
     \frac{\prePthreeopt\prePtwoopt\preQthreeopt}
    {\prekzeroopt \prekoneopt\prektwoopt\prenthreeopt\premthreeopt
      \preTthreeopt\preTtwoopt
      \premtwoopt
    }
    \sim\frac{1}{ r^{5/2}4\sqrt{6}}.
  \end{equation*}
  % Prob Sigma
  Next we estimate $\prob{\Sigma(x)}$. We will use~\textred{\cite[Theorem 4]{PWa}},
   applied with $\nthree$ as the
  parameter $n$ in~\textred{\cite[Theorem 4]{PWa}} and $\cthree =
  \Qthree/\nthree$ as $c$ in~\textred{\cite[Theorem 4]{PWa}}.  By
  Lemma~\ref{lem:B-in-S} and~\eqref{eq:optpre-rel-hyper}, we have that
  $\Qthree(x) - 3\nthree(x) \sim (\Qthree(\xopt)-\nthree(\xopt)) \sim
  12 R^2/n = \omega \ln(n)$. Thus, by~\textred{\cite[Theorem 4]{PWa}},
  \begin{equation*}
    \prob{\Sigma(x)}\sim
    \frac{1}{\sqrt{2\pi \Qthree(x)  (1+\etathree(x)-\cthree(x))}},
  \end{equation*}
  where $\etathree(x) = \lambda(x) \f(\lambda(x))/\ff(\lambda(x))$ and
  $\cthree(x) = \Qthree(x)/\nthree(x) = \lambda(x)
  \ff(\lambda(x))/\fff(\lambda(x))$. Since $\Qthree(x)/\nthree(x) \sim
  \preQthree(\prexopt)/\prenthree(\prexopt)$,
  Lemma~\ref{lem:lambda-close-hyper} implies that $\lambda(x) \sim
  \lambdaopt\to 0$ and so (omitting the $(x)$ in the following)
  \begin{equation*}
    \begin{split}
      &1+\eta_3-c_3
      = \frac{\ff(\lambda)\fff(\lambda)
               +\lambda\f(\lambda)\fff(\lambda)
               -\lambda\ff(\lambda)^2}
              {\ff(\lambda)\fff(\lambda)}
      \\
      &=
      \frac{ {\displaystyle
          \paren[\Big]{\frac{\lambda^2}{2}+\frac{\lambda^3}{6}}
          \paren[\Big]{\frac{\lambda^3}{6}+\frac{\lambda^4}{24}}
          +\lambda \paren[\Big]{\lambda+\frac{\lambda^2}{2}}
          \paren[\Big]{\frac{\lambda^3}{6}+\frac{\lambda^4}{24}}
          +\lambda \paren[\Big]{\frac{\lambda^2}{2}+\frac{\lambda^3}{6}}^2
          +O(\lambda^7)}}
        {\displaystyle
          \paren[\Big]{\frac{\lambda^2}{2}+\frac{\lambda^3}{6}}
          \paren[\Big]{\frac{\lambda^3}{6}+\frac{\lambda^4}{24}}
        +O(\lambda^7)}
      \\
      &=
      \frac{\lambda^6/144}{\lambda^5/12}
      \paren[\big]{1+O(\lambda)}
      \sim
      \frac{\lambda}{12}
      \sim
      \frac{\lambdaopt}{12}\sim r,
    \end{split}
  \end{equation*}
by  Lemma~\ref{lem:B-in-S} and~\eqref{eq:optpre-rel-hyper}. Moreover,
  $\Qthree\sim 6R$ by~\eqref{eq:optpre-rel-hyper}. Hence,
  \begin{equation*}
    \prob{\Sigma}\sim
    \frac{1}{\sqrt{2\pi (6R) (1+\etathree-\cthree)}}
    \sim
    \frac{1}{r\sqrt{12\pi n}}.
  \end{equation*}
% Summation
Thus,
\begin{equation}
  \label{eq:approx-gprex-pre-hyper}
  \gpre(x)
  =
  n!
  \frac{1}{144(\pi n)^3 r^{7/2}}
  \sum_{x\in B}
  \exp(n\fpre(x))(1+o(1)),
\end{equation}
for all $x\in(\xopt+B)$. Since $(\xopt+B)\cap \setZ^4$ is a finite set
for each~$n$, \crisc{removed a reference to a uniformity lemma in the
  appendix}\Old{Lemma~\ref{lem:unif-deg-seq} implies }\New{we have
}that there is a function $q(n) = o(1)$ such that the error
in~\eqref{eq:approx-gprex-pre-hyper} is bounded by $q(n)$ uniformly
for all $x\in (\xopt+B)\cap \setZ^4$. Thus,
\begin{equation*}
  \begin{split}
    \sum_{x\in (\xopt+B)\cap\setZ^4}
    \gpre(x)
    &\sim
    n!
    \frac{1}{144(\pi n)^3 r^{7/2}}
    \sum_{x\in (\xopt+B)}
    \exp(n\fpre(x))
    \\
    &\sim
    n!
    \frac{1}{144(\pi n)^3 r^{7/2}}
    \cdot
    144\sqrt{3} \pi^2n^2r^{7/2}
    \exp\paren{ n \fpre(\prexopt)},
  \end{split}
\end{equation*}
by Lemma~\ref{lem:approx-pre-hyper}. Thus,
\begin{equation*}
  \sum_{x\in (\xopt+B)\cap\setZ^4}
  \gpre(x)
  \sim
  n!
  \frac{\sqrt{3}}{\pi n}
  \exp(n\fpre(\prexopt)).
  \end{equation*}
  Together with Lemma~\ref{lem:tail-pre-hyper}, this finishes the
  proof of Theorem~\ref{thm:formula-pre-hyper}.
  
\subsection{Partial derivatives}
\label{sec:partial-prek-hyper}
In this section, we will analyse the first, second, and third partial
derivatives of $\fpre$. This will be used in the proof that $\prexopt$
achieves the maximum for $\fpre$ (Lemma~\ref{lem:max-pre-hyper}) and
also to approximate the summation around $\prexopt$
(Lemma~\ref{lem:approx-pre-hyper}). \nickca{deleted: See Section~ for a
Maple spreadsheet.}

Recall that $\hpre(y) = y\ln(yn)-y$  and, for $\prex =
(\prenn,\prekzero,\prekone,\prektwo)$,
\begin{equation*}
  \begin{split}
    \fpre(\prex)
    =
    &
    \ \hpre(\prePthree)
    +\hpre(\prePtwo)
    +\hpre(\preQthree)
    +\hpre(\mtwo)
    \\
    &-\hpre(\prekzero)
    -\hpre(\prekone)
    -\hpre(\prektwo)
    -\hpre(\prenthree)
    -\hpre(\premthree)
    \\
    &-\hpre(\preTthree)
    -\hpre(\preTtwo)
    -2\hpre(\premtwoprime)
    \\
    &
    -\prektwo \ln 2
    - \premtwoprime \ln 2
    - \premthree \ln 6
    \\
    &+\prenthree\ln f_3(\lambda(x))
    -\preQthree\ln\lambda(x),
  \end{split}
\end{equation*}
where $\lambda(x)$ is the unique positive solution to $\lambda
\ff(\lambda)/\fff(\lambda) = \cthree$, where $\cthree =
\preQthree/\prenthree$.

Using~\eqref{eq:difdeg} to compute the partial derivatives of
$\prenthree\ln f_3(\lambda(x)) -\preQthree\ln\lambda(x)$
(w.r.t.~$\prenn$, $\prekzero$, $\prekone$ and $\prektwo$), we obtain
  \begin{align}
    \label{eq:derivative-pre-hyper}
    \exp\left(\frac{\dif \fpre(x)}{\dif \prenn}\right) &=\frac{4 \preTthree^3 \prenthree \prenn\lambda}
    {9\premthree^2\preQthree\preTtwo^2 \fff{\lambda}};\\
    \exp\left(\frac{\dif\fpre(x)}{\dif \prekzero}\right) 
    &= \frac{\prenthree \preTtwo^2\lambda^2}
    {2\preQthree^2\prekzero \fff(\lambda)};\\
    \exp\left(\frac{\dif \fpre(x)}{\dif \prekone}\right) 
    &=\frac{\preTthree\prenthree \preTtwo\lambda^2}
    {\prekone\preQthree^2 \fff(\lambda)};\\
    \exp\left(\frac{\dif \fpre(x)}{\dif \prektwo}\right)
    &=
    \frac{\preTthree^2\prenthree\lambda^2}
    {2\prektwo\preQthree^2 \fff(\lambda)};
  \end{align}

  For the second partial derivatives, we need to compute
  \begin{equation*}
    \frac{ \partial^2 (\prenthree \ln \fff(\lambda(x)) 
          - \preQthree\ln \lambda(x))}
        {\partial a \partial b},
\end{equation*}
for any $a,b\in
\set{\prenn,\prekzero,\prekone,\prektwo}$. Using~\eqref{eq:difdeg},
this is

\begin{equation}
  \label{eq:second-lambda-prek}
  \begin{split}
    \frac{\partial}{\partial a}
    \left(\frac{\partial \prenthree}{\partial b} \ln f_{3}(\lambda)
      - \frac{\partial \preQthree}{\partial b}\ln\lambda\right)
    &=
    \frac{\partial}{\partial a}\left(-\ln f_{3}(\lambda)
      - \frac{\partial\preQthree}{\partial b}\ln\lambda
    \right)
    =    \frac{\partial \lambda}{\partial a}
    \left(-\frac{f_{2}(\lambda)}{f_3(\lambda)}
          - \frac{\partial\preQthree}{\partial b}\frac{1}{\lambda}
    \right)
    \\
    &=
    \frac{\partial c_3}{\partial a}
    \frac{\lambda}{c_3 (1+\eta_3-c_3)}
    \left(-\frac{f_{2}(\lambda)}{f_3(\lambda)}
      - \frac{\partial\preQthree}{\partial b}\frac{1}{\lambda}
    \right),
    \\
    &=
    \left(\frac{\partial{\preQthree}}{\partial a}
      \frac{1}{\prenthree}
      -\frac{\partial{\prenthree}}{\partial a}
      \frac{\preQthree}{\prenthree^2}
    \right)
    \frac{1}{c_3 (1+\eta_3-c_3)}
    \left(-c_3- \frac{\partial\preQthree}{\partial b} \right)
    \\
    &=
    -\left(c_3+\frac{\partial{\preQthree}}{\partial a}\right)
    \left(c_3+ \frac{\partial\preQthree}{\partial b} \right)
    \frac{1}{\preQthree (1+\eta_3-c_3)}.
  \end{split}
\end{equation}
The second partial derivatives now are
\begin{equation}
  \label{eq:second-pre-hyper}
  \begin{split}
    &\frac{\partial^2 \fpre(\prex)}{\partial \prenn\partial\prenn} =
  \frac{9}{\Pthree} +\frac{4}{\Ptwo}-\frac{9}{\Tthree}+\frac{1}{\Qthree}
  -\frac{1}{\nthree}-\frac{1}{\mthree}-\frac{4}{\Ttwo}
  -\frac{2}{\mtwoprime} +\frac{1}{\nn}
  +D_1\\
  & \frac{\partial^2 \fpre(\prex)}{\partial \prenn\partial\prekzero} =
  -\frac{4}{\Ptwo}+\frac{2}{\Qthree}-\frac{1}{\nthree}
  +\frac{4}{\Ttwo}+\frac{2}{\mtwoprime}
  +D_k\\
  & \frac{\partial^2 \fpre(\prex)}{\partial \prenn\partial\prekone} =
  -\frac{3}{\Tthree}+\frac{2}{\Qthree}
  -\frac{1}{\nthree}+\frac{2}{\Ttwo}
  +D_k\\
  &\frac{\partial^2 \fpre(\prex)}{\partial \prenn\partial\prektwo} =  
  -\frac{6}{\Tthree}+\frac{2}{\Qthree}-\frac{1}{\nthree}
  +D_k\\
  &\frac{\partial^2 \fpre(\prex)}{\partial \prekzero\partial\prekzero} =  
  \frac{4}{\Ptwo}+\frac{4}{\Qthree}-\frac{1}{\nthree}
  -\frac{4}{\Ttwo}-\frac{2}{\mtwoprime}-\frac{1}{\kzero} 
  +D_{kk}\\
  &\frac{\partial^2 \fpre(\prex)}{\partial \prekzero\partial\prekone} =  
  \frac{4}{\Qthree}-\frac{1}{\nthree}-\frac{2}{\Ttwo}
  +D_{kk} \\
  &\frac{\partial^2 \fpre(\prex)}{\partial \prekzero\partial\prektwo} =  
  \frac{4}{\Qthree}-\frac{1}{\nthree}  
  +D_{kk}\\
  &\frac{\partial^2 \fpre(\prex)}{\partial \prekone\partial\prekone} =  
  -\frac{1}{\kone}-\frac{1}{\Tthree}+\frac{4}{\Qthree}
  -\frac{1}{\nthree}-\frac{1}{\Ttwo}
  +D_{kk}\\
  & \frac{\partial^2 \fpre(\prex)}{\partial \prekone\partial\prektwo} =  
  -\frac{2}{\Tthree}+\frac{4}{\Qthree}-\frac{1}{\nthree}
  +D_{kk}\\
  &\frac{\partial^2 \fpre(\prex)}{\partial \prektwo\partial\prektwo} =  
  -\frac{1}{\ktwo}-\frac{4}{\Tthree}+\frac{4}{\Qthree}-\frac{1}{\nthree}
  +D_{kk},
  \end{split}
\end{equation}
where
\begin{align*}
  &D_1 = -\frac{(\cthree-1)^2}{(1+\etathree-\cthree)\preQthree};\\
  &D_k =  -\frac{(\cthree-1)(\cthree-2)}{(1+\etathree-\cthree)\preQthree};\\
  &D_{kk} =  -\frac{(\cthree-2)^2}{(1+\etathree-\cthree)\preQthree}.
\end{align*}

In the next lemma, we find an approximation for the Hessian $\fpre$ at
$\prexopt$. It follows immediately by computing the series of each
partial second derivative with $\lambda\to 0$. \nickca{Deleted: See
Section~ for a Maple spreadsheet.}
\begin{lem}
  \label{lem:hessian-hyper}
  The Hessian of $\fpre$ at $\prexopt$ is $(-1/r^2) H_0 - (1/r) T +
  O(J)$, where $H_0$ and $T$ are defined in~\eqref{eq:def-H0-T-pre-hyper}
  and $J$ is a $4\times 4$ matrix  with all entries equal to $1$.
\end{lem}

We will bound the third partial derivatives for points close to
$\xopt$.
\begin{lem}
  \label{lem:third-pre-hyper} 
  Suppose that $\delta_1^3=o(r/n)$ and $\delta^3 = o(r^4/n)$. Then for
  any $x \in B$ we have that
  \begin{equation*}
    n\frac{\partial \fpre(\prexopt+\prex)}
          {\partial t_1 \partial t_2 \partial t_3} 
    t_1(\prex) t_2(\prex) t_3(\prex)
    =o(1),
  \end{equation*}
  for any $t_1,t_2,t_3 \in\set{\prenn,\prekzero,\prekone,\prektwo}$.
\end{lem}
\begin{proof}
  Let $x\in B$. Then $x = \alpha z_1 + b$, where
  $\abs{\alpha}\leq \delta_1$ and $b= (0,b_2,b_3,b_4)$ and
  $|b_i|\leq \delta$ and $b^T z_1 = 0$. Recall that $z_1 =
  (1,1,-3,0)$ and so $x = (\alpha, \alpha+b_2,
  -3\alpha+b_3,b_4)$. Then, by using~\eqref{eq:second-pre-hyper}, we
  may compute each partial derivative $\frac{\partial \fpre}{\partial
    t_1 \partial t_2 \partial t_3} t_1(\prex)t_2(\prex)(t_3(\prex))$
  exactly. We omit the lengthy computations here. \nickca{Deleted: (See
  Section~ for a Maple spreadsheet.)}
  The third derivative is the sum of the part involving
  $\lambda$ and the part that does not involve $\lambda$.  The part
  not involving $\lambda$ can be written as 
  \begin{equation*}
    \sum_{\substack{a=(a_1,a_2,a_3,a_4)\in\setN^4,\\
        a_1+a_2+a_3+a_4=3}}
    T(a)\alpha^{a_1}(\alpha+b_2)^{a_2}(-3\alpha+b_3)^{a_3}b_4^{a_4},
  \end{equation*}
  where each $T(a)$ is a sum of terms in the format $1/z^2$, where
  \begin{equation*}
    z
    \in
    \set{\prenn,\prekzero,\prekone,\prektwo,\prenthree,
      \prePtwo,\prePthree,\preQthree,\preTtwo,\preTthree}.
  \end{equation*}
  This can be expanded so that it is
  \begin{equation*}
    \sum_{\substack{f=(f_1,f_2,f_3,f_4)\in\set{0,1,2,3}\times\set{0,1}^3,\\
        f_1+f_2+f_3+f_4=3}}
    T_2(f)\alpha^{f_1} b_2^{f_2} b_3^{f_3}b_4^{f_4},
    \end{equation*}
    where each $T_2(f)$ is also a sum of terms in the format $1/z^2$,
    where
  \begin{equation*}
    z
    \in
    \set{\prenn,\prekzero,\prekone,\prektwo,\prenthree,
      \prePtwo,\prePthree,\preQthree,\preTtwo,\preTthree}.
  \end{equation*}
  Since $\delta_1^3=o(r/n)$ and $\delta^3 = o(r^4/n)$ and $R^3 =
  \omega(N)$, we have that $\delta_1=o(r)$ and $\delta =
  o(r^2)$. Thus, by Lemma~\ref{lem:B-in-S}, we have that $z\sim
  z(x^*)$.  Using this fact and computing the series of each term with
  $r\to 0$, we obtain $T_2(f) = O(1/r^{4-f_1})$, and so $\delta =
  o(r^4/n)$ and $\delta_1 = o(1/\sqrt{n})$ ensure
  $|\alpha^{f_1}b_2^{f_2}b_3^{f_3}b_4^{f_4}T_2(f)| \leq
  \delta_1^{f_1}\delta^{f_2+f_3+f_4}|T_2(f)| = o(1)$.

  Similarly the part involving $\lambda$ can be written
  as \begin{equation*}
    \sum_{\substack{f=(f_1,f_2,f_3,f_4)\in\set{0,1,2,3}\times\set{0,1}^3,\\
        f_1+f_2+f_3+f_4=3}}
    U(f)\alpha^{f_1} b_2^{f_2} b_3^{f_3}b_4^{f_4}
  \end{equation*}
  where each $U(f)$ is a sum of terms in the following format
  \begin{equation*}
    \begin{split}
      \frac{1}{\preQthree(1+\eta-c_3)}
   &{\Bigg(}
    -\frac{(c3-e_1)(2 c_3- e_2- e_3)}{\prenthree}
    \\
    &+ 
    (c_3-e_2)(c_3-e_3)\left(\frac{e_1}{\preQthree}+
      \frac{(c_3-e_1)}{(1+\eta-c_3)^2}
      \left(
        \frac{\eta(1+\lambda e^{\lambda}/\f(\lambda) - \eta)}
        {\preQthree(1+\eta-c_3)}
        -\frac{1}{\prenthree}
      \right)
    \right)
    {\Bigg )}
    \end{split}
  \end{equation*}
  where $e_1,e_2,e_3\in\set{1,2}$.  Since $\delta_1=o(r)$ and $\delta
  = o(r^2)$, by Lemma~\ref{lem:B-in-S}, we have that
  $\lambda(\xopt+x)\sim \lambda(\xopt)$.  Using this fact and
  computing the series of $U(f)$ with $r\to 0$, we have that $U(f) =
  O(1/r^{4-f_1})$, and so $\delta = o(r^4/n)$ and $\delta_1^3 =
  o(r/n)$ ensure $|\alpha^{f_1}b_2^{f_2}b_3^{f_3}b_4^{f_4}U(f)| \leq
  \delta_1^{f_1}\delta^{f_2+f_3+f_4}|U(f)| = o(1)$.
\end{proof}

\subsection{Establishing the maximum}
\label{sec:max-pre-hyper}

In this section, we prove Lemma~\ref{lem:max-pre-hyper} which
establishes the maximum of $\fpre$ in $\preS$. Recall that the region
$\preS$ where we want to optimise $\fpre(\prex)$ over is defined by
conditions (C1)--(C4). We rewrite these conditions as follows:
\begin{itemize}
\item[(D1)] $\preQthree \geq 3\prenthree\geq 0$ and, if $\prenthree=0$, then $\preQthree=0$.
\item[(D2)] $\prePtwo\geq 0$;
\item[(D3)] $\prePthree\geq 0$;
\item[(D4)] $\prekzero,\prekone,\prektwo\geq 0$ and $\preTtwo\geq 0$ and
  $\preTthree \geq 0$;
\end{itemize}
These conditions are obviously a subset of the conditions (C1)--(C5),
with the $\prenn \geq 0$ being the only constraint missing, which is
implied by $\preTtwo \geq 0$. First we will show that $\prexopt$ is
the only local maximum in the interior of $\preS$:
\begin{lem}
  \label{lem:max-interior-prek}
  The point $\prexopt= (\prennopt, \prekzeroopt, \prekoneopt,
  \prektwoopt)$ is the unique local maximum for $\fpre$ in the
  interior of $\preS$ and its value is
  \begin{equation*}
    2r\ln n -4r\ln r + 
      \left(
        -\frac{2}{3}\ln(2)-\frac{1}{3}\ln(3)+\frac{1}{3}\right)
      \lambdaopt
      +
      \left(
        -\frac{2}{9}\ln(2)-\frac{1}{9}\ln(3)+\frac{7}{36}
      \right) (\lambdaopt)^2
      +O((\lambdaopt)^3).
  \end{equation*}
\end{lem}

We will then analyse local maximums when some condition in
(D1)--(D4) is tight. The following lemma will be useful to reduce the
number of cases to be analysed by giving sufficient conditions 
for a point not being a local maximum.
\begin{lem}
  \label{lem:aux-boundary}
  Let $k$ be a fixed positive integer and $S\subseteq\setR$ be a
  bounded set. Let $f: S \to \setR$ be a continuous function such that
  $f(x) = -\sum_{i=1}^q \ell_i(x)\ln\ell_i(x) + g(x)$, where
  $\ell_i(x)=\sum_{j=1}^k \alpha_{i,j} x_j \geq 0$ for all $x\in
  S$. Suppose $x^{(0)}\in S$ is such that $\ell_i(x^{(0)}) =0$ for
  some $i$. Suppose there is $v\in\setR^k$ such that $x^{(0)}+tv$ is
  in the interior of $S$ for small enough $t$ and 
  \begin{equation*}
    \frac{\dif
      g(x^{(0)}+t v)}{\dif t}|_{t=0}>C,    
  \end{equation*}
 for some (possibly negative)
  constant $C$. Then $x^{(0)}$ is not a local maximum for $f$ in~$S$.
\end{lem}

The following lemma gives a bound for the value of $\fpre(\prex)$ for
any local maximum other than~$\prexopt$:
\begin{lem}
  \label{lem:boundary-prek}
  Let $\preS_1$ be the points in $\preS_m$ such that any of the
  constraints in (D1)--(D4) is tight.  There exists a constant $\beta
  < -(2/9)\ln(2)-(1/9)\ln(3)+(7/36)$ such that any local maximum of
  $\preS_m$ in $\preS_1$ for $\fpre$ has value at most
  \begin{equation*}
    2r\ln n -4r\ln r +  
      \left(
        -\frac{2}{3}\ln(2)-\frac{1}{3}\ln(3)+\frac{1}{3}\right)
      \lambdaopt
      +
      \beta (\lambdaopt)^2.
  \end{equation*}
\end{lem}
Note that the constraint $\preQthree = 0$ whenever $\prenthree=0$
makes $\preS_m$ not closed. We analyse the value of any sequence of
points converging to a point with $\preQthree >0$ and $\nthree = 0$:
\begin{lem}
  \label{lem:open-prek}
  Let $(\prex(i))_{i\in \setN}$ be a sequence of points in 
  $\preS_m$ converging to a point $z$ with $\preQthree(z)>0$ and
  $\prenthree(z)=0$. Then $\lim_{i\to \infty}\fpre(x_i) = -\infty$.
\end{lem}

Lemma~\ref{lem:max-pre-hyper} is trivially implied by
Lemmas~\ref{lem:max-interior-prek},~\ref{lem:boundary-prek},
and~\ref{lem:open-prek}. In the rest of this section, we prove these
lemmas.

\begin{proof}[Proof of Lemma~\ref{lem:max-interior-prek}]
  The computations in this proof are elementary (such as computing
  resultants) but very lengthy. \nickca{Deleted: See Section~ for a
  Maple spreadsheet.}

  Since any local maximum must have value $\exp(\frac{\dif \fpre}{\dif
    t}) = 1$ for any $t\in\set{\prenn,\prekzero,\prekone,\prektwo}$,
  by~\eqref{eq:derivative-pre-hyper}
  \begin{align}
    & \label{eq:n1} 
    4\preTthree^3 
    \prenthree \prenn\lambda
    -9\premthree^2\preQthree
    \preTtwo^2 \fff(\lambda)
    =0\\
    &\label{eq:k0} 
    \prenthree \preTtwo^2\lambda^2
    -2\preQthree^2\prekzero f_3(\lambda)
    =0\\
    &\label{eq:k1} 
    \preTthree\prenthree 
    \preTtwo\lambda^2
    - \prekone\preQthree^2 f_3(\lambda)
    =0\\
    &\label{eq:k2} 
    \preTthree^2\prenthree\lambda^2
    - 2\prektwo\preQthree^2 f_3(\lambda)
    =0.
  \end{align}
  Next we proceed to take resultants between the LHS of these
  equations to show that there is only one solution in the interior of
  $\preS_m$ satisfying all of them. In these computations, we consider
  $f_3(\lambda)$ and $\lambda$ as independent variables.
  The resultant of the RHS of~\eqref{eq:k0} and~\eqref{eq:k1}
  by eliminating $f_3(\lambda)$ is
  \begin{equation*}
    \lambda^2
    \preTtwo
    \prenthree
    \preQthree^2
    (6\prenn\prekzero+4\prektwo\prekzero+2\prenn\prekone-\prekone^2-6\mpre\prekzero)
    =0
  \end{equation*}
  and, since only the last term may possibly be zero in the interior of
  $\preS_m$, this implies that any local maximum in the interior of $\preS_m$ must
  satisfy
  \begin{equation}
    \label{eq:r1}
    6\prenn\prekzero+4\prektwo\prekzero+2\prenn\prekone-\prekone^2-6\mpre\prekzero=0,
  \end{equation}
  and note that this determines $\prektwo$ in terms of $\prenn$,
  $\prekone$ and $\prekzero$ for any local maximum in the interior of~$\preS_m$.
  Similarly, the  resultant of the RHS of~\eqref{eq:k1} and~\eqref{eq:k2}
  by eliminating $\lambda$ is
  \begin{equation*}
    f_3(\lambda)^2
    \preTthree^2
    \prenthree^2
    \preQthree^4 
    (4\prektwo\prekzero+3\mpre\prekone-3\prenn \prekone-\prekone^2-4\prenn\prektwo)^2=0
  \end{equation*}
  and it implies that any local maximum in the interior of $\preS_m$ must
  satisfy
  \begin{equation}
    \label{eq:r2}
    4\prektwo\prekzero+3\mpre\prekone-3\prenn \prekone-\prekone^2-4\prenn\prektwo=0.
  \end{equation}
  The resultant of the RHS of~\eqref{eq:r1} and~\eqref{eq:r2}
  by eliminating $\prektwo$ is
  \begin{equation*}
    4\preTtwo
    (3\mpre \prekzero-3\prenn\prekzero-\prenn\prekone)=0,
  \end{equation*}
  and it implies that any local maximum in the interior of $\preS_m$ must
  satisfy
  \begin{equation}
    \label{eq:r3}
    3\mpre \prekzero-3\prenn\prekzero-\prenn\prekone=0,
  \end{equation}
  which gives determines $\prekone$ in terms of $\prekzero$ and $\prenn$.

  Taking the resultant of the RHS of~\eqref{eq:n1} and~\eqref{eq:r1}
  by eliminating $\prektwo$ and ignoring the factors that cannot be
  zero in $\preS_m$ gives us
  \begin{equation}
    \label{eq:r4}
    \begin{split}
      &-4\lambda\prekone^3\prenn\prekzero
      -2\lambda\prekone^4\prenn^2
      +\lambda\prekone^5\prenn
      -2\lambda\prekone^3\prenn^2\prekzero
      +6\lambda\prekone^3\mpre\prenn\prekzero
      +4\lambda\prekone^4\prenn\prekzero
      +4\lambda\prekone^3\prenn\prekzero^2
      +36\prekzero^3 f_3(\lambda)\prekone\prenn^2
      \\
      &
      -72\prekzero^3 f_3(\lambda) \prekone\mpre\prenn
      +36\prekzero^3 f_3(\lambda) \prekone\mpre^2
      +72 f_3(\lambda) \prekzero^4\prenn^2
      -144 f_3(\lambda) \prenn\prekzero^4\mpre
      +72 f_3(\lambda) \mpre^2\prekzero^4
      =0
    \end{split}
  \end{equation}
  and then we take the resultant of the RHS of~\eqref{eq:r3}
  and~\eqref{eq:r4} by eliminating $\prekone$ and ignoring the factors
  that cannot be zero in $\preS_m$ gives us
  \begin{equation}
    \label{eq:r5}
    \begin{split}
      &27\prekzero \lambda \mpre^3
      -45\prekzero \lambda \mpre^2\prenn
      +21\prekzero \lambda \mpre\prenn^2
      -3 \lambda \prenn^3\prekzero
      +12\mpre\prenn^3f_3(\lambda)
      +12\prenn^3 \lambda \mpre
      \\
      &-12 \lambda \mpre\prenn^2
      -12\prenn^4 \lambda 
      -4f_3(\lambda)\prenn^4
      +12 \lambda \prenn^3
      =0.
    \end{split}
\end{equation}

Taking the resultant of the RHS of~\eqref{eq:k0} and~\eqref{eq:r1}
  by eliminating $\prektwo$ and ignoring the factors that cannot be zero in $\preS_m$ gives us
  \begin{equation}
    \label{eq:r6}
8\prekzero^2f_3(\lambda)
+4 \lambda^2\prekzero^2
+6 \lambda^2\mpre\prekzero
+4 \lambda^2\prekone\prekzero
-4\prekzero \lambda^2 
-2 \lambda^2\prenn\prekzero
+8\prekone f_3(\lambda)\prekzero
+\lambda^2\prekone^2
+2f_3(\lambda)\prekone^2
-2 \lambda^2\prenn\prekone
=0
  \end{equation}
and then we take the resultant of the RHS of~\eqref{eq:r3} and~\eqref{eq:r6}
  by eliminating $\prekone$ and ignoring the factors that cannot be zero in $\preS_m$ gives us
  \begin{equation}
    \label{eq:r7}
2\prekzero f_3(\lambda)\prenn^2
+\lambda^2\prenn^2\prekzero
-6 \lambda^2\mpre\prenn\prekzero
-12\prekzero f_3(\lambda)\mpre\prenn
+9 \lambda^2\mpre^2\prekzero
+18\prekzero f_3(\lambda)\mpre^2
-4 \lambda^2\prenn^2
+4 \lambda^2\prenn^3
=0,
  \end{equation}
  and note that this determines $\prekzero$ in terms of $\prenn$ and $\lambda$.

  Finally we take the resultant of the RHS of~\eqref{eq:r5}
  and~\eqref{eq:r7} by eliminating $\prekzero$ and ignoring the
  factors that cannot be zero in $\preS_m$, we get
  \begin{equation}
    \label{eq:r8}
    6 \lambda \mpre\prenn
    +6f_3(\lambda)\mpre\prenn
    +3 \lambda^2\mpre\prenn
    -6\mpre \lambda 
    -6 \lambda \prenn^2
    -2f_3(\lambda)\prenn^2
    -\lambda^2\prenn^2
    +6\prenn \lambda =0.
  \end{equation}
  We can then use the equation determining $\lambda$ (that is,
  $\lambda f_2(\lambda)/f_3(\lambda) = \preQthree/\prenthree$) by
  replacing $\prekzero,\prekone$ and $\prektwo$ by the values
  determined by $\prenn,\lambda$ and $m$ and taking the resultant
  with~\eqref{eq:r8} by eliminating $\prekzero$ and ignoring the
  factors that cannot be zero in $\preS_m$:
  \begin{equation*}
    3\mpre e^{2\lambda}
    -9\mpre^2 e^{2\lambda}
    +3\mpre e^{2\lambda}\lambda
    -\lambda e^{2\lambda}
    +3\mpre e^{\lambda}\lambda
    -e^{\lambda}\lambda
    +2\lambda-12\mpre\lambda+9\mpre^2-3\mpre+18\lambda\mpre^2=0
  \end{equation*}
  which has two solutions for $\mpre$: $\mpre=1/3$ (which is false) or
  \begin{equation*}
    \mpre =\frac{1}{3}\frac{\lambda f_1(\lambda) g_2(\lambda)}{\FF(2\lambda)},
  \end{equation*}
  which has a unique positive solution $\lambdaopt$ by
  Lemma~\ref{lem:unique-pre-tools-hyper}, which defines $\prexopt$.  Thus,
  $\prexopt$ is the only point in the interior of $\preS_m$ such that
  all partial derivatives at it are zero. We now show that $\prexopt$
  is a local maximum. Using the second partial derivatives computed
  in~\eqref{eq:second-pre-hyper} and the series of the determinants of
  each leading principal submatrix with $\lambda \to 0$, we have that
  the Hessian at $\prexopt$ is negative definite, which implies that
  $\prexopt$ is a local maximum.

  By writing $\fpre(\xopt)$ in terms of $\lambdaopt$ and computing its
  series with $\lambda\to 0$, we obtain 
  \begin{equation*}
    2r\ln n -4r\ln r + 
      \left(
        -\frac{2}{3}\ln(2)-\frac{1}{3}\ln(3)+\frac{1}{3}\right)
      \lambdaopt
      +
      \left(
        -\frac{2}{9}\ln(2)-\frac{1}{9}\ln(3)+\frac{7}{36}
      \right) (\lambdaopt)^2
      +O((\lambdaopt)^3).
  \end{equation*}
\end{proof}

\clearpage
\begin{proof}[Proof of Lemma~\ref{lem:aux-boundary}]
  Let $I\in[q]$ be the set of indices such that $\ell_{i}(x^{(0)})=0$. We
  compute the derivative of $f(x^{(0)}+tv)$ at $t=0$, using the fact that
  $\ell_i$ is a linear function,
  \begin{equation*}
    \begin{split}
      \frac{\dif f(x^{(0)}+tv)}{\dif t} {\Big|_{t=0}}
      &\geq
      C
      +
      \sum_{i=1}^{q} \lim_{t\to 0^{+}}
      \frac{\left( -\ell_i(x^{(0)}+tv) \ln\ell_i(x^{(0)}+tv)+\ell_i(x^{(0)}) \ln\ell_i(x^{(0)})
        \right)}{t}\\
      &=
      C
      +
      \sum_{i=1}^{q} \lim_{t\to 0^{+}}
      \frac{\left( -\ell_i(tv) \ln\ell_i(x^{(0)}+tv)+
          \ell_i(x^{(0)}) (\ln\ell_i(x^{(0)})-\ln(\ell_i(x^{(0)}+tv)
        \right)}{t}\\
      &=
      C
      +
      \sum_{i=1}^{q} \lim_{t\to 0^{+}}
      \left( -\ell_i(v) \ln\ell_i(x^{(0)}+tv)\right)
      -\sum_{i\in[q]\setminus I} \lim_{t\to 0^{+}}
      \frac{\ell_i(x^{(0)})}{t}
      \ln\left(1+t\frac{\ell_i(v)}{\ell_{i}(x^{(0)})}
        \right)\\
        &=
      C
      +
      \sum_{i=1}^{q} \lim_{t\to 0^{+}}
      \left( -\ell_i(v) \ln\ell_i(x^{(0)}+tv)\right)
      -\sum_{i\in[q]\setminus I} \ell_i(v).
      \end{split}
\end{equation*}
Since $x_0+tv$ is in the interior of $S$ for
small enough but positive $t$, we have that $\ell_{i}(v) > 0$ for
all $i\in I$.  For $i\in [q]\setminus I$, we have that $\ell_i(v) \ln\ell_i(x^{(0)}
   + tv)+\ell_i(v)$ is bounded. For $i\in I$, using the fact that
  $\ell_i(v)>0$, we have that $\ell_i(v) \lim_{t\to
      0^+}\ln\ell_i(x^{(0)}+tv)= -\infty$. Thus, we
  conclude that
  \begin{equation*}
    \frac{\dif f(x^{(0)}+tv)}{\dif t}{\Big|}_{t=0} > 0,
  \end{equation*}
which shows that $x^{(0)}$ is not a local maximum.
\end{proof}

\begin{proof}[Proof of Lemma~\ref{lem:boundary-prek}]
  We want to find the local maximums in $\preS_1$, which is the set of
  points in~$\preS_m$ such that any of the constraints in (D1)--(D4) is
  tight. Recall that the constraints (D1)--(D4) are the following:
  \begin{itemize}
  \item[(D1)] $\preQthree \geq 3\prenthree\geq 0$ and, if $\prenthree=0$, then $\preQthree=0$.
  \item[(D2)] $\prePtwo\geq 0$;
  \item[(D3)] $\prePthree\geq 0$;
  \item[(D4)] $\prekzero,\prekone,\prektwo\geq 0$ and $\preTtwo\geq 0$ and
    $\preTthree \geq 0$;
  \end{itemize}
  We split the analysis in the following cases:
  \begin{enumerate}
  \item[Case 1:] $\preQthree = \prenthree = 0$;
  \item[Case 2:] $\preQthree = 3\prenthree > 0$ and $\prePthree=0$;
  \item[Case 3:] $\preQthree = 3\prenthree > 0$ and $\prePtwo=0$;
  \item[Case 4:] $\preQthree = 3\prenthree > 0$ and $\prePthree\neq 0$ and $\prePtwo\neq 0$;
  \item[Case 5:] $\preQthree > 3\prenthree > 0$ and $\prePthree = 0$;
  \item[Case 6:] $\preQthree > 3\prenthree > 0$ and $\prePtwo = 0$.
\end{enumerate}
We will use the definitions in~\eqref{eq:pre-param-def} many times in
the analysis. \nickca{Deleted: For Maple spreadsheets with the
  computations below see 
Section~ for Case 1, Section~
for Case 2, Section~ for Case 3,
Section~ for Case 4, Section~
for Case 5, and Section~ for Case 6.} \nickka{Maple was used for several computations in the following.}

\noindent{\textbf{Case 1:}} Assume that $\preQthree = \prenthree =
0$. Recall that, by definition, we have that $\preQthree =
3\mpre-\prenn-2\prekzero-2\prekone-2\prektwo$, $\preTtwo =
2\prenn-2\prekzero-\prekone$, and $\preTthree =
3\mpre-3\prenn-\prekone-2\prektwo$. Thus,
\begin{equation}
  \label{eq:Qthree-Tthree-Ttwo-hyper}
  \preQthree = \preTtwo + \preTthree.
\end{equation}
Moreover, $\preTtwo \geq 0$ and $\preTthree \geq 0$ are constraints in
the definition of $\preS_m$. Thus, since $\preQthree =0$, we have that
$\preTtwo = \preTthree = 0$. Recall that $\prenthree =
1-\prenn-\prekzero-\prekone-\prektwo$. Hence, we obtain the following
equations:
\begin{align*}
  &1-\prenn-\prekzero-\prekone-\prektwo=0,\\
  &2\prenn-2\prekzero-\prekone=0,\\
  &3\mpre-3\prenn-\prekone-2\prektwo.
\end{align*}
By solving this system of equation, we obtain the following values for
$\prenn$, $\prekone$, and $\prektwo$ in terms of $\prekzero$ and
$\mpre$:
\begin{align*}
  &\prenn = 2-3\mpre;\\
  &\prekone = 4-6\mpre-2\prekzero;\\
  &\prektwo = -5+9\mpre+\prekzero.
\end{align*}
Moreover, $\prePthree = 3(m-\prenn) = -6+12\mpre$, $\prePtwo = 2(\prenn-\prekzero)=4-6\mpre-2\prekzero$ Thus,
$\fpre(x)$ depends only on $\prekzero$ and we get
\begin{equation*}
  \begin{split}
  \fpre(x)
  =
  f(\prekzero)
  &\eqdef
      \hpre(\prePthree)
      +\hpre(\prePtwo)
      +\hpre(\premtwo)
      -\hpre(\prekzero)
      -\hpre(\prekone)
      -\hpre(\prektwo)
      \\
      &\quad-\hpre(\premthree)
      -2\hpre(\premtwoprime)
      -\prektwo \ln 2
      - \premtwoprime \ln 2
      - \premthree \ln 6,
         \end{split}
   \end{equation*}
where $\prekzero \in[5-9\mpre,2-3\mpre]$. We have that
\begin{equation*}
  \exp\left(\frac{\dif f}{\dif \prekzero}\right)
  =
  \frac{(3\mpre-2+\prekzero)^2}{(-5+9\mpre+\prekzero)\prekzero}
  \quad\text{and}\quad
  \frac{\dif^2 f}{\dif^2 \prekzero}
  =\frac{3\mpre\prekzero-\prekzero+33\mpre-10-27\mpre^2}{\prekzero(-5+9\mpre+\prekzero)(3\mpre-2+\prekzero)}.
\end{equation*}
For $\prekzero \in[5-9\mpre,2-3\mpre]$, the denominator of the second
derivative is always nonnegative and its numerator is always negative
for sufficiently small $r$ (that is, sufficiently large $n$). Hence,
$f$ is strictly concave. Thus, there is a unique maximum and it
satisfies:
\begin{equation*}
  \frac{(3\mpre-2+\prekzero)^2}{(-5+9\mpre+\prekzero)\prekzero}
  =1,
\end{equation*}
that is,
\begin{equation*}
  \prekzero =
  \frac{(3\mpre-2)^2}{3\mpre-1}.
\end{equation*}
We then compute the series for $f(\prekzero)$ at this point with
$\lambdaopt$ going to zero (by using~\eqref{eq:pre-r-opt}):
\begin{equation*}
  f(\prekzero)=
  2r\ln n -4r\ln r
  +  \left(-\ln(2)-\frac{1}{3}\ln(3)+\frac{1}{3}\right)\lambdaopt
  +(\lambdaopt)^2\ln(\lambdaopt)
  +O((\lambdaopt)^2).
\end{equation*}

\noindent{\textbf{Case 2:}} Assume that $\preQthree = 3\prenthree>0$
and $\prePthree=0$. Since $\prePthree=0$ and
$\prePthree=3(\mpre-\prenn)$ by definition
(see~\eqref{eq:pre-param-def}), we have that $\prenn =
\mpre$. Moreover, since $\preTthree =\prePthree-\prekone-2\prektwo$
and $\preTthree,\prekone,\prekone\geq 0$ are constraints in the
definition of $\preS_1$, we have that $\prekone=0$ and $\prektwo=0$.
Using $\preQthree=3\prenthree$ and their definitions
in~\eqref{eq:pre-param-def}, we have that
$3\mpre-\prenn-2\prekzero-2\prekone-2\prektwo =
3(1-\prenn-\prekzero-\prekone-\prektwo)$ and so $\prekzero = 3 -
3\mpre-2\prenn = 3-5\mpre$. Thus, we only have to compute the value of
$\fpre$ in the point $(\mpre,3-5\mpre,0,0)$. By computing the series
of $\fpre$ in this point with $\lambdaopt$ going to zero (by
using~\eqref{eq:pre-r-opt}), we get
\begin{equation*}
  2r\ln(n)-4r\ln r
  +
  \left(\frac{1}{3}-\ln(2)-\frac{1}{3}\ln(3)\right)\lambdaopt
  +O((\lambdaopt)^2).
\end{equation*}

\noindent{\textbf{Case 3:}} Assume that $\preQthree = 3\prenthree>0$
and $\prePtwo=0$. Since $\prePtwo=0$ and
$\prePtwo=2(\prenn-\prekzero)$ by definition
(see~\eqref{eq:pre-param-def}), we have that $\prekzero =
\prenn$. Moreover, since $\preTtwo =\prePtwo-\prekone$ and
$\preTtwo,\prekone\geq 0$ are constraints in the definition of
$\preS_m$, we have that $\prekone=0$.  Using $\preQthree=3\prenthree$
and their definition in~\eqref{eq:pre-param-def}, we have that
$3\mpre-\prenn-2\prekzero-2\prekone-2\prektwo =
3(1-\prenn-\prekzero-\prekone-\prektwo)$ and so $\prektwo
=3-3\mpre-3\prenn$. So let $f(\prenn) \eqdef
\fpre(\prenn,\prenn,0,3-3\mpre-3\prenn)$ and $\prenn
\in[2-3\mpre,1-\mpre]$. We have that
\begin{equation*}
  \exp\left(\frac{\dif f}{\dif \prenn}\right)
  =
  \frac{8(1-\prenn-\mpre)^3}
  {(\mpre-\prenn)^2 (-2\prenn+3\mpre)}
  \quad\text{and}\quad
  \frac{\dif^2 f}{\dif^2 \prenn}
  =\frac{(2\mpre-1)(4-7\mpre-\prenn)}
  {(-2+\prenn+3\mpre)(1-\prenn-m)(\mpre-\prenn)}
\end{equation*}
For $\prenn \in[2-3\mpre,1-\mpre]$, the denominator of the second
derivative is always nonnegative and its numerator is always negative
for sufficiently small $r$. Hence, $f$ is strictly concave. Thus,
there is unique maximum satisfying
\begin{equation*}
  8(1-\prenn-\mpre)^3-(\mpre-\prenn)^2 (-2\prenn+3\mpre))=0,
\end{equation*}
which has a unique real solution at $1/2+ \alpha r$, where
$\alpha\approx -2.03566$, which is the real solution for
\begin{equation*}
  9\alpha^3+25\alpha^2+19\alpha+11=0.
\end{equation*}
We then compute the value of the function $f$ at $1/2+\alpha r$:
\begin{equation*}
  2r\ln(n) + 2r\ln r+ \beta,
\end{equation*}
with $\beta\approx 1.9389$.

\noindent\textbf{Case 4:} Now suppose that $\preQthree = 3\prenthree > 0$ and
$\prePthree >0$ and $\prePtwo > 0$. By Lemma~\ref{lem:aux-boundary}, we
do not need to consider the cases $\prekzero=0$, $\prekone=0$, $\prektwo=0$,
$\preTthree=0$, $\preTtwo=0$ and $\premthree=0$.

Since $\preQthree = 3\prenthree$, we have that $\prekzero =
3-3\mpre-2\prenn-\prekone-\prektwo$. Thus we analyse the function
\begin{equation*}
  f(\prenn,\prekone,\prektwo)
  \eqdef
  \fpre(\prenn,3-3\mpre-2\prenn-\prekone-\prektwo,\prekone,\prektwo).
\end{equation*}
We have 
that, for any local maximum in this case,
\begin{align*}
  &\exp\left(\frac{\dif f}{\dif \prenn}\right)
  = \frac{8\prePthree^3 \prekzero^2 \prenthree^2 \prenn}{\premthree^2
    \preTtwo^6}
  =1;\\
  &\exp\left(\frac{\dif f}{\dif \prekone}\right)
  =  \frac{2\prePthree \prekzero}{\preTtwo \prekone}
  =1;\\
  &\exp\left(\frac{\dif f}{\dif \prektwo}\right)
  =  \frac{\prePthree^2 \prekzero}{\preTtwo^2 \prektwo}
  =1;
\end{align*}
and so
\begin{align}
  &\label{eq:fn1-case4} 8\prePthree^3 \prekzero^2 \prenthree^2 \prenn-\premthree^2 \preTtwo^6=0;\\
  &\label{eq:fk1-case4} 2\prePthree \prekzero-\preTtwo \prekone = 0;\\
  &\label{eq:fk2-case4} \prePthree^2 \prekzero-\preTtwo^2 \prektwo = 0.
\end{align}
By taking the resultant of the RHS of~\eqref{eq:fk1-case4}
and~\eqref{eq:fk2-case4}, by eliminating $\prekone$, we get
\begin{multline*}
  9\prenthree^2
  (\prektwo-3+3\mpre+4\prenn)
  \\(9\mpre^2\prektwo-6\mpre\prenn\prektwo+n1^2\prektwo+27\mpre^3-27\mpre^2
  -36\mpre^2\prenn-9\mpre\prenn^2+54\mpre\prenn-27\prenn^2+18\prenn^3)
  =0.
\end{multline*}
Using $\preQthree=3\prenthree$ and their definition
in~\eqref{eq:pre-param-def}, we have that
$3\mpre-\prenn-2\prekzero-2\prekone-2\prektwo =
3(1-\prenn-\prekzero-\prekone-\prektwo)$ and so $3\mpre -3 =
\prekzero+\prekone+\prektwo-2\prenn$. Thus, $\prektwo-3+3\mpre+4\prenn
= \prekzero+\prekone+2\prektwo+2\prenn>0$ since we already excluded
the case $\prekzero =0$. Recall that in this case we have
$\prenthree>0$.  Thus, for any local maximum in this case,
\begin{equation}
  \label{eq:eq1-case4} 
  9\mpre^2\prektwo-6\mpre\prenn\prektwo+n1^2\prektwo+27\mpre^3-27\mpre^2
  -36\mpre^2\prenn-9\mpre\prenn^2+54\mpre\prenn-27\prenn^2+18\prenn^3
  =0.
\end{equation}
This implies that $\prektwo$ can be determined in terms of $\prenn$:
\begin{equation*}
  \prektwo
  =
\frac{9(-3\mpre+3-2\prenn)(\mpre-\prenn)^2}{(3\mpre-\prenn)^2}.
\end{equation*}
By taking the resultant of the RHS of~\eqref{eq:fk1-case4}
and~\eqref{eq:fk2-case4}, by eliminating $\prektwo$, we get
\begin{equation*}
  44
  \prenthree^2
  (\prekone-2\prenn)
  (9\mpre^2\prekone-6\mpre\prenn\prekone+n1^2\prekone
  +36\mpre^2\prenn-36\mpre\prenn+36\prenn^2-24\prenn^3-12\mpre\prenn^2)
  =0.
\end{equation*}
In this case $\prenthree>0$. Moreover, $2\prenn-\prekone = 0$ implies,
by the definitions in~\eqref{eq:pre-param-def}, that $\preTtwo =
\prePtwo-\prekone = 2\prenn-2\prekzero-\prekone\leq 0$ since $\prekzero
\geq 0$ in $\preS_m$. But we have already excluded the case $\preTtwo =
0$. Thus,
\begin{equation}
  \label{eq:eq2-case4}
  9\mpre^2\prekone-6\mpre\prenn\prekone+n1^2\prekone
  +36\mpre^2\prenn-36\mpre\prenn+36\prenn^2-24\prenn^3-12\mpre\prenn^2
  =0.
\end{equation}
This implies that $\prekone$ can be determined in terms of $\prenn$:
\begin{equation*}
  \prekone = \frac{12\prenn (-3\mpre+3-2\prenn) (\mpre-\prenn)}{(3\mpre-\prenn)^2}.
\end{equation*}
We take the resultant of the RHS of~\eqref{eq:fn1-case4}
and~\eqref{eq:eq2-case4} by eliminating $\prekone$ and then the resultant
of the polynomial obtained with the RHS of~\eqref{eq:eq1-case4} by
eliminating $\prektwo$ and ignoring the factors that cannot be zero in $\preS_m$ and we obtain:
\begin{equation*}
  18\mpre-36\mpre^2+18\mpre^3-18\prenn+18\mpre\prenn-3\mpre^2\prenn+22\prenn^2-16\mpre\prenn^2-7\prenn^3
  =0.
\end{equation*}
This cubic equation has one real solution for $\prenn$ and two complex
solutions because the discriminant $\Delta$ of the polynomial above is
$-63/4+O(r)$, which is negative for sufficiently large $n$. For we
have that the real solution is $1/2-r-6r^2-O(r^3)$ and so the value of the
function $\fpre$ at this point is, by using~\eqref{eq:pre-r-opt},
\begin{equation*}
2r\ln n-4r\ln r+
\left(\frac{1}{3}-\frac{1}{3}\ln(3)-\frac{2}{3}\ln(2)\right)\lambdaopt
+ 
\left(\frac{11}{72}-\frac{2}{9}\ln(2)-\frac{1}{9}\ln(3)\right)(\lambdaopt)^2
+
O((\lambdaopt)^3).  
\end{equation*}

\noindent\textbf{Case 5:} Now suppose that $\prePthree = 0$ and $\preQthree >
\prenthree >0$. Since $\prePthree = 0$ and $\prePthree =
3(\mpre-\prenn)$ by definition (see~\eqref{eq:pre-param-def}), we have
that $\prenn=\mpre$. Moreover, since $\preTthree,\prekone,\prektwo\geq
0$ in $\preS_m$ and $\preTthree = \prePthree - \prekone-2\prektwo$ by
definition, we have that $\prekone = 0$ and $\prektwo = 0$. Thus, for
any local maximum with $\prePthree =0$, it suffices to analyse
\begin{equation*}
  f(\prenthree) \eqdef
  \fpre(\mpre,1-\mpre-\prenthree,0,0),
\end{equation*}
where $\prenthree \in(0,1-\mpre)$, since by definition
$\prekzero=1-\prenn-\prekone-\prektwo-\prenthree =
1-\mpre-\prenthree\geq 0$
and $\preQthree = 3\mpre-\prenn-2\prekzero-2\prekone-2\prektwo =
4\mpre-2+2\prenthree\geq 0$. We do not have to analyse the value at the
endpoints of the interval for $\prenthree$ as they were already
considered in cases before. Also, in this case $\preQthree=\prePtwo$,
thus we do not have to check the case $\prePtwo=0$. Thus, it suffices
to consider points satisfying
\begin{equation*}
  \exp\left(\frac{\dif f}{\dif \prenthree}\right)
  =
  \frac{2(1-\mpre-\prenthree)f_3(\lambda)}{\prenthree \lambda^2}=1,
\end{equation*}
where $\lambda f_2(\lambda)/f_3(\lambda) = \preQthree/\prenthree$. The equation
below is equivalent to
\begin{equation*}
  \prenthree = \frac{(1-m)f_3(\lambda)}{f_2(\lambda)}.
\end{equation*}
Combining this with the equation defining $\lambda$ implies:
\begin{equation*}
r = \frac{1}{2} 
\frac{-2 e^\lambda+2+\lambda+\lambda e^{\lambda}}{2 e^{\lambda}-2-3\lambda+\lambda
  e^{\lambda}},
\end{equation*}
and since $r$ goes to zero so does $\lambda$. We have that
\begin{equation*}
  r = \frac{1}{24} \lambda +O(\lambda^2),
\end{equation*}
which implies
\begin{equation*}
  \lambda = 2\lambdaopt+O(\lambdaopt)^2.
\end{equation*}
We then compute the series of $f(\prenthree)$ with $\lambda$ going to zero:
\begin{align*}
&2r\ln n-4r\ln r+  \left(-\frac{1}{2}\ln(2)-\frac{1}{6}\ln(3)+\frac{1}{6}\right)\lambda + O(\lambda^2)\\
  =
&2r\ln n-4r\ln r+
  \left(-\ln(2)-\frac{1}{3}\ln(3)+\frac{1}{3}\right)\lambdaopt + O((\lambdaopt)^2).
\end{align*}
\noindent\textbf{Case 6:}  Now suppose that $\prePtwo = 0$ and $\preQthree >
\prenthree >0$. We have that $\prePtwo = 2(\prenn-\prekzero)$. Thus,
we have $\prenn=\prekzero$ since $\prePtwo=0$. Moreover, since
$\preTtwo,\prekone\geq 0$ in $\preS_m$ and $\preTtwo =
\prePtwo-\prekone$ by definition, we have that $\prekone = 0$. Thus,
we only need to analyse
\begin{equation*}
  f(\prenn,\prektwo) \eqdef
  \fpre(\prenn,\prenn,0,\prektwo),
\end{equation*}
where $\preQthree > 3\prenthree>0$ and $\prePthree > 0$.
Thus, it suffices to
consider points satisfying
\begin{equation*}
  \exp\left(\frac{\dif f}{\dif \prenn}\right)
  =
  \frac{2}{9}
  \frac{\prenthree^2\lambda^3}
    {\premthree^2f_3(\lambda)^2}
    =1
  \quad\text{and}\quad
  \exp\left(\frac{\dif f}{\dif \prektwo}\right)
  =
  \frac{1}{2}
  \frac{\prenthree \lambda^2}{\prektwo f_3(\lambda)}=1,
\end{equation*} 
where $\lambda f_2(\lambda)/f_3(\lambda) = \preQthree/\prenthree$. The second
equation implies that for any local maximum
\begin{equation*}
  \prektwo = \frac{1}{2}\frac{(1-2\prenn)\lambda^2}{f_2(\lambda)}.
\end{equation*}
By using this with the derivative w.r.t.~$\prenn$,
we get
\begin{equation*}
\prenn =
\frac{\lambda^{3/2}\sqrt{2}-3 f_2(\lambda)\mpre}
  {2\lambda^{3/2}\sqrt{2}-3f_2(\lambda)}.
\end{equation*}
By putting this together with the equation defining $\lambda$, we have that
\begin{equation*}
\frac{(-e^{\lambda}+1+\sqrt{2\lambda})\lambda}
{f_3(\lambda)}=0,
\end{equation*}
which has a unique solution $\ell^* \approx 0.8267$. For $\lambda
=\ell^*$, we have $\prenn =
\frac{1}{2}+\alpha r$, with $\alpha\approx 1.4887$ and $\prektwo =
\beta(1/2-\prenn)$ with $\beta \approx 0.1173$. By using this values
of $\prenn$ and $\prektwo$, we evaluate the function
$f(\prenn,\prektwo)$ as 
\begin{equation*}
2r\ln n - 4r\ln r+  6\ln r + O(r),
\end{equation*}
since $\alpha < 0$, $0< \beta < 2$ and $\lambda > 0$.
\end{proof}

\begin{proof}[Proof of Lemma~\ref{lem:open-prek}]
  Let $\prex(i) = (\prenn(i),\prekzero(i),\prekone(i),\prektwo(i))$ and similarly
  for $\preQthree(i), \prenthree(i)$, etc. Let $\lambda(i)$ be such that
  $\lambda(i)f_2(\lambda(i))/f_1(\lambda(i))= \preQthree(i)/\prenthree(i)$.
  Recall that
\begin{equation*}
  \begin{split}
    \fpre(\prenn,\prekzero,\prekone,\prektwo)
    =
    &\hpre(\prePthree)
    +\hpre(\prePtwo)
    +\hpre(\preQthree)
    +\hpre(\premtwo)
    \\
    &-\hpre(\prekzero)
    -\hpre(\prekone)
    -\hpre(\prektwo)
    -\hpre(\prenthree)
    -\hpre(\premthree)
    \\
    &-\hpre(\preTthree)
    -\hpre(\preTtwo)
    -2\hpre(\premtwoprime)
    \\
    &
    -\prektwo \ln 2
    - \premtwoprime \ln 2
    - \premthree \ln 6
    \\
    &+\prenthree\ln f_3(\lambda)
    -\preQthree\ln\lambda.
  \end{split}
\end{equation*}
Since $S\subseteq[0,1]^4$ and the fact that $|y\ln y|\leq 1/e$ for
$y\in [0,1]$, we have that $\fpre(x)\leq C +\prenthree\ln f_3(\lambda)
-\preQthree\ln\lambda$ for some constant $C$. Thus, it suffices to show
that $\prenthree(i)\ln f_3(\lambda(i)) -\preQthree(i)\ln\lambda(i) \to -\infty$
as $i\to\infty$.

Since $\preQthree(i)$ converges to a positive number and $\prenthree(i)$
converges to $0$, we have that $\preQthree(i)/\prenthree(i)\to\infty$. This
implies that $\lambda(i)\to\infty$. Thus,
\begin{equation*}
  \begin{split}
    &\prenthree(i)\ln f_3(\lambda(i)) -\preQthree(i)\ln\lambda(i) \\
    &\leq
    \prenthree(i)\lambda(i) -\preQthree(i)\ln\lambda(i),\quad \text{since
    }f_3(\lambda)\leq \exp(\lambda)\\
    &\leq
    \prenthree(i)\frac{\preQthree(i)}{\prenthree(i)} -\preQthree(i)\ln\lambda(i),\quad \text{since
    }\lambda(i)\leq \preQthree(i)/\prenthree(i)
    \\
    &= \preQthree(i)(1-\ln(\lambda(i))) \to-\infty,\quad \text{since
    }\lambda(i)\to\infty \text{ and }\lim\inf_{i\to \infty} \preQthree(i)>0.
  \end{split}
\end{equation*}

\end{proof}

\subsection{Approximation around the maximum and bounding the tail}
\label{sec:approx-tail-pre-hyper}

In this section, we approximate the sum of
$\exp(n\fpre(x))$ over a set of points `close' to $\xopt$ and  bound
the sum for the points `far' from $\xopt$. More specifically, we prove
Lemmas~\ref{lem:approx-pre-hyper} and~\ref{lem:tail-pre-hyper}.

\begin{proof}[Proof of Lemma~\ref{lem:approx-pre-hyper}]
  We use Lemma~\ref{lem:third-pre-hyper} and
  Lemma~\ref{lem:max-interior-prek}, which were proved in
  Section~\ref{sec:partial-prek-hyper} and
  Section~\ref{sec:max-pre-hyper}, resp. Let $x\in B$. By
  Lemma~\ref{lem:third-pre-hyper}, since $\delta_1^3=o(r/n)$ and
  $\delta^3 = o(r^4/n)$, we have that
  \begin{equation*}
    n\frac{\partial \fpre(\prexopt+\prex)}
          {\partial t_1 \partial t_2 \partial t_3} 
    t_1(\prex) t_2(\prex) t_3(\prex)
    =o(1),
  \end{equation*}
  for any $t_1,t_2,t_3
  \in\set{\prenn,\prekzero,\prekone,\prektwo}$. By
  Lemma~\ref{lem:max-pre-hyper}, we have that
  \begin{equation*}
    \frac{\partial \fpre(\prexopt)}
    {\partial t} 
    =0,
  \end{equation*}
  for any $t\in\set{\prenn,\prekzero,\prekone,\prektwo}$. Thus, by
  Taylor's approximation \crisc{Removed reference to Taylor in
    preliminaries},
  \begin{equation}
    \label{eq:taylor-aux-pre-hyper}
    \exp\left( n \fpre(\prexopt+\prex)\right)
    =
    \exp\left( n \fpre(\prexopt)  
      +\frac{n\prex^T H \prex}{2}+o(1)
    \right),
  \end{equation}
  where $H$ is the Hessian of $\fpre$ at $\xopt$. Using the fact that
  $\preB\cap ((\setZ^4-\xopt)/n)$ is a finite set for
  each~$n$\crisc{removed a reference to a uniformity lemma in the appendix}, this implies that
  \begin{equation}
    \label{eq:taylor-pre-hyper}
    \sum_{\substack{\prex\in \preB\\ \prex\in (\setZ^4-\xopt)/n}}
    \exp\left( n \fpre(\prexopt+\prex)\right)
    \sim
    \sum_{\substack{\prex\in \preB\\ \prex\in (\setZ^4-\xopt)/n}}
    \exp\left( n \fpre(\prexopt)  
    +\frac{n\prex^T H \prex}{2}
  \right).
  \end{equation}
  So we need to show that
  \begin{equation}
    \label{eq:taylor-integral-aux-1-prek-hyper}
    \sum_{\substack{\prex\in \preB\\ \prex\in (\setZ^4-\xopt)/n}}
    \exp\left(
      \frac{n\prex^T H \prex}{2}
    \right)
    \sim
    144\sqrt{3}\pi^2 r^{7/2}n^2.
  \end{equation}
  Let 
  \begin{equation*}
    A =
    \left(\begin{array}{cccc}
    1&0&0&0\\
    1&1&0&0\\
    -3&0&1&0\\
    0&0&0&1
  \end{array}
\right).
  \end{equation*}
  We rewrite the summation in the LHS
  of~\eqref{eq:taylor-integral-aux-1-prek-hyper} over $\preC \eqdef
  \set{y: Ay \in\preB}$ as
  \begin{multline}
    \label{eq:taylor-integral-aux-2-prek-hyper}
          \sum_{\substack{y\in \preC\\ y\in (\setZ^4-A^{-1}\xopt)/n}}
      \exp\left(
        \frac{ny^T (A^T HA) y}{2}
      \right)
      \\=
    {\sum_{\substack{y\in \preC\\ y\in (\setZ^4-A^{-1}\xopt)/n}}}
    \exp\left(
      \frac{-ny^T (A^T H_0A) y }{2r^2}
      -\frac{ny^T (A^T TA) y }{2r}
    + \frac{O(ny^T J y)}{2}
    \right),
  \end{multline}
  by~Lemma~\ref{lem:hessian-hyper} (for the definitions of $H_0$ and
  $T$, see~\eqref{eq:def-H0-T-pre-hyper}).  Note
  that the condition ``$\prex\in (\setZ^4-\xopt)/n$'' became ``$y\in
  (\setZ^4-A^{-1}\xopt)/n$'' because $A$ is an integer invertible
  matrix and
  \begin{equation*}
    A^{-1}
    =
    \left(\begin{array}{cccc}
    1&0&0&0\\
    -1&1&0&0\\
    3&0&1&0\\
    0&0&0&1
  \end{array}
\right)
  \end{equation*}
  is also an integer matrix. Using the definition of $H_0$ and
  $T$ in~\eqref{eq:def-H0-T-pre-hyper}, we have that
  \begin{equation}
    \label{eq:taylor-integral-aux-3-prek-hyper}
    \begin{split}
      &\sum_{\substack{y\in \preC\\ y\in (\setZ^4-A^{-1}\xopt)/n}}
      \exp\left(
        \frac{-ny (A^T H_0A) y }{2r^2}
        -\frac{ny (A^T TA) y }{2r}
        + \frac{nO(y^T J y)}{2}
      \right)=\\
      &=
      \sum_{\substack{y\in \preC\\ y\in (\setZ^4-A^{-1}\xopt)/n}}
      \exp{\Bigg(}
      -\frac{n}{12r^2}y_2^2
      -\frac{n}{6r^2}y_2 y_3
      -\frac{n}{6r^2}y_2 y_4
      -\frac{7n}{72r^2} y_3^2
      -\frac{4n}{9r^2}y_3 y_4
      - \frac{n}{6r^2}y_4^2
      \\
      &\qquad \qquad \qquad \qquad\qquad
      -\frac{n}{2r} y_1^2
      +\frac{n}{r} y_1 y_2
      +\frac{n}{r} y_1 y_3
      \\&\qquad\qquad\qquad\qquad\qquad
      -\frac{11n}{30r} y_2^2
      -\frac{2n}{5r} y_2 y_3
      -\frac{n}{15r} y_2 y_4
      -\frac{31n}{180r} y_3^2
      +\frac{2n}{45r} y_3 y_4
      +\frac{n}{45r} y_4^2
      + \frac{nO(y^T J y)}{2}
    \Bigg{)},
    \end{split}
  \end{equation}
    The set $\preC = \set{y:
    Ay \in\preB}$ can be described as
  \begin{equation*}
    \preC
    =
    {\big\{}
    y\in\setR^4:
    \abs{y_1}\leq \delta_1, \
    \abs{y_i}\leq \delta\text{ for }i=2,3,4
    {\big\}},
  \end{equation*}
  since $\preB$ was defined as
  \begin{equation*}
    \preB = \set[\Big]{
      \prex\in\setR^4 :
      \prex = \gamma_1 z_1 + \gamma_2 e_2 + \gamma_3 e_3 + \gamma_4 e_4,\
      |\gamma_1|\leq \delta_1 
      \text{ and }
      |\gamma_i|\leq \delta   \text{ for }i=2,3,4
    }.
  \end{equation*}
  Thus, the ranges of the summation for different variables $y_i$'s
  are independent. We have that
  \begin{equation*}
    \begin{split}
      &\sum_{\substack{|y_1|\leq \delta_1\\
                y_1\in (\setZ-(A^{-1}\xopt)_1)/n}}
      \exp{\Bigg(}
      -\frac{n}{2r} y_1^2
      +\frac{n}{r} y_1 y_2
      +\frac{n}{r} y_1 y_3
      +\sum_{j=1}^4 O\paren{n y_1  y_j}
      \Bigg{)}
    \\
    &=
    \sum_{\substack{|\tilde y_1|\leq
        \delta_1 \sqrt{n/r}\\
                \tilde y_1\in (\setZ-(A^{-1}\xopt)_1)/\sqrt{rn}}}
      \exp{\Bigg(}
      - y_1^2/2
      + \tilde y_1 \tilde y_2
      + \tilde y_1 \tilde y_3
      +\sum_{j=1}^4 O\paren{r\tilde y_1 \tilde y_j}
    \Bigg{)},
    \end{split}
  \end{equation*}
  where $\tilde y_i = \sqrt{n} y_i/\sqrt{r}$ for $i=2,3$. We
  apply Lemma~\ref{lem:integral-tools} with $\alpha = 1/2$, $\beta =
  \tilde y_2+\tilde y_3$, $\phi = O(r) = o(1)$, $\psi = O(r\tilde y_2+
  r\tilde y_3 + r\tilde y_4) = O(r) =o(1)$, $s_n = \sqrt{rn}
  \to\infty$ and $T_n = \delta_1 \sqrt{n/r}\to \infty$:
  \begin{equation*}
    \begin{split}
      \sum_{\substack{|\tilde y_1|\leq
        \delta_1 \sqrt{n/r}\\
                y_1\in (\setZ-(A^{-1}\xopt)_1)/(r\sqrt{n})}}
      \exp{\Bigg(}
      - y_1^2/2
      + \tilde y_1 \tilde y_2
      + \tilde y_1 \tilde y_3
      +\sum_{j=1}^4 O\paren{r\tilde y_1 \tilde y_j}
    \Bigg{)}
    &\sim
      \sqrt{2rn\pi}
      \exp((\tilde y_2+\tilde y_3)^2/2).
    \end{split}
  \end{equation*}
  We then proceed similarly for $y_2$, $y_3$ and $y_4$. Fix $y_3$ and
  $y_4$. Set $\check y_i = \sqrt{n}y_i/r$ for $i=3,4$.  We apply
  Lemma~\ref{lem:integral-tools} with $\alpha = 1/12$, $\beta =
  -(1/6) \check y_3 - (1/6) \check y_4 $, $\phi =
  -(2r/15) +O(r^2) = o(1)$, $\psi = (3r/5) \check y_3
  -r/15 +\sum_{j=3}^4 O\paren{r^2\check y_j} =o(1)$, $s_n =
  r\sqrt{n} \to\infty$ and $T_n = \delta\sqrt{n}/r\to \infty$:
  \begin{equation*}
    \begin{split}
      &\sum_{\substack{|y_2|\leq \delta\\
          y_2\in (\setZ-(A^{-1}\xopt)_2)/n}} \hspace{-14pt}\exp{\Bigg(}
      -\frac{n}{12r^2}y_2^2 -\frac{n}{6r^2}y_2 y_3 -\frac{n}{6r^2}y_2
      y_4 
      -\frac{2n}{15r} y_2^2 +\frac{3n}{5r} y_2 y_3 -\frac{n}{15r}
      y_2 y_4 +\sum_{j=2}^4 O\paren{n y_2 y_j}
      \Bigg{)}\\
      &= \sum_{\substack{|\check y_2|\leq \delta\sqrt{n}/r \\ \check
          y_2\in (\setZ-(A^{-1}\xopt)_2)/(r\sqrt{n})}} \exp{\Bigg(}
      -\frac{1}{12}\check y_2^2 -\frac{1}{6}\check y_2 \check y_3
      -\frac{1}{6} \check y_2 \check y_4 
      -\frac{2r}{15} \check y_2^2
      +\frac{3r}{5} \check y_2 \check y_3 -\frac{r}{15} \check y_2
      +\sum_{j=2}^4 O\paren{r^2\check y_2 \check y_j} \Bigg{)}
      \\
      &\sim  2\sqrt{3\pi}r\sqrt{n} \exp((\check y_3+\check
      y_4)^2/12).
    \end{split}
  \end{equation*}
  Fix $y_4$ and set $\check y_4 = \sqrt{n}y_4/r$.  We apply
  Lemma~\ref{lem:integral-tools} with $\alpha = 1/72$, $\beta = -(1/8)\tilde
  y_4$, $\phi = O(r) = o(1)$, $\psi = O(r\tilde y_4) = O(r) =o(1)$, $s_n = r\sqrt{n}
  \to\infty$ and $T_n = \delta\sqrt{n}/r\to \infty$:
  \begin{equation*}
    \begin{split}
      &\sum_{\substack{|y_3|\leq \delta\\
          y_3\in (\setZ-(A^{-1}\xopt)_3)/n}}
      \exp{\Bigg(}
      -\frac{n}{72r^2}y_3^2
      -\frac{n}{18r^2}y_3y_4
      +\sum_{j=3}^4 O(ny_3y_j/r)
      \Bigg{)}\\
      =
      &\sum_{\substack{|\check y_3|\leq \delta\sqrt{n}/r\\
          \check y_3\in (\setZ-(A^{-1}\xopt)_3)/(r\sqrt{n})}}
      \exp{\Bigg(}
      -\frac{1}{72}\check y_3^2
      -\frac{1}{18}\check y_3 \check y_4
      +\sum_{j=3}^4 O(r\check y_3\check y_j)
      \Bigg{)}\\
      &\sim
      6\sqrt{2\pi}r\sqrt{n}
      \exp(\check y_4^2/18)).
    \end{split}
  \end{equation*}
  Finally, for $y_4$, we apply Lemma~\ref{lem:integral-tools} with
  $\alpha = 1/36$, $\beta = 0$, $\phi = O(r) = o(1)$, $\psi = 0$, $s_n
  = r\sqrt{n} \to\infty$ and $T_n = \delta\sqrt{n}/r\to \infty$:
  \begin{equation*}
    \begin{split}
      \sum_{\substack{|y_4|\leq \delta\\
          y_4\in (\setZ-(A^{-1}\xopt)_4)/n}} \exp{\Bigg(}
      -\frac{n}{36r^2} y_4^2 +O(ny_4y_4/r) \Bigg{)} &=
      \sum_{\substack{|\check y_4|\leq \delta\sqrt{n}/r\\
          \check y_4\in (\setZ-(A^{-1}\xopt)_4)/(r\sqrt{n})}}
      \exp{\Bigg(} -\frac{1}{36} \check y_4^2 +O(r \check y_4 \check
      y_4)
      \Bigg{)}\\
      &\sim 6r\sqrt{\pi n}.
    \end{split}
  \end{equation*}
  Hence,
  \begin{equation*}
    \sum_{\substack{\prex\in \preB\\ \prex\in (\setZ^4-\xopt)/n}}
    \exp\left(
      \frac{n\prex^T H \prex}{2}
    \right)
    \sim
    \sqrt{2rn\pi}
    \cdot
    2\sqrt{3\pi}r\sqrt{n}
    \cdot
    6\sqrt{2\pi}r\sqrt{n}
    \cdot
    6r\sqrt{\pi n}
    =144\sqrt{3}\pi^2n^2r^{7/2},
  \end{equation*}
  completing the proof.
\end{proof}

\begin{proof}[Proof of Lemma~\ref{lem:tail-pre-hyper}]
  Recall that $\hpre(y) = y\ln(yn)-y$,
\begin{equation*}
  \wpre(x)
  =
  \begin{cases}
    {\displaystyle\frac{\Pthree!\Ptwo!\Qthree!(\mtwo-1)!}
    {\kzero!\kone!\ktwo!
      \nthree!\mthree!
      \Tthree!\Ttwo!
      (\mtwoprime-1)!\mtwoprime!
      2^{\ktwo} 2^{\mtwoprime} 6^{\mthree}}
      \frac{\fff(\lambda)^{\nthree}}{\lambda^{\Qthree}}},&\text{if
      }\Qthree > 3\nthree;\\
      {\displaystyle\frac{\Pthree!\Ptwo!\Qthree!(\mtwo-1)!}
      {\kzero!\kone!\ktwo!
      \nthree!\mthree!
      \Tthree!\Ttwo!
      (\mtwoprime-1)!\mtwoprime!
      2^{\ktwo} 2^{\mtwoprime} 6^{\mthree}}
      \frac{1}{6^{\nthree}}},&\text{otherwise.}
  \end{cases}
\end{equation*}
and, for $x\in S$ such that $\Qthree > 3\nthree$
\begin{equation*}
  \begin{split}
    \fpre(\prex) =
    &\hpre(\prePthree)
    +\hpre(\prePtwo)
    +\hpre(\preQthree)
    +\hpre(\mtwo)
    \\
    &-\hpre(\prekzero)
    -\hpre(\prekone)
    -\hpre(\prektwo)
    -\hpre(\prenthree)
    -\hpre(\premthree)
    \\
    &-\hpre(\prePthree-\prekone-2\prektwo)
    -\hpre(\prePtwo-\prekone)
    -2\hpre(\premtwoprime)
    \\
    &
    -\prektwo \ln 2
    - \premtwoprime \ln 2
    - \premthree \ln 6
    \\
    &+\prenthree\ln f_3(\lambda)
    -\preQthree\ln\lambda,
  \end{split}
\end{equation*}
and, if $\Qthree = 3\nthree$,
\begin{equation*}
  \begin{split}
    \fpre(\prex) = &\hpre(\prePthree)
    +\hpre(\prePtwo)
    +\hpre(\preQthree)
    +\hpre(\mtwo)
    \\
    &-\hpre(\prekzero)
    -\hpre(\prekone)
    -\hpre(\prektwo)
    -\hpre(\prenthree)
    -\hpre(\premthree)
    \\
    &-\hpre(\prePthree-\prekone-2\prektwo)
    -\hpre(\prePtwo-\prekone)
    -2\hpre(\premtwoprime)
    \\
    &
    -\prektwo \ln 2
    - \premtwoprime \ln 2
    - \premthree \ln 6
    \\
    & -\prenthree\ln 6.
  \end{split}
\end{equation*}
Thus, by \Old{Lemma~\ref{lem:stirling-cte-prelim} (which states that
Stirling's approximation is correct up to a constant
factor)}\New{Stirling's approximation}, there
is a polynomial $Q(n)$ such that for $\prex\in \preS_m$
\begin{equation*}
    \wpre(x)
    \leq
    Q(n) n!\exp(n\fpre(\prex)).
  \end{equation*}
  Hence, if we obtain an upper bound for the tail $\sum_{x\in
    (S\setminus(\xopt+B))\cap \setZ^4}n!  \exp(n\fpre(\prex))$, we
  also get an upper bound for the tail $\sum_{x\in
    (S\setminus(\xopt+B))\cap \setZ^4} \wpre(x)$ although it is a
  weaker bound because of the polynomial factor $Q(n)$.

  Let $x\in (S\setminus(\xopt+B))\cap \setZ^4$.  Let
  $\gamma_1,\gamma_2,\gamma_3,\gamma_4$ be such that $x = \xopt+\gamma_1 z_1
  + \gamma_2 e_2 + \gamma_3 e_3 + \gamma_4 e_4$. Let
  $\delta_1'=\omega(\delta_1)$ be such that $\delta_1'/\delta_1$ goes
  to infinity arbitrarily slowly and let $\delta'$ be such that
  $\delta'/\delta$ goes to infinity arbitrarily slowly. If
  $\delta_1\leq |\gamma_1|\leq \delta_1'$ and $\delta\leq |\gamma_i|\leq \delta'$ for
  $i=2,3,4$,   by~\eqref{eq:taylor-aux-pre-hyper},
  \begin{equation*}
    \begin{split}
      \frac{\exp\left( n \fpre(\prex)\right)}{\exp\left( n \fpre(\prexopt)\right)}
      \sim
      \exp\left(\frac{n(\prex-\prexopt)^T H (\prex-\prexopt)}{2}
      \right),
    \end{split}
  \end{equation*}
  where $H$ is the Hessian of $\fpre$ at $\xopt$. Recall that, by
  Lemma~\ref{lem:hessian-hyper}, $H = (-1/r^2) H_0 - (1/r) T + J$,
  where $H_0$ and $T$ are defined in~\eqref{eq:def-H0-T-pre-hyper} and
  $J= J(n)$ is a matrix with bounded entries. Thus, there exists
  a  positive constant $\alpha$ such that
  \begin{equation*}
    \begin{split}
      \exp\left(\frac{n(\prex-\prexopt)^T H (\prex-\prexopt)}{2}
      \right)
      &=
      \exp\left(\frac{n (\prex-\prexopt)^T \left((-1/r^2) H_0 - (1/r) T + J\right) (\prex-\prexopt)}{2}
      \right)
      \\
      &\leq
      \exp\left(-\frac{\alpha\gamma_1^2 n}{r}
        -\sum_{i=2}^{4}\frac{\alpha\gamma_i^2 n}{r^2}\right)
      \\
      &\leq
      \max\paren[\bigg]{\exp\left(-\frac{\alpha \delta_1^2n}{r}\right),
        \exp\left(-\frac{\alpha \delta^2 n}{r^2}\right)}
      \\
      &= \frac{1}{n^{\omega(1)}},
    \end{split}
  \end{equation*}
  where the last relation follows from $\delta_1^2n/r =\omega(\ln
  n)$ and $\delta^2n/r^2 =\omega(\ln n)$.
    
  By Lemma~\ref{lem:max-pre-hyper}, for any local maximum $x$ in $S$
  other than $x^*$,
  \begin{equation*}
    \frac{\exp(n\fpre(x))}{\exp(n\fpre(x^*))}
    = \frac{1}{\exp(\Omega(r^2n))}
    = \frac{1}{\exp(\Omega(R^2/n))}
    = \frac{1}{\exp(\Omega(\ln^{3/2} n))},
  \end{equation*}
  since $R=\omega(n^{1/2}\ln^{3/2} (n))$. Hence, for any $x\in S\setminus (\xopt+B)$,
    \begin{equation*}
    \frac{\exp(n\fpre(x))}{\exp(n\fpre(x^*))}= \frac{1}{\exp(\Omega(\ln^{3/2} n))}.
  \end{equation*}
  Thus,
  \begin{equation*}
    \begin{split}
      \sum_{x\in (S\setminus(\xopt+B))\cap \setZ^4}
      \wpre(x)
      &\leq 
      Q(n)n!\sum_{x\in (S\setminus(\xopt+B))\cap \setZ^4}
      \exp(n\fpre(\prex))
      \leq
      Q(n)n! n^4 \frac{\exp(n\fpre(\prexopt))}{\exp(\Omega(\ln^{3/2} n))}
      \\
      &=
      o\left(n! {\exp(n\fpre(\prexopt))}\right).
    \end{split}
  \end{equation*}
\end{proof}

%%%%%%%%%%%%%%%%%%%%%%%%%%%%%%%%%%%%%%%%%%%%%%%%%%%%%%%%%%%%%%%%
% Combining
%%%%%%%%%%%%%%%%%%%%%%%%%%%%%%%%%%%%%%%%%%%%%%%%%%%%%%%%%%%%%%%%
\section{Combining pre-kernels and forests}
\label{sec:core-cacti-hyper}

In this section, we will obtain a formula for the number of connected
graphs with vertex set $[n]$ and $m$ edges, proving
Theorem~\ref{thm:main-hyper}. We defer the proof of some lemmas to
Section~\ref{sec:lemmas-core-cacti-hyper}. \Old{We will perform Step~4
  as described in the overview of the proof in
  Section~\ref{sec:overview-hyper}: we will analyse the summation
  $\sum_n \gcacti(N,n)\gcore(n,M-(N-n)/2)$ by combining the formula
  obtained for forests (Section~\ref{sec:cacti-hyper}) and for cores
  (Section~\ref{sec:core-hyper}). We relate the formulae for cores and
  \prekernels\ (Section~\ref{sec:pre-hyper}) so that we can deduce the
  asymptotic value of $\sum_n \gcacti(N,n)\gpre(n,M-(N-n)/2)$, which
  is the number of connected graphs.}\New{We will perform Steps~4
  and~5 as described in the overview of the proof in
  Section~\ref{sec:overview-hyper}.}

For $\nN\in [0,1]$, let
\begin{equation}
\label{eq:t-def-hyper}
  t(\nN)
  =
  -\frac{(1-\nN)}{2} \ln(1-\nN)
      +\frac{1-\nN}{2}
      +\nN\fcore(\nncopt),
\end{equation}
where $\nncopt=\nncopt(n) = {3\mpre}/{\fgg(\lambdaopt)}$ and
$\lambdaopt=\lambdaopt(n)$ is the unique positive solution of the equation
${\lambda \f(\lambda)\fgg(\lambda)}/{\FF(2\lambda)} = 3m/n$. We
have already discussed the existence and uniqueness of $\lambdaopt$
in Section~\ref{sec:core-hyper}.

Elementary but lengthy computations show that
\begin{equation}
  \label{eq:t-simple-hyper}
  \begin{split}
    t(\nN) = &-\frac{(1-\nN)}{2} \ln(1-\nN) +\frac{1-\nN}{2}\\
      &2\RN\ln(N)+(2\ln(3)
      - \ln(2)-2)\RN
      + 2\RN\ln(\nN)
      + \left(\ln 3- \frac{1}{2}\ln(2)\right)\nN
      \\
      &+ \ln\left(\frac{\FF(2\lambdaopt)}{g_1(\lambdaopt)}\right)\nN
      +\left(\frac{1}{2}\nN+\RN\right)\ln\left(\frac{\mcore^2
          g_1(\lambdaopt)^3}{g_2^2(\lambdaopt)f_1(\lambdaopt)(\lambdaopt)^3}
      \right),  
  \end{split}
\end{equation}
where $\RN = R/N$. \nickca{Deleted: See Section for a Maple spreadsheet.}
In this section, we use $\check y$ to denote $y/N$. We obtain the
following asymptotic formulae.
\begin{thm}
  \label{thm:formula-core-cacti-hyper}
  We have that\crisc{I deleted the formula for cores$+$forests from here}
\Old{ \begin{equation}
    \label{eq:formula-core-cacti}
    \sum_{\substack{n\in[N]\\ N-n\text{ even }}}
    \binom{N}{n}\gcacti(N,n)\gcore(n,m)
    \sim
    \frac{\sqrt{3}}
    {\sqrt{\pi \lambdaoptopt N} }\exp(N t(\nNopt)+ N\ln N-N)
  \end{equation}
  and}
  \begin{equation}
    \label{eq:formula-pre-core-cacti}
    \sum_{\substack{n\in[N]\\ N-n\text{ even }}}
    \binom{N}{n}\gcacti(N,n)\gpre(n,m)
    \sim
    \frac{\sqrt{3}}
    {\sqrt{\pi N} }\exp(N t(\nNopt)+ N\ln N-N)
  \end{equation}
  where   \begin{equation}
    \label{eq:n-sol-core-cacti}
    \nNopt =
    \frac{\FF(2\lambdaoptopt)}{f_1(\lambdaoptopt) g_1(\lambdaoptopt)}.
  \end{equation} 
  and $\lambdaoptopt$ is the unique positive solution to 
  \begin{equation}
    \label{eq:R-sol-core-cacti}
    \frac{2\lambda f_1(\lambda) g_2(\lambda)-3 \FF(2\lambda)
      }{f_1(\lambda)
      g_1(\lambda)}=\frac{6R}{N}.
  \end{equation}
\end{thm}

Theorem~\ref{thm:main-hyper} follows immediately from
Theorem~\ref{thm:formula-core-cacti-hyper} by simplifying $t(\nNopt)$
by using~\eqref{eq:t-simple-hyper} with~\eqref{eq:n-sol-core-cacti}
and~\eqref{eq:m-sol-core}.  The rest of this section is dedicated to
prove Theorem~\ref{thm:formula-core-cacti-hyper}.

The following lemma shows that $\lambdaoptopt$ is well-defined. 
\begin{lem}
 \label{lem:R-sol-increasing}
  The equation
  \begin{equation*}
    \frac{2\lambda f_1(\lambda)
      g_2(\lambda)-3\FF(2\lambda)}{f_1(\lambda) g_1(\lambda)}=
    \alpha_n
  \end{equation*}
  has a unique solution for $\alpha_n > 0$ and it goes to $0$ if
  $\alpha_n\to 0$. 
\end{lem}
\begin{proof}
  For the first part, it suffices to show that the function 
  \begin{equation*}
    f(\lambda) = \frac{2\lambda f_1(\lambda)
      g_2(\lambda)-3\FF(2\lambda)}{f_1(\lambda) g_1(\lambda)}
  \end{equation*} is
  strictly increasing and it goes to zero as $\lambda\to 0$. By
  computing the series of $f(\lambda)$ with $\lambda\to 0$, we obtain
  $f(\lambda) = \lambda^2/2 +O(\lambda^3) \to 0$ as $\lambda\to 0$. To
  show $f(\lambda)$ is strictly increasing, we compute its
  derivative:
  \begin{equation*}
    \frac{\dif f(\lambda)}{\dif \lambda}
    =
    \frac{2 (e^{4\lambda}+ e^{3\lambda}-e^{\lambda}-1-\lambda e^{3\lambda}-4\lambda e^{2\lambda}-\lambda e^{\lambda})}
    {\f(\lambda)^2\fg(\lambda)^2}
  \end{equation*}
  while it is obvious that the denominator is positive for $\lambda
  >0$, it is not immediate that so is the numerator.

  Let $g(\lambda)= e^{4\lambda}+ e^{3\lambda}-e^{\lambda}-1-\lambda
  e^{3\lambda}-4\lambda e^{2\lambda}-\lambda e^{\lambda} $. We will
  use the following strategy: starting with $i=1$, we check that
  $\frac{\dif^{i-1} g(\lambda)}{\dif^{i-1} \lambda}|_{\lambda=0}=0$
  and compute $\frac{\dif^i g(\lambda)}{\dif^i \lambda}$. If for some
  $i$ we can show that $\frac{\dif^i g(\lambda)}{\dif^i \lambda} > 0$
  for any $\lambda$, then we obtain $g(\lambda) > 0$ for $\lambda>0$.
  We omit the computations here. \nickca{Deleted: See
  Section in the Appendix for a maple
  spreadsheet.} We have that
  \begin{equation*}
    \frac{\dif^5 g(x)}{\dif^5 x}
    =
    2048e^{4\lambda}-12e^{\lambda}-324e^{3\lambda}-486\lambda
    e^{3\lambda}
    -640e^{2\lambda}-256\lambda e^{2\lambda}-2\lambda e^{x},
  \end{equation*}
  which is trivially positive since $\exp(x) > 1+x$ for all $x\in\setR$
  and the sum of the coefficients of the negative terms is less than $2048$.
  \begin{comment}
  For the second statement of the lemma, note that $\lambdaoptopt\to
  0$ since $\RN\to 0$. Moreover, the first derivative goes to $0$ with
  $\lambda\to 0$ and it is a continuous function. By choosing a
  constant $\eps > 0$, we can bound the first derivative in $(0,\eps)$
  by some positive constant $\eps'$. This implies that for any
  $\nN\in(0,\eps)$ we have $|\lambdaopt(n)-\lambdaoptopt| \leq \alpha
  |\nN-\nNopt|$.
  \end{comment}
\end{proof}

It will be useful to know how $\lambdaoptopt$ compares to $R$ and
$\nNopt$. By Lemma~\ref{lem:R-sol-increasing}, $\lambdaoptopt=o(1)$
since $R = o(N)$. We can write $\RN$ and $\nNopt$ in terms of
$\lambdaoptopt$ by using~\eqref{eq:R-sol-core-cacti}
and~\eqref{eq:n-sol-core-cacti}. By expanding the LHS
of~\eqref{eq:R-sol-core-cacti} and the RHS~\eqref{eq:n-sol-core-cacti}
as functions of $\lambdaoptopt$ about $0$, we have that
\begin{equation}
  \label{eq:opt-sol-lambda-core-cacti-hyper}
  \begin{split}
    &\RN
    = \frac{(\lambdaoptopt)^2}{12} + O((\lambdaoptopt)^4),\\
    &\nNopt =
    \lambdaoptopt -\frac{(\lambdaoptopt)^2}{3} + O((\lambdaoptopt)^4).
  \end{split}
\end{equation}

Next, we state the main lemmas for the proof of
Theorem~\ref{thm:formula-core-cacti-hyper}. We defer their proofs to
Section~\ref{sec:lemmas-core-cacti-hyper}.

\Old{First we show the relation
between $\gpre(n,m)$ and $\gcore(n,m)$ for a certain range of~$n$.} The
next lemma follows from \Old{Theorem~\ref{thm:corevalue-hyper} and}
Theorem~\ref{thm:formula-pre-hyper} and a series of simplifications
that show that $\fcore\paren{\nnopt} = \fpre(\xopt)$. \crisc{Removed: For the
simplifications see~ Section.} 
\begin{lem}
  \label{lem:pre-core-rel}
  Let $\alpha_1<\alpha_2$ be positive constants. If
  $\alpha_1\sqrt{RN}\leq n\leq \alpha_2\sqrt{RN}$, then
 \Old{\begin{equation}
    \label{eq:ratio-applies-hyper}
    \frac{\gpre(n,m)}{\gcore(n,m)}
    \sim
    2\sqrt{3r}.
  \end{equation}
  Moreover, for all $n\geq 0$,
  \begin{equation}
    \gpre(n,m)\leq \gcore(n,m)
  \end{equation}}
\New{   \begin{equation}
      \label{eq:gpre-simpler-overview}
      \gpre(n,m)\sim
      \frac{\sqrt{3}}{\pi n}\cdot
    n!\exp(n\fcore(\nncopt)),
    \end{equation}
where $\nncopt=\nncopt(n) = {3\mpre}/{\fgg(\lambdaopt)}$ and
$\lambdaopt=\lambdaopt(n)$ is the unique positive solution of the equation
${\lambda \f(\lambda)\fgg(\lambda)}/{\FF(2\lambda)} = 3m/n$.}
\end{lem}
This will allow us to obtain the formula for connected hypergraphs
from the formula for simple hypergraphs. We compute the point of
maximum for $t(\nN)$:
\begin{lem}
  \label{lem:max-core-cacti-hyper}
  The point $\nNopt$ is the unique maximum of the function $t(\nN)$ in
  the interval $[0,1]$. Moreover, $\nNopt$ is the unique point such
  that the derivative of $t(\nN)$ is $0$ in $(0,1)$, and
  $t'(\nN) > 0$ for $\nN <\nNopt$ and $t'(\nN) < 0$ for $\nN <\nNopt$.
\end{lem}

We then expand the summation around this maximum and approximate it by
an integral that can be easily computed.
\begin{lem}
  \label{lem:approx-core-cacti-hyper}
  Suppose $\delta^3=o(\lambdaoptopt/N)$ and $\delta=\omega(1/N^{1/2})$. Then
  \begin{equation*}
    \sum_{\substack{n\in[\nopt-\delta N,\nopt+\delta N]\\ N-n\textrm{ even}}}
    \exp(N t(\nN))
    \sim
    \sqrt{\frac{\pi N}{2}}\exp(N t(\nNopt)).
  \end{equation*}
\end{lem}

Finally, we show that the terms far from the maximum do not
contribute significantly to the summation:
\begin{lem}
  \label{lem:tail-core-cacti-hyper}
  Suppose that $\delta^3=o(\lambdaoptopt/N)$ and $\delta^2 = \omega((\ln N)/N)$. Then
  \begin{equation*}
    \sum_{n\in [0,N]\setminus[\nopt-\delta N,\nopt+\delta N]}
    \binom{N}{n}\gcacti(N,n)\Old{\gcore(n,m)}\New{\gpre(n,m)}
    =
    \frac{N!\exp(Nt(\nNopt))}{N^{\omega(1)}}.
  \end{equation*}
\end{lem}

We are now ready to prove Theorem~\ref{thm:formula-core-cacti-hyper}.
\begin{proof}[Proof of Theorem~\ref{thm:formula-core-cacti-hyper}]
In order to use
Lemma~\ref{lem:approx-core-cacti-hyper} and
Lemma~\ref{lem:tail-core-cacti-hyper}, we need to check if there
exists $\delta$ such that $\delta^3=o(\lambdaoptopt/N)$ and
$\delta^2 =\omega(\ln N/N)$. This is true if and only if
\begin{equation*}
  (\lambdaoptopt)^2
  =\omega
  \paren[\bigg]{
    \frac{\log^3 N}
    {N}
  },
\end{equation*}
which, by~\eqref{eq:opt-sol-lambda-core-cacti-hyper}, is true if and
only if
\begin{equation*}
R=\omega(\log^3 N),
\end{equation*}
which is true by assumption. Thus, assume that $\delta$ satisfies
$\delta^3=o(\lambdaoptopt/N)$ and $\delta^2 =\omega(\ln N/N)$.

Let $J(\delta) = [\nopt-\delta N,\nopt+ \delta N]\cap (2\setZ -N)$.
By~\eqref{eq:opt-sol-lambda-core-cacti-hyper}, we have that
\begin{equation*}
  \nopt = \Theta( \lambdaoptopt N)
  = \Theta(\sqrt{RN}).
\end{equation*}
Moreover, since $\delta^3 = o(\lambdaoptopt/N)$ and $R\to\infty$,
\begin{equation*}
  \delta N = o(\sqrt[6]{RN^3}) = 
  o\paren[\bigg]{\frac{\sqrt{RN}}{R^{1/3}}}
  =
  o(\nopt).
\end{equation*}
Thus, there are constants $\alpha_1>0$ and $\alpha_2>0$ such that any
$n\in J(\delta)$ satisfies $\alpha_1\sqrt{RN} < n <
\alpha_2\sqrt{RN}$. By Lemma~\ref{lem:pre-core-rel}
\Old{\begin{equation*}
  \frac{\gpre(n,m)}{\gcore(n,m)}
  \sim
  2\sqrt{3r},
\end{equation*}}
\New{
  \begin{equation*}
          \gpre(n,m)\sim
      \frac{\sqrt{3}}{\pi n}\cdot
    n!\exp(n\fcore(\nncopt))
  \end{equation*}
}
for any $n\in J(\delta)$ and $m = n/2+R$. \Old{Since $R = o(n)$ and
$R=\omega(\log n)$, by~Theorem~\ref{thm:corevalue-hyper} and by
Theorem~\ref{thm:cacti-formula}, for $n\in J(\delta)$,}  \New{By
Theorem~\ref{thm:cacti-formula}, for $n\in J(\delta)$}
\begin{equation*}
\Old{\gcore(n,m) 
  \sim    
  \frac{1}{2\pi n \sqrt{r}}
  \cdot n!
  \exp\paren[\Big]{
    n\fcore(\nncopt)}
  \quad
  \text{and}
  \quad}
  \gcacti(n,N)
  =
  {\displaystyle \frac{n}{N}\cdot\frac{(N-n)! N^{(N-n)/2}}
    {\left(\frac{N-n}{2}\right)! 2^{(N-n)/2}}}.
\end{equation*}
Thus, for $n\in J(\delta)$, with $m=n/2+R$, by Stirling's
approximation and using the fact that $n = o(N)$ by~\eqref{eq:opt-sol-lambda-core-cacti-hyper},
\begin{equation}
\label{eq:aux-core-cacti-hyper}
  \begin{split}
   \binom{N}{n}\gcacti(N,n)\Old{\gcore(n,m)}\New{\gpre(n,m)}
    &\sim
    \binom{N}{n}\cdot
    {\displaystyle \frac{n}{N}\cdot\frac{(N-n)! N^{\paren{N-n}/{2}}}
      {\left(\frac{N-n}{2}\right)! 2^{\paren{N-n}/{2}}}}
    \cdot
    \Old{\frac{1}{2\pi n \sqrt{r}}}
    \New{\frac{\sqrt{3}}{\pi n}}
    \cdot n!
    \exp\paren[\Big]{
      n\fcore(\nncopt)}
    \\
    &=
    \Old{\frac{\sqrt{n}}{2\pi N \sqrt{R}}}
    \New{\frac{\sqrt{3}}{\pi N}}
    \cdot
    \frac{N! N^{\paren{N-n}/{2}}}{\left(\frac{N-n}{2}\right)!
      2^{\paren{N-n}/{2}}}
    \exp\paren[\Big]{
      n\fcore(\nncopt)}
    \\
    &\sim
    \Old{\frac{\sqrt{2n N}}{2\pi N \sqrt{R(N-n)}}}
    \New{\frac{\sqrt{6}}{\pi N}}
    \cdot
    \exp\paren[\bigg]{
     N t(\nN)+N\ln N-N
    },\text{ by~\eqref{eq:t-def-hyper}}\Old{\\
    &\sim
    \frac{\sqrt{n}}{\pi N \sqrt{2R}}
    \cdot
    \exp\paren[\bigg]{
     N t(\nN)+N\ln N-N
    }}.
  \end{split}
\end{equation}
\Old{By~\eqref{eq:opt-sol-lambda-core-cacti-hyper},
\begin{equation*}
  \frac{\sqrt{n}}{\pi N \sqrt{2R}}
  \sim
  \frac{\sqrt{6}}{\pi N \sqrt{\lambdaoptopt}}
\end{equation*}
and so
\begin{equation}
  \label{eq:aux-core-cacti-hyper}
  \binom{N}{n}\gcacti(N,n)\Old{\gcore(n,m)}\New{\gpre(n,m)}
  \sim
  \frac{\sqrt{6}}{\pi N \sqrt{\lambdaoptopt}}
\end{equation}
for $n\in J(\delta)$.} Since $J(\delta)$ is a finite set for each
$n$,\crisc{removed a reference to uniformity lemma in appendix} we have that there
exists a function $q(n)=o(1)$ such that the $o(1)$
in~\eqref{eq:aux-core-cacti-hyper} is bounded by $q(n)$ for any $n\in
J(\delta)$. Thus,
\begin{equation}
  \label{eq:core-cacti-aux1}
  \begin{split}
    \sum_{n\in J(\delta)}\binom{N}{n}\gcacti(N,n)\gcore(n,m)
    &\sim
    \sum_{n\in J(\delta)}
   \Old{ \frac{\sqrt{6}}{\pi N \sqrt{\lambdaoptopt}}}
   \New{\frac{\sqrt{6}}{\pi N}}
  \cdot
  \exp\paren[\bigg]{
     N t(\nN)+N\ln N-N
   }
   \\
   &\sim
   \Old{\frac{\sqrt{6}}{\pi N \sqrt{\lambdaoptopt}}}
   \New{\frac{\sqrt{6}}{\pi N}}
   \sqrt{\frac{\pi N}{2}}\exp(N t(\nNopt) + N\ln N -N)
  \end{split}
\end{equation}
by Lemma~\ref{lem:approx-core-cacti-hyper}. Together with
Lemma~\ref{lem:tail-core-cacti-hyper}, this proves
\Old{Equation~\eqref{eq:formula-core-cacti} of}
Theorem~\ref{thm:formula-core-cacti-hyper}.
  
\Old{Equations~\eqref{eq:core-cacti-aux1}
and~\eqref{eq:ratio-applies-hyper} implies that
\begin{equation}
    \begin{split}
    \sum_{n\in J(\delta)}\binom{N}{n}\gcacti(N,n)\gpre(n,m)
      &\sim
   2\sqrt{\frac{3R}{n}}\cdot\frac{\sqrt{6}}{\pi N \sqrt{\lambdaoptopt}}
   \sqrt{\frac{\pi N}{2}}\exp(N t(\nNopt) + N\ln N -N)
   \\
      &\sim
   2\sqrt{\frac{3}{12}}\cdot\frac{\sqrt{3}}{\sqrt{\pi N}}
 \exp(N t(\nNopt) + N\ln N -N)
 \\
 &\sim
 \frac{\sqrt{3}}{\sqrt{\pi N}}
 \exp(N t(\nNopt) + N\ln N -N),
  \end{split}
\end{equation}
which, together with Lemma~\ref{lem:tail-core-cacti-hyper} and the
fact that $\gpre(n,m)\leq \gcore(n,m)$, proves
Equation~\eqref{eq:formula-pre-core-cacti} of
Theorem~\ref{thm:formula-core-cacti-hyper}.
}\end{proof}

\subsection{Proof of the lemmas in Section~\ref{sec:core-cacti-hyper}}
\label{sec:lemmas-core-cacti-hyper}

In this section, we prove
Lemmas~\ref{lem:max-core-cacti-hyper},~\ref{lem:approx-core-cacti-hyper}
and~\ref{lem:tail-core-cacti-hyper}. \nickca{Deleted:  See Section
for a Maple spreadsheet.} We start by computing the derivatives of
$t$. For that, we need to compute $\frac{\dif \lambdaopt(\nN)}{\dif
  \nN}$. This can be done by implicit differentiation using
Equation~\eqref{eq:m-sol-core} that defines $\lambdaopt$ and recalling
$m=n/2+R$. We obtain
\begin{equation}
  \label{eq:dlambda-core-cacti}
  \frac{\dif\lambdaopt}{\dif \nN}
  =
  -\frac{\RN}{\nN^2 \mcore a(\lambdaopt)},
\end{equation}
where
\begin{equation}
  \label{eq:a-def-hyper}
  a(\lambda) = \frac{1}{\lambda}
  +\frac{\exp(\lambda)}{f_1(\lambda)}
  +\frac{\exp(\lambda)}{g_2(\lambda)}
  -\frac{2\exp(2\lambda)}{\FF(2\lambda)}
  +\frac{2}{\FF(2\lambda)}.
\end{equation}
Thus, the first derivative of $t(\nN)$, which is defined in~\eqref{eq:t-simple-hyper}, is
\begin{equation}
  \label{first-derivative-t}
  \frac{\ln(1-\nN)}{2}+\ln(3)-\frac{\ln 2}{2}
  +\ln\left(\frac{\FF(2\lambdaopt)}{g_1(\lambdaopt)}\right)
  +\frac{1}{2}\ln\left(\frac{\mcore^2g_1(\lambdaopt)^3}
    {g_2^2(\lambdaopt)f_1(\lambdaopt)(\lambdaopt)^3}\right).
\end{equation}
The second derivative is 
\begin{equation}
  \label{second-derivative-t}
  -\frac{1}{2(1-\nN)}-\frac{2\RN}{(\nN+2\RN)\nN}
  -\frac{4\RN^2}{\nN(\nN+2\RN)^2}
  \frac{b(\lambdaopt)}{a(\lambdaopt)\FF(2\lambdaopt)}
\end{equation}
where
\begin{equation*}
  \begin{split}
    b(\lambda) &=
    2F_1(\lambda)
    -\frac{\FF(2\lambda)\exp(\lambda)}{g_1(\lambda)}.
  \end{split}
\end{equation*}
The third derivative is 
\begin{equation}
  \label{third-derivative-t}
  \begin{split}
  &-\frac{1}{2(1-\nN)^2}
  +\frac{4\RN(\nN+\RN)}{\nN^2(\nN+2\RN)^2}
  \\
  &+\frac{\dif}{\dif\nN}\left(-\frac{4\RN^2}{\nN(\nN+2\RN)^2}\right)
  \frac{b(\lambdaopt)}{a(\lambdaopt)\FF(2\lambdaopt)}
  -\frac{4\RN^2}{\nN(\nN+2\RN)^2}
  \frac{\dif }
  {\dif\lambdaopt}\left(\frac{b(\lambdaopt)}{a(\lambdaopt)\FF(2\lambdaopt)}\right)
  \frac{\dif\lambdaopt}{\dif \nN}.
  \end{split}
\end{equation}

\begin{lem}
   For $\delta = o(\nNopt)$ and   $n\in
  [\nopt-\delta N, \nopt+\delta N]$, we have that
  $|\lambdaopt(n) -\lambdaoptopt| = o(\nNopt)$.
\end{lem}
\begin{proof}
  Given a connected $(N,M)$-graph such that its core has $n$ vertices
  and $m$ edges, we have that $m = M-(N-n)/2$. Recall that $\nNopt=
  \FF(2\lambdaoptopt)/f_1(\lambdaoptopt) g_1(\lambdaoptopt)$
  by~\eqref{eq:n-sol-core-cacti} and 
  \begin{equation*}
    \frac{6R}{N}= \frac{2\lambdaoptopt
    f_1(\lambdaoptopt) g_2(\lambdaoptopt)-3 \FF(2\lambdaoptopt)}{f_1(\lambdaoptopt)
    g_1(\lambdaoptopt)},
  \end{equation*}
  by~\eqref{eq:R-sol-core-cacti}. Thus,
  \begin{equation*}
    3M =\frac{\lambdaoptopt(1+\exp(2\lambdaoptopt)+\exp(\lambdaoptopt))}
      {\exp(2\lambdaoptopt)-1}.
  \end{equation*}
  Hence,
  \begin{equation*}
    \frac{\lambdaopt(\nopt) \f(\lambdaopt(\nopt)) 
      g_2(\lambdaopt(\nopt))}{\FF(2\lambdaopt(\nopt))}
    =\frac{3m}{\nopt}=
    \frac{3M}{N}\frac{1}{\nNopt}-\frac{3}{2\nNopt}+\frac{3}{2}
    =
    \frac{\lambdaoptopt) \f(\lambdaoptopt) 
      g_2(\lambdaoptopt)}{\FF(2\lambdaoptopt)}
  \end{equation*}
  and so $\lambdaopt = \lambdaopt(\nopt)$.  The lemma then follows
  directly from the fact that $\lambdaoptopt = \lambdaopt(\nNopt)$ and
  Lemma~\ref{lem:unique-pre-tools-hyper}.
\end{proof}

 Now we bound the
third derivative for points close to $\nNopt$:
\begin{lem}
  \label{lem:third-derivative-core-cacti} The third derivative of $t(\nN)$ is
  $O(1/\lambdaoptopt)$ for
  $|\nN-\nNopt|=o(\nNopt)$.
\end{lem}
\begin{proof}
  We analyse the terms in~\eqref{third-derivative-t}. By
  Lemma~\ref{lem:unique-pre-tools-hyper}, since
  $n=\nN(1+o(1))$,
  \begin{equation*}
    \begin{split}
      \frac{\dif}{\dif\nN}\left(-\frac{1}{2(1-\nN)^2}
        +\frac{4\RN(\nN+\RN)}{\nN^2(\nN+2\RN)^2}\right)
      &=
      -\frac{1}{2(1-\nN)^2}
      +\frac{4\RN(\nN+\RN))}{\nN^2 (\nN +2\RN)^2}
      \\
      &=
      \left(-\frac{1}{2(1-\nNopt)^2}
        +\frac{4\RN(\nNopt+\RN)}{(\nNopt)^2 (\nNopt +2\RN)^2}\right)
      (1+o(\nNopt))
      \\
      &=
      \frac{1}{3\lambdaoptopt}+O(1),
    \end{split}
  \end{equation*} 
  where the last equality is obtained by computing the series of the
  expression in the previous equation
  using~\eqref{eq:opt-sol-lambda-core-cacti-hyper}.
  For $\lambda\to 0$,
  \begin{align}
    a(\lambda) &= \frac{1}{6}+\frac{\lambda}{12}+O(\lambda^2)\\ 
    b(\lambda) &= 4\lambda+O(\lambda^2);
  \end{align}
  Thus,  by Lemma~\ref{lem:unique-pre-tools-hyper} and~\eqref{eq:opt-sol-lambda-core-cacti-hyper},
  \begin{equation*}
    \begin{split}
      \frac{\dif}{\dif\nN}\left(-\frac{4\RN^2}{\nN(\nN+2\RN)^2}\right)
      \frac{b(\lambdaopt)}{a(\lambdaopt)\FF(2\lambdaopt)}
      &=
      \frac{4\RN^2 (3\nN -2\RN)}{\nN^2(\nN+2\RN)^3}
      \frac{b(\lambdaopt)}{a(\lambdaopt)\FF(2\lambdaopt)}
      \\&\sim
      \frac{4\RN^2 (3\nNopt -2\RN)}{(\nNopt)^2(\nNopt+2\RN)^3}
      \frac{b(\lambdaoptopt)}{a(\lambdaoptopt)\FF(2\lambdaoptopt)}
      =
      \frac{1}{\lambdaoptopt}+O(1).
    \end{split}
  \end{equation*}
We have that
 \begin{align*}
   \frac{\dif b(\lambda)}{\dif \lambda} = 4\exp(2\lambda)
   -\frac{\exp(\lambda) (3\exp(2\lambda)-3-2\lambda)}{g_1(\lambda)}
   +\frac{\FF(2\lambda)\exp(2\lambda)}{g_1(\lambda)^2}
 \end{align*}
and
\begin{align*}
  \frac{\dif a(\lambda)}{\dif \lambda} =
  -\frac{1}{\lambda^2}
  +\frac{\exp(\lambda)}{f_1(\lambda)}
  -\frac{\exp(2\lambda)}{f_1(\lambda)^2}
  +\frac{\exp(\lambda)}{g_2(\lambda)}
  -\frac{\exp(2\lambda)}{g_2(\lambda)^2}
  -\frac{4\exp(2\lambda)}{\FF(2\lambda)}
  -\frac{4 F_1(\lambda)^2}{\FF(2\lambda)^2}.
\end{align*}
Thus, by~\eqref{eq:dlambda-core-cacti}
\begin{equation*}
  \begin{split}
    &-\frac{4\RN^2}{\nN(\nN+2\RN)^2}
    \frac{\dif
      }    {\dif\lambdaopt}\left(\frac{b(\lambdaopt)}{a(\lambdaopt)\FF(2\lambdaopt)}\right)
    \frac{\dif\lambdaopt}{\dif \nN}
    \\
    &=
    -\frac{4\RN^2}{\nN(\nN+2\RN)^2}
    \Bigg(\frac{\dif b(\lambda)}{\dif
        \lambda}{\Big |}_{\lambda=\lambdaopt}\frac{1}{a(\lambdaopt)\FF(2\lambdaopt)}
      \\
      &\qquad\qquad\qquad\qquad
      -\frac{b(\lambdaopt)}{a(\lambdaopt)^2 \FF(2\lambdaopt)^2}
      \left(\frac{\dif a(\lambda)}{\dif
          \lambda}{\Big |}_{\lambda=\lambdaopt} \FF(2\lambdaopt)
        +2F_1(\lambdaopt) a(\lambdaopt) \right)
    \Bigg)
    \left(-\frac{\RN}{\nN^2 \mcore a(\lambdaopt)}\right).
  \end{split}
  \end{equation*}
  By Lemma~\ref{lem:unique-pre-tools-hyper}, the above is the value
  applied at $\lambdaoptopt$ with an error of $o(\lambdaoptopt)$ and
  the series for it with $\lambdaoptopt\to 0$ is
  \begin{equation*}
    \frac{2}{3\lambdaoptopt}+O(1).
  \end{equation*}
\end{proof}

We now present the proofs of
Lemmas~\ref{lem:max-core-cacti-hyper},~\ref{lem:approx-core-cacti-hyper}
and~\ref{lem:tail-core-cacti-hyper}.
\begin{proof}[Proof of Lemma~\ref{lem:max-core-cacti-hyper}]
  By setting~\eqref{first-derivative-t} to zero and using $\mcore =
  \lambdaopt f_1(\lambdaopt) g_2(\lambdaopt)/\FF(2\lambdaopt)$, we get
  following value for $\nN$
  \begin{equation}
    \label{eq:max-core-cacti-aux1}
    \nNopt =
    \frac{\FF(2\lambdaopt)}{f_1(\lambdaopt) g_1(\lambdaopt)}.
  \end{equation}
  We also know that, by~\eqref{eq:m-sol-core},
  \begin{equation}
    \label{eq:max-core-cacti-aux2}
    \frac{\lambdaopt f_1(\lambdaopt)g_2(\lambdaopt)}{\FF(2\lambdaopt)} =  3\mcore= \frac{3}{2}+\frac{3\RN}{\nN}.
  \end{equation}
  Thus, by combining~\eqref{eq:max-core-cacti-aux1}
  and~\eqref{eq:max-core-cacti-aux2}, we get the following equation:
  \begin{equation}
    \label{RN-lambda-optimal}
    \RN =  \frac{1}{6}\frac{-3 \FF(2\lambdaopt)+2
    \lambdaopt f_1(\lambdaopt) g_2(\lambdaopt)}{f_1(\lambdaopt)
    g_1(\lambdaopt)},
\end{equation}
which has a unique solution $\lambdaoptopt$ for $\RN > 0$ by
Lemma~\ref{lem:R-sol-increasing}.
By computing the series of the second derivative as $\lambdaopt\to 0$,
we get that the second derivative at $\nNopt$ is 
\begin{equation*}
  -1+O(\lambdaopt),
\end{equation*}
which is negative for big enough $n$ and so $\nNopt$ is a local
maximum. 
\end{proof}

\begin{proof}[Proof of Lemma~\ref{lem:approx-core-cacti-hyper}]
  Let $J(\delta) = [\nopt-\delta N, \nopt+\delta N]\cap (2\setZ-N)$.
  Using Taylor's approximation, Lemma~\ref{lem:max-core-cacti-hyper} and
  Lemma~\ref{lem:third-derivative-core-cacti}, for $n\in J(\delta)$,
    \begin{equation*}
    \begin{split}
      \exp(N t(\nN))
    &=
    \exp\left(N t(\nNopt)+ \frac{Nt''(\nNopt)|\nN-\nNopt|^2}{2} +
      O \left( \frac{\delta^3 N}{\lambdaoptopt}\right)\right)
    \\
    &\sim
        \exp\left(N t(\nNopt)+ \frac{Nt''(\nNopt)|\nN-\nNopt|^2}{2}\right)
    \end{split}
  \end{equation*}
  since $\delta^3=o(\lambdaopt / N)$, and so
  \begin{equation}
    \label{eq:taylor-comb-hyper}
    \begin{split}
      \sum_{\nN \in J(\delta)} \exp(N t(\nN))
      &\sim
    \sum_{\nN \in J(\delta)}
    \exp\left(N t(\nNopt)+ \frac{Nt''(\nNopt)|\nN-\nNopt|^2}{2}\right)
    \\
    &=
    \exp\left(N t(\nNopt)\right)
    \sum_{\substack{x\in[-\delta N,\delta N]\\ (\nopt + x) \in
        (-N+2\setZ)}}
    \exp\left(\frac{t''(\nNopt)\prex}{2N}\right).
    \end{split}
  \end{equation}
We change variables from $x$ to $y = \ell x/2$ with $\ell =
\sqrt{|t''(\nNopt)|/2}\sim \frac{1}{2}$. Using $\delta =
\omega(1/\sqrt{N})$ and Lemma~\ref{lem:integral-tools},
\begin{equation*}
  \begin{split}
    \sum_{\substack{x\in[-\delta N,\delta N]\\ N-(\nNopt N + x)\in 2\setZ}}
    \exp\left(\frac{t''(\nNopt)x^2}{2N}\right) 
    &=
    \sum_{\substack{y\in[-\delta \ell\sqrt{N}/2,\delta
        \ell\sqrt{N}/2]\\y\in \setZ\cdot(\ell/\sqrt{N})}}
    \exp\left(-4y^2\right)
    \sim
    \sqrt{\frac{\pi N}{2}}.
  \end{split}
\end{equation*}
\end{proof}

\begin{proof}[Proof of Lemma~\ref{lem:tail-core-cacti-hyper}]
  \New{We have that $\gpre(n,m)\leq\gcore(n,m)$ since every pre-kernel
    is a core and $\gcore(n,m)$ is an upper bound for the number of
    cores with vertex-set $[n]$ and $m$ edges by
    Theorem~\ref{thm:corevalue-hyper}.} \Old{From
    Theorems~\ref{thm:corevalue-hyper} and} \New{Using
    Theorem~\ref{thm:cacti-formula} and} the definition of $t$, we
    have that there is a polynomial $Q(N)$ such that
  \begin{equation*}
    \begin{split}
      \New{ \sum_{n\not\in [(\nNopt-\delta)N,(\nNopt+\delta)N]}
    \binom{N}{n}\gcacti(N,n)\gpre(n,m)}
    &\leq
    \sum_{n\not\in [(\nNopt-\delta)N,(\nNopt+\delta)N]}
    \binom{N}{n}\gcacti(N,n)\gcore(n,m)
    \\
    &\leq Q(N) N!
    \sum_{{\substack{n\in[0,N]\\
      n\not\in[\nopt-\delta N,\nopt+\delta N]}}}
    \exp(N t(\nN))
  \end{split}
 \end{equation*}
 Using Lemma~\ref{lem:max-core-cacti-hyper}
 and~\eqref{eq:taylor-comb-hyper}, we have that
 \begin{equation*}
   \begin{split}
     \sum_{{\substack{n\in[0,N]\\
      n\not\in [\nopt-\delta N,\nopt+\delta N]}}}
    \exp(N t(\nN))
    &\leq
    N
    \exp(Nt(\nNopt)-\Omega(N\delta^2)),
  \end{split}
\end{equation*}
and $N\delta^2 = \omega(\ln N)$ for $\delta^2=\omega(\ln
  N)/{N})$.
\end{proof}

\bigskip

\nickkb{
\noindent
{\bf Acknowledgment\ }
\smallskip

The authors would like to thank Huseyin Acan for pointing out a mistake in~\cite{Sato13} which led to an incorrect claim on the number of cores, not necessarily connected.}
\newpage
\section*{Glossary}

\begin{enumerate} [itemsep=0pt,labelsep=0.5cm,leftmargin=2cm]
\item [$\cNM$] number of connected $3$-uniform hypergraphs on $[N]$
  with $M$ edges
\item [$N$] used for the number of vertices in the graph
\item [$M$] used for the number of edges in the graph
\item [$R$] $M-N/2$ used as an excess function in the graph
\item [$n$] used for the number of edges in the core
\item [$r$] $R/n$, scaled $R$
\item [$f_{k}(\lambda)$] $e^\lambda -\sum_{i=0}^{k-1} \lambda^i/i!$
\item [$g_k(\lambda)$] $e^{\lambda}+k$
\item [$\lambdakc$]  the unique positive solution to
  ${\lambda \fpo{k-1}(\lambda)}/{\fpo{k}(\lambda)}=c$
\item [$\gcore(n,m)$] number of (simple) cores with vertex set $[n]$
  and $m$ edges
\item [$\gcacti(N,n)$] number of forest with vertex set $[N]$
and $[n]$ as its roots
\item [$\gpre(n,m)$] number of (simple) pre-kernels with vertex set
  $[n]$ and $m$ edges that are connected
\item [$\lambdaoptopt$] unique positive solution to $\lambda
  {e^{2\lambda}+e^{\lambda}+1}/\paren {\fpo{1}(\lambda) g_1(\lambda)} =
  {3M}/{N}$. This is used to define a point achieving the maximum
  when combining cores and pre-kernels, p.~\pageref{glo:thm-main}.
\item [$\nNopt$] ${\FF(2\lambdaoptopt)}/\paren{\fpo{1}(\lambdaoptopt)
    g_1(\lambdaoptopt)}$.  This is the point achieving the maximum when
  combining cores and pre-kernels, p.~\pageref{glo:thm-main}.
\item[$2$-edge] an edge that contains exactly one vertex of degree~$1$
\item[$3$-edge] an edge that contains no  vertices of degree~$1$
\end{enumerate}

\noindent\textbf{For the core:}
\begin{enumerate} [itemsep=0pt,labelsep=0.5cm,leftmargin=2cm]
\item[] For any symbol $y$, $\hat y=y/n$ denotes the scaled version of $y$
\item [$\hcore(y)$] $y\ln(yn)-y$.
\item [$\fcore$] a function used to approximate the exponential part
  of $\wcore$, p.~\pageref{eq:fcore-hyper}
\item [$\wcore$] a function used to count cores, p.~\pageref{eq:wcore-def-hyper}
\item [$\nn$]  used as the number of vertices of degree~$1$ 
\item [$\Dcal_{\nn}$] set of all
  $\ds\in(\setN\setminus\set{0,1})^{n-\nn}$ with $\sum_{i} d_i =
  3m-\nn$
\item [$\lambda_{\nn}$] unique positive solution to ${\lambda
    \f(\lambda)}/{\ff(\lambda)} = \ctwo(\nn)$
\item [$\ntwo(\nn)$] $n -\nn$, the number of vertices of
  degree at least~$2$.
\item [$\mthree(\nn)$] $ m -\nn$, the number of $3$-edges
\item [$\Qtwo(\nn)$] $ 3m-\nn$, the sum of degrees of vertices of
  degree at least $2$ 
\item [$\ctwo(\nn)$] $\Qtwo(\nn)/\ntwo(\nn)$, the average degree of
  the vertices of degree at least $2$. 
\item [$\etatwo(\nn)$] $\lambda_{\nn} \exp(\lambda_{\nn}/\f(\lambda_{\nn})$
\item [$\Gcore(\nn,\ds)$ ] random core with $\nn$ vertices of
  degree $1$ and degree sequence $\ds$ for the vertices of degree at
  least $2$, p.~\pageref{glo:gcore}
\item [$\lambdaopt$] unique positive solution to ${\lambda
    \f(\lambda)\fgg(\lambda)}/{\FF(2\lambda)} = 3m/n$. This is used to
  define a point achieving the maximum for $\fcore$,
  p.~\pageref{eq:lambdaopt-def-aux-hyper}
\item [$\nnopt$] ${3m}/{n\fgg(\lambdaopt)}$. This the point
  achieving the maximum for $\fcore$, p.~\pageref{eq:gcore-ineq}
\item [$\Ys$] $(Y_1,\dotsc, Y_{\ntwo}$, where the $Y_i$'s are
  independent random variables with truncated Poisson distribution
  $\tpoisson{2}{\lambda_{\nncore}}$
\item [$\Sigma_{\nn}$] event that a random variable $\Ys$ satisfies
  $\sum_{i} Y_i = 3m-\nn$
\end{enumerate}

\noindent\textbf{For the \prekernel:}

\begin{enumerate} [itemsep=0pt,labelsep=0.5cm,leftmargin=2cm]
\item[] For any symbol $y$, $\hat y = y/n$ denotes the scaled version
  of $y$
\item [$\hpre(y)$]  $y\ln(yn)-y$.
\item [$\fpre$] a function used to approximate the exponential part of
  $\wpre$, p.~\pageref{eq:fpre-def-hyper}
\item [$\wpre$] a function used to count \prekernels, p.~\pageref{eq:wpre-def-hyper}
\item [$\nn$] used as the number of vertices of degree~$1$ 
\item [$\kzero$] used as the number of vertices of degree~$2$ that are
  in two $2$-edges
\item [$\kone$] used as the number of vertices of degree~$2$ that are
  in one $2$-edge and in one $3$-edge
\item [$\ktwo$] used as the number of vertices of degree~$2$ that are
  in two $3$-edges
\item[$x$] used as $(\nn,\kzero,\kone, \ktwo)$
\item[$\Dcal(x)$] subset of $\setN^{\nthree(x)}$ such that
  $\ds\in\Dcal(x)$ if $d_i\geq 3$ for all $i$ and
  $\sum_{i=1}^{\nthree(x)} d_i = \Qthree(x)$, p.~\pageref{Dcalx}
\item[$\ntwoequal(x)$] $\kzero +\kone +\ktwo$, the number of vertices
  of degree $2$ 
\item[$\nthree(x)$] $ n-\nn-\ntwoequal(x)$, the number of vertices
  of degree at least $3$
\item[$\mtwo(x)$] $ \nn$, the number of $2$-edges in the pre-kernel
\item[$\mtwoprime(x)$] $ \nn-\kzero$ , the number of $2$-edges in the
  kernel
\item[$\Ptwo(x)$]   $ 2\mtwoprime(x)$, the number of points in
  $2$-edges in the kernel
\item[$\mthree(x)$] $ m-\nn$, the number of $3$-edges in the
  pre-kernel
\item[$\Pthree(x)$] $ 3\mthree(x)$, the number of points in
  $3$-edges in the pre-kernel
\item[$\Qthree(x)$] $ 3m-\nn-2\ntwoequal(x)$, the sum of the degrees
  of the vertices of degree at least $3$
\item[$\cthree(x)$] $\Qthree(x)/\nthree(x)$, the average degree of the
  vertices of degree at least $3$
\item [$\Tthree(x)$] $\Pthree(x)-\kone(x)-2\ktwo(x)$, the number of
  points in $3$-edges that will be matched to points in vertices of
  degree at least~$3$
\item [$\Ttwo(x)$] $\Ptwo(x)-\kone(x)$, the number of points in
  $2$-edges that will be matched to points in vertices of degree at
  least~$3$
 \item[$\lambda(x)$] unique positive solution to ${\lambda
    \ff(\lambda)}/{\fff(\lambda)} = \cthree(x)$
\item[$\etathree(x)$] $\Qthree(x)/\nthree(x)$
\item[$\lambdaopt$] unique positive solution to ${\lambda
    \f(\lambda)\fgg(\lambda)}/{\FF(2\lambda)} = 3m/n$. This is used to
  define a point achieving the maximum for $\fpre$, p.~\pageref{eq:lambdaopt-def-aux-hyper}
\item[$\xopt$] This the point achieving the maximum for $\fpre$, p.~\pageref{eq:pre-param-opt}
\item[$\Kcal$] random kernel (it receives parameters
  $(V,M_3,\kone,\ktwo,\ds)$), p.~\pageref{glo:kernel}
\item[$\Gpre(x,\ds)$] random pre-kernel with parameters
  $x=(\nn,\kzero,\kone,\ktwo)$ and degree sequence $\ds$ for the
  vertices of degree at least $3$, p.~\pageref{glo:gpre}
\item[$\Ys$] $(Y_1,\dotsc,Y_{\nthree})$, where the $Y_i$'s are
  independent random variables with truncated Poisson distribution
  $\tpoisson{3}{\lambda(x)}$.
\item[$\Sigma(x)$] event that a random variable $\Ys$ satisfies
  $\sum_{i} Y_i = 3m-\nn-2\ntwo$
\item[$S_{\psi}^*$] a set of points `close' to $\xopt$, p.~\pageref{eq:Spsi}
\end{enumerate}

\bibliographystyle{plain}
\bibliography{hyper}

\end{document}